\documentclass{article}
\usepackage[T1]{fontenc}
\usepackage{amsthm, fullpage}
\usepackage{amsmath,enumerate}
\usepackage{amssymb}
\usepackage{amsfonts}
\usepackage{eucal}
\usepackage[all]{xy}
\usepackage[frenchb, english]{babel}
\usepackage{pslatex}
\usepackage{hyperref}
\providecommand{\scr}{\mathcal}
\usepackage{mathrsfs}

\newtheorem{prop}{Proposition}[subsection]
\newtheorem{theo}[prop]{Théor\`eme}
\newtheorem*{theo**}{Théorème}
\newtheorem{coro}[prop]{Corollaire}
\newtheorem*{conj*}{Conjecture}
\newtheorem{lemm}[prop]{Lemme}
\newtheorem{lemm*}{Lemme}[prop]

\theoremstyle{definition}

\newtheorem{vide}[prop]{}
\newtheorem{defi}[prop]{Définition}
\newtheorem*{defi*}{Définition}

\theoremstyle{remark}
\newtheorem{rema}[prop]{Remarques}

\newtheorem{nota}[prop]{Notations}

\numberwithin{equation}{prop}

\newcommand{\riso}{ \overset{\sim}{\longrightarrow}\, }
\newcommand{\liso}{ \overset{\sim}{\longleftarrow}\, }

\newcommand{\Spec}{\mathrm{Spec}\,}
\newcommand{\Spf}{\mathrm{Spf}\,}

\renewcommand{\AA}{{\mathcal{A}}}

\newcommand{\FF}{{\mathcal{F}}}
\newcommand{\B}{{\mathcal{B}}}
\newcommand{\CC}{{\mathcal{C}}}
\newcommand{\E}{{\mathcal{E}}}
\newcommand{\G}{{\mathcal{G}}}
\renewcommand{\H}{{\mathcal{H}}}
\newcommand{\M}{{\mathcal{M}}}
\newcommand{\NN}{{\mathcal{N}}}
\newcommand{\D}{{\mathcal{D}}}
\newcommand{\I}{{\mathcal{I}}}
\newcommand{\PP}{{\mathcal{P}}}
\newcommand{\QQ}{{\mathcal{Q}}}
\renewcommand{\O}{{\mathcal{O}}}

\newcommand{\V}{\mathcal{V}}

\newcommand{\T}{{\mathcal{T}}}
\newcommand{\Y}{\mathcal{Y}}
\newcommand{\ZZ}{\mathcal{Z}}
\newcommand{\X}{\mathfrak{X}}

\newcommand{\U}{\mathfrak{U}}

\newcommand{\A}{\mathbb{A}}

\renewcommand{\L}{\mathbb{L}}
\newcommand{\R}{\mathbb{R}}
\newcommand{\Q}{\mathbb{Q}}
\newcommand{\Z}{\mathbb{Z}}
\newcommand{\N}{\mathbb{N}}

\newcommand{\hdag}{  \phantom{}{^{\dag} }    }

\begin{document}

\title{Systèmes inductifs cohérents de $\mathcal{D}$-modules arithmétiques logarithmiques, stabilité par opérations cohomologiques}
\author{Daniel Caro} 

\date{}

\maketitle

\begin{abstract}
 
Let $\mathcal{V}$ be a complete discrete valuation ring of unequal characteristic with perfect residue field, 
$\mathcal{P}$ be a smooth, quasi-compact, separated formal scheme over $\mathcal{V}$, 
$\mathcal{Z}$ be a strict normal crossing divisor of $\mathcal{P}$ and 
$\PP ^{\sharp}:= (\PP,\ZZ)$ the induced smooth formal log-scheme over $\V$. 
In Berthelot's theory of arithmetic $\mathcal{D}$-modules,
we work with the inductive system of sheaves of rings 
$\smash{\widehat{\mathcal{D}}} _{\mathcal{P} ^{\sharp}} ^{(\bullet)} : =
(\smash{\widehat{\mathcal{D}}} _{\mathcal{P}^{\sharp}} ^{(m)})_{m\in \N}$,
where $\smash{\widehat{\mathcal{D}}} _{\mathcal{P}^{\sharp}} ^{(m)}$ is the $p$-adic completion of the ring
of differential operators of level $m$ over $\mathcal{P}^{\sharp}$.
Moreover, he introduced the sheaf 
$\mathcal{D} ^{\dag} _{\mathcal{P} ^{\sharp},\mathbb{Q}}:=\underset{\underset{m}{\longrightarrow}}{\lim}\,
\smash{\widehat{\mathcal{D}}} _{ \mathcal{P} ^{\sharp} } ^{(m)} 
\otimes _{\mathbb{Z} }\mathbb{Q}$ of differential operators over $\mathcal{P} ^{\sharp}$ of finite level.
In this paper, we define the notion of (over)coherence for complexes
of $\smash{\widehat{\mathcal{D}}} _{\mathcal{P} ^{\sharp}} ^{(\bullet)} $-modules.
In this inductive system context, we prove some classical properties including 
that of Berthelot-Kashiwara's theorem. 
Moreover, when $\ZZ$ is empty, we check this notion
is compatible to that already known of (over)coherence for complexes of 
$\mathcal{D} ^{\dag} _{\mathcal{P},\mathbb{Q}}$-modules. 
\end{abstract}

\selectlanguage{frenchb}
\tableofcontents
\date{}

\section*{Introduction}

Soit $\V$ un anneau de valuation discrète complet d'inégales caractéristiques $(0,p)$, 
de corps résiduel parfait.
Soient $\PP$ un $\V$-schéma formel séparé, quasi-compact, lisse
et $P$ sa fibre spéciale. 
Pour tout entier $m\in \N$, 
Berthelot a construit dans \cite{Be1} le faisceau d'anneaux 
$\smash{\D} _{\PP} ^{(m)}$
des opérateurs différentiels sur $\PP$ de niveau $m$.
En le complétant $p$-adiquement, on obtient le faisceau d'anneaux
$\smash{\widehat{\D}} _{\PP} ^{(m)}$. 
On dispose de plus des morphismes canoniques de changement de niveaux
$\smash{\widehat{\D}} _{\PP} ^{(m)}
\to 
\smash{\widehat{\D}} _{\PP} ^{(m+1)}$ (voir \cite{Be1}), ce qui donne le 
système inductif d'anneaux 
$\smash{\widehat{\D}} _{\PP} ^{(\bullet)} : =
(\smash{\widehat{\D}} _{\PP} ^{(m)})_{m\in \N}$. 
Berthelot construit le faisceau des opérateurs différentiels de niveau fini
en posant 
$\D ^{\dag} _{\PP,\Q}:=\underset{\underset{m}{\longrightarrow}}{\lim}\,
\smash{\widehat{\D}} _{ \PP } ^{(m)} \otimes _{\Z }\Q$.
Par tensorisation par $\Q$ et passage à la limite sur le niveau, on obtient le foncteur noté
$\underrightarrow{\lim}\colon 
D ^{\mathrm{b}} ( \smash{\widehat{\D}} _{\PP} ^{(\bullet)})
\to 
D ^{\mathrm{b}} (\D ^{\dag} _{\PP,\Q})$.
Afin d'obtenir un foncteur pleinement fidèle 
qui factorise ce foncteur $\underrightarrow{\lim}$, 
Berthelot a introduit la catégorie 
$\smash{\underrightarrow{LD}} ^{\mathrm{b}} _{\Q}
 ( \smash{\widehat{\D}} _{\PP} ^{(\bullet)})$ qui est une localisation 
 de 
$D ^{\mathrm{b}} ( \smash{\widehat{\D}} _{\PP} ^{(\bullet)})$.
Il a défini la sous-catégorie pleine des complexes cohérents de 
$\smash{\underrightarrow{LD}} ^{\mathrm{b}} _{\Q}
 ( \smash{\widehat{\D}} _{\PP} ^{(\bullet)})$
qu'il note
$\smash{\underrightarrow{LD}} ^{\mathrm{b}} _{\Q, \mathrm{coh}} ( \smash{\widehat{\D}} _{\PP} ^{(\bullet)})$.
Il a alors établi que le foncteur 
$\underrightarrow{\lim}$
induit l'équivalence de catégories
\begin{equation}
\notag
(*)\hspace{1cm}
\underrightarrow{\lim} 
\colon 
\smash{\underrightarrow{LD}} ^{\mathrm{b}} _{\Q, \mathrm{coh}} ( \smash{\widehat{\D}} _{\PP} ^{(\bullet)})
\cong
D ^{\mathrm{b}} _{\mathrm{coh}}( \smash{\D} ^\dag _{\PP,\Q} ).
\end{equation}
Nous avions défini dans \cite{caro_surcoherent} 
la sous-catégorie pleine notée 
$D ^{\mathrm{b}} _{\mathrm{surcoh}}( \smash{\D} ^\dag _{\PP,\Q} )$
de $D ^{\mathrm{b}} _{\mathrm{coh}}( \smash{\D} ^\dag _{\PP,\Q} )$ des complexes surcohérents.
Par définition, un complexe $\E$ de $D ^{\mathrm{b}} _{\mathrm{coh}}( \smash{\D} ^\dag _{\PP,\Q} )$ 
est surcohérent s'il vérifie la propriété suivante : pour n'importe quel morphisme lisse $f \colon \PP' \to \PP$,
pour n'importe quel diviseur $T' $ de la fibre spéciale de $\PP'$, on a 
$(\hdag T') \circ f ^{!} (\E) \in D ^{\mathrm{b}} _{\mathrm{coh}}( \smash{\D} ^\dag _{\PP',\Q} )$,
où $(\hdag T') $ est le foncteur de localisation en dehors de $T'$.
Si $\E$ est un objet $D ^{\mathrm{b}} _{\mathrm{surcoh}}( \smash{\D} ^\dag _{\PP,\Q} )$
et $\E ^{(\bullet)}$ est un objet 
de $\smash{\underrightarrow{LD}} ^{\mathrm{b}} _{\Q, \mathrm{coh}} ( \smash{\widehat{\D}} _{\PP} ^{(\bullet)})$
tel que 
$\underrightarrow{\lim}  ~ (\E ^{(\bullet)}) \riso \E$, 
il n'est pas évident que $\E ^{(\bullet)}$ vérifie les mêmes propriétés de stabilité que $\E$. 
La raison est que 
pour tout objet $\FF ^{(\bullet)} $
de 
$\smash{\underrightarrow{LD}} ^{\mathrm{b}} _{\Q} ( \smash{\widehat{\D}} _{\PP} ^{(\bullet)})$,
la propriété que
$\underrightarrow{\lim}  ~ (\FF ^{(\bullet)})$ soit cohérent n'implique pas que $\FF ^{(\bullet)}$ soit cohérent.

Soient $\ZZ$ un diviseur à croisements normaux strict de $\PP$,
$\PP ^{\sharp}:= (\PP, \ZZ)$ le schéma formal logarithmique lisse sur $\V$ induit
et
$T$ un diviseur de $P$. 
Toutes les constructions des faisceaux des opérateurs différentiels ci-dessus restent valable en rajoutant
des singularités surconvergentes le long de $T$ (voir \cite{Be1}) ou en rajoutant des singularités logarithmiques (voir
\cite{these_montagnon} ou \cite{caro_log-iso-hol}). 
Dans ce papier, nous introduisons de manière analogue la notion de (sur)cohérence pour un complexe de 
$D ( \smash{\widehat{\D}} _{\PP ^{\sharp}} ^{(\bullet)}(T) )$ ou de 
$D ^{\mathrm{b}} _{\mathrm{coh}}( \smash{\D} ^\dag _{\PP ^{\sharp}} (\hdag T) _{\Q} )$ (pour ce dernier, on retrouve la notion 
usuelle de surcohérence lorsque $\ZZ$ est vide).
Nous étendons aussi quelques propriétés standards de la théorie des $\D$-modules arithmétiques à ce contexte
logarithmique de systèmes inductifs, notamment le théorème de Berthelot-Kashiwara (voir le théorème \ref{u!u+=id}).
En notant 
$\smash{\underrightarrow{LD}} ^{\mathrm{b}} _{\Q, \mathrm{surcoh}} ( \smash{\widehat{\D}} _{\PP} ^{(\bullet)} (T))$
la sous-catégorie pleine de 
$\smash{\underrightarrow{LD}} ^{\mathrm{b}} _{\Q, \mathrm{coh}} ( \smash{\widehat{\D}} _{\PP } ^{(\bullet)}(T))$
des complexes surcohérents (la notion de surcohérence lorsque $\ZZ$ est non vide n'est pas intéressante), 
nous vérifions que l'on dispose de l'équivalence de catégories
\begin{equation}
(**)\hspace{1cm}
\notag
\underrightarrow{\lim} 
\colon 
\smash{\underrightarrow{LD}} ^{\mathrm{b}} _{\Q, \mathrm{surcoh}} ( \smash{\widehat{\D}} _{\PP} ^{(\bullet)}(T))
\cong
D ^{\mathrm{b}} _{\mathrm{surcoh}}( \smash{\D} ^\dag _{\PP} (\hdag T) _{\Q} )
\end{equation}
Le point crucial de cette équivalence (**) est le théorème
\ref{limTouD} dont le corollaire \ref{coro1limTouD} signifie que si la localisation en dehors d'un certain diviseur $T$ de $P$
de l'image par $\underrightarrow{\lim} $ d'un complexe cohérent $\E ^{(\bullet)}$ reste cohérent, alors
la localisation en dehors de ce diviseur $T$ de $\E ^{(\bullet)}$ reste cohérent. 
Pour comprendre la pertinence  fondamentale de cette équivalence de catégories dans les problèmes de stabilité 
en théorie des $\D$-modules arithmétiques, 
on peut aussi consulter la remarque \ref{rema-interet papier}.
\bigskip

Précisons à présent le contenu de ce papier. 
Notons $M (\smash{\widehat{\D}} _{\PP ^{\sharp}} ^{(\bullet)} (T))$ la catégorie
des $\smash{\widehat{\D}} _{\PP ^{\sharp}} ^{(\bullet)} (T)$-modules (toujours à gauche par défaut).
Dans le premier chapitre, 
en localisant la catégorie
$M (\smash{\widehat{\D}} _{\PP ^{\sharp}} ^{(\bullet)} (T))$ 
de manière identique à Berthelot pour les complexes, on introduit la catégorie 
$\underrightarrow{LM} _{\Q} (\smash{\widehat{\D}} _{\PP ^{\sharp}} ^{(\bullet)} (T))$.
Nous vérifions que c'est une catégorie abélienne quotient à la Serre de $M (\smash{\widehat{\D}} _{\PP ^{\sharp}} ^{(\bullet)} (T))$. 
Dans le second chapitre, nous définissons les objets de 
$\underrightarrow{LM} _{\Q} (\smash{\widehat{\D}} _{\PP ^{\sharp}} ^{(\bullet)} (T))$
qui sont cohérents à lim-ind-isogénies près. Cette notion de cohérence à lim-ind-isogénies près est locale sur $\PP ^{\sharp}$, ce qui est la motivation principale 
de son introduction. 
De plus, en notant 
$D ^{\mathrm{b}} _{\mathrm{coh}} (\underrightarrow{LM} _{\Q} (\smash{\widehat{\D}} _{\PP ^{\sharp}} ^{(\bullet)} (T)))$
la sous-catégorie pleine de 
$D ^{\mathrm{b}}  (\underrightarrow{LM} _{\Q} (\smash{\widehat{\D}} _{\PP ^{\sharp}} ^{(\bullet)} (T)))$
des complexes dont les espaces de cohomologie sont cohérents à lim-ind-isogénie près, 
on vérifie l'équivalence canonique de catégories de la forme
$$ \smash{\underrightarrow{LD}} ^{\mathrm{b}} _{\Q, \mathrm{coh}} ( \smash{\widehat{\D}} _{\PP ^{\sharp}} ^{(\bullet)} (T))
\cong
D ^{\mathrm{b}} _{\mathrm{coh}} (\underrightarrow{LM} _{\Q} (\smash{\widehat{\D}} _{\PP ^{\sharp}} ^{(\bullet)} (T))).$$
Nous établissons aussi un des résultats techniques utiles dans la preuve du résultat central du troisième chapitre (i.e. le théorème 
\ref{limTouD} déjà évoqué ci-dessus) 
qui est que 
pour vérifier qu'un morphisme de
$ \smash{\underrightarrow{LD}} ^{\mathrm{b}} _{\Q, \mathrm{coh}} ( \smash{\widehat{\D}} _{\PP ^{\sharp}} ^{(\bullet)} (T))$
soit un isomorphisme, il suffit de le prouver pour ses espaces de cohomologie
 calculés dans $\underrightarrow{LM} _{\Q} (\smash{\widehat{\D}} _{\PP ^{\sharp}} ^{(\bullet)} (T))$, ce qui nous ramène ainsi au cas des modules.

Dans le troisième chapitre, nous rappelons et étudions le foncteur de localisation en dehors d'un diviseur puis nous établissons le point clé de ce papier
évoqué ci-dessus
concernant la stabilité de la cohérence par foncteur de localisation en dehors d'un diviseur. 
Dans le quatrième chapitre, nous rappelons la construction 
des foncteurs cohomologiques à support strict dans un fermé et de localisation en dehors d'un fermé 
de \cite{caro_surcoherent} (ou de \cite{caro-Tsuzuki} avec des structures logarithmiques)
et complétons, lorsque cela est utile, la vérification de certaines de leurs propriétés (qui étendent celles de \cite{caro_surcoherent} au cas logarithmiques), 
notamment concernant la fonctorialité.
Enfin, dans le dernier chapitre, nous définissons la notion de surcohérence dans
le contexte des systèmes inductifs de $\D$-modules arithmétiques logarithmiques
et établissons l'équivalence $(**)$.

\subsection*{Remerciement}
Je remercie Tomoyuki Abe pour ses encouragements
à rédiger en détail la partie concernant le foncteur cohomologique local ainsi que pour ses commentaires 
sur une version préliminaire. Je remercie Pierre Berthelot pour son avis sur les résultats de ce papier. 
Je remercie enfin l'IUF pour son soutien.

\section*{Notations}
Dans ce papier, 
on désigne par $\V$ un anneau de valuation discrète complet d'inégales caractéristiques $(0,p)$, 
$k$ son corps résiduel supposé parfait, $K$ son corps des fractions et $\pi$ une uniformisante. 
Les faisceaux seront notés par des lettres calligraphiques, 
leurs sections globales par la lettre droite associée. 
Les modules sont par défaut à gauche. 
On notera avec des chapeaux les complétions $p$-adiques et si 
$\E$ est un faisceau en groupes abéliens alors on posera $\E _{\Q}:= \E \otimes _{\Z} \Q$. 
Soient $\AA$ un faisceau d'anneaux sur un espace topologique $X$.
Si $*$ est l'un des symboles $+$, $-$, ou $\mathrm{b}$, $D ^* ( \AA )$ désigne
la catégorie dérivée des complexes de $\AA$-modules (à gauche) vérifiant les conditions correspondantes d'annulation
des faisceaux de cohomologie. Lorsque l'on souhaite préciser entre droite et gauche, on précise alors comme suit
$D ^* ( \overset{ ^\mathrm{g}}{}\AA )$ ou $D ^* ( \AA \overset{ ^\mathrm{d}}{})$.
On note $D ^{\mathrm{b}} _{\mathrm{coh}} ( \AA )$
la sous-catégorie pleine de $D  ( \AA )$
des complexes à cohomologie cohérente et bornée.
On suppose (sans nuire à la généralité) que tous les
$k$-schémas sont réduits et on pourra confondre les diviseurs avec leur support.

Les $\V$-schémas formels (logarithmiques) seront indiqués par des lettres calligraphiques ou gothiques 
et leur fibre spéciale par la lettre droite correspondante.
Par défaut, tous les schémas ou schémas formels seront quasi-compacts. 
On se donne de plus $\PP$ un $\V$-schéma formel séparé, lisse (et quasi-compact), 
$T$ un diviseur de $P$,
$\ZZ$ un diviseur à croisements normaux strict sur $\PP$
et $\PP ^{\sharp}:= (\PP, \ZZ)$ le $\V$-schéma formel logarithmique lisse dont la log-structure est donnée par $\ZZ$.
Si $\U$ est un ouvert de $\PP$, on notera $\U ^{\sharp}:= (\U, \ZZ \cap \U)$.
Pour alléger les notations, on notera alors
$\smash{\widehat{\D}} _{\PP ^{\sharp}} ^{(m)} (T):=
\widehat{\B} ^{(m)} _{\PP } ( T)  \smash{\widehat{\otimes}} _{\O _{\PP}} \smash{\widehat{\D}} _{\PP ^{\sharp}} ^{(m)}$, 
où $\widehat{\B} ^{(m)} _{\PP} ( T) $ désigne les faisceaux d'anneaux construits par Berthelot dans
\cite[4.2.3]{Be1} et 
$\smash{\D} _{\PP ^{\sharp}} ^{(m)}$ est le faisceau des opérateurs différentiels de niveau $m$ sur $\PP$
(voir \cite[2.2]{Be1} pour la version non-logarithmique ou \cite[1]{caro_log-iso-hol} en général).
On fixe $\lambda _0\colon \N \to \N$ une application croissante telle que 
$\lambda _{0} (m) \geq m$. On pose alors 
$\widetilde{\B} ^{(m)} _{\PP} ( T):= \widehat{\B} ^{(\lambda _0 (m))} _{\PP} ( T)$
et
$\smash{\widetilde{\D}} _{\PP ^{\sharp}} ^{(m)} (T):=
\widetilde{\B} ^{(m)} _{\PP} ( T)  \smash{\widehat{\otimes}} _{\O _{\PP}} \smash{\widehat{\D}} _{\PP ^{\sharp}} ^{(m)}$.
Nous verrons d'ailleurs a posteriori (voir \ref{prop-MQlambda2LMQ} et \ref{bis-prop-MQlambda2LMQ}), 
que l'hypothèse $\lambda _0= id$ ne nuit pas à la généralité. 
Enfin, si $f \colon \X \to \PP$ 
(resp. $f ^{\sharp} \colon \X ^{\sharp} \to \PP ^{\sharp}$)
est un morphisme de $\V$-schémas formels (resp. logarithmiques) lisses,
pour tout entier $i \in \N$,
on note $ f _{i} \colon X _i \to P _i$ (resp. $ f _{i} ^{\sharp}  \colon X ^{\sharp} _i \to P ^{\sharp} _i$)
le morphisme induit modulo $\pi ^{i+1}$.
On pose enfin
$\smash{\D} _{P ^{\sharp} _i} ^{(m)} (T):= \V / \pi ^{i+1} \otimes _{\V} \smash{\widehat{\D}} _{\PP ^{\sharp} } ^{(m)} (T) 
=
\B ^{(m)} _{P _i} ( T)  \otimes _{\O _{P _i}} \smash{\D} _{P ^{\sharp} _i} ^{(m)}$
et
$\smash{\widetilde{\D}} _{P ^{\sharp} _i} ^{(m)} (T):=\widetilde{\B} ^{(m)} _{P _i} ( T)  \otimes _{\O _{P _i}} \smash{\D} _{P ^{\sharp} _i} ^{(m)}$.

\section{Localisations de catégories dérivées de systèmes inductifs sur le niveau de $\D$-modules arithmétiques}

\subsection{Rappels et définitions de Berthelot dans le cas des complexes}
Nous utiliserons dans ce papier toutes les notations qui suivent:
\begin{vide}
[Catégories localisées de la forme $\smash{\underrightarrow{LD}} _{\Q}$ de Berthelot]
\label{loc-LM}
On rappelle ici les constructions de Berthelot de \cite[4.2.1 et 4.2.2]{Beintro2}
(Ce sont des cas particuliers pour simplifier la présentation 
de cette introduction \cite{Beintro2}).
On dispose du système inductif d'anneaux 
$\smash{\widetilde{\D}} _{\PP ^\sharp} ^{(\bullet)}(T) : =
(\smash{\widetilde{\D}} _{\PP ^\sharp} ^{(m)}(T) )_{m\in \N}$ 
(les morphismes de transition sont construits de manière analogue à \cite{Be1}).
Nous disposons de la catégorie dérivée
$D ^{\sharp}( \smash{\widetilde{\D}} _{\PP ^\sharp} ^{(\bullet)}(T))$, où $\sharp \in \{\emptyset, +,-, \mathrm{b}\}$.
Les objets de $D ( \smash{\widetilde{\D}} _{\PP ^\sharp} ^{(\bullet)}(T))$
seront notés 
 $\E ^{(\bullet)}= (\E ^{(m)} , \alpha ^{(m',m)})$, 
 où $m,m'$ parcourent les entiers positifs tels que $m' \geq m$,
 où $\E ^{(m)} $ est un complexe de $\smash{\widetilde{\D}} _{\PP ^\sharp} ^{(m)}(T)$-modules
 et $\alpha ^{(m',m)} \colon \E ^{(m)}\to \E ^{(m')}$ sont les morphismes $\smash{\widetilde{\D}} _{\PP ^\sharp} ^{(m)}(T)$-linéaires de transition.

\begin{itemize}
\item Soit $M$ l'ensemble filtrant (muni de la relation d'ordre canonique) 
des applications croissantes $\chi \colon \N \to \N$. Pour toute application $\chi \in M$, on note
$\chi ^{*} (\E ^{(\bullet)}) := (\E ^{(m)} , p ^{\chi (m') -\chi (m)}\alpha ^{(m',m)})$.
On obtient en fait un foncteur 
$\chi ^{*} \colon D ( \smash{\widetilde{\D}} _{\PP ^\sharp} ^{(\bullet)}(T))\to D ( \smash{\widetilde{\D}} _{\PP ^\sharp} ^{(\bullet)}(T))$ de la manière suivante:
si $f ^{(\bullet)} \colon \E ^{(\bullet)} \to \FF ^{(\bullet)}$ de $D ( \smash{\widetilde{\D}} _{\PP ^\sharp} ^{(\bullet)}(T))$, 
le morphisme de niveau $m$ de $\chi ^{*} f ^{(\bullet)} $ est $f ^{(m)}$.
Si $\chi _1, \chi _2 \in M$, on calcule que 
$\chi _1 ^* \circ \chi _2 ^* = (\chi _1 +\chi _2)^*$, en particulier $\chi _1 ^*$ et $\chi _2 ^*$ commutent.
De plus, si $\chi _1 \leq \chi _2$, alors on dispose du morphisme canonique 
$\chi _1 ^* ( \E ^{(\bullet)})\to \chi _2 ^* ( \E ^{(\bullet)})$ défini par 
$p ^{\chi _2 (m) -\chi _1(m)}\colon \E ^{(m)} \to \E ^{(m)}$. 
Un morphisme $f ^{(\bullet)} \colon \E ^{(\bullet)} \to \FF ^{(\bullet)}$ de $D ( \smash{\widetilde{\D}} _{\PP ^\sharp} ^{(\bullet)}(T))$
est une ind-isogénie s'il existe $\chi \in M$ 
et un morphisme 
$g ^{(\bullet)} \colon \FF ^{(\bullet)} \to \chi ^{*} \E ^{(\bullet)}$ de $D ( \smash{\widetilde{\D}} _{\PP ^\sharp} ^{(\bullet)}(T))$
tels que les morphismes 
$g ^{(\bullet)}\circ f ^{(\bullet)}$ et $\chi ^{*} (f ^{(\bullet)}) \circ g ^{(\bullet)}$ 
de $D ( \smash{\widetilde{\D}} _{\PP ^\sharp} ^{(\bullet)}(T))$
sont les morphismes canoniques (on prend $\chi _1 =0$ et $\chi _2=\chi$).
L'ensemble des ind-isogénies est un système inductif (cela découle de 
la proposition \cite[I.4.2]{HaRD} et du lemme \ref{JG-foncteurcohomol} ci-dessous). 
La catégorie obtenue en localisant $D ^{\sharp}( \smash{\widetilde{\D}} _{\PP ^\sharp} ^{(\bullet)}(T))$
par rapport aux ind-isogénies se note
$\smash{\underrightarrow{D}} _{\Q} ^{\sharp}( \smash{\widetilde{\D}} _{\PP ^\sharp} ^{(\bullet)}(T))$.

\item Soit $L$ l'ensemble filtrant (muni de la relation d'ordre canonique)
des applications croissantes $\lambda \colon \N \to \N$ telles que $\lambda (m ) \geq m$.
Pour tout $\lambda \in L$, on note 
$\lambda ^{*} (\E ^{(\bullet)}) := (\E ^{(\lambda(m))} , \alpha ^{(\lambda(m'),\lambda(m))})_{m'\geq m}$.
Si $\lambda _1, \lambda _2 \in L$, on calcule que 
$\lambda _1 ^* \circ \lambda _2 ^* = (\lambda _1 \circ\lambda _2)^*$.
De plus, si $\lambda _1 \leq \lambda _2$, on dispose alors du morphisme canonique
$\lambda _1 ^* (\E ^{(\bullet)}) \to 
\lambda _2 ^* (\E ^{(\bullet)})$ défini au niveau $m$ par 
le morphisme $\alpha ^{(\lambda _2(m),\lambda _1(m))} \colon 
\E ^{(\lambda _1(m))} \to \E ^{(\lambda _2(m))}$. 
Comme pour \cite[4.2.2]{Beintro2}, 
on note $\Lambda ^{\sharp}$ l'ensemble des 
morphismes
$f ^{(\bullet)} \colon \E ^{(\bullet)} \to \FF ^{(\bullet)}$
de 
$\smash{\underrightarrow{D}} _{\Q} ^{\sharp} ( \smash{\widetilde{\D}} _{\PP ^\sharp} ^{(\bullet)}(T))$
tels qu'il existe $\lambda \in L$ et un morphisme 
$g ^{(\bullet)} \colon \FF ^{(\bullet)} \to \lambda ^{*} \E ^{(\bullet)}$ de $\smash{\underrightarrow{D}} _{\Q} ( \smash{\widetilde{\D}} _{\PP ^\sharp} ^{(\bullet)}(T))$
tels que les morphismes
$g ^{(\bullet)}\circ f ^{(\bullet)}$ et $\lambda ^{*} (f ^{(\bullet)}) \circ g ^{(\bullet)}$ 
de $\smash{\underrightarrow{D}} _{\Q} ( \smash{\widetilde{\D}} _{\PP ^\sharp} ^{(\bullet)}(T))$
sont les morphismes canoniques (i.e. on prend $\lambda _1 =id$ et $\lambda _2=\lambda$).
Par commodité pour en référer, convenons que les morphismes de $\Lambda ^{\sharp}$ sont les {\og lim-isomorphismes\fg}.
On vérifie que $\Lambda ^{\sharp}$ est un système multiplication 
(à nouveau, utiliser  \cite[I.4.2]{HaRD} et  \ref{JG-foncteurcohomol}). 
La catégorie obtenue en localisant 
$\smash{\underrightarrow{D}} ^{\sharp} _{\Q}
( \smash{\widetilde{\D}} _{\PP ^\sharp} ^{(\bullet)}(T))$
par rapport aux lim-isomorphismes 
sera notée
$\smash{\underrightarrow{LD}} ^{\sharp} _{\Q}
( \smash{\widetilde{\D}} _{\PP ^\sharp} ^{(\bullet)}(T))$.

\item Si $\chi _1\leq \chi _2 $ dans $M$ et $\lambda _1 \leq \lambda _2$ dans $L$, on obtient alors par composition 
le morphisme canonique 
$\lambda  _1^{*} \chi  _1^{*} \to \lambda _2^{*} \chi _2^{*}$.
De plus, en considérant $\chi _1 \circ \lambda _1$ comme un élément de $M$, on obtient l'égalité
$\lambda _1 ^* \chi _1 ^* = (\chi _1 \circ \lambda _1) ^* \lambda _1 ^*$.
On note $S ^{\sharp}$ l'ensemble des 
morphismes
$f ^{(\bullet)} \colon \E ^{(\bullet)} \to \FF ^{(\bullet)}$
de 
$D ^{\sharp} ( \smash{\widetilde{\D}} _{\PP ^\sharp} ^{(\bullet)}(T))$
tels qu'il existe $\chi \in M$, 
$\lambda \in L$ et un morphisme 
$g ^{(\bullet)} \colon \FF ^{(\bullet)} \to \lambda ^{*} \chi ^{*}\E ^{(\bullet)}$ de $D ( \smash{\widetilde{\D}} _{\PP ^\sharp} ^{(\bullet)}(T))$
tels que les morphismes
$g ^{(\bullet)}\circ f ^{(\bullet)}$ et $\lambda ^{*} \chi ^{*}(f ^{(\bullet)}) \circ g ^{(\bullet)}$ 
de $D ( \smash{\widetilde{\D}} _{\PP ^\sharp} ^{(\bullet)}(T))$
sont les morphismes canoniques.
Les morphismes de $S ^{\sharp}$ sont appelés les {\og lim-ind-isogénies\fg}.
De même que pour $\Lambda ^{\sharp}$, on vérifie que $S ^{\sharp}$ est un système multiplicatif.
Plus précisément, cela découle de la proposition \cite[I.4.2]{HaRD} et du lemme \ref{JG-foncteurcohomol} ci-dessous.
\end{itemize}
\end{vide}

\begin{lemm}
\label{JG-foncteurcohomol}
Pour tout $\G ^{(\bullet)} \in D ^{\sharp} ( \smash{\widetilde{\D}} _{\PP ^\sharp} ^{(\bullet)}(T))$,
notons $H _{\G ^{(\bullet)} } ,\ I _{\G ^{(\bullet)} } ,\ J _{\G ^{(\bullet)} } \colon D ^{\sharp} ( \smash{\widetilde{\D}} _{\PP ^\sharp} ^{(\bullet)}(T)) \to \mathfrak{Ab}$
les foncteurs cohomologiques à valeur dans la catégorie des groupes abéliens définis respectivement en posant pour tout 
$\E ^{(\bullet)} \in D ^{\sharp} ( \smash{\widetilde{\D}} _{\PP ^\sharp} ^{(\bullet)}(T))$
\begin{gather}
\notag
H _{\G ^{(\bullet)} } (\E ^{(\bullet)}):=
\underset{\chi \in M}{\underrightarrow{\lim}}~
\mathrm{Hom} _{D (\smash{\widetilde{\D}} _{\PP ^\sharp} ^{(\bullet)} (T))}
(\G ^{(\bullet)},\chi ^{*}\E ^{(\bullet)} ),
\
I _{\G ^{(\bullet)} } (\E ^{(\bullet)}):=
\underset{\lambda \in L}{\underrightarrow{\lim}}~
\mathrm{Hom} _{D (\smash{\widetilde{\D}} _{\PP ^\sharp} ^{(\bullet)} (T))}
(\G ^{(\bullet)}, \lambda ^{*} \E ^{(\bullet)} ),
\\
\notag
J _{\G ^{(\bullet)} } (\E ^{(\bullet)}):=
\underset{\lambda \in L}{\underrightarrow{\lim}}~
\underset{\chi \in M}{\underrightarrow{\lim}}~
\mathrm{Hom} _{D (\smash{\widetilde{\D}} _{\PP ^\sharp} ^{(\bullet)} (T))}
(\G ^{(\bullet)}, \lambda ^{*} \chi ^{*}\E ^{(\bullet)} ).
\end{gather}

Soit $f \colon \E ^{(\bullet)} 
\to \FF ^{(\bullet)} $ 
un morphisme de 
$D ^{\sharp} ( \smash{\widetilde{\D}} _{\PP ^\sharp} ^{(\bullet)}(T))$.
Le morphisme  $f ^{(\bullet)}$ est une ind-isogénie (resp. 
un lim-isomorphisme, resp. une  lim-ind-isogénie) si et seulement si 
$H _{\G ^{(\bullet)} } (f ^{(\bullet)})$ 
(resp. $I _{\G ^{(\bullet)} } (f ^{(\bullet)})$,
resp. $J _{\G ^{(\bullet)} } (f ^{(\bullet)})$)
est un isomorphisme pour tout $\G ^{(\bullet)} \in D ^{\sharp} ( \smash{\widetilde{\D}} _{\PP ^\sharp} ^{(\bullet)}(T))$.
\end{lemm}

\begin{proof}
0) Les deux autres cas se vérifiant de la même manière (on omet soit les $\chi$ soit les $\lambda$), 
traitons seulement le dernier cas respectif.  

1) Supposons que $f ^{(\bullet)}\in  S ^{\sharp}$.
Il existe alors $\chi _0 \in M$, 
$\lambda _0 \in L$ et un morphisme 
$g ^{(\bullet)} \colon \FF ^{(\bullet)} \to \lambda _0 ^{*} \chi _0^{*}\E ^{(\bullet)}$ de $D ( \smash{\widetilde{\D}} _{\PP ^\sharp} ^{(\bullet)}(T))$
tels que les morphismes
$g ^{(\bullet)}\circ f ^{(\bullet)}$ et $\lambda _0 ^{*} \chi _0 ^{*}(f ^{(\bullet)}) \circ g ^{(\bullet)}$ 
de $D ( \smash{\widetilde{\D}} _{\PP ^\sharp} ^{(\bullet)}(T))$
sont les morphismes canoniques.

a) Vérifions que $J _{\G ^{(\bullet)} } (f ^{(\bullet)})$ est injective. 
Soit un morphisme 
$u^{(\bullet)}\colon \G ^{(\bullet)}\to \lambda ^{*} \chi ^{*}\E ^{(\bullet)} $ tel que 
$\lambda ^{*} \chi ^{*} (f ^{(\bullet)} )\circ u^{(\bullet)} =0$.
Il s'agit de vérifier que, quitte à augmenter $\lambda$ et $\chi$, alors 
$u^{(\bullet)}=0$.
Quitte à augmenter $\chi _0$, $\chi$, $\lambda _0$ ou $\lambda$, on peut supposer
que $\lambda =\lambda _0$ et $\chi = \chi _0 \circ \lambda _0$ (on remarque que $\chi _0 \circ \lambda _0 \in M$ et sera toujours considéré
comme un élément de $M$). 
Il en résulte
$\lambda _0 ^{*} (\chi _0 \circ \lambda _0)  ^{*}(g ^{(\bullet)} ) \circ \lambda _0 ^{*} (\chi _0 \circ \lambda _0)^{*} (f ^{(\bullet)} )\circ u^{(\bullet)}=0$. 
Cela signifie que le morphisme 
$u^{(\bullet)}\colon \G ^{(\bullet)}
\to 
\lambda _0 ^{*} (\chi _0 \circ \lambda _0)  ^{*}\E ^{(\bullet)} $
composé avec le morphisme canonique
$\lambda _0 ^{*} (\chi _0 \circ \lambda _0)  ^{*}\E ^{(\bullet)} 
\to 
\lambda _0 ^{*} (\chi _0 \circ \lambda _0)  ^{*}\lambda _0 ^{*} \chi _0^{*} \E ^{(\bullet)} $
est le morphisme nul.
Quoique les foncteurs $\lambda _0 ^{*}$ et $( \chi _0  \circ \lambda _0) ^{*}$  ne commutent pas,
on dispose cependant du diagramme commutatif de foncteurs
dont tous les morphismes sont les morphismes canoniques
\begin{equation}
\label{diag1J_G}
\xymatrix @R=0,3cm {
{(\chi _0 \circ \lambda _0) ^{*}} 
\ar[r] ^-{}
\ar[d] ^-{}
& 
{ \lambda _0 ^{*} (\chi _0 \circ \lambda _0) ^{*}} 
\ar[r] ^-{}
&
{ \lambda _0 ^{*} (\chi _0 \circ \lambda _0) ^{*}   \chi _0^{*}} 
\\ 
{(\chi _0 \circ \lambda _0) ^{*}  \lambda _0 ^{*} }
\ar@{=}[r] ^-{} 
\ar[d] ^-{}
& 
{\lambda _0 ^{*} \chi _0^{*}} 
\ar[u] ^-{}
\ar[d] ^-{}
\ar[r] ^-{}
& 
{\lambda _0 ^{*} \chi _0^{*} \chi _0^{*}} 
\ar[u] ^-{}
\\ 
{(\chi _0 \circ \lambda _0) ^{*}  \lambda _0 ^{*}   \chi _0^{*}} 
\ar@{=}[r] ^-{} 
& 
{\lambda _0 ^{*} \chi _0^{*}  \chi _0^{*}.}
\ar@{=}[ru] ^-{}  
& 
{ } 
} 
\end{equation}
En appliquant $\lambda _0 ^{*}$ au diagramme \ref{diag1J_G}, 
on vérifie alors qu'en composant 
le morphisme $u^{(\bullet)}$
avec 
le morphisme canonique
$\lambda _0 ^{*} (\chi _0 \circ \lambda _0)  ^{*}\E ^{(\bullet)} 
\to 
\lambda _0 ^{*}  \lambda _0 ^{*} (\chi _0 \circ \lambda _0) ^{*}   \chi _0^{*}\E ^{(\bullet)}$,
on obtient le morphisme nul. Ainsi la classe de $u^{(\bullet)}$ dans $J _{\G ^{(\bullet)} } (\E ^{(\bullet)})$ est nulle.

b) Vérifions la surjectivité de $J _{\G ^{(\bullet)} } (f ^{(\bullet)})$. 
Soit un morphisme 
$v^{(\bullet)}\colon \G ^{(\bullet)}\to \lambda ^{*} \chi ^{*}\FF ^{(\bullet)} $.
Quitte à augmenter $\chi _0$, $\chi$, $\lambda _0$ ou $\lambda$, on peut supposer
que $\lambda =\lambda _0$ et $\chi = \chi _0 \circ \lambda _0$.
Posons 
$u^{(\bullet)}:= 
\lambda ^{*} \chi ^{*} (g ^{(\bullet)}) \circ v ^{(\bullet)}
\colon 
\G ^{(\bullet)}\to
 \lambda _0 ^{*} (\chi _0 \circ \lambda _0) ^{*}   \lambda _0^{*}\chi _0^{*} \E ^{(\bullet)}
 =
  \lambda _0 ^{*} \lambda _0 ^{*}   \chi _0^{*}\chi _0^{*} \E ^{(\bullet)}$.
En utilisant de nouveau la commutativité de \ref{diag1J_G} auquel on a appliqué $\lambda _0 ^*$,
on vérifie de même la classe de $u^{(\bullet)}$ s'envoie sur la classe de $v ^{(\bullet)}$ via
$J _{\G ^{(\bullet)} } (f ^{(\bullet)})$.

2) Supposons à présent que 
$J _{\G ^{(\bullet)} } (f ^{(\bullet)})$ soit un isomorphisme quelque soit 
$\G ^{(\bullet)} \in D ^{\sharp} ( \smash{\widetilde{\D}} _{\PP ^\sharp} ^{(\bullet)}(T))$.
Comme $J _{\FF ^{(\bullet)} } (f ^{(\bullet)})$ est en particulier surjectif,
il existe alors $\chi _0 \in M$, 
$\lambda _0 \in L$ et un morphisme 
$g ^{(\bullet)} \colon \FF ^{(\bullet)} \to \lambda _0 ^{*} \chi _0^{*}\E ^{(\bullet)}$ de $D ( \smash{\widetilde{\D}} _{\PP ^\sharp} ^{(\bullet)}(T))$
tel que le morphisme
$\lambda _0 ^{*} \chi _0 ^{*}(f ^{(\bullet)}) \circ g ^{(\bullet)}$ 
de $D ( \smash{\widetilde{\D}} _{\PP ^\sharp} ^{(\bullet)}(T))$
est le morphisme canonique. 
Le foncteur $J _{\E ^{(\bullet)} } (f ^{(\bullet)})$ envoie la classe de 
$g ^{(\bullet)}\circ f ^{(\bullet)}$
sur la classe de $\lambda _0 ^{*} \chi _0 ^{*}(f ^{(\bullet)}) \circ g ^{(\bullet)}\circ f ^{(\bullet)}$,
qui est aussi la classe de $f ^{(\bullet)}$.
Comme $J _{\E ^{(\bullet)} } (f ^{(\bullet)})$ est en particulier injectif,
la classe de $g ^{(\bullet)}\circ f ^{(\bullet)}$
est égale à la classe de l'identité de $\E ^{(\bullet)}$.
Quitte à augmenter $\lambda _0$ et $\chi _0$, il en résulte que le morphisme
$g ^{(\bullet)}\circ f ^{(\bullet)}$ est le morphisme canonique. 
\end{proof}

\begin{vide}
\label{preHomLDQ}
Un morphisme 
$f ^{(\bullet)}
\colon \E ^{(\bullet)} 
\to 
\FF ^{(\bullet)}$ de 
$\smash{\underrightarrow{D}} _{\Q} ( \smash{\widetilde{\D}} _{\PP ^\sharp} ^{(\bullet)}(T)) $
 est représenté par un morphisme 
 $\phi ^{(\bullet)}
 \colon 
 \E ^{(\bullet)} 
\to 
\chi ^{*}\FF ^{(\bullet)}$
pour un certain $\chi \in M$. 
De plus, deux morphismes $\phi _1 ^{(\bullet)}
 \colon 
 \E ^{(\bullet)} 
\to 
\chi _1 ^{*}\FF ^{(\bullet)}$
et
$\phi _2 ^{(\bullet)}
 \colon 
 \E ^{(\bullet)} 
\to 
\chi _2 ^{*}\FF ^{(\bullet)}$
de $D( \smash{\widetilde{\D}} _{\PP ^\sharp} ^{(\bullet)}(T)) $
induisent la même flèche de 
$\smash{\underrightarrow{D}} _{\Q} ( \smash{\widetilde{\D}} _{\PP ^\sharp} ^{(\bullet)}(T)) $
si et seulement s'il existe $\chi \geq \chi _1, \chi _2$, tels que les deux flèches composées
$\E ^{(\bullet)} 
\overset{\phi _1 ^{(\bullet)}}{\longrightarrow}
\chi _1 ^{*}\FF ^{(\bullet)}
\to 
\chi ^{*}\FF ^{(\bullet)}$
et
$\E ^{(\bullet)} 
\overset{\phi _2 ^{(\bullet)}}{\longrightarrow}
\chi _2 ^{*}\FF ^{(\bullet)}
\to 
\chi ^{*}\FF ^{(\bullet)}$
soient égales.
Pour résumer, 
pour tous $\E ^{(\bullet)}, \FF ^{(\bullet)} \in \underrightarrow{D} _{\Q} (\smash{\widetilde{\D}} _{\PP ^\sharp} ^{(\bullet)} (T))$,
on dispose de la formule
\begin{equation}
\label{4.2.2Beintropre}
\mathrm{Hom} _{\underrightarrow{D} _{\Q} (\smash{\widetilde{\D}} _{\PP ^\sharp} ^{(\bullet)} (T))}
(\E ^{(\bullet)}, \FF ^{(\bullet)} )
=
\underset{\chi \in M}{\underrightarrow{\lim}}\;
\mathrm{Hom} _{D (\smash{\widetilde{\D}} _{\PP ^\sharp} ^{(\bullet)} (T))}
(\E ^{(\bullet)}, \chi ^{*}\FF ^{(\bullet)} ).
\end{equation}
De même, 
on dispose,
pour tous $\E ^{(\bullet)}, \FF ^{(\bullet)} \in S ^{-1} D (\smash{\widetilde{\D}} _{\PP ^\sharp} ^{(\bullet)} (T))$,
de l'égalité:
\begin{equation}
\label{pre4.2.2Beintro}
\mathrm{Hom} _{S ^{-1} D (\smash{\widetilde{\D}} _{\PP ^\sharp} ^{(\bullet)} (T))}
(\E ^{(\bullet)}, \FF ^{(\bullet)} )
=
\underset{\lambda \in L}{\underrightarrow{\lim}}~
\underset{\chi \in M}{\underrightarrow{\lim}}~
\mathrm{Hom} _{D (\smash{\widetilde{\D}} _{\PP ^\sharp} ^{(\bullet)} (T))}
(\E ^{(\bullet)}, \lambda ^{*} \chi ^{*}\FF ^{(\bullet)} ).
\end{equation}
\end{vide}

\begin{lemm}
\label{lemm-locSQ}
\begin{enumerate}
\item 
\label{lemm-locSQ1}
Soient $\chi _1\in M$ et
$f ^{(\bullet)} \colon \E ^{(\bullet)} \to \chi _1^{*} \FF ^{(\bullet)}$
un morphisme de 
$D ( \smash{\widetilde{\D}} _{\PP ^\sharp} ^{(\bullet)}(T))$.
S'il existe $\chi _2 \in M$ 
et un morphisme 
$g ^{(\bullet)} \colon \FF ^{(\bullet)} \to \chi _2^{*} \E ^{(\bullet)}$ de $D ( \smash{\widetilde{\D}} _{\PP ^\sharp} ^{(\bullet)}(T))$
tel que 
$\chi _1 ^{*}(g ^{(\bullet)}) \circ f ^{(\bullet)}$ et $\chi _2 ^{*} (f ^{(\bullet)}) \circ g ^{(\bullet)}$ sont dans
$\smash{\underrightarrow{D}} _{\Q} 
( \smash{\widetilde{\D}} _{\PP ^\sharp} ^{(\bullet)}(T))$
les morphismes canoniques, 
alors 
$f ^{(\bullet)}$ est une ind-isogénie.

\item 
\label{lemm-locSQ2}
Soient $\chi _1\in M$, $\lambda _1 \in L$ et
$f ^{(\bullet)} \colon \E ^{(\bullet)} \to \lambda _1 ^* \chi _1^{*} \FF ^{(\bullet)}$
un morphisme de 
$D ( \smash{\widetilde{\D}} _{\PP ^\sharp} ^{(\bullet)}(T))$.
S'il existe $\chi _2 \in M$, $\lambda _2 \in L$
et un morphisme 
$g ^{(\bullet)} \colon \FF ^{(\bullet)} \to \lambda _2 ^* \chi _2^{*} \E ^{(\bullet)}$ de $D ( \smash{\widetilde{\D}} _{\PP ^\sharp} ^{(\bullet)}(T))$
tels que les morphismes 
$\lambda _1 ^* \chi _1 ^{*}(g ^{(\bullet)}) \circ f ^{(\bullet)}$ et $\lambda _2 ^* \chi _2 ^{*} (f ^{(\bullet)}) \circ g ^{(\bullet)}$ 
sont les morphismes canoniques
dans $S ^{-1} D
( \smash{\widetilde{\D}} _{\PP ^\sharp} ^{(\bullet)}(T))$, 
alors 
$f ^{(\bullet)}$ est une lim-ind-isogénie.
\end{enumerate}

\end{lemm}

\begin{proof}
1) Traitons d'abord la première assertion. 
a) Dans un premier temps supposons $\chi _1 =0$, i.e. $\chi _1 ^* =id$.
Par hypothèse, il existe alors $\chi  \geq \chi _2$ tel que
$g ^{(\bullet)} \circ f ^{(\bullet)}$ composé avec la flèche canonique 
$\chi _2 ^{*} \E ^{(\bullet)} \to  \chi ^{*} \E ^{(\bullet)}$
soit la flèche canonique dans 
$D ( \smash{\widetilde{\D}} _{\PP ^\sharp} ^{(\bullet)}(T))$
et tel que 
$\chi _2 ^{*} (f ^{(\bullet)}) \circ g ^{(\bullet)}$
composé avec la flèche canonique $\chi _2 ^{*} \FF ^{(\bullet)} \to  \chi ^{*} \FF ^{(\bullet)}$
soit la flèche canonique dans 
$D ( \smash{\widetilde{\D}} _{\PP ^\sharp} ^{(\bullet)}(T))$.
En notant $h ^{(\bullet)}$ le composé de $g ^{(\bullet)}$ avec le morphisme canonique
$\chi _2 ^{*} \E ^{(\bullet)} \to  \chi ^{*} \E ^{(\bullet)}$, 
on vérifie alors que 
$h ^{(\bullet)} \circ f ^{(\bullet)}$ et $\chi ^{*} (f ^{(\bullet)}) \circ h ^{(\bullet)}$ sont dans
$D
( \smash{\widetilde{\D}} _{\PP ^\sharp} ^{(\bullet)}(T))$
les morphismes canoniques.

b) Revenons à présent au cas général. 
En composant deux flèches consécutives de la suite
$\E ^{(\bullet)} 
\underset{f ^{(\bullet)}}{\longrightarrow}
\chi _1^{*} \FF ^{(\bullet)}
\underset{\chi _1 ^{*}(g ^{(\bullet)})}{\longrightarrow}
\chi _1^{*} \chi _2^{*} \E ^{(\bullet)}
\underset{\chi _1^{*} \chi _2^{*}(f ^{(\bullet)})}{\longrightarrow}
\chi _1^{*} \chi _2^{*} (\chi _1^{*} \FF ^{(\bullet)})
$
on obtient les morphismes canoniques
dans $\smash{\underrightarrow{D}} _{\Q} ^{\sharp}
( \smash{\widetilde{\D}} _{\PP ^\sharp} ^{(\bullet)}(T))$.
D'où le résultat d'après le cas a) traité ci-dessus. 

2) Pour la seconde assertion, 
on procède de manière analogue: on traite d'abord le cas où $\lambda _{1} = id$ et $\chi _1 =0$, puis le cas général
(on remplace partout $\chi _i ^{*}$ par $\lambda _i ^* \chi _i ^{*}$ pour $i= 1,2$).
\end{proof}

\begin{lemm}
\label{S=LDQ}
On dispose de l'équivalence canonique de catégories
$S ^{\sharp -1} D ^{\sharp}
( \smash{\widetilde{\D}} _{\PP ^\sharp} ^{(\bullet)}(T))
\cong
\smash{\underrightarrow{LD}} _{\Q} ^{\sharp}
( \smash{\widetilde{\D}} _{\PP ^\sharp} ^{(\bullet)}(T))$
qui est l'identité sur les objets.
\end{lemm}

\begin{proof}
Comme le foncteur canonique 
$D ^{\sharp}( \smash{\widetilde{\D}} _{\PP ^\sharp} ^{(\bullet)}(T))
\to
\smash{\underrightarrow{D}} _{\Q} ^{\sharp}
( \smash{\widetilde{\D}} _{\PP ^\sharp} ^{(\bullet)}(T))$
envoie une lim-ind-isogénie sur un lim-isomorphisme, 
le foncteur canonique $D ^{\sharp}( \smash{\widetilde{\D}} _{\PP ^\sharp} ^{(\bullet)}(T))
\to
\smash{\underrightarrow{LD}} _{\Q} ^{\sharp}
( \smash{\widetilde{\D}} _{\PP ^\sharp} ^{(\bullet)}(T))$
se factorise canoniquement 
en le foncteur 
$S ^{\sharp -1} D ^{\sharp}( \smash{\widetilde{\D}} _{\PP ^\sharp} ^{(\bullet)}(T))
\to
\smash{\underrightarrow{LD}} _{\Q} ^{\sharp}
( \smash{\widetilde{\D}} _{\PP ^\sharp} ^{(\bullet)}(T))$.
Réciproquement, comme une ind-isogénie est en particulier une lim-ind-isogénie, on bénéficie du foncteur canonique
$\smash{\underrightarrow{D}} _{\Q} ^{\sharp}
( \smash{\widetilde{\D}} _{\PP ^\sharp} ^{(\bullet)}(T))
\to 
S ^{\sharp -1} D ^{\sharp}( \smash{\widetilde{\D}} _{\PP ^\sharp} ^{(\bullet)}(T))$.
Soit $f ^{(\bullet)} \colon \E ^{(\bullet)} \to \FF ^{(\bullet)}$
un morphisme de 
$\smash{\underrightarrow{D}} _{\Q} ^{\sharp} ( \smash{\widetilde{\D}} _{\PP ^\sharp} ^{(\bullet)}(T))$
tel qu'il existe $\lambda _2 \in L$ et un morphisme 
$g ^{(\bullet)} \colon \FF ^{(\bullet)} \to \lambda _2 ^{*} \E ^{(\bullet)}$ de $\smash{\underrightarrow{D}} _{\Q} ( \smash{\widetilde{\D}} _{\PP ^\sharp} ^{(\bullet)}(T))$
tels que 
$g ^{(\bullet)}\circ f ^{(\bullet)}$ et $\lambda _2 ^{*} (f ^{(\bullet)}) \circ g ^{(\bullet)}$ sont les morphismes canoniques.
Il existe $\chi _1\in M$ tel que 
$f ^{(\bullet)} $
soit représenté par un morphisme de 
$D ^{\sharp} ( \smash{\widetilde{\D}} _{\PP ^\sharp} ^{(\bullet)}(T))$
de la forme
$\phi ^{(\bullet)}\colon \E ^{(\bullet)} \to \chi _1 ^* \FF ^{(\bullet)}$.
Il existe $\chi _2 \in M$ tel que 
$g ^{(\bullet)} $
soit représenté par un morphisme de 
$D ^{\sharp} ( \smash{\widetilde{\D}} _{\PP ^\sharp} ^{(\bullet)}(T))$
de la forme
$\psi ^{(\bullet)}\colon \FF ^{(\bullet)} \to \chi _2 ^* \lambda _2 ^{*} \E ^{(\bullet)}$.
On vérifie alors que
$\chi _1 ^* (\psi ^{(\bullet)})
\circ \phi ^{(\bullet)}$
et 
$\chi _2 ^* \lambda _2 ^* (\phi ^{(\bullet)})
\circ \psi ^{(\bullet)}$
sont dans 
$\smash{\underrightarrow{D}} _{\Q} ^{\sharp} ( \smash{\widetilde{\D}} _{\PP ^\sharp} ^{(\bullet)}(T))$
(et donc dans $S ^{-1} D ^{\sharp} ( \smash{\widetilde{\D}} _{\PP ^\sharp} ^{(\bullet)}(T))$)
les morphismes canoniques.
D'après le lemme \ref{lemm-locSQ}.\ref{lemm-locSQ2}, 
on en déduit que $\phi ^{(\bullet)}$ est une lim-ind-isogénie.
\end{proof}

\begin{vide}
\label{HomLDQ}
Il découle du lemme \ref{S=LDQ} et de \ref{pre4.2.2Beintro}
que l'on dispose,
pour tous $\E ^{(\bullet)}, \FF ^{(\bullet)} \in \underrightarrow{LD} _{\Q} (\smash{\widetilde{\D}} _{\PP ^\sharp} ^{(\bullet)} (T))$,
de l'égalité:
\begin{equation}
\label{4.2.2Beintro}
\mathrm{Hom} _{\underrightarrow{LD} _{\Q} (\smash{\widetilde{\D}} _{\PP ^\sharp} ^{(\bullet)} (T))}
(\E ^{(\bullet)}, \FF ^{(\bullet)} )
=
\underset{\lambda \in L}{\underrightarrow{\lim}}~
\underset{\chi \in M}{\underrightarrow{\lim}}~
\mathrm{Hom} _{D (\smash{\widetilde{\D}} _{\PP ^\sharp} ^{(\bullet)} (T))}
(\E ^{(\bullet)}, \lambda ^{*} \chi ^{*}\FF ^{(\bullet)} ).
\end{equation}

\end{vide}

\begin{rema}
\label{rema-iso-0}
\begin{itemize}
\item Soit $\phi ^{(\bullet)}
\colon \E ^{(\bullet)} 
\to 
\FF ^{(\bullet)}$ un morphisme de 
$D ( \smash{\widetilde{\D}} _{\PP ^\sharp} ^{(\bullet)}(T)) $.
Grâce à \ref{preHomLDQ} et \ref{lemm-locSQ} (resp. et \ref{HomLDQ}), on vérifie alors que
$\phi ^{(\bullet)}$ est un isomorphisme dans 
$\underrightarrow{D} _{\Q} (\smash{\widetilde{\D}} _{\PP ^\sharp} ^{(\bullet)} (T))$
(resp. $\underrightarrow{LD} _{\Q} (\smash{\widetilde{\D}} _{\PP ^\sharp} ^{(\bullet)} (T))$)
si et seulement si 
$\phi ^{(\bullet)}$ est une ind-isogénie 
(resp. $\phi ^{(\bullet)}\in S$).

\item On déduit du premier point qu'un complexe $\E ^{(\bullet)} $
de $D ( \smash{\widetilde{\D}} _{\PP ^\sharp} ^{(\bullet)}(T)) $
est isomorphe à $0$ dans
$\underrightarrow{D} _{\Q} (\smash{\widetilde{\D}} _{\PP ^\sharp} ^{(\bullet)} (T))$
(resp. $\underrightarrow{LD} _{\Q} (\smash{\widetilde{\D}} _{\PP ^\sharp} ^{(\bullet)} (T))$)
si et seulement s'il existe $\chi \in M$ (resp. $\chi \in M$ et $\lambda \in L$)
tels que la flèche canonique
$\E ^{(\bullet)}  \to \chi ^* \E ^{(\bullet)} $
(resp. 
$\E ^{(\bullet)}  \to \lambda ^* \chi ^* \E ^{(\bullet)} $)
est le morphisme nul.
\end{itemize}

\end{rema}

\subsection{Point de vue de la catégorie dérivée d'une catégorie abélienne quotient à la Serre}

\begin{vide}
\label{defi-M-L}
Notons $M (\smash{\widetilde{\D}} _{\PP ^\sharp} ^{(\bullet)} (T))$ la catégorie
des $\smash{\widetilde{\D}} _{\PP ^\sharp} ^{(\bullet)} (T)$-modules.
Les $\smash{\widetilde{\D}} _{\PP ^\sharp} ^{(\bullet)} (T)$-modules
seront notés 
$\E ^{(\bullet)}= (\E ^{(m)} , \alpha ^{(m',m)})$, 
où $m,m'$ parcourent les entiers positifs tels que $m' \geq m$,
où $\E ^{(m)} $ désigne un $\smash{\widetilde{\D}} _{\PP ^\sharp} ^{(m)}(T)$-module
et $\alpha ^{(m',m)} \colon \E ^{(m)}\to \E ^{(m')}$ sont les morphismes $\smash{\widetilde{\D}} _{\PP ^\sharp} ^{(m)}(T)$-linéaires de transition.
Pour tout $\chi \in M$, on note de manière analogue à 
\ref{loc-LM}
$\chi ^{*} (\E ^{(\bullet)}) := (\E ^{(m)} , p ^{\chi (m') -\chi (m)}\alpha ^{(m',m)})$.
En fait, on obtient le foncteur 
$\chi ^{*}\colon M (\smash{\widetilde{\D}} _{\PP ^\sharp} ^{(\bullet)} (T)) \to M (\smash{\widetilde{\D}} _{\PP ^\sharp} ^{(\bullet)} (T))$, 
ce dernier étant l'identité sur les flèches entre les modules au niveau $m$ fixé.
De plus, comme pour 
\ref{loc-LM},
pour tout $\lambda \in L$, on note 
$\lambda ^{*} (\E ^{(\bullet)}) := (\E ^{(\lambda(m))} , \alpha ^{(\lambda(m'),\lambda(m))})$.

Comme pour 
\ref{loc-LM},
on définit le système multiplicatif des ind-isogénies (resp. des lim-ind-isogénies) de
$M (\smash{\widetilde{\D}} _{\PP ^\sharp} ^{(\bullet)} (T))$
et on note
$\smash{\underrightarrow{M}} _{\Q}  ( \smash{\widetilde{\D}} _{\PP ^\sharp} ^{(\bullet)}(T))$
(resp. $S ^{-1}M (\smash{\widetilde{\D}} _{\PP ^\sharp} ^{(\bullet)} (T))$)
la catégorie localisée par les ind-isogénies (resp. par les lim-ind-isogénies). 
Comme pour 
\ref{loc-LM} (on remplace ici l'utilisation de \cite[I.4.2]{HaRD} par 
\cite[Exercice 7.5]{Kashiwara-schapira-book}),
on définit le système multiplicatif des lim-isomorphismes de
$\smash{\underrightarrow{M}} _{\Q}  ( \smash{\widetilde{\D}} _{\PP ^\sharp} ^{(\bullet)}(T))$
et on note
$\underrightarrow{LM} _{\Q} (\smash{\widetilde{\D}} _{\PP ^\sharp} ^{(\bullet)} (T))$
la catégorie localisée. 

Tous les résultats entre le lemme  \ref{preHomLDQ} et les remarques \ref{rema-iso-0} restent valables 
pour les modules à la place des complexes. 
Par exemple, de manière identique à \ref{S=LDQ}, 
on vérifie l'équivalence canonique de catégories
$S ^{-1}M (\smash{\widetilde{\D}} _{\PP ^\sharp} ^{(\bullet)} (T))
\cong 
\underrightarrow{LM} _{\Q} (\smash{\widetilde{\D}} _{\PP ^\sharp} ^{(\bullet)} (T))$.
Ainsi, les objets de 
$\underrightarrow{LM} _{\Q} (\smash{\widetilde{\D}} _{\PP ^\sharp} ^{(\bullet)} (T))$
sont ceux de $M (\smash{\widetilde{\D}} _{\PP ^\sharp} ^{(\bullet)} (T))$ et ses flèches sont comme pour \ref{4.2.2Beintro}
données par la formule: 
pour tous $\E ^{(\bullet)}, \FF ^{(\bullet)} \in \underrightarrow{LM} _{\Q} (\smash{\widetilde{\D}} _{\PP ^\sharp} ^{(\bullet)} (T))$ 
\begin{equation}
\label{4.2.2BeintroLMQ}
\mathrm{Hom} _{\underrightarrow{LM} _{\Q} (\smash{\widetilde{\D}} _{\PP ^\sharp} ^{(\bullet)} (T))}
(\E ^{(\bullet)}, \FF ^{(\bullet)} )
=
\underset{\lambda \in L}{\underrightarrow{\lim}}\;
\underset{\chi \in M}{\underrightarrow{\lim}}\;
\mathrm{Hom} _{M (\smash{\widetilde{\D}} _{\PP ^\sharp} ^{(\bullet)} (T))}
(\E ^{(\bullet)}, \lambda ^{*} \chi ^{*}\FF ^{(\bullet)} ).
\end{equation}
\end{vide}

\begin{vide}
\label{defi-M-L-Gamma}
En remplaçant $\D$ par $D$, 
on définit de manière identique à \ref{defi-M-L} les catégories
$M (\smash{\widetilde{D}} _{\PP ^{\sharp}} ^{(\bullet)} (T))$,
$\smash{\underrightarrow{M}} _{\Q}  ( \smash{\widetilde{D}} _{\PP ^{\sharp}} ^{(\bullet)}(T))$,
$S ^{-1}M (\smash{\widetilde{D}} _{\PP ^{\sharp}} ^{(\bullet)} (T))$
et 
$\underrightarrow{LM} _{\Q} (\smash{\widetilde{D}} _{\PP ^{\sharp}} ^{(\bullet)} (T))$.
De même, les résultats entre le lemme  \ref{preHomLDQ} et les remarques \ref{rema-iso-0}
sont toujours valables dans ce contexte.  
\end{vide}

\begin{lemm}
\label{Lm-LD-plfid}
Le foncteur canonique 
$\underrightarrow{LM} _{\Q} (\smash{\widetilde{\D}} _{\PP ^\sharp} ^{(\bullet)} (T))
\to 
\underrightarrow{LD} _{\Q} (\smash{\widetilde{\D}} _{\PP ^\sharp} ^{(\bullet)} (T))$
est pleinement fidèle.

\end{lemm}

\begin{proof}
Cela résulte du fait que l'application 
$\mathrm{Hom} _{M (\smash{\widetilde{\D}} _{\PP ^\sharp} ^{(\bullet)} (T))}
(\E ^{(\bullet)}, \lambda ^{*} \chi ^{*}\FF ^{(\bullet)} )
\to 
\mathrm{Hom} _{D (\smash{\widetilde{\D}} _{\PP ^\sharp} ^{(\bullet)} (T))}
(\E ^{(\bullet)}, \lambda ^{*} \chi ^{*}\FF ^{(\bullet)} )$
est bijective
et que l'on dispose des égalités \ref{4.2.2Beintro} et \ref{4.2.2BeintroLMQ}.
\end{proof}

\begin{lemm}
\label{defi-M-L-Serre}
Soit $N (\smash{\widetilde{\D}} _{\PP ^\sharp} ^{(\bullet)} (T))$
la sous-catégorie pleine de 
$M (\smash{\widetilde{\D}} _{\PP ^\sharp} ^{(\bullet)} (T))$
des modules nuls dans 
$\underrightarrow{LM} _{\Q} (\smash{\widetilde{\D}} _{\PP ^\sharp} ^{(\bullet)} (T))$.
\begin{enumerate}
\item La catégorie $N (\smash{\widetilde{\D}} _{\PP ^\sharp} ^{(\bullet)} (T))$ est une sous-catégorie de Serre
de $M (\smash{\widetilde{\D}} _{\PP ^\sharp} ^{(\bullet)} (T))$.
\item On dispose de l'égalité:
$M (\smash{\widetilde{\D}} _{\PP ^\sharp} ^{(\bullet)} (T))/N (\smash{\widetilde{\D}} _{\PP ^\sharp} ^{(\bullet)} (T))
=
S ^{-1}M (\smash{\widetilde{\D}} _{\PP ^\sharp} ^{(\bullet)} (T))
\cong 
\underrightarrow{LM} _{\Q} (\smash{\widetilde{\D}} _{\PP ^\sharp} ^{(\bullet)} (T))$.
En particulier, 
$\underrightarrow{LM} _{\Q} (\smash{\widetilde{\D}} _{\PP ^\sharp} ^{(\bullet)} (T))$ est une catégorie abélienne
et le système multiplicatif des lim-ind-isogénies de 
$M (\smash{\widetilde{\D}} _{\PP ^\sharp} ^{(\bullet)} (T))$ est saturé.

\end{enumerate}
\end{lemm}

\begin{proof}
1) Il est clair que $N (\smash{\widetilde{\D}} _{\PP ^\sharp} ^{(\bullet)} (T))$ 
est stable par isomorphismes. 
D'après l'analogue de la remarque 
\ref{rema-iso-0} (on remplace les complexes par les modules),
un objet
$\E ^{(\bullet)} $
de
$M (\smash{\widetilde{\D}} _{\PP ^\sharp} ^{(\bullet)} (T))$
est dans 
$N (\smash{\widetilde{\D}} _{\PP ^\sharp} ^{(\bullet)} (T))$ 
si et seulement s'il existe 
$\chi \in M$, 
$\lambda \in L$
tels que le morphisme canonique
$\E ^{(\bullet)}
\to
\lambda ^{*} \chi ^{*}\E ^{(\bullet)}$
soit le morphisme nul.
Il en résulte aussitôt que 
$N (\smash{\widetilde{\D}} _{\PP ^\sharp} ^{(\bullet)} (T))$ 
est stable par sous-objets, sous-quotients. 
De plus, soit
$0
\to 
\FF ^{(\bullet)}
\to 
\E ^{(\bullet)}
\to
\G ^{(\bullet)}
\to 0$
une suite exacte de 
$M (\smash{\widetilde{\D}} _{\PP ^\sharp} ^{(\bullet)} (T))$
telle que 
$\FF ^{(\bullet)}, \G ^{(\bullet)}
\in 
N (\smash{\widetilde{\D}} _{\PP ^\sharp} ^{(\bullet)} (T))$.
Soient $\chi \in M$, 
$\lambda \in L$
tels que les morphismes canoniques
$\FF ^{(\bullet)}
\to
\lambda ^{*} \chi ^{*}\FF ^{(\bullet)}$
et
$\G ^{(\bullet)}
\to
\lambda ^{*} \chi ^{*}\G ^{(\bullet)}$
sont les morphismes nuls.
Comme les foncteurs $\chi ^*$ et $\lambda ^*$
sont exacts,
on obtient la suite exacte courte
$0
\to 
\lambda ^{*} \chi ^{*}\FF ^{(\bullet)}
\to 
\lambda ^{*} \chi ^{*}\E ^{(\bullet)}
\to
\lambda ^{*} \chi ^{*}\G ^{(\bullet)}
\to 0$
dans $M (\smash{\widetilde{\D}} _{\PP ^\sharp} ^{(\bullet)} (T))$.
L'image du morphisme canonique
$\E ^{(\bullet)}
\to
\lambda ^{*} \chi ^{*}\E ^{(\bullet)}$
est alors incluse dans
$\lambda ^{*} \chi ^{*}\FF ^{(\bullet)}$.
On en déduit que le morphisme canonique
$\E ^{(\bullet)}
\to
(\lambda \circ \lambda) ^{*} (2\chi) ^{*}\E ^{(\bullet)}$
est le morphisme nul.

2) Notons $S _{N}$ le système multiplicatif saturé associé à la 
sous-catégorie de Serre 
$N (\smash{\widetilde{\D}} _{\PP ^\sharp} ^{(\bullet)} (T))$.
Rappelons que les éléments de $S _{N}$ sont les flèches 
de $M (\smash{\widetilde{\D}} _{\PP ^\sharp} ^{(\bullet)} (T))$
dont le noyau et le conoyau sont dans 
$N (\smash{\widetilde{\D}} _{\PP ^\sharp} ^{(\bullet)} (T))$.
Il s'agit de vérifier que $S _N =S$.

a) Soit 
$0
\to 
\FF ^{(\bullet)}
\to 
\E ^{(\bullet)}
\to
\G ^{(\bullet)}
\to 0$
une suite exacte de 
$M (\smash{\widetilde{\D}} _{\PP ^\sharp} ^{(\bullet)} (T))$.
Supposons 
$\FF ^{(\bullet)}
\in 
N (\smash{\widetilde{\D}} _{\PP ^\sharp} ^{(\bullet)} (T))$.
Soient $\chi \in M$, 
$\lambda \in L$
tel que le morphisme canonique
$\FF ^{(\bullet)}
\to
\lambda ^{*} \chi ^{*}\FF ^{(\bullet)}$
soit nul. 
On obtient alors un 
morphisme
$\G ^{(\bullet)}
\to 
\lambda ^{*} \chi ^{*}\E ^{(\bullet)}$
qui donne un inverse 
à 
$\E ^{(\bullet)}
\to
\G ^{(\bullet)}$
dans 
$S ^{-1}M (\smash{\widetilde{\D}} _{\PP ^\sharp} ^{(\bullet)} (T))$.
De même, si 
$\G ^{(\bullet)}
\in 
N (\smash{\widetilde{\D}} _{\PP ^\sharp} ^{(\bullet)} (T))$,
on vérifie que 
$\FF ^{(\bullet)}
\to 
\E ^{(\bullet)}$
est un isomorphisme dans 
$S ^{-1}M (\smash{\widetilde{\D}} _{\PP ^\sharp} ^{(\bullet)} (T))$.

b) Soit
$\phi \colon \E ^{(\bullet)}
\to 
\FF ^{(\bullet)}$
un morphisme de
$M (\smash{\widetilde{\D}} _{\PP ^\sharp} ^{(\bullet)} (T))$.
En utilisant $a)$, on vérifie que si 
 $\ker \phi$ et $\mathrm{coker} \phi$ appartiennent à 
$N (\smash{\widetilde{\D}} _{\PP ^\sharp} ^{(\bullet)} (T))$, 
alors $\phi$ est un isomorphisme dans
$\underrightarrow{LM} _{\Q} (\smash{\widetilde{\D}} _{\PP ^\sharp} ^{(\bullet)} (T))$, et donc 
$\phi \in S $ d'après le premier point de la remarque \ref{rema-iso-0}.
Réciproquement, supposons que 
$\phi $ soit une lim-ind-isogénie de $M (\smash{\widetilde{\D}} _{\PP ^\sharp} ^{(\bullet)} (T))$.
Soient $\chi \in M$, 
$\lambda \in L$
tels qu'il existe un  morphisme canonique
$\psi\colon \FF ^{(\bullet)}
\to
\lambda ^{*} \chi ^{*}\E ^{(\bullet)}$
dans
$M (\smash{\widetilde{\D}} _{\PP ^\sharp} ^{(\bullet)} (T))$
tel que 
$\psi \circ \phi$
et $\lambda ^{*} \chi ^{*} (\phi) \circ \psi$ 
sont les morphismes canoniques.
On vérifie alors que la composition
$\ker \phi \to
\lambda ^{*} \chi ^{*} 
\ker \phi
\to 
\lambda ^{*} \chi ^{*} 
 \E ^{(\bullet)}$
est le morphisme nul. 
Comme $\lambda ^{*} \chi ^{*} 
\ker \phi
\to 
\lambda ^{*} \chi ^{*} 
 \E ^{(\bullet)}$
est un monomorphisme, 
il en résulte que le morphisme canonique
$\ker \phi \to
\lambda ^{*} \chi ^{*} 
\ker \phi$
est
nul. 
De même, on vérifie que 
la composition
$\FF ^{(\bullet)} \to 
\mathrm{coker} (\phi)
\to 
\lambda ^{*} \chi ^{*} \mathrm{coker} (\phi)$
est le morphisme nul. 
Comme $\FF ^{(\bullet)} \to 
\mathrm{coker} (\phi)$
est un épimorphisme, on en déduit que le morphisme 
canonique
$\mathrm{coker} (\phi)
\to 
\lambda ^{*} \chi ^{*} \mathrm{coker} (\phi)$
est le morphisme nul. 
\end{proof}

\begin{rema}
\label{rema-Sqi}
Pour $\sharp \in \{ +,-, \mathrm{b}, \emptyset\}$, 
notons $D ^{\sharp} _{N (\smash{\widetilde{\D}} _{\PP ^\sharp} ^{(\bullet)} (T))} (\smash{\widetilde{\D}} _{\PP ^\sharp} ^{(\bullet)} (T))$
la sous-catégorie pleine de 
$D ^{\sharp} (\smash{\widetilde{\D}} _{\PP ^\sharp} ^{(\bullet)} (T))$
des complexes dont les espaces de cohomologie sont dans
$N (\smash{\widetilde{\D}} _{\PP ^\sharp} ^{(\bullet)} (T))$. 
Notons $S _{Nqi} ^{\sharp}$ le système multiplicatif saturé compatible à la triangulation 
de
$D ^{\sharp} (\smash{\widetilde{\D}} _{\PP ^\sharp} ^{(\bullet)} (T))$
qui correspond à 
$D ^{\sharp} _{N (\smash{\widetilde{\D}} _{\PP ^\sharp} ^{(\bullet)} (T))} (\smash{\widetilde{\D}} _{\PP ^\sharp} ^{(\bullet)} (T))$.
Par définition, un morphisme $f ^{(\bullet)} $ de 
$D ^{\sharp} (\smash{\widetilde{\D}} _{\PP ^\sharp} ^{(\bullet)} (T))$
appartient à $S _{Nqi} ^{\sharp}$ si et seulement si, 
pour tout triangle distingué (il en suffit d'un en fait)
dans $D ^{\sharp} (\smash{\widetilde{\D}} _{\PP ^\sharp} ^{(\bullet)} (T))$
de la forme 
$\E ^{(\bullet)} 
\overset{f ^{(\bullet)}}{\longrightarrow}
\FF ^{(\bullet)} 
\to 
\G  ^{(\bullet)} \to \E ^{(\bullet)}  [1]$, 
pour tout entier $n \in \Z$, 
on ait 
$ \mathcal{H} ^{n} (\G  ^{(\bullet)})
 \in N  (\smash{\widetilde{\D}} _{\PP ^\sharp} ^{(\bullet)} (T))$).
Il découle du théorème \cite[3.1]{Miyachi-localization}
(on ne peut pas utiliser sa version plus faible et plus agréable à démontrer de \cite[2.6.2]{Bockle-Pink-cohomological-theory} car 
comme la sous-catégorie 
$N (\smash{\widetilde{\D}} _{\PP ^\sharp} ^{(\bullet)} (T))$ 
de $M (\smash{\widetilde{\D}} _{\PP ^\sharp} ^{(\bullet)} (T))$ n'est pas stable par limite inductive filtrante,
le foncteur oubli $N (\smash{\widetilde{\D}} _{\PP ^\sharp} ^{(\bullet)} (T)) \to M (\smash{\widetilde{\D}} _{\PP ^\sharp} ^{(\bullet)} (T))$
ne possède pas de foncteur adjoint à droite),
le foncteur canonique
$D ^{\sharp} (\smash{\widetilde{\D}} _{\PP ^\sharp} ^{(\bullet)} (T)) 
\to
D ^{\sharp}
(\underrightarrow{LM} _{\Q} (\smash{\widetilde{\D}} _{\PP ^\sharp} ^{(\bullet)} (T)))$ 
induit canoniquement l'équivalence de catégories 
\begin{equation}
\label{rema-Sqi-eqcat}
S _{Nqi} ^{\sharp-1} D ^{\sharp} (\smash{\widetilde{\D}} _{\PP ^\sharp} ^{(\bullet)} (T)) 
\cong 
D ^{\sharp}
(\underrightarrow{LM} _{\Q} (\smash{\widetilde{\D}} _{\PP ^\sharp} ^{(\bullet)} (T))).
\end{equation}
\end{rema}

\begin{vide} 
\label{defi-Hn-LMQ}
Pour tout entier $n\in \Z$, 
on dispose du foncteur $n$-ième espace de cohomologie
$\mathcal{H} ^{n}
\colon 
D  (\smash{\widetilde{\D}} _{\PP ^\sharp} ^{(\bullet)} (T))
\to 
 M(\smash{\widetilde{\D}} _{\PP ^\sharp} ^{(\bullet)} (T))$
défini pour  
$\E ^{(\bullet)}= (\E ^{(m)} , \alpha ^{(m',m)}) \in D  (\smash{\widetilde{\D}} _{\PP ^\sharp} ^{(\bullet)} (T))$
en posant
$\mathcal{H} ^{n} (\E ^{(\bullet)})
= (\mathcal{H} ^{n} (\E ^{(m)} ), \mathcal{H} ^{n} (\alpha ^{(m',m)}))$. 
On dispose de l'isomorphisme canonique 
$\mathcal{H} ^{n} \lambda ^{*} \chi ^{*} (\E ^{(\bullet)})
\riso
\lambda ^{*} \chi ^{*}\mathcal{H} ^{n} (\E ^{(\bullet)})$
qui s'inscrit dans le diagramme commutatif
\begin{equation}
\notag
\xymatrix @ R=0,4cm{
{\mathcal{H} ^{n} (\E ^{(\bullet)})}
\ar@{=}[d] ^-{} 
\ar[r] ^-{}
& 
{\mathcal{H} ^{n} \lambda ^{*} \chi ^{*} (\E ^{(\bullet)})} 
\ar[d] ^-{\sim}
\\ 
{\mathcal{H} ^{n} (\E ^{(\bullet)})} 
\ar[r] ^-{}
& 
{\lambda ^{*} \chi ^{*}\mathcal{H} ^{n} (\E ^{(\bullet)})} 
}
\end{equation}
dont la flèche horizontale du bas est la flèche canonique
tandis que celle du haut est l'image par $\mathcal{H} ^{n} $
du morphisme canonique.
On en déduit que le foncteur $\mathcal{H} ^{n}$ envoie les lim-ind-isogénies sur les lim-ind-isogénies.
Le foncteur $\mathcal{H} ^{n}$se factorise alors en 
$\mathcal{H} ^{n} \colon 
\underrightarrow{LD} _{\Q} (\smash{\widetilde{\D}} _{\PP ^\sharp} ^{(\bullet)} (T))
\to 
\underrightarrow{LM} _{\Q} (\smash{\widetilde{\D}} _{\PP ^\sharp} ^{(\bullet)} (T))$.
\end{vide}

\begin{lemm}
\label{Hn-fctcoho}
Le foncteur 
$\mathcal{H} ^{n} \colon 
\underrightarrow{LD} _{\Q} (\smash{\widetilde{\D}} _{\PP ^\sharp} ^{(\bullet)} (T))
\to 
\underrightarrow{LM} _{\Q} (\smash{\widetilde{\D}} _{\PP ^\sharp} ^{(\bullet)} (T))$
défini dans \ref{defi-Hn-LMQ}
est un foncteur cohomologique.
\end{lemm}

\begin{proof}
Par construction, un triangle distingué de
 $\underrightarrow{LD} _{\Q} (\smash{\widetilde{\D}} _{\PP ^\sharp} ^{(\bullet)} (T))$ 
 est isomorphe dans  $\underrightarrow{LD} _{\Q} (\smash{\widetilde{\D}} _{\PP ^\sharp} ^{(\bullet)} (T))$ 
 à l'image d'un triangle distingué de 
 $K (\smash{\widetilde{\D}} _{\PP ^\sharp} ^{(\bullet)} (T))$
 par le foncteur canonique de localisation
 $K (\smash{\widetilde{\D}} _{\PP ^\sharp} ^{(\bullet)} (T))
 \to 
 \underrightarrow{LD} _{\Q} (\smash{\widetilde{\D}} _{\PP ^\sharp} ^{(\bullet)} (T))$.
 Comme $\mathcal{H} ^{n} \colon 
K (\smash{\widetilde{\D}} _{\PP ^\sharp} ^{(\bullet)} (T))
\to 
M(\smash{\widetilde{\D}} _{\PP ^\sharp} ^{(\bullet)} (T))$
est un foncteur cohomologique,
comme le foncteur de localisation
$M(\smash{\widetilde{\D}} _{\PP ^\sharp} ^{(\bullet)} (T))
\to 
\underrightarrow{LM} _{\Q} (\smash{\widetilde{\D}} _{\PP ^\sharp} ^{(\bullet)} (T))$
préserve les suites exactes (cela résulte des propriétés des localisations par une sous-catégorie de Serre et de \ref{defi-M-L-Serre}),
on en déduit le résultat.

\end{proof}

\begin{nota}
\label{nota-LDQ0}
Notons $\underrightarrow{LD} ^{0} _{\Q} (\smash{\widetilde{\D}} _{\PP ^\sharp} ^{(\bullet)} (T))$ la sous-catégorie strictement pleine de 
 $\underrightarrow{LD} ^{\mathrm{b}}_{\Q} (\smash{\widetilde{\D}} _{\PP ^\sharp} ^{(\bullet)} (T))$
 des complexes 
 $\E ^{(\bullet)} $
 tels que pour tout entier $n \not =0$ on ait  
$\mathcal{H} ^{n} (\E ^{(\bullet)} ) \riso 0$ 
dans 
$\underrightarrow{LM} _{\Q} (\smash{\widetilde{\D}} _{\PP ^\sharp} ^{(\bullet)} (T))$.
\end{nota}

\begin{rema}
Soit
$\E ^{(\bullet)}
\in 
\underrightarrow{LD}  _{\Q} (\smash{\widetilde{\D}} _{\PP ^\sharp} ^{(\bullet)} (T))$.
Par définition, 
$\E ^{(\bullet)}$ est un objet de 
$\underrightarrow{LD} ^{\mathrm{b}} _{\Q} (\smash{\widetilde{\D}} _{\PP ^\sharp} ^{(\bullet)} (T))$
si et seulement si
$\E ^{(\bullet)}$ est un objet de 
$D ^{\mathrm{b}} (\smash{\widetilde{\D}} _{\PP ^\sharp} ^{(\bullet)} (T))$.
Cette condition est a priori plus forte que de demander 
qu'il existe $N \in \N $ assez grand tel que,
pour tout $j \not \in [-N,N] \cap \Z$ on ait l'isomorphisme $\mathcal{H} ^{j} (\E ^{(\bullet)} ) \riso 0$
dans $\underrightarrow{LM}  _{\Q} (\smash{\widetilde{\D}} _{\PP ^\sharp} ^{(\bullet)} (T))$. 
Il n'est donc pas clair que  $\underrightarrow{LD} ^{\mathrm{b}}_{\Q} (\smash{\widetilde{\D}} _{\PP ^\sharp} ^{(\bullet)} (T))$
soit une sous-catégorie strictement pleine de 
 $\underrightarrow{LD} _{\Q} (\smash{\widetilde{\D}} _{\PP ^\sharp} ^{(\bullet)} (T))$.
\end{rema}

\begin{lemm}
\label{eq-cat-LM-LD0}
Le foncteur canonique
$\underrightarrow{LM} _{\Q} (\smash{\widetilde{\D}} _{\PP ^\sharp} ^{(\bullet)} (T))
\to 
\underrightarrow{LD} ^{0} _{\Q} (\smash{\widetilde{\D}} _{\PP ^\sharp} ^{(\bullet)} (T))$
est une équivalence de catégories. 
De plus, le foncteur 
$\mathcal{H} ^{0}\colon 
\underrightarrow{LD} ^{0} _{\Q} (\smash{\widetilde{\D}} _{\PP ^\sharp} ^{(\bullet)} (T))
\to 
\underrightarrow{LM} _{\Q} (\smash{\widetilde{\D}} _{\PP ^\sharp} ^{(\bullet)} (T))$
est une équivalence de catégories quasi-inverse. 

\end{lemm}

\begin{proof}
Il résulte de \ref{Lm-LD-plfid} que le premier foncteur est pleinement fidèle. 
Soit 
$\E ^{(\bullet)} \in \underrightarrow{LD} ^{0} _{\Q} (\smash{\widetilde{\D}} _{\PP ^\sharp} ^{(\bullet)} (T))$.
Il reste à prouver qu'il existe dans 
$\underrightarrow{LD} ^{\mathrm{b}} _{\Q} (\smash{\widetilde{\D}} _{\PP ^\sharp} ^{(\bullet)} (T))$
un isomorphisme de la forme
$\E ^{(\bullet)} \riso \mathcal{H} ^{0}(\E ^{(\bullet)})$.
Comme 
$\E ^{(\bullet)}$ est un objet de 
$D ^{\mathrm{b}} (\smash{\widetilde{\D}} _{\PP ^\sharp} ^{(\bullet)} (T))$,
il existe $N \in \N $ assez grand tel que,
pour tout $j \not \in [-N,N] \cap \Z$ on ait l'isomorphisme $\mathcal{H} ^{j} (\E ^{(\bullet)} ) \riso 0$
dans $M(\smash{\widetilde{\D}} _{\PP ^\sharp} ^{(\bullet)} (T))$.
Pour tout entier $n \in \Z$, on note $\sigma _{\geq n} (\E ^{(\bullet)} )$ et 
$\sigma _{\leq n} (\E ^{(\bullet)} )$ les troncations cohomologiques usuelles (voir les notations après la proposition \cite[I.7.1]{HaRD}).
On dispose ainsi par définition pour tout entier $n \in \Z$ de la suite exacte dans 
$C (\smash{\widetilde{\D}} _{\PP ^\sharp} ^{(\bullet)} (T))$ de la forme: 
$0\to 
\sigma _{\leq -n-1} (\E ^{(\bullet)} )
\to
\E ^{(\bullet)}
\to 
\sigma _{\geq -n} (\E ^{(\bullet)} )\to 0$.

1) {\it Vérifions $\sigma _{\leq -1} (\E ^{(\bullet)} ) \riso 0$ 
dans 
$\underrightarrow{LD} ^{\mathrm{b}} _{\Q} (\smash{\widetilde{\D}} _{\PP ^\sharp} ^{(\bullet)} (T))$.}

a) Posons
$\FF ^{(\bullet)} 
:= \sigma _{\leq -1} (\E ^{(\bullet)} )
$.
On dispose alors dans $D(\smash{\widetilde{\D}} _{\PP ^\sharp} ^{(\bullet)} (T))$ de l'isomorphisme 
$\sigma _{\leq -N-1} (\FF ^{(\bullet)} ) 
=\sigma _{\leq -N-1} (\E ^{(\bullet)} )
\riso 0$.
En considérant la suite exacte de $C (\smash{\widetilde{\D}} _{\PP ^\sharp} ^{(\bullet)} (T))$: 
$0\to 
\sigma _{\leq -N-1} (\FF ^{(\bullet)} )
\to
\FF ^{(\bullet)}
\to 
\sigma _{\geq -N} (\FF ^{(\bullet)} )\to 0$, 
on en déduit que 
le morphisme canonique
$\FF ^{(\bullet)}
\to 
\sigma _{\geq -N} (\FF ^{(\bullet)} )$
est un isomorphisme de 
$D(\smash{\widetilde{\D}} _{\PP ^\sharp} ^{(\bullet)} (T))$.

b) Pour tout $j \in [-N,-1]\cap \Z$, on dispose d'après \cite[I.7.2]{HaRD} du triangle distingué dans
$D (\smash{\widetilde{\D}} _{\PP ^\sharp} ^{(\bullet)} (T))$:
\begin{equation}
\notag
\mathcal{H} ^{j}
(\FF ^{(\bullet)} )
\to 
\sigma _{\geq j} (\FF ^{(\bullet)} )
\to
\sigma _{\geq j+1} (\FF ^{(\bullet)} )
\to +1.
\end{equation}
Comme, pour tout $j \in [-N,-1]\cap \Z$,
$\mathcal{H} ^{j}
(\FF ^{(\bullet)} )
=
\mathcal{H} ^{j}
(\E ^{(\bullet)} )
\riso 
0$
dans
$\underrightarrow{LD} ^{\mathrm{b}} _{\Q} (\smash{\widetilde{\D}} _{\PP ^\sharp} ^{(\bullet)} (T))$,
il en résulte que 
la flèche 
$\sigma _{\geq j} (\FF ^{(\bullet)} )
\to
\sigma _{\geq j+1} (\FF ^{(\bullet)} )
$
est un isomorphisme 
dans 
$\underrightarrow{LD} ^{\mathrm{b}} _{\Q} (\smash{\widetilde{\D}} _{\PP ^\sharp} ^{(\bullet)} (T))$.
D'où 
$\sigma _{\geq -N} (\FF ^{(\bullet)} )
\riso 
\sigma _{\geq 0} (\FF ^{(\bullet)} )$
dans 
$\underrightarrow{LD} ^{\mathrm{b}} _{\Q} (\smash{\widetilde{\D}} _{\PP ^\sharp} ^{(\bullet)} (T))$.

c) 
Comme on a dans $D (\smash{\widetilde{\D}} _{\PP ^\sharp} ^{(\bullet)} (T))$ l'isomorphisme
$\sigma _{\geq 0} (\FF ^{(\bullet)} )
=
\sigma _{\geq 0} ( \sigma _{\leq -1} (\E ^{(\bullet)} ))
\riso 
0$, il résulte des étapes a) et b) que 
$\FF ^{(\bullet)} \riso 0$ 
dans 
$\underrightarrow{LD} ^{\mathrm{b}} _{\Q} (\smash{\widetilde{\D}} _{\PP ^\sharp} ^{(\bullet)} (T))$.

2) {\it Démontrons à présent que 
 le morphisme canonique
$\mathcal{H} ^{0}
(\E ^{(\bullet)} )
\to 
\sigma _{\geq 0} (\E ^{(\bullet)} )$
(voir sa construction dans \cite[I.7.2]{HaRD})
est un isomorphisme
dans 
$\underrightarrow{LD} ^{\mathrm{b}} _{\Q} (\smash{\widetilde{\D}} _{\PP ^\sharp} ^{(\bullet)} (T))$.}

Posons $\G ^{(\bullet)} := \sigma _{\geq 0} (\E ^{(\bullet)} )$.
De même que pour l'étape 1.b), on vérifie que 
le morphisme canonique
$\sigma _{\geq 1} (\E ^{(\bullet)} )
=\sigma _{\geq 1} (\G ^{(\bullet)} )
\to
\sigma _{\geq N +1} (\G ^{(\bullet)} )$
est un isomorphisme 
dans 
$\underrightarrow{LD} ^{\mathrm{b}} _{\Q} (\smash{\widetilde{\D}} _{\PP ^\sharp} ^{(\bullet)} (T))$.
Comme 
$\sigma _{\geq N +1} (\G ^{(\bullet)} )
=
\sigma _{\geq N +1} (\E ^{(\bullet)} )
\riso 0$
dans $D (\smash{\widetilde{\D}} _{\PP ^\sharp} ^{(\bullet)} (T))$,
il en résulte que 
$\sigma _{\geq 1} (\E ^{(\bullet)} )
\riso 0$
dans 
$\underrightarrow{LD} ^{\mathrm{b}} _{\Q} (\smash{\widetilde{\D}} _{\PP ^\sharp} ^{(\bullet)} (T))$.
En considérant le triangle distingué de $D (\smash{\widetilde{\D}} _{\PP ^\sharp} ^{(\bullet)} (T))$ (voir \cite[I.7.2]{HaRD})
\begin{equation}
\notag
\mathcal{H} ^{0}
(\E ^{(\bullet)} )
\to 
\sigma _{\geq 0} (\E ^{(\bullet)} )
\to
\sigma _{\geq 1} (\E ^{(\bullet)} )
\to +1,
\end{equation}
on en déduit le résultat.

3) Via 1) et 2), on conclut grâce à la suite exacte  
$0\to 
\sigma _{\leq -1} (\E ^{(\bullet)} )
\to
\E ^{(\bullet)} 
\to 
\sigma _{\geq 0} (\E ^{(\bullet)} )
\to 0$ dans $C (\smash{\widetilde{\D}} _{\PP ^\sharp} ^{(\bullet)} (T))$.

\end{proof}

\begin{coro}
\label{eqcatLD=DSM}
Avec les notations \ref{rema-Sqi}, 
on dispose de l'égalité
$S ^{\mathrm{b}}= S _{Nqi} ^{\mathrm{b}}$. 
Pour $\sharp \in \{ +,-, \mathrm{b}, \emptyset\}$, 
on dispose de l'inclusion
$S ^{\sharp}\subset S _{Nqi} ^{\sharp}$.
Le morphisme canonique
$D ^{\mathrm{b}} (\smash{\widetilde{\D}} _{\PP ^\sharp} ^{(\bullet)} (T))
\to
D ^{\mathrm{b}}
(\underrightarrow{LM} _{\Q} (\smash{\widetilde{\D}} _{\PP ^\sharp} ^{(\bullet)} (T)))$
induit par le foncteur de localisation
$M(\smash{\widetilde{\D}} _{\PP ^\sharp} ^{(\bullet)} (T))
\to
\underrightarrow{LM} _{\Q} (\smash{\widetilde{\D}} _{\PP ^\sharp} ^{(\bullet)} (T))$
se factorise canoniquement en l'équivalence de catégories
\begin{equation}
\label{eqcatLD=DSM-fonct}
\underrightarrow{LD} ^{\mathrm{b}} _{\Q} (\smash{\widetilde{\D}} _{\PP ^\sharp} ^{(\bullet)} (T))
\cong 
D ^{\mathrm{b}}
(\underrightarrow{LM} _{\Q} (\smash{\widetilde{\D}} _{\PP ^\sharp} ^{(\bullet)} (T))).
\end{equation}
\end{coro}

\begin{proof}
1) Vérifions d'abord l'inclusion $S ^{\mathrm{b}}_{Nqi}\subset S^{\mathrm{b}}$.
Soient $f ^{(\bullet)} \in S  ^{\mathrm{b}}_{Nqi}$ 
et un triangle distingué 
dans $D ^{\mathrm{b}} (\smash{\widetilde{\D}} _{\PP ^\sharp} ^{(\bullet)} (T))$
de la forme 
$\E ^{(\bullet)} 
\overset{f ^{(\bullet)}}{\longrightarrow}
\FF ^{(\bullet)} 
\to 
\G  ^{(\bullet)} \to \E ^{(\bullet)}  [1]$.
Par définition, 
pour tout entier $n \in \Z$, 
$\mathcal{H} ^{n} (\G  ^{(\bullet)}) \in N  (\smash{\widetilde{\D}} _{\PP ^\sharp} ^{(\bullet)} (T))$.
Autrement dit, pour tout entier $n \in \Z$, on a
$\mathcal{H} ^{n}(\G ^{(\bullet)})\riso 0$ dans $\underrightarrow{LM} _{\Q} (\smash{\widetilde{\D}} _{\PP ^\sharp} ^{(\bullet)} (T))$.
D'après \ref{eq-cat-LM-LD0}, cela entraîne que 
$\G ^{(\bullet)}\riso 0$ dans $\underrightarrow{LD} ^{\mathrm{b}} _{\Q} (\smash{\widetilde{\D}} _{\PP ^\sharp} ^{(\bullet)} (T))$.
D'après les propriétés concernant les catégories triangulées,
$f ^{(\bullet)}$ est donc un isomorphisme dans $\underrightarrow{LD} ^{\mathrm{b}} _{\Q} (\smash{\widetilde{\D}} _{\PP ^\sharp} ^{(\bullet)} (T))$.
D'après la remarque \ref{rema-iso-0}, on obtient alors que
$f ^{(\bullet)}\in S ^{\mathrm{b}}$.

2) Vérifions à présent l'inclusion $S ^{\sharp}\subset S ^{\sharp} _{Nqi}$.
Soit $f ^{(\bullet)} \colon
\E ^{(\bullet)} 
\to
\FF ^{(\bullet)} $ 
un morphisme de 
$D ^{\sharp} (\smash{\widetilde{\D}} _{\PP ^\sharp} ^{(\bullet)} (T))$.
Comme le foncteur 
espace de cohomologie 
$\mathcal{H} ^{0}\colon 
D ^{\sharp} (\smash{\widetilde{\D}} _{\PP ^\sharp} ^{(\bullet)} (T))
\to 
\underrightarrow{LM} _{\Q} (\smash{\widetilde{\D}} _{\PP ^\sharp} ^{(\bullet)} (T))$
est un foncteur cohomologique local, 
en considérant la suite exacte longue 
dans 
$\underrightarrow{LM} _{\Q} (\smash{\widetilde{\D}} _{\PP ^\sharp} ^{(\bullet)} (T))$
déduite du triangle distingué 
dans $D ^{\sharp} (\smash{\widetilde{\D}} _{\PP ^\sharp} ^{(\bullet)} (T))$
de la forme 
$\E ^{(\bullet)} 
\overset{f ^{(\bullet)}}{\longrightarrow}
\FF ^{(\bullet)} 
\to 
\G  ^{(\bullet)} \to \E ^{(\bullet)}  [1]$, 
on vérifie que
$f ^{(\bullet)}  \in S ^{\sharp} _{Nqi}$ si et seulement si, 
pour tout entier $n \in \Z$, 
$\mathcal{H} ^{n} (f ^{(\bullet)})$ est un isomorphisme
dans 
$\underrightarrow{LM} _{\Q} (\smash{\widetilde{\D}} _{\PP ^\sharp} ^{(\bullet)} (T))$ (qui est une catégorie abélienne).
Or, si $f ^{(\bullet)}\in S^{\sharp}$, alors son image  
dans $\underrightarrow{LD}  ^{\sharp} _{\Q} (\smash{\widetilde{\D}} _{\PP ^\sharp} ^{(\bullet)} (T))$
est un isomorphisme.
Comme le foncteur $\mathcal{H} ^{n} \colon 
D ^{\sharp} (\smash{\widetilde{\D}} _{\PP ^\sharp} ^{(\bullet)} (T))
\to 
\underrightarrow{LM} _{\Q} (\smash{\widetilde{\D}} _{\PP ^\sharp} ^{(\bullet)} (T))$ se factorise 
par 
$\underrightarrow{LD}  _{\Q} ^{\sharp} (\smash{\widetilde{\D}} _{\PP ^\sharp} ^{(\bullet)} (T))
\to 
\underrightarrow{LM} _{\Q} (\smash{\widetilde{\D}} _{\PP ^\sharp} ^{(\bullet)} (T))$, 
on en déduit l'inclusion $S^{\sharp} \subset S ^{\sharp}_{Nqi}$ voulue.

3) Enfin l'équivalence canonique de catégories résulte de celle
$S _{Nqi} ^{\sharp -1} D ^{\sharp} (\smash{\widetilde{\D}} _{\PP ^\sharp} ^{(\bullet)} (T)) 
\cong 
D ^{\sharp}
(\underrightarrow{LM} _{\Q} (\smash{\widetilde{\D}} _{\PP ^\sharp} ^{(\bullet)} (T)))$ de \ref{rema-Sqi-eqcat}
ainsi que du lemme \ref{S=LDQ}.
\end{proof}

\begin{coro}
\label{LDLMiso=isoHn}
Soit $\phi \colon 
\E ^{(\bullet)}
\to 
\FF ^{(\bullet)}$
un morphisme dans 
$\underrightarrow{LD} ^{\mathrm{b}} _{\Q} (\smash{\widetilde{\D}} _{\PP ^\sharp} ^{(\bullet)} (T))$.
Le morphisme $\phi$ est un isomorphisme 
dans $\underrightarrow{LD} ^{\mathrm{b}} _{\Q} (\smash{\widetilde{\D}} _{\PP ^\sharp} ^{(\bullet)} (T))$
si et seulement si,
pour tout entier $n \in \Z$, 
le morphisme 
$\mathcal{H} ^{n} (\phi) \colon 
\mathcal{H} ^{n}  (\E ^{(\bullet)})
\to 
\mathcal{H} ^{n} (\FF ^{(\bullet)})$
est un isomorphisme 
de 
$\underrightarrow{LM}  _{\Q} (\smash{\widetilde{\D}} _{\PP ^\sharp} ^{(\bullet)} (T))$.
\end{coro}

\begin{proof}
Il existe un triangle distingué dans
$\underrightarrow{LD} ^{\mathrm{b}} _{\Q} (\smash{\widetilde{\D}} _{\PP ^\sharp} ^{(\bullet)} (T))$
de la forme
 $\E ^{(\bullet)}
\underset{\phi}{\longrightarrow} 
\FF ^{(\bullet)}
\to 
\G ^{(\bullet)}
\to 
\E ^{(\bullet)}[1]$.
D'après les propriétés concernant les catégories triangulées,
$\phi$ est un isomorphisme si et seulement si 
$\G ^{(\bullet)}\riso 0$ dans $\underrightarrow{LD} ^{\mathrm{b}} _{\Q} (\smash{\widetilde{\D}} _{\PP ^\sharp} ^{(\bullet)} (T))$.
D'après \ref{eq-cat-LM-LD0}, cela est équivalent à
dire que, pour tout entier $n \in \Z$, on ait 
$\mathcal{H} ^{n}(\G ^{(\bullet)})\riso 0$ dans $\underrightarrow{LM} _{\Q} (\smash{\widetilde{\D}} _{\PP ^\sharp} ^{(\bullet)} (T))$.
Le lemme \ref{Hn-fctcoho} nous permet de conclure.
\end{proof}

\subsection{Propriétés locales}

\begin{lemm}
\label{Nlocal}
Le fait qu'un objet de 
$M (\smash{\widetilde{\D}} _{\PP ^\sharp} ^{(\bullet)} (T))$
soit un objet de 
$N (\smash{\widetilde{\D}} _{\PP ^\sharp} ^{(\bullet)} (T))$
est local en $\PP$.
\end{lemm}

\begin{proof}
Soit $\E ^{(\bullet)}$ un $\smash{\widetilde{\D}} _{\PP ^\sharp} ^{(\bullet)} (T)$-module. 
D'après l'analogue de la seconde remarque 
\ref{rema-iso-0} (on remplace les complexes par les modules),
le fait que $\E ^{(\bullet)}$ soit un objet de 
$N (\smash{\widetilde{\D}} _{\PP ^\sharp} ^{(\bullet)} (T))$
équivaut à dire 
qu'il existe $\chi \in M$ et $\lambda \in L$
tels que la flèche canonique
$\E ^{(\bullet)}  \to \lambda ^* \chi ^* \E ^{(\bullet)} $
soit le morphisme nul (dans la catégorie $M (\smash{\widetilde{\D}} _{\PP ^\sharp} ^{(\bullet)} (T))$). 
Soit $(\PP _i) _{i \in I}$ un recouvrement ouvert de $\PP $ tel que, pour tout $i\in I$, il existe 
$\chi _i \in M$ et $\lambda _i \in L$
tels que la flèche canonique
$\E ^{(\bullet)} |\PP _i  \to \lambda ^* _i  \chi ^* _i  \E ^{(\bullet)} |\PP _i $
soit le morphisme nul. 
Comme $\PP$ est quasi-compact, on peut supposer $I$ fini.
Il existe donc $\chi \in M$ et $\lambda \in L$ tels que, pour tout $i \in I$, 
on ait $\chi \geq \chi _i$ et $\lambda \geq \lambda _i$.
Pour tout $i \in I$, le morphisme canonique
$\E ^{(\bullet)} |\PP _i  \to \lambda ^* \chi ^* \E ^{(\bullet)} |\PP _i $
est donc le morphisme nul. 

\end{proof}

\begin{lemm}
\label{LMQ-iso:local}
Soit $f\colon \E ^{(\bullet)} \to \FF ^{(\bullet)}$ un morphisme de 
$\underrightarrow{LM} _{\Q} (\smash{\widetilde{\D}} _{\PP ^\sharp} ^{(\bullet)} (T))$. 
Le fait que ce morphisme soit un monomorphisme (resp. un épimorphisme, resp. un isomorphisme) 
est local en $\PP$.
\end{lemm}

\begin{proof}
On sait déjà que 
$\underrightarrow{LM} _{\Q} (\smash{\widetilde{\D}} _{\PP ^\sharp} ^{(\bullet)} (T))$
est une catégorie abélienne (voir \ref{defi-M-L-Serre}). 
Or, d'après le lemme \ref{Nlocal}, 
le fait d'être nul dans $\underrightarrow{LM} _{\Q} (\smash{\widetilde{\D}} _{\PP ^\sharp} ^{(\bullet)} (T))$ est local en $\PP$.
D'où le résultat.
\end{proof}

\begin{prop}
\label{LDiso-local}
Soit $\phi \colon 
\E ^{(\bullet)}
\to 
\FF ^{(\bullet)}$
un morphisme dans 
$\underrightarrow{LD} ^{\mathrm{b}} _{\Q} (\smash{\widetilde{\D}} _{\PP ^\sharp} ^{(\bullet)} (T))$.
Le fait que le morphisme $\phi$ soit un isomorphisme 
est local en $\PP$.
\end{prop}

\begin{proof}
Cela résulte de \ref{LMQ-iso:local} et \ref{LDLMiso=isoHn}.
\end{proof}

\begin{lemm}
\label{localLD0}
Soit $\E ^{(\bullet)} \in \underrightarrow{LD} ^{\mathrm{b}} _{\Q} (\smash{\widetilde{\D}} _{\PP ^\sharp} ^{(\bullet)} (T))$. 
La propriété 
$\E ^{(\bullet)} \in \underrightarrow{LD} ^0 _{\Q} (\smash{\widetilde{\D}} _{\PP ^\sharp} ^{(\bullet)} (T))$
est locale en $\PP$. 
\end{lemm}

\begin{proof}
D'après 
\ref{LMQ-iso:local},
pour tout entier $n \in \Z$, 
la propriété
$\mathcal{H} ^{n} (\E ^{(\bullet)} ) \riso 0$ 
dans 
$\underrightarrow{LM} _{\Q} (\smash{\widetilde{\D}} _{\PP ^\sharp} ^{(\bullet)} (T))$
est locale en $\PP$.
\end{proof}

\begin{lemm}
\label{LMQ-morph:local}
Soient $\E ^{(\bullet)}, \, \FF ^{(\bullet)}$ deux objets de
$M (\smash{\widetilde{\D}} _{\PP ^\sharp} ^{(\bullet)} (T))$. 
La construction d'un morphisme dans $\underrightarrow{LM} _{\Q} (\smash{\widetilde{\D}} _{\PP ^\sharp} ^{(\bullet)} (T))$
de la forme
$f\colon \E ^{(\bullet)} \to \FF ^{(\bullet)}$ est locale en $\PP$.
\end{lemm}

\begin{proof}
Soient $(\PP _i) _{i \in I}$ un recouvrement ouvert de $\PP $
et pour tout $i\in I$ des morphismes 
$f _i\colon \E ^{(\bullet)} | \PP _i \to \FF ^{(\bullet)} | \PP _i $
dans 
$\underrightarrow{LM} _{\Q} (\smash{\widetilde{\D}} _{\PP _i} ^{(\bullet)} (T \cap P _i))$
tels 
$f _i | \PP _{i} \cap \PP _{j} =f _j | \PP _{i} \cap \PP _{j}$
pour tous $i,j\in I$. 
Comme $\PP$ est quasi-compact, on peut supposer $I$ fini.
Comme une famille finie d'éléments de $L$ (resp. de $M$) est majorée par un élément
de $L$ (resp. de $M$),
quitte à augmenter les éléments de $L$ ou $M$ qui apparaissent 
dans un choix de représentant des morphismes $f _i$, on peut supposer qu'il existe $\lambda \in L$ et $\chi \in M$,
des morphismes 
$a _i\colon \E ^{(\bullet)}| \PP _i \to \lambda ^{*} \chi ^{*} \FF ^{(\bullet)}| \PP _i$ dans 
$M (\smash{\widetilde{\D}} _{\PP _i} ^{(\bullet)} (T\cap P _i))$
représentant $f _i$.
Quitte à nouveau à augmenter $  \lambda $ et $\chi$, on peut en outre supposer
que 
$a _i | \PP _{i} \cap \PP _{j} =a _j | \PP _{i} \cap \PP _{j}$.
On obtient donc un morphisme 
$a\colon \E ^{(\bullet)} \to \lambda ^{*} \chi ^{*} \FF ^{(\bullet)}$ dans 
$M (\smash{\widetilde{\D}} _{\PP ^\sharp} ^{(\bullet)} (T))$
tel que 
$a _i = a  | \PP _{i}$.
D'où le résultat.

\end{proof}

\begin{lemm}
\label{lemm-f=glocal}
Soient $f,~g\colon \E ^{(\bullet)} \to \FF ^{(\bullet)}$ deux morphismes de 
 $\underrightarrow{LM} _{\Q} (\smash{\widetilde{\D}} _{\PP ^\sharp} ^{(\bullet)} (T))$.
 L'égalité $f =g$ dans $\underrightarrow{LM} _{\Q} (\smash{\widetilde{\D}} _{\PP ^\sharp} ^{(\bullet)} (T))$ est locale en $\PP$. 
\end{lemm}

\begin{proof}
L'égalité $f =g$ équivaut à dire que le morphisme canonique 
$\ker (f-g) \to \E ^{(\bullet)} $ est un isomorphisme. 
Cela résulte alors de \ref{LMQ-iso:local}.
\end{proof}

\subsection{Bifoncteurs homomorphismes}

\begin{prop}
\label{const-calHomLMQ}
Soient $\E ^{(\bullet)}, \, \FF ^{(\bullet)}$ deux objets de
 $\underrightarrow{LM} _{\Q} (\smash{\widetilde{\D}} _{\PP ^\sharp} ^{(\bullet)} (T))$.
Notons $\mathcal{H}om _{\underrightarrow{LM} _{\Q} (\smash{\widetilde{\D}} _{\PP ^\sharp} ^{(\bullet)} (T))}
(\E ^{(\bullet)},~\FF ^{(\bullet)})$ le préfaisceau en groupes abéliens sur $\PP$ défini par 
$\U \mapsto \mathrm{Hom} _{\underrightarrow{LM} _{\Q} (\smash{\widetilde{\D}} _{\U ^{\sharp}} ^{(\bullet)} (T \cap U))}
(\E ^{(\bullet)} |\U,~\FF ^{(\bullet)}|\U)$.
Alors $\mathcal{H}om _{\underrightarrow{LM} _{\Q} (\smash{\widetilde{\D}} _{\PP ^\sharp} ^{(\bullet)} (T))}
(\E ^{(\bullet)},~\FF ^{(\bullet)})$ est un faisceau.
\end{prop}
\begin{proof}
La proposition équivaut aux lemmes
\ref{LMQ-morph:local} et \ref{lemm-f=glocal}.
\end{proof}

\begin{nota}
\label{def-Hom-LD}
Notons $\mathrm{Ab} _\PP$
la catégorie abélienne des faisceaux en groupes abéliens sur $\PP$.
Soient $\E ^{(\bullet),~\bullet},~\FF ^{(\bullet),~ \bullet}\in 
K 
(\underrightarrow{LM} _{\Q} (\smash{\widetilde{\D}} _{\PP ^\sharp} ^{(\bullet)} (T)))$
(exceptionnellement, on indique le deuxième $\bullet$ pour préciser les notations qui suivent).
Avec les notations de la proposition \ref{const-calHomLMQ},
on dispose du bifoncteur 
$$\mathcal{H}om ^{\bullet} _{\underrightarrow{LM} _{\Q} (\smash{\widetilde{\D}} _{\PP ^\sharp} ^{(\bullet)} (T))}
(-,-)
\colon 
K 
(\underrightarrow{LM} _{\Q} (\smash{\widetilde{\D}} _{\PP ^\sharp} ^{(\bullet)} (T))) 
\times
K 
(\underrightarrow{LM} _{\Q} (\smash{\widetilde{\D}} _{\PP ^\sharp} ^{(\bullet)} (T))) 
\to
K (\mathrm{Ab} _\PP)$$
dont le $n$-ème terme pour tout entier $n\in \Z$ 
est  défini en posant:
\begin{equation}
\label{def-Hom-nLD}
\mathcal{H}om ^{n}
_{\underrightarrow{LM} _{\Q} (\smash{\widetilde{\D}} _{\PP ^\sharp} ^{(\bullet)} (T))}
(\E ^{(\bullet),~\bullet},~\FF ^{(\bullet),~\bullet})
:= 
\prod _{p\in \Z}
\mathcal{H}om _{\underrightarrow{LM} _{\Q} (\smash{\widetilde{\D}} _{\PP ^\sharp} ^{(\bullet)} (T))}
(\E ^{(\bullet),p},~\FF ^{(\bullet),~p+n})
\end{equation}
et les morphismes de transition sont donnés par la formule
$d = d _{\E}  + (-1) ^{n+1} d _{\FF}$.
\end{nota}

\begin{nota}
\label{def-rmHom-LD}
Notons $\mathrm{Ab}$ la catégorie des groupes abéliens.
De manière analogue à la construction de \ref{def-Hom-LD}, en y 
 remplaçant partout  {\og $\mathcal{H}om$\fg} par {\og $\mathrm{Hom}$\fg}, 
on définit le bifoncteur (qui est d'ailleurs le bifoncteur classique des homomorphismes de la catégorie 
abélienne 
$\underrightarrow{LM} _{\Q} (\smash{\widetilde{\D}} _{\PP ^\sharp} ^{(\bullet)} (T))$):
$$\mathrm{Hom} ^{\bullet} _{\underrightarrow{LM} _{\Q} (\smash{\widetilde{\D}} _{\PP ^\sharp} ^{(\bullet)} (T))}
(-,-)
\colon 
K 
(\underrightarrow{LM} _{\Q} (\smash{\widetilde{\D}} _{\PP ^\sharp} ^{(\bullet)} (T))) 
\times
K 
(\underrightarrow{LM} _{\Q} (\smash{\widetilde{\D}} _{\PP ^\sharp} ^{(\bullet)} (T))) 
\to
K (\mathrm{Ab}).$$

\end{nota}

\begin{lemm}
\label{lemm-trans-fonctderdroi}
Soient $\mathcal{C} _1, ~
\mathcal{C} '_1,~
\mathcal{C}$ des catégories triangulées, 
$\NN _1, ~\NN '_1, \NN$ des systèmes nuls (voir la définition de \cite[10.2.2]{Kashiwara-schapira-book})
de respectivement
$\mathcal{C} _1, ~
\mathcal{C} '_1,~
\mathcal{C}$.
Soient 
$\NN _2, ~\NN '_2$ des systèmes nuls
de respectivement
$\mathcal{C} _2:= \mathcal{C} _1/ \NN _1, ~
\mathcal{C} ' _2:= \mathcal{C} '_1/ \NN '_1$.
Soient 
$\NN _3, ~\NN '_3$ des systèmes nuls
de respectivement
$\mathcal{C} _1, ~
\mathcal{C} ' _1$
tels que $\mathcal{C} _1/ \NN _3 =\mathcal{C} _2/ \NN _2$,
$\mathcal{C} '_1/ \NN '_3 =\mathcal{C} '_2/ \NN '_2$.
Soit $F \colon \CC _1 \times \CC '_1 \to \CC$ un bifoncteur triangulé. 
On suppose que la localisation à droite de $F$ par rapport à 
$( \NN _1 \times ~\NN ' _1, \NN)$ existe (voir la définition \cite[10.3.7]{Kashiwara-schapira-book})
et se note 
$\R _{\NN _1 \times ~\NN ' _1} ^{\NN} F$.
Si l'une des conditions:
\begin{enumerate}
\item la localisation à droite de $F$ par rapport à 
$( \NN _3 \times ~\NN ' _3, \NN)$ existe,
\item la localisation à droite de $\R _{\NN _1 \times ~\NN ' _1} ^{\NN} F$ par rapport à 
$( \NN _2 \times ~\NN ' _2, \NN)$ existe
\end{enumerate}
est satisfaite, alors la seconde aussi et on dispose dans ce cas de l'isomorphisme de bifoncteurs 
$$\R _{\NN _2 \times ~\NN ' _2} ^{\NN} \R _{\NN _1 \times ~\NN ' _1} ^{\NN} F
\riso
\R _{\NN _3 \times ~\NN ' _3} ^{\NN} F.$$

\end{lemm}

\begin{proof}
Cela découle aussitôt de la propriété universelle des localisations à droite.
\end{proof}

\begin{nota}
[Bifoncteur faisceau des homomorphismes de Berthelot]
\label{def-Hom-LD-Berthelot}
Soient $\E ^{(\bullet),~\bullet}\in 
K ^{-} 
(\smash{\widetilde{\D}} _{\PP ^\sharp} ^{(\bullet)} (T))$,
$\FF ^{(\bullet),~ \bullet}\in 
K ^{+}
(\smash{\widetilde{\D}} _{\PP ^\sharp} ^{(\bullet)} (T))$.
Berthelot a considéré le bifoncteur
$$\underset{\lambda \in L,~\chi \in M}{\underrightarrow{\lim}}~  
\mathcal{H}om ^{\bullet} _{\smash{\widetilde{\D}} _{\PP ^\sharp} ^{(\bullet)} (T))}
(-,~ \lambda ^{*} \chi ^{*}-)\colon 
K ^{-} 
(\smash{\widetilde{\D}} _{\PP ^\sharp} ^{(\bullet)} (T))^{\circ}
\times
K ^{+}
(\smash{\widetilde{\D}} _{\PP ^\sharp} ^{(\bullet)} (T))
\to
K (\mathrm{Ab} _\PP)$$
dont le $n$-ème terme pour tout entier $n\in \Z$ 
est  défini en posant 
\begin{equation}
\label{def-Hom-nLD}
\underset{\lambda \in L,~\chi \in M}{\underrightarrow{\lim}}~  \mathcal{H}om ^{n}
_{\smash{\widetilde{\D}} _{\PP ^\sharp} ^{(\bullet)} (T))}
(\E ^{(\bullet),~\bullet},~ \lambda ^{*} \chi ^{*}\FF ^{(\bullet),~\bullet})
:= 
\underset{\lambda \in L,~\chi \in M}{\underrightarrow{\lim}}~ 
\prod _{p\in \Z}
\mathcal{H}om _{\smash{\widetilde{\D}} _{\PP ^\sharp} ^{(\bullet)} (T))}
(\E ^{(\bullet),p},~ \lambda ^{*} \chi ^{*}\FF ^{(\bullet),~p+n})
\end{equation}
et les morphismes de transition sont donnés par la formule
$d = d _{\E}  + (-1) ^{n+1} d _{\FF}$.
Berthelot a vérifié que ce bifoncteur
$\underset{\lambda \in L,~\chi \in M}{\underrightarrow{\lim}}~  
\mathcal{H}om ^{\bullet} _{\smash{\widetilde{\D}} _{\PP ^\sharp} ^{(\bullet)} (T))}
(-,~ \lambda ^{*} \chi ^{*}-)$
est localisable à droite en le dérivé droit 
de la forme
\begin{equation}
\label{RcalHomBerth}
\underset{\lambda \in L,~\chi \in M}{\R \underrightarrow{\lim}}~  \mathcal{H}om ^{\bullet} _{\smash{\widetilde{\D}} _{\PP ^\sharp} ^{(\bullet)} (T))}
(-,~ \lambda ^{*} \chi ^{*}-)
\colon 
\smash{\underrightarrow{LD}} ^{\mathrm{-}} _{\Q} ( \smash{\widetilde{\D}} _{\PP ^\sharp} ^{(\bullet)}(T)) ^{\circ} 
\times 
\smash{\underrightarrow{LD}} ^{\mathrm{+}} _{\Q} ( \smash{\widetilde{\D}} _{\PP ^\sharp} ^{(\bullet)}(T))
\to
D (\mathrm{Ab} _\PP).
\end{equation}
Il note 
$\R \mathcal{H}om _{\smash{\underrightarrow{LD} } _{\Q} ( \smash{\widetilde{\D}} _{\PP ^\sharp} ^{(\bullet)}(T))}
(-,-) 
:=
\underset{\lambda \in L,~\chi \in M}{\R \underrightarrow{\lim}}~  \mathcal{H}om ^{\bullet} _{\smash{\widetilde{\D}} _{\PP ^\sharp} ^{(\bullet)} (T))}
(-,~ \lambda ^{*} \chi ^{*}-)$, ce foncteur dérivé droit qui se résout 
(i.e., on applique le théorème d'existence \cite[10.3.9]{Kashiwara-schapira-book}) en prenant des résolutions injectives
du terme de droite et des résolutions plates d'un certain type du terme de gauche.

\end{nota}

\begin{vide}
[Bifoncteur homomorphisme de Berthelot]
\label{bif-HomrmBert-const}
Soient $\E ^{(\bullet)}\in 
\smash{\underrightarrow{LD}} ^{-} _{\Q} ( \smash{\widetilde{\D}} _{\PP ^\sharp} ^{(\bullet)}(T))$,
$\FF ^{(\bullet)}
\in 
\smash{\underrightarrow{LD}} ^{+} _{\Q} ( \smash{\widetilde{\D}} _{\PP ^\sharp} ^{(\bullet)}(T))$.
De manière analogue à \ref{def-Hom-LD-Berthelot}, 
en remplaçant 
$\mathcal{H}om$ par $\mathrm{Hom}$, 
Berthelot a construit le bifoncteur 
\begin{equation}
\label{RHomLD}
\R \mathrm{Hom} _{\smash{\underrightarrow{LD} } _{\Q} ( \smash{\widetilde{\D}} _{\PP ^\sharp} ^{(\bullet)}(T))}
(-, - )\colon 
\smash{\underrightarrow{LD}} ^{\mathrm{-}} _{\Q} ( \smash{\widetilde{\D}} _{\PP ^\sharp} ^{(\bullet)}(T)) ^{\circ} 
\times 
\smash{\underrightarrow{LD}} ^{\mathrm{+}} _{\Q} ( \smash{\widetilde{\D}} _{\PP ^\sharp} ^{(\bullet)}(T))
\to
D ^{+}(\mathrm{Ab})
\end{equation}
tel que 
$H ^{0} (\R \mathrm{Hom} _{\smash{\underrightarrow{LD}}  _{\Q} ( \smash{\widetilde{\D}} _{\PP ^\sharp} ^{(\bullet)}(T))}
(\E ^{(\bullet)}, \FF ^{(\bullet)} ))
= 
\mathrm{Hom} _{\smash{\underrightarrow{LD}}  _{\Q} ( \smash{\widetilde{\D}} _{\PP ^\sharp} ^{(\bullet)}(T))}
(\E ^{(\bullet)}, \FF ^{(\bullet)} )$.
Il a de plus vérifié l'isomorphisme de bifoncteurs
\begin{equation}
\label{Berth-Rhomcal-vs-rm}
\R \Gamma (\PP, -) \circ \R \mathcal{H}om _{\smash{\underrightarrow{LD} } _{\Q} ( \smash{\widetilde{\D}} _{\PP ^\sharp} ^{(\bullet)}(T))}
(-,-) \riso
\R \mathrm{Hom} _{\smash{\underrightarrow{LD} } _{\Q} ( \smash{\widetilde{\D}} _{\PP ^\sharp} ^{(\bullet)}(T))}
(-,-).
\end{equation}

\end{vide}

\begin{vide}
\label{exist-RcalHom}
Pour $\sharp \in \{ -,~+, \mathrm{b} \}$,
notons 
$Q _N
\colon 
K ^{\sharp} (\smash{\widetilde{\D}} _{\PP ^\sharp} ^{(\bullet)} (T))
\to 
K ^{\sharp} (\underrightarrow{LM} _{\Q} (\smash{\widetilde{\D}} _{\PP ^\sharp} ^{(\bullet)} (T)))$
le foncteur canonique.
D'après la remarque après le théorème \cite[3.2]{Miyachi-localization}, le foncteur $Q _N$
est un quotient. Soit $\NN _N$ le système nul correspondant à ce quotient.  
De plus, le foncteur $Q _N$ induit par localisation l'équivalence de catégories de
\ref{eqcatLD=DSM-fonct} notée 
$Q _{eq}\colon \underrightarrow{LD} ^{\mathrm{b}} _{\Q} (\smash{\widetilde{\D}} _{\PP ^\sharp} ^{(\bullet)} (T))
\cong 
D ^{\mathrm{b}}
(\underrightarrow{LM} _{\Q} (\smash{\widetilde{\D}} _{\PP ^\sharp} ^{(\bullet)} (T)))$.
Notons enfin $\QQ is $ le système nul des quasi-isomorphismes de 
$K (\mathrm{Ab} _\PP)$ et $Q \colon K (\mathrm{Ab} _\PP) \to D (\mathrm{Ab} _\PP)$,
le foncteur canonique.
Avec les notations des paragraphes \ref{def-Hom-LD} et \ref{def-Hom-LD-Berthelot},
on dispose de l'égalité
$$\underset{\lambda \in L,~\chi \in M}{\underrightarrow{\lim}}~  
\mathcal{H}om ^{\bullet} _{\smash{\widetilde{\D}} _{\PP ^\sharp} ^{(\bullet)} (T))}
(-,~ \lambda ^{*} \chi ^{*}-)
=
\mathcal{H}om ^{\bullet} _{\underrightarrow{LM} _{\Q} (\smash{\widetilde{\D}} _{\PP ^\sharp} ^{(\bullet)} (T))}
(Q _N (-),Q _N (-)).$$
En d'autres termes, 
la localisation à droite du bifoncteur  
$\underset{\lambda \in L,~\chi \in M}{\underrightarrow{\lim}}~  
\mathcal{H}om ^{\bullet} _{\smash{\widetilde{\D}} _{\PP ^\sharp} ^{(\bullet)} (T))}
(-,~ \lambda ^{*} \chi ^{*}-)$
par rapport à $(\NN _N\times \NN _N, \QQ is)$ existe et est égal à
$Q \circ \mathcal{H}om ^{\bullet} _{\underrightarrow{LM} _{\Q} (\smash{\widetilde{\D}} _{\PP ^\sharp} ^{(\bullet)} (T))}
(-,-)$.
Grâce à l'existence vérifiée par Berthelot 
du foncteur dérivé droit de $\underset{\lambda \in L,~\chi \in M}{\underrightarrow{\lim}}~  
\mathcal{H}om ^{\bullet} _{\smash{\widetilde{\D}} _{\PP ^\sharp} ^{(\bullet)} (T))}
(-,~ \lambda ^{*} \chi ^{*}-)$
(voir \ref{def-Hom-LD-Berthelot}),
grâce à l'équivalence de catégories 
$Q _{eq}$,
en utilisant le lemme \ref{lemm-trans-fonctderdroi},
on obtient alors que le bifoncteur $\mathcal{H}om ^{\bullet} _{\underrightarrow{LM} _{\Q} (\smash{\widetilde{\D}} _{\PP ^\sharp} ^{(\bullet)} (T))} (-,-)$ 
est localisable à droite en le bifoncteur que l'on notera
$$\R \mathcal{H}om _{D (\underrightarrow{LM} _{\Q} (\smash{\widetilde{\D}} _{\PP ^\sharp} ^{(\bullet)} (T)))} (-,-)
\colon 
D ^{\mathrm{b}}(\underrightarrow{LM} _{\Q} (\smash{\widetilde{\D}} _{\PP ^\sharp} ^{(\bullet)} (T))) ^\circ 
\times
D ^{\mathrm{b}} (\underrightarrow{LM} _{\Q} (\smash{\widetilde{\D}} _{\PP ^\sharp} ^{(\bullet)} (T)))
\to 
D (\mathrm{Ab} _\PP).$$
Par propriété universelle des foncteurs dérivés droits, 
on dispose de l'isomorphisme canonique de bifoncteurs 
$\underrightarrow{LD} ^{\mathrm{b}} _{\Q} (\smash{\widetilde{\D}} _{\PP ^\sharp} ^{(\bullet)} (T)) ^\circ 
\times
\underrightarrow{LD} ^{\mathrm{b}} _{\Q} (\smash{\widetilde{\D}} _{\PP ^\sharp} ^{(\bullet)} (T)) 
\to 
D (\mathrm{Ab} _\PP)$
de la forme 
$$\R \mathcal{H}om _{\smash{\underrightarrow{LD} } _{\Q} ( \smash{\widetilde{\D}} _{\PP ^\sharp} ^{(\bullet)}(T))} (-,-)
\riso
\R \mathcal{H}om _{D (\underrightarrow{LM} _{\Q} (\smash{\widetilde{\D}} _{\PP ^\sharp} ^{(\bullet)} (T)))} (Q _{eq}(-),Q _{eq}(-)).$$
\end{vide}

\begin{vide}
\label{exist-RrmHom}
De manière identique à \ref{exist-RcalHom}, on vérifie que 
le bifoncteur $\mathrm{Hom} ^{\bullet} _{\underrightarrow{LM} _{\Q} (\smash{\widetilde{\D}} _{\PP ^\sharp} ^{(\bullet)} (T))} (-,-)$ 
est localisable à droite en le bifoncteur que l'on notera
$$\R \mathrm{Hom} _{D (\underrightarrow{LM} _{\Q} (\smash{\widetilde{\D}} _{\PP ^\sharp} ^{(\bullet)} (T)))} (-,-)
\colon 
D ^{\mathrm{b}}(\underrightarrow{LM} _{\Q} (\smash{\widetilde{\D}} _{\PP ^\sharp} ^{(\bullet)} (T))) ^\circ 
\times
D ^{\mathrm{b}} (\underrightarrow{LM} _{\Q} (\smash{\widetilde{\D}} _{\PP ^\sharp} ^{(\bullet)} (T)))
\to 
D (\mathrm{Ab}).$$
Par propriété universelle des foncteurs dérivés droits, 
on dispose de plus de l'isomorphisme canonique de bifoncteurs 
$\underrightarrow{LD} ^{\mathrm{b}} _{\Q} (\smash{\widetilde{\D}} _{\PP ^\sharp} ^{(\bullet)} (T)) ^\circ 
\times
\underrightarrow{LD} ^{\mathrm{b}} _{\Q} (\smash{\widetilde{\D}} _{\PP ^\sharp} ^{(\bullet)} (T)) 
\to 
D (\mathrm{Ab} )$
de la forme 
$$\R \mathrm{Hom} _{\smash{\underrightarrow{LD} } _{\Q} ( \smash{\widetilde{\D}} _{\PP ^\sharp} ^{(\bullet)}(T))} (-,-)
\riso
\R \mathrm{Hom} _{D (\underrightarrow{LM} _{\Q} (\smash{\widetilde{\D}} _{\PP ^\sharp} ^{(\bullet)} (T)))} (Q _{eq}(-),Q _{eq}(-)).$$
Ce morphisme provient du morphisme analogue sans les dérivés droits. 
On déduit alors de \ref{bif-HomrmBert-const} que l'on dispose de 
l'isomorphisme de bifoncteurs
$D ^{\mathrm{b}}(\underrightarrow{LM} _{\Q} (\smash{\widetilde{\D}} _{\PP ^\sharp} ^{(\bullet)} (T))) ^\circ 
\times
D ^{\mathrm{b}} (\underrightarrow{LM} _{\Q} (\smash{\widetilde{\D}} _{\PP ^\sharp} ^{(\bullet)} (T)))
\to 
\mathrm{Ab} $ de la forme:
\begin{equation}
\label{H0Homrm-DLM}
\mathcal{H} ^{0} (\R \mathrm{Hom} _{D (\underrightarrow{LM} _{\Q} (\smash{\widetilde{\D}} _{\PP ^\sharp} ^{(\bullet)} (T)))} (-,-))
\riso 
\mathrm{Hom} _{D (\underrightarrow{LM} _{\Q} (\smash{\widetilde{\D}} _{\PP ^\sharp} ^{(\bullet)} (T)))} (-,-).
\end{equation}

On déduit de \ref{Berth-Rhomcal-vs-rm} l'isomorphisme 
de bifoncteurs 
$D ^{\mathrm{b}} (\underrightarrow{LM} _{\Q} (\smash{\widetilde{\D}} _{\PP ^\sharp} ^{(\bullet)} (T))) ^\circ 
\times
D ^{\mathrm{b}} (\underrightarrow{LM} _{\Q} (\smash{\widetilde{\D}} _{\PP ^\sharp} ^{(\bullet)} (T)))\to D (\mathrm{Ab})$ de la forme
\begin{equation}
\label{Rhomcal-vs-rm}
\R \Gamma (\PP,- )\circ \R \mathcal{H}om _{D (\underrightarrow{LM} _{\Q} (\smash{\widetilde{\D}} _{\PP ^\sharp} ^{(\bullet)} (T)))} (-,-)
\riso 
\R \mathrm{Hom}_{D (\underrightarrow{LM} _{\Q} (\smash{\widetilde{\D}} _{\PP ^\sharp} ^{(\bullet)} (T)))} (-,-).
\end{equation}
\end{vide}

\begin{vide}
Notons $M ( \smash{\D} ^\dag _{\PP ^\sharp} (\hdag T) _{\Q} )$ 
la catégorie abélienne des 
$\smash{\D} ^\dag _{\PP ^\sharp} (\hdag T) _{\Q}$-modules.
En tensorisant par $\Q$ puis en passant à la limite inductive sur le niveau, on obtient le foncteur
$\underrightarrow{\lim} 
\colon
M
(\smash{\widetilde{\D}} _{\PP ^\sharp} ^{(\bullet)}(T))
\to
M ( \smash{\D} ^\dag _{\PP ^\sharp} (\hdag T) _{\Q} )$.
Comme il transforme les lim-ind-isomorphismes en isomorphisme, il se factorise canoniquement en le foncteur
\begin{equation}
\label{M-eq-lim}
\underrightarrow{\lim} 
\colon
\smash{\underrightarrow{LM}}  _{\Q}
(\smash{\widetilde{\D}} _{\PP ^\sharp} ^{(\bullet)}(T))
\to
M ( \smash{\D} ^\dag _{\PP ^\sharp} (\hdag T) _{\Q} ).
\end{equation}
\end{vide}

\begin{vide}
Soient $\E ^{(\bullet)}, \, \FF ^{(\bullet)}$ deux objets de
$K ^{\mathrm{b}} (\underrightarrow{LM} _{\Q} (\smash{\widetilde{\D}} _{\PP ^\sharp} ^{(\bullet)} (T)))$.
Via le foncteur 
$\underrightarrow{\lim} 
\colon
K ^{\mathrm{b}} (\underrightarrow{LM} _{\Q} (\smash{\widetilde{\D}} _{\PP ^\sharp} ^{(\bullet)} (T)))
\to
K ^{\mathrm{b}} ( \smash{\D} ^\dag _{\PP ^\sharp} (\hdag T) _{\Q} )$
déduit \ref{M-eq-lim},
on obtient le morphisme de $K (\mathrm{Ab} _\PP)$ suivant:
\begin{equation}
\label{pre-fleche-RhomLM}
\mathcal{H}om ^{\bullet} _{\underrightarrow{LM} _{\Q} (\smash{\widetilde{\D}} _{\PP ^\sharp} ^{(\bullet)} (T))}
(\E ^{(\bullet)},~\FF ^{(\bullet)})
\to 
\mathcal{H}om ^{\bullet} _{\smash{\D} ^\dag _{\PP ^\sharp} (\hdag T) _{\Q}}
(\underrightarrow{\lim} ~\E ^{(\bullet)},\underrightarrow{\lim} ~\FF ^{(\bullet)}).
\end{equation}
Comme le foncteur 
$\underrightarrow{\lim} 
\colon
K ^{\mathrm{b}} (\underrightarrow{LM} _{\Q} (\smash{\widetilde{\D}} _{\PP ^\sharp} ^{(\bullet)} (T)))
\to
K ^{\mathrm{b}} ( \smash{\D} ^\dag _{\PP ^\sharp} (\hdag T) _{\Q} )$ transforme complexe acyclique en complexe acyclique, 
il se factorise donc en le foncteur
$\underrightarrow{\lim} 
\colon
D ^{\mathrm{b}} (\underrightarrow{LM} _{\Q} (\smash{\widetilde{\D}} _{\PP ^\sharp} ^{(\bullet)} (T)))
\to
D ^{\mathrm{b}} ( \smash{\D} ^\dag _{\PP ^\sharp} (\hdag T) _{\Q} )$
qui commute au foncteur canonique de localisation
$Q \colon 
K ^{\mathrm{b}} (\underrightarrow{LM} _{\Q} (\smash{\widetilde{\D}} _{\PP ^\sharp} ^{(\bullet)} (T)))
\to 
D ^{\mathrm{b}} (\underrightarrow{LM} _{\Q} (\smash{\widetilde{\D}} _{\PP ^\sharp} ^{(\bullet)} (T)))$
et
$Q\colon K ^{\mathrm{b}} ( \smash{\D} ^\dag _{\PP ^\sharp} (\hdag T) _{\Q} )
\to 
D ^{\mathrm{b}} (\smash{\D} ^\dag _{\PP ^\sharp} (\hdag T) _{\Q} )$, 
i.e. 
$Q \circ \underrightarrow{\lim} \riso \underrightarrow{\lim}  \circ Q$.
On dispose ainsi des morphisme de bifoncteurs
$K ^{\mathrm{b}} 
(\underrightarrow{LM} _{\Q} (\smash{\widetilde{\D}} _{\PP ^\sharp} ^{(\bullet)} (T))) 
\times
K ^{\mathrm{b}} 
(\underrightarrow{LM} _{\Q} (\smash{\widetilde{\D}} _{\PP ^\sharp} ^{(\bullet)} (T))) 
\to
D (\mathrm{Ab} _\PP)$:
$$Q \circ \mathcal{H}om ^{\bullet} _{\underrightarrow{LM} _{\Q} (\smash{\widetilde{\D}} _{\PP ^\sharp} ^{(\bullet)} (T))}
\underset{\ref{pre-fleche-RhomLM}}{\longrightarrow}
Q \circ \mathcal{H}om ^\bullet _{ \smash{\D} ^\dag _{\PP ^\sharp} (\hdag T) _{\Q} } 
\circ \underrightarrow{\lim} 
\to 
\R\mathcal{H}om _{ \smash{\D} ^\dag _{\PP ^\sharp} (\hdag T) _{\Q} } 
\circ Q \circ \underrightarrow{\lim}
\riso 
\R\mathcal{H}om _{ \smash{\D} ^\dag _{\PP ^\sharp} (\hdag T) _{\Q} } 
\circ \underrightarrow{\lim}\circ Q .$$
Par propriété universelle du bifoncteur dérivé droit, il existe donc un morphisme canonique 
de la forme
\begin{equation}
\label{fleche-RhomLM}
\R \mathcal{H}om _{D (\underrightarrow{LM} _{\Q} (\smash{\widetilde{\D}} _{\PP ^\sharp} ^{(\bullet)} (T)))} 
(-,-)
\to
\R\mathcal{H}om _{ \smash{\D} ^\dag _{\PP ^\sharp} (\hdag T) _{\Q} } 
(-,- ) \circ \underrightarrow{\lim}
\end{equation}
de bifoncteurs 
$D ^{\mathrm{b}} 
(\underrightarrow{LM} _{\Q} (\smash{\widetilde{\D}} _{\PP ^\sharp} ^{(\bullet)} (T))) 
\times
D ^{\mathrm{b}}  
(\underrightarrow{LM} _{\Q} (\smash{\widetilde{\D}} _{\PP ^\sharp} ^{(\bullet)} (T))) 
\to
D (\mathrm{Ab} _\PP)$.

Par propriété universelle du bifoncteur dérivé droit, on vérifie de la même manière qu'il existe un morphisme canonique 
de la forme
\begin{equation}
\label{fleche-mathrmRhomLM}
\R \mathrm{Hom} _{D (\underrightarrow{LM} _{\Q} (\smash{\widetilde{\D}} _{\PP ^\sharp} ^{(\bullet)} (T)))} 
(-,-)
\to
\R \mathrm{Hom} _{ \smash{\D} ^\dag _{\PP ^\sharp} (\hdag T) _{\Q} } 
(-,- ) \circ \underrightarrow{\lim}
\end{equation}
de bifoncteurs 
$D ^{\mathrm{b}} 
(\underrightarrow{LM} _{\Q} (\smash{\widetilde{\D}} _{\PP ^\sharp} ^{(\bullet)} (T))) 
\times
D ^{\mathrm{b}}  
(\underrightarrow{LM} _{\Q} (\smash{\widetilde{\D}} _{\PP ^\sharp} ^{(\bullet)} (T))) 
\to
D (\mathrm{Ab})$.
\end{vide}

\begin{rema}
Si la catégorie abélienne
$\underrightarrow{LM} _{\Q} (\smash{\widetilde{\D}} _{\PP ^\sharp} ^{(\bullet)} (T))$
possèdait assez d'injectifs (ce que j'ignore),
avec les deux lemmes qui suivent, on pourrait alors résoudre comme d'habitude le bifoncteur 
$\mathcal{H}om ^{\bullet} _{\underrightarrow{LM} _{\Q} (\smash{\widetilde{\D}} _{\PP ^\sharp} ^{(\bullet)} (T))}(-,-)$.
\end{rema}

\begin{vide}
[Stabilité de l'injectivité par image inverse par une immersion ouverte]
\label{stab-inj-j*}
Soit $j\colon \U \hookrightarrow \PP$ une immersion ouverte de $\V$-schémas formels lisses. 
Le foncteur canonique
$j ^{*}\colon 
M (\smash{\widetilde{\D}} _{\PP ^\sharp} ^{(\bullet)} (T))
\to
M (\smash{\widetilde{\D}} _{\U ^{\sharp}} ^{(\bullet)} (T \cap U))$
possède un adjoint à gauche exact (le foncteur extension par zéro en dehors de $U$)
que l'on notera, 
afin de ne pas le confondre avec le foncteur image directe extraordinaire,
$j ^{\mathrm{top}} _!
\colon
M (\smash{\widetilde{\D}} _{\U ^{\sharp}} ^{(\bullet)} (T \cap U))
\to
M (\smash{\widetilde{\D}} _{\PP ^\sharp} ^{(\bullet)} (T ))$.
Pour tout $\E ^{(\bullet)}
\in M (\smash{\widetilde{\D}} _{\U ^{\sharp}} ^{(\bullet)} (T \cap U))$,
on dispose de l'isomorphisme canonique 
$j ^{\mathrm{top}} _! \lambda ^{*} \chi ^{*} (\E ^{(\bullet)})
\riso
\lambda ^{*} \chi ^{*}j ^{\mathrm{top}} _! (\E ^{(\bullet)})$
qui s'inscrit dans le diagramme commutatif
\begin{equation}
\notag
\xymatrix @ R=0,4cm{
{j ^{\mathrm{top}} _! (\E ^{(\bullet)})}
\ar@{=}[d] ^-{} 
\ar[r] ^-{}
& 
{j ^{\mathrm{top}} _! \lambda ^{*} \chi ^{*} (\E ^{(\bullet)})} 
\ar[d] ^-{\sim}
\\ 
{j ^{\mathrm{top}} _! (\E ^{(\bullet)})} 
\ar[r] ^-{}
& 
{\lambda ^{*} \chi ^{*}j ^{\mathrm{top}} _! (\E ^{(\bullet)})} 
}
\end{equation}
dont la flèche horizontale du bas est la flèche canonique
tandis que celle du haut est l'image par $j ^{\mathrm{top}} _! $
du morphisme canonique. On dispose des mêmes propriétés concernant $j ^*$.
Soient 
$\E ^{(\bullet)}
\in 
M (\smash{\widetilde{\D}} _{\U ^{\sharp}} ^{(\bullet)} (T \cap U))$,
$\FF ^{(\bullet)}
\in 
M (\smash{\widetilde{\D}} _{\PP ^\sharp} ^{(\bullet)} (T))$,
$\lambda \in L$, $\chi \in M$.
On a l'égalité $\E ^{(\bullet)}=j ^* j ^{\mathrm{top}} _! (\E ^{(\bullet)})$
qui est compatible avec les isomorphismes ci-dessus, 
i.e. l'isomorphisme composé
$j ^* j ^{\mathrm{top}} _! \lambda ^{*} \chi ^{*} (\E ^{(\bullet)})
\riso 
j ^* \lambda ^{*} \chi ^{*} j ^{\mathrm{top}} _!  (\E ^{(\bullet)})
\riso
\lambda ^{*}  \chi ^{*} j ^* j ^{\mathrm{top}} _!  (\E ^{(\bullet)})$
est l'identité.
De même, le morphisme d'adjonction 
$ j ^{\mathrm{top}} _! j ^* (\FF ^{(\bullet)})
\to 
\FF ^{(\bullet)}$ est compatible aux isomorphismes ci-dessus.

On déduit de ce qui précède que
le foncteur $j ^{\mathrm{top}} _!$ envoie les lim-ind-isogénies
sur les lim-ind-isogénies.
On obtient donc sa factorisation sous la forme
$j ^{\mathrm{top}} _!
\colon
\underrightarrow{LM} _{\Q} (\smash{\widetilde{\D}} _{\U ^{\sharp}} ^{(\bullet)} (T \cap U))
\to
\underrightarrow{LM} _{\Q} (\smash{\widetilde{\D}} _{\PP ^\sharp} ^{(\bullet)} (T ))$.
Comme le foncteur canonique
$M(\smash{\widetilde{\D}} _{\PP ^\sharp} ^{(\bullet)} (T ))
\to 
\underrightarrow{LM} _{\Q} (\smash{\widetilde{\D}} _{\PP ^\sharp} ^{(\bullet)} (T ))$
est exact (e.g. voir \cite[2.3.4]{Bockle-Pink-cohomological-theory}),
comme une suite exacte courte 
de $\underrightarrow{LM} _{\Q} (\smash{\widetilde{\D}} _{\U ^{\sharp}} ^{(\bullet)} (T \cap U))$
est isomorphe à l'image par le foncteur canonique quotient d'une suite exacte courte
de
$M (\smash{\widetilde{\D}} _{\U ^{\sharp}} ^{(\bullet)} (T \cap U))$ (e.g. voir \cite[2.3.5]{Bockle-Pink-cohomological-theory}), 
comme 
$j ^{\mathrm{top}} _!
\colon
M (\smash{\widetilde{\D}} _{\U ^{\sharp}} ^{(\bullet)} (T \cap U))
\to
M (\smash{\widetilde{\D}} _{\PP ^\sharp} ^{(\bullet)} (T ))$ 
est exact, 
alors le foncteur
$j ^{\mathrm{top}} _!
\colon
\underrightarrow{LM} _{\Q} (\smash{\widetilde{\D}} _{\U ^{\sharp}} ^{(\bullet)} (T \cap U))
\to
\underrightarrow{LM} _{\Q} (\smash{\widetilde{\D}} _{\PP ^\sharp} ^{(\bullet)} (T ))$ est exact.
De la même façon, on dispose du foncteur exact
$j ^{*}\colon 
\underrightarrow{LM} _{\Q} (\smash{\widetilde{\D}} _{\PP ^\sharp} ^{(\bullet)} (T))
\to
\underrightarrow{LM} _{\Q} (\smash{\widetilde{\D}} _{\U ^{\sharp}} ^{(\bullet)} (T \cap U))$.

On définit l'application 
$\alpha \colon \mathrm{Hom} _{\underrightarrow{LM} _{\Q} (\smash{\widetilde{\D}} _{\U ^{\sharp}} ^{(\bullet)} (T\cap U))}
(\E ^{(\bullet)}, j ^{*} \FF ^{(\bullet)})
\to 
\mathrm{Hom} _{\underrightarrow{LM} _{\Q} (\smash{\widetilde{\D}} _{\PP ^\sharp} ^{(\bullet)} (T))}
(j ^{\mathrm{top}} _! \E ^{(\bullet)},  \FF ^{(\bullet)})$ de la manière suivante:
si $\phi \colon \E ^{(\bullet)} \to  \lambda ^* \chi ^* j ^{*} \FF ^{(\bullet)} $ est un représentant d'une flèche, 
alors un représentant de $\alpha (\phi)$ est donné par
$j ^{\mathrm{top}} _! \E ^{(\bullet)} \underset{j ^{\mathrm{top}} _! (\phi)}{\longrightarrow}  
j ^{\mathrm{top}} _! \lambda ^* \chi ^* j ^{*} \FF ^{(\bullet)}  
\riso 
\lambda ^* \chi ^* j ^{\mathrm{top}} _! j ^{*} \FF ^{(\bullet)}  
\to 
\lambda ^* \chi ^* \FF ^{(\bullet)}  $.
De plus, on définit l'application 
$\beta 
\colon 
\mathrm{Hom} _{\underrightarrow{LM} _{\Q} (\smash{\widetilde{\D}} _{\PP ^\sharp} ^{(\bullet)} (T))}
(j ^{\mathrm{top}} _! \E ^{(\bullet)},  \FF ^{(\bullet)})
\to 
\mathrm{Hom} _{\underrightarrow{LM} _{\Q} (\smash{\widetilde{\D}} _{\U ^{\sharp}} ^{(\bullet)} (T\cap U))}
(\E ^{(\bullet)}, j ^{*} \FF ^{(\bullet)})$ de la manière suivante:
si $\psi \colon j ^{\mathrm{top}} _! \E ^{(\bullet)} \to  \lambda ^* \chi ^* \FF ^{(\bullet)} $ est un représentant d'une flèche, 
alors $\beta (\psi)$ est représenté par 
$\E ^{(\bullet)}= j^*  j ^{\mathrm{top}} _!  (\E ^{(\bullet)}) \underset{ j ^* (\psi)}{\longrightarrow}  
j ^{*}  \lambda ^* \chi ^*  \FF ^{(\bullet)}  
\riso 
\lambda ^* \chi ^*  j ^{*} \FF ^{(\bullet)}  $.
On vérifie facilement que $\alpha$ et $\beta$ sont des bijections réciproques. 
Ainsi, le foncteur exact $  j ^{\mathrm{top}} _!  $ est un adjoint à gauche de $j ^*$.
On en déduit que si $\FF ^{(\bullet)}$ est un objet injectif de la catégorie
$\underrightarrow{LM} _{\Q} (\smash{\widetilde{\D}} _{\PP ^\sharp} ^{(\bullet)} (T))$, 
alors $j ^{*} \FF ^{(\bullet)}  $ est un objet injectif de la catégorie
$\underrightarrow{LM} _{\Q} (\smash{\widetilde{\D}} _{\U ^{\sharp}} ^{(\bullet)} (T\cap U))$.
\end{vide}

\begin{lemm}
\label{flasque}
Soient $\E ^{(\bullet)}, \, \I ^{(\bullet)}$ deux objets de
$\underrightarrow{LM} _{\Q} (\smash{\widetilde{\D}} _{\PP ^\sharp} ^{(\bullet)} (T))$.
On suppose de plus que $\I ^{(\bullet)}$ est un objet injectif de la catégorie 
$\underrightarrow{LM} _{\Q} (\smash{\widetilde{\D}} _{\PP ^\sharp} ^{(\bullet)} (T))$.
Le faisceau 
$\mathcal{H}om _{\underrightarrow{LM} _{\Q} (\smash{\widetilde{\D}} _{\PP ^\sharp} ^{(\bullet)} (T))}
(\E ^{(\bullet)},~\I ^{(\bullet)})$
(voir sa définition en \ref{const-calHomLMQ}) 
est flasque.
\end{lemm}

\begin{proof}
Il s'agit de reprendre les arguments du lemme analogue de \cite[II.7.3.2]{Godement-topo-alg}
qui restent valable grâce au paragraphe \ref{stab-inj-j*}.
\end{proof}

\begin{lemm}
\label{der-HomLD}
Soient $\E ^{(\bullet)}\in 
K 
(\underrightarrow{LM} _{\Q} (\smash{\widetilde{\D}} _{\PP ^\sharp} ^{(\bullet)} (T)))$
et
$\I ^{(\bullet)}\in 
K ^{+}
(\underrightarrow{LM} _{\Q} (\smash{\widetilde{\D}} _{\PP ^\sharp} ^{(\bullet)} (T)))$
un complexe dont les termes sont des objets injectifs
de 
$\underrightarrow{LM} _{\Q} (\smash{\widetilde{\D}} _{\PP ^\sharp} ^{(\bullet)} (T))$.
Si $\E ^{(\bullet)}$ ou 
$\I ^{(\bullet)}$ sont acycliques, alors il en est de même de 
$\mathcal{H}om ^{\bullet} _{\underrightarrow{LM} _{\Q} (\smash{\widetilde{\D}} _{\PP ^\sharp} ^{(\bullet)} (T))}
(\E ^{(\bullet)},~\I ^{(\bullet)})$
\end{lemm}

\begin{proof}
La preuve du lemme analogue de \cite[II.3.1]{HaRD} fonctionne encore: 
il suffit de vérifier que le complexe
$\Gamma (\U, \mathcal{H}om ^{\bullet} _{\underrightarrow{LM} _{\Q} (\smash{\widetilde{\D}} _{\PP ^\sharp} ^{(\bullet)} (T))}
(\E ^{(\bullet)},~\I ^{(\bullet)}))=
 \mathrm{Hom} ^{\bullet} _{\underrightarrow{LM} _{\Q} (\smash{\widetilde{\D}} _{\U ^\sharp} ^{(\bullet)} (T))}
(\E ^{(\bullet)}|\U,~\I ^{(\bullet)}|\U)$ est acylique. 
Or, d'après \ref{stab-inj-j*}, 
$\I ^{(\bullet)}|\U$ 
est un complexe dont 
les termes sont des objets injectifs
de 
$\underrightarrow{LM} _{\Q} (\smash{\widetilde{\D}} _{\U ^{\sharp}} ^{(\bullet)} (T\cap U))$.

\end{proof}

\section{Systèmes inductifs cohérents de $\D$-modules arithmétiques}

\subsection{Cohérence à ind-isogénie près}
Dans cette section, on fixe $\lambda \in L$.

\begin{nota}
\label{rema-prolonlambda}
Soit $\E ^{(\bullet)}$ est un 
$\lambda ^{*}\smash{\widetilde{\D}} _{\PP ^\sharp} ^{(\bullet)} (T)$-module.
\begin{itemize}
\item On note $M (\lambda ^{*} \smash{\widetilde{\D}} _{\PP ^\sharp} ^{(\bullet)}(T))$
la catégorie des 
$\lambda ^{*} \smash{\widetilde{\D}} _{\PP ^\sharp} ^{(\bullet)}(T)$-modules.
Via le morphisme canonique d'anneaux
$\smash{\widetilde{\D}} _{\PP ^\sharp} ^{(\bullet)}(T)
\to 
\lambda ^{*} \smash{\widetilde{\D}} _{\PP ^\sharp} ^{(\bullet)}(T)$,
on dispose du foncteur oubli
$\mathrm{oub} _{\lambda}
\colon 
M (\lambda ^{*} \smash{\widetilde{\D}} _{\PP ^\sharp} ^{(\bullet)}(T))
\to 
M ( \smash{\widetilde{\D}} _{\PP ^\sharp} ^{(\bullet)}(T))$.

\item Soit $\chi \in M$. On remarque que 
$\chi ^{*}\lambda ^{*}\smash{\widetilde{\D}} _{\PP ^\sharp} ^{(\bullet)} (T)$ n'est plus un système inductif d'anneaux (à moins que $\chi $ soit constante).
On calcule que $\chi ^{*} \E ^{(\bullet)}$ reste muni d'une structure canonique de 
$\lambda ^{*}\smash{\widetilde{\D}} _{\PP ^\sharp} ^{(\bullet)} (T)$-module.
On obtient même le foncteur 
$\chi ^{*}
\colon 
M (\lambda ^{*} \smash{\widetilde{\D}} _{\PP ^\sharp} ^{(\bullet)}(T))
\to
M (\lambda ^{*} \smash{\widetilde{\D}} _{\PP ^\sharp} ^{(\bullet)}(T))$.

\item Soit $\mu \in L$. 
On note
$\mu ^{*}
\colon 
M (\lambda ^{*} \smash{\widetilde{\D}} _{\PP ^\sharp} ^{(\bullet)}(T))
\to
M (\mu ^{*}\lambda ^{*} \smash{\widetilde{\D}} _{\PP ^\sharp} ^{(\bullet)}(T))$,
le foncteur induit par $\mu ^*$.
\end{itemize}

\end{nota}

\begin{nota}
\begin{itemize}
\item Un morphisme $f ^{(\bullet)} \colon \E ^{(\bullet)} \to \FF ^{(\bullet)}$ de $M (\lambda ^{*} \smash{\widetilde{\D}} _{\PP ^\sharp} ^{(\bullet)}(T))$
est une ind-isogénie s'il existe $\chi \in M$ 
et un morphisme 
$g ^{(\bullet)} \colon \FF ^{(\bullet)} \to \chi ^{*} \E ^{(\bullet)}$ de $M (\lambda ^{*} \smash{\widetilde{\D}} _{\PP ^\sharp} ^{(\bullet)}(T))$
tels que les morphismes 
$g ^{(\bullet)}\circ f ^{(\bullet)}$ et $\chi ^{*} (f ^{(\bullet)}) \circ g ^{(\bullet)}$ 
de $M (\lambda ^{*} \smash{\widetilde{\D}} _{\PP ^\sharp} ^{(\bullet)}(T))$
sont les morphismes canoniques.
La catégorie obtenue en localisant $M (\lambda ^{*} \smash{\widetilde{\D}} _{\PP ^\sharp} ^{(\bullet)}(T))$
par rapport aux ind-isogénies se note
$\smash{\underrightarrow{M}} _{\Q}  (\lambda ^{*} \smash{\widetilde{\D}} _{\PP ^\sharp} ^{(\bullet)}(T))$.
Comme pour \ref{defi-M-L-Serre}, on vérifie que celle-ci est abélienne.

\item Un morphisme 
$f ^{(\bullet)} \colon \E ^{(\bullet)} \to \FF ^{(\bullet)}$
de 
$M (\lambda ^{*} \smash{\widetilde{\D}} _{\PP ^\sharp} ^{(\bullet)}(T))$
est une lim-ind-isogénie s'il existe
 $\chi \in M$, 
$\mu \in L$ et un morphisme 
$g ^{(\bullet)} \colon \FF ^{(\bullet)} \to \mu ^{*} \chi ^{*}\E ^{(\bullet)}$ de $M (\lambda ^{*} \smash{\widetilde{\D}} _{\PP ^\sharp} ^{(\bullet)}(T))$
tels que les morphismes
$g ^{(\bullet)}\circ f ^{(\bullet)}$ et $\mu ^{*} \chi ^{*}(f ^{(\bullet)}) \circ g ^{(\bullet)}$ 
de $M (\lambda ^{*} \smash{\widetilde{\D}} _{\PP ^\sharp} ^{(\bullet)}(T))$
sont les morphismes canoniques.
On note 
$\smash{\underrightarrow{LM}} _{\Q}  (\lambda ^{*} \smash{\widetilde{\D}} _{\PP ^\sharp} ^{(\bullet)}(T))$
la catégorie localisée par rapport aux lim-ind-isogénies.
Comme pour \ref{defi-M-L-Serre}, on vérifie que celle-ci est abélienne.
\end{itemize}

\end{nota}

\begin{lemm}
\label{facto-oub-otimes}
On bénéficie des factorisations
$\mathrm{oub} _{\lambda}
\colon \smash{\underrightarrow{M}} _{\Q}  (\lambda ^{*} \smash{\widetilde{\D}} _{\PP ^\sharp} ^{(\bullet)}(T))
\to 
\smash{\underrightarrow{M}} _{\Q}  ( \smash{\widetilde{\D}} _{\PP ^\sharp} ^{(\bullet)}(T))$
et
$\mathrm{oub} _{\lambda}
\colon \smash{\underrightarrow{LM}} _{\Q}  (\lambda ^{*} \smash{\widetilde{\D}} _{\PP ^\sharp} ^{(\bullet)}(T))
\to 
\smash{\underrightarrow{LM}} _{\Q}  ( \smash{\widetilde{\D}} _{\PP ^\sharp} ^{(\bullet)}(T))$.
De plus, on dispose des factorisation
$\lambda ^{*}\smash{\widetilde{\D}} _{\PP ^\sharp} ^{(\bullet)}(T) \otimes _{\smash{\widetilde{\D}} _{\PP ^\sharp} ^{(\bullet)}(T)}-
\colon 
\smash{\underrightarrow{M}} _{\Q}  (\smash{\widetilde{\D}} _{\PP ^\sharp} ^{(\bullet)}(T))
\to 
\smash{\underrightarrow{M}} _{\Q}  (\lambda ^{*}  \smash{\widetilde{\D}} _{\PP ^\sharp} ^{(\bullet)}(T))$
et
$\lambda ^{*}\smash{\widetilde{\D}} _{\PP ^\sharp} ^{(\bullet)}(T) \otimes _{\smash{\widetilde{\D}} _{\PP ^\sharp} ^{(\bullet)}(T)}-
\colon \smash{\underrightarrow{LM}} _{\Q}  (\smash{\widetilde{\D}} _{\PP ^\sharp} ^{(\bullet)}(T))
\to 
\smash{\underrightarrow{LM}} _{\Q}  (\lambda ^{*}  \smash{\widetilde{\D}} _{\PP ^\sharp} ^{(\bullet)}(T))$.
\end{lemm}

\begin{proof}
Il est immédiat qu'une ind-isogénie (resp. une lim-ind-isogénie) de 
$M (\lambda ^{*} \smash{\widetilde{\D}} _{\PP ^\sharp} ^{(\bullet)}(T))$
est (via le foncteur oubli) une ind-isogénie  (resp. une lim-ind-isogénie) 
de 
$M ( \smash{\widetilde{\D}} _{\PP ^\sharp} ^{(\bullet)}(T))$. 
D'où les factorisations concernant le foncteur oubli. 
Pour tout $\chi \in M$, 
pour tout $\smash{\widetilde{\D}} _{\PP ^\sharp} ^{(\bullet)} (T)$-module $\E ^{(\bullet)}$,
l'isomorphisme canonique 
$\chi ^* (\lambda ^{*}\smash{\widetilde{\D}} _{\PP ^\sharp} ^{(\bullet)}(T) 
\otimes _{\smash{\widetilde{\D}} _{\PP ^\sharp} ^{(\bullet)}(T)} 
\E ^{(\bullet)})
\riso
\lambda ^{*}\smash{\widetilde{\D}} _{\PP ^\sharp} ^{(\bullet)}(T) 
\otimes _{\smash{\widetilde{\D}} _{\PP ^\sharp} ^{(\bullet)}(T)} 
\chi ^* \E ^{(\bullet)}$ 
qui commute aux morphismes canoniques de la forme 
$ \E ^{(\bullet)} \to \chi ^*  \E ^{(\bullet)}$. D'où la première factorisation du foncteur
$ \lambda ^{*}\smash{\widetilde{\D}} _{\PP ^\sharp} ^{(\bullet)}(T) 
\otimes _{\smash{\widetilde{\D}} _{\PP ^\sharp} ^{(\bullet)}(T)} -$. 
De plus, soit $f ^{(\bullet)}\colon  \E ^{(\bullet)}  \to  \FF ^{(\bullet)} $
un morphisme de 
$\smash{\underrightarrow{M}} _{\Q}  (\smash{\widetilde{\D}} _{\PP ^\sharp} ^{(\bullet)}(T))$. 
Dire que $f ^{(\bullet)}$ est inversible dans 
$\smash{\underrightarrow{LM}} _{\Q}  (\smash{\widetilde{\D}} _{\PP ^\sharp} ^{(\bullet)}(T))$
équivaut à dire qu'il existe un morphisme 
$\mu \in L$ et un morphisme 
$g ^{(\bullet)}\colon  \E ^{(\bullet)}  \to  \mu ^{*}\FF ^{(\bullet)} $ de
$\smash{\underrightarrow{M}} _{\Q}  (\smash{\widetilde{\D}} _{\PP ^\sharp} ^{(\bullet)}(T))$
rendant commutatif le carré de gauche ci-dessous dans le cas où $\lambda =id$:
\begin{equation}
\label{facto-oub-otimes-2carres}
\xymatrix @ R=0,4cm{
{ \lambda ^{*}\smash{\widetilde{\D}} _{\PP ^\sharp} ^{(\bullet)}(T) 
\otimes _{\smash{\widetilde{\D}} _{\PP ^\sharp} ^{(\bullet)}(T)} 
\FF ^{(\bullet)}} 
\ar[r] ^-{}
\ar@{.>}[rd] ^-{id \otimes g ^{(\bullet)}}
&
{ \lambda ^{*}\smash{\widetilde{\D}} _{\PP ^\sharp} ^{(\bullet)}(T) 
\otimes _{\smash{\widetilde{\D}} _{\PP ^\sharp} ^{(\bullet)}(T)} 
\mu ^* \FF ^{(\bullet)}} 
\ar[r] ^-{}
& 
{\mu ^*  \left ( \lambda ^{*}\smash{\widetilde{\D}} _{\PP ^\sharp} ^{(\bullet)}(T) 
\otimes _{\smash{\widetilde{\D}} _{\PP ^\sharp} ^{(\bullet)}(T)} 
\FF ^{(\bullet)}\right )}
\\ 
{ \lambda ^{*}\smash{\widetilde{\D}} _{\PP ^\sharp} ^{(\bullet)}(T) 
\otimes _{\smash{\widetilde{\D}} _{\PP ^\sharp} ^{(\bullet)}(T)} 
\E ^{(\bullet)}} 
\ar[r] ^-{}
\ar[u] ^-{id \otimes f ^{(\bullet)}}
& 
{ \lambda ^{*}\smash{\widetilde{\D}} _{\PP ^\sharp} ^{(\bullet)}(T) 
\otimes _{\smash{\widetilde{\D}} _{\PP ^\sharp} ^{(\bullet)}(T)} 
\mu ^* \E ^{(\bullet)}} 
\ar[r] ^-{}
\ar[u] ^-{id \otimes \mu ^* f ^{(\bullet)}}
& 
{\mu ^*  \left ( \lambda ^{*}\smash{\widetilde{\D}} _{\PP ^\sharp} ^{(\bullet)}(T) 
\otimes _{\smash{\widetilde{\D}} _{\PP ^\sharp} ^{(\bullet)}(T)} 
\E ^{(\bullet)}\right ),} 
\ar[u] ^-{\mu ^*(id \otimes  f ^{(\bullet)})}
}
\end{equation}
 dont les flèches horizontales sont les morphismes canoniques.
Il en résulte que la dernière factorisation voulue.

\end{proof}

\begin{vide}
[Adjonction]
\label{adjonction-oubotimes}
On vérifie que le foncteur 
$\lambda ^{*}\smash{\widetilde{\D}} _{\PP ^\sharp} ^{(\bullet)}(T) \otimes _{\smash{\widetilde{\D}} _{\PP ^\sharp} ^{(\bullet)}(T)}-
\colon M ( \smash{\widetilde{\D}} _{\PP ^\sharp} ^{(\bullet)}(T))
\to
M (\lambda ^{*} \smash{\widetilde{\D}} _{\PP ^\sharp} ^{(\bullet)}(T))$
est adjoint à gauche du foncteur oubli
$\mathrm{oub} _{\lambda}
\colon 
M (\lambda ^{*} \smash{\widetilde{\D}} _{\PP ^\sharp} ^{(\bullet)}(T))
\to 
M ( \smash{\widetilde{\D}} _{\PP ^\sharp} ^{(\bullet)}(T))$.
 Comme le foncteur oubli $\mathrm{oub} _{\lambda}$ commute canoniquement aux foncteurs de la forme $\mu ^{*}$ et $\chi ^*$, on vérifie 
que $\mathrm{oub} _{\lambda}
\colon \smash{\underrightarrow{M}} _{\Q}  (\lambda ^{*} \smash{\widetilde{\D}} _{\PP ^\sharp} ^{(\bullet)}(T))
\to 
\smash{\underrightarrow{M}} _{\Q}  ( \smash{\widetilde{\D}} _{\PP ^\sharp} ^{(\bullet)}(T))$
est adjoint à droite de 
$\lambda ^{*}\smash{\widetilde{\D}} _{\PP ^\sharp} ^{(\bullet)}(T) \otimes _{\smash{\widetilde{\D}} _{\PP ^\sharp} ^{(\bullet)}(T)}-
\colon \smash{\underrightarrow{M}} _{\Q}  (\smash{\widetilde{\D}} _{\PP ^\sharp} ^{(\bullet)}(T))
\to 
\smash{\underrightarrow{M}} _{\Q}  (\lambda ^{*}  \smash{\widetilde{\D}} _{\PP ^\sharp} ^{(\bullet)}(T))$
et $\mathrm{oub} _{\lambda}
\colon \smash{\underrightarrow{LM}} _{\Q}  (\lambda ^{*} \smash{\widetilde{\D}} _{\PP ^\sharp} ^{(\bullet)}(T))
\to 
\smash{\underrightarrow{LM}} _{\Q}  ( \smash{\widetilde{\D}} _{\PP ^\sharp} ^{(\bullet)}(T))$
est adjoint à droite de 
 $\lambda ^{*}\smash{\widetilde{\D}} _{\PP ^\sharp} ^{(\bullet)}(T) \otimes _{\smash{\widetilde{\D}} _{\PP ^\sharp} ^{(\bullet)}(T)}-
\colon \smash{\underrightarrow{LM}} _{\Q}  (\smash{\widetilde{\D}} _{\PP ^\sharp} ^{(\bullet)}(T))
\to 
\smash{\underrightarrow{LM}} _{\Q}  (\lambda ^{*}  \smash{\widetilde{\D}} _{\PP ^\sharp} ^{(\bullet)}(T))$.
\end{vide}

\begin{lemm}
\label{lem-MQlambda2LMQ}
Les foncteurs 
$\mathrm{oub} _{\lambda}$ et 
$\lambda ^{*}\smash{\widetilde{\D}} _{\PP ^\sharp} ^{(\bullet)}(T) \otimes _{\smash{\widetilde{\D}} _{\PP ^\sharp} ^{(\bullet)}(T)}-$ 
sont des équivalences de catégories quasi-inverses entre 
$\smash{\underrightarrow{LM}} _{\Q}  (\lambda ^{*} \smash{\widetilde{\D}} _{\PP ^\sharp} ^{(\bullet)}(T))$
et
$\smash{\underrightarrow{LM}} _{\Q}  ( \smash{\widetilde{\D}} _{\PP ^\sharp} ^{(\bullet)}(T))$.
On en déduit en particulier que le foncteur oubli de la forme
\begin{equation}
\label{MQlambda2LMQ}
\mathrm{oub} _{\lambda}
\colon 
\smash{\underrightarrow{M}} _{\Q}  (\lambda ^{*} \smash{\widetilde{\D}} _{\PP ^\sharp} ^{(\bullet)}(T))
\to 
\smash{\underrightarrow{LM}} _{\Q}  ( \smash{\widetilde{\D}} _{\PP ^\sharp} ^{(\bullet)}(T))
\end{equation}
est pleinement fidèle et exacte (car c'est un quotient par une sous-catégorie de Serre. 
\end{lemm}

\begin{proof}
Soient $\E ^{(\bullet)}$ un $\smash{\widetilde{\D}} _{\PP ^\sharp} ^{(\bullet)} (T)$-module 
et
$\FF ^{(\bullet)}$ un $\lambda ^{*}\smash{\widetilde{\D}} _{\PP ^\sharp} ^{(\bullet)} (T)$-module.
Grâce à \ref{adjonction-oubotimes}, il s'agit de vérifier que les morphismes d'adjonction sont des isomorphismes.
Cela résulte des diagrammes commutatifs de $ M ( \smash{\widetilde{\D}} _{\PP ^\sharp} ^{(\bullet)}(T))$
(resp. $ M (\lambda ^* \smash{\widetilde{\D}} _{\PP ^\sharp} ^{(\bullet)}(T))$ pour le carré de droite):
\begin{equation}
\label{lem-MQlambda2LMQ-diag}
\xymatrix @ R=0,4cm @C=0,3cm {
{ \lambda ^{*}\smash{\widetilde{\D}} _{\PP ^\sharp} ^{(\bullet)}(T) 
\otimes _{\smash{\widetilde{\D}} _{\PP ^\sharp} ^{(\bullet)}(T)} 
\E ^{(\bullet)}} 
\ar[r] ^-{}
\ar@{.>}[dr] ^-{}
 & 
 { \lambda ^{*} ( \lambda ^{*}\smash{\widetilde{\D}} _{\PP ^\sharp} ^{(\bullet)}(T) 
\otimes _{\smash{\widetilde{\D}} _{\PP ^\sharp} ^{(\bullet)}(T)} 
\E ^{(\bullet)})} 
\\ 
{\E ^{(\bullet)}} 
\ar[r] ^-{}
\ar[u] ^-{}
& 
{\lambda ^{*}\E ^{(\bullet)} ,} 
\ar[u] ^-{}
}
~
\xymatrix @ R=0,4cm @C=0,2cm {
{\FF ^{(\bullet)}} 
\ar[r] ^-{}
\ar@{.>}[dr] ^-{}
 & 
{\lambda ^{*}\FF ^{(\bullet)} } 
\\ 
{ \lambda ^{*}\smash{\widetilde{\D}} _{\PP ^\sharp} ^{(\bullet)}(T) 
\otimes _{\smash{\widetilde{\D}} _{\PP ^\sharp} ^{(\bullet)}(T)} 
\FF ^{(\bullet)}} 
\ar[r] ^-{}
\ar[u] ^-{}
& 
{ \lambda ^{*} ( \lambda ^{*}\smash{\widetilde{\D}} _{\PP ^\sharp} ^{(\bullet)}(T) 
\otimes _{\smash{\widetilde{\D}} _{\PP ^\sharp} ^{(\bullet)}(T)} 
\FF ^{(\bullet)}).} 
\ar[u] ^-{}
}
\end{equation}

\end{proof}

\begin{defi}
\label{coh-loc-pres-fini}
Soit 
$\E ^{(\bullet)}$ un 
$\lambda ^{*}\smash{\widetilde{\D}} _{\PP ^\sharp} ^{(\bullet)} (T)$-module.

\begin{itemize}
\item Le module 
$\E ^{(\bullet)}$
est un 
$\lambda ^{*}\smash{\widetilde{\D}} _{\PP ^\sharp} ^{(\bullet)} (T)$-module localement de présentation finie s'il existe
un recouvrement ouvert 
$(\PP _i ) _{i\in I}$ de $\PP$ tel que, pour tout $i \in I$, on dispose  
d'une suite exacte dans 
$M (\lambda ^{*}\smash{\widetilde{\D}} _{\PP ^\sharp} ^{(\bullet)} (T))$ de la forme: 
\begin{equation}
\notag
\left ( \lambda ^{*}\smash{\widetilde{\D}} _{\PP _i} ^{(\bullet)} (T \cap P _i) \right) ^{r _i}
\to 
\left ( \lambda ^{*}\smash{\widetilde{\D}} _{\PP _i} ^{(\bullet)} (T  \cap P _i) \right) ^{s _i}
\to 
\E ^{(\bullet)} | \PP _i
\to 0,
\end{equation}
où $r _i , s _i \in \N$.

\item Si de plus, on peut prendre $\{\PP\}$ comme tel recouvrement ouvert de $\PP$, 
on dira que le module 
$\E ^{(\bullet)}$
est un 
$\lambda ^{*}\smash{\widetilde{\D}} _{\PP ^\sharp} ^{(\bullet)} (T)$-module globalement de présentation finie.

\end{itemize}
\end{defi}

\begin{lemm}
\label{eqcat-locpfotimes}
Les foncteurs 
$ \lambda ^{*}\smash{\widetilde{\D}} _{\PP ^\sharp} ^{(\bullet)} (T ) 
\otimes _{\smash{\widetilde{\D}} _{\PP ^\sharp} ^{(\lambda (0))} (T)} -$
et
$\E ^{(\bullet)}\mapsto \E ^{(0)}$
induisent des équivalences quasi-inverses
entre la catégorie 
des $\smash{\widetilde{\D}} _{\PP ^\sharp} ^{(\lambda (0))} (T)$-modules cohérents 
(resp. des $\smash{\widetilde{\D}} _{\PP ^\sharp} ^{(\lambda (0))} (T)$-modules globalement de 
présentation finie)
et celle des
$\lambda ^{*}\smash{\widetilde{\D}} _{\PP ^\sharp} ^{(\bullet)} (T)$-modules localement de présentation finie
(resp. des 
$\lambda ^{*}\smash{\widetilde{\D}} _{\PP ^\sharp} ^{(\bullet)} (T)$-modules globalement de présentation finie).
\end{lemm}

\begin{proof}
Comme ces deux foncteurs sont exacts à droite, 
ils se factorisent comme voulu (on vérifie d'abord le cas respectif, le cas non respectif s'en déduisant par cohérence 
de $\smash{\widetilde{\D}} _{\PP ^\sharp} ^{(\lambda (0))} (T)$). 
Or, le foncteur 
$ \lambda ^{*}\smash{\widetilde{\D}} _{\PP ^\sharp} ^{(\bullet)} (T ) 
\otimes _{\smash{\widetilde{\D}} _{\PP ^\sharp} ^{(\lambda (0))} (T)} -$
est adjoint à gauche de
$\E ^{(\bullet)}\mapsto \E ^{(0)}$ pour les catégories
de  $\smash{\widetilde{\D}} _{\PP ^\sharp} ^{(\lambda (0))} (T)$-modules et 
de $\lambda ^{*}\smash{\widetilde{\D}} _{\PP ^\sharp} ^{(\bullet)} (T)$-modules.
Il suffit donc de vérifier que les morphismes d'adjonction sont des isomorphismes, ce qui est aisé:  
si $\E ^{(\bullet)}$ un $\lambda ^{*}\smash{\widetilde{\D}} _{\PP ^\sharp} ^{(\bullet)} (T)$-module localement de présentation finie, 
alors le morphisme canonique 
$ \lambda ^{*}\smash{\widetilde{\D}} _{\PP ^\sharp} ^{(\bullet)} (T ) 
\otimes _{\smash{\widetilde{\D}} _{\PP ^\sharp} ^{(\lambda (0))} (T)} \E ^{(0)}
\to 
\E ^{(\bullet)}$ 
est un isomorphisme car, par exactitude à droite,  additivité de nos foncteurs et via \ref{LDiso-local}, il suffit  de le vérifier 
pour $\E ^{(\bullet)}= \lambda ^{*}\smash{\widetilde{\D}} _{\PP ^\sharp} ^{(\bullet)} (T)$ (le fait que le second morphisme d'adjonction soit un isomorphisme
est immédiat).
\end{proof}

La caractérisation du lemme qui suit est à rapprocher avec la notion de cohérence de Berthelot que l'on rappellera dans 
\ref{defi-LDQ0coh}.
\begin{lemm}
Soit 
$\E ^{(\bullet)}$ un 
$\lambda ^{*}\smash{\widetilde{\D}} _{\PP ^\sharp} ^{(\bullet)} (T)$-module.
Le  
$\lambda ^{*}\smash{\widetilde{\D}} _{\PP ^\sharp} ^{(\bullet)} (T)$-module $\E ^{(\bullet)}$
est localement de présentation finie 
si et seulement  
s'il satisfait les conditions suivantes:
\begin{enumerate}
\item Pour tout $m \in \N$, $\E ^{(m)}$ est un $\smash{\widetilde{\D}} _{\PP ^\sharp} ^{(\lambda (m))} (T)$-module cohérent ; 
\item Pour tous entiers $0\leq m \leq m'$,  
le morphisme canonique
\begin{equation}
\label{Beintro-4.2.3M}
\smash{\widetilde{\D}} _{\PP ^\sharp} ^{(\lambda (m'))} (T) \otimes _{\smash{\widetilde{\D}} _{\PP ^\sharp} ^{(\lambda (m))} (T)}
\E ^{(m)} \to \E ^{(m')} 
\end{equation}
est un isomorphisme.
\end{enumerate}
\end{lemm}

\begin{proof}
Cela découle aussitôt de \ref{eqcat-locpfotimes}.
\end{proof}

\begin{rema}
\label{rema-otimes-eqcat}
\begin{itemize}
\item 
Comme les extensions
$\smash{\widetilde{\D}} _{\PP ^\sharp} ^{(\lambda (m))} (T) 
\to 
\smash{\widetilde{\D}} _{\PP ^\sharp} ^{(\lambda (m +1))} (T)$
ne sont pas plates, 
la catégorie des
$\lambda ^{*}\smash{\widetilde{\D}} _{\PP ^\sharp} ^{(\bullet)} (T ) $-modules localement de présentation finie 
n'est pas stable par noyau. 
 
\item Si $\PP$ est affine, d'après le théorème de type $A$, les notions de
$\smash{\widetilde{\D}} _{\PP ^\sharp} ^{(\lambda (0))} (T)$-modules cohérents
et de $\smash{\widetilde{\D}} _{\PP ^\sharp} ^{(\lambda (0))} (T)$-modules globalement de présentation finie sont alors égales.  
D'après \ref{eqcat-locpfotimes}, 
il en résulte que les notions de 
$\lambda ^{*}\smash{\widetilde{\D}} _{\PP ^\sharp} ^{(\bullet)} (T ) $-modules localement de présentation finie 
et de 
$\lambda ^{*}\smash{\widetilde{\D}} _{\PP ^\sharp} ^{(\bullet)} (T ) $-modules globalement de présentation finie 
sont alors égales. 
De plus, 
le foncteur 
$ \lambda ^{*}\smash{\widetilde{\D}} _{\PP ^\sharp} ^{(\bullet)} (T ) 
\otimes _{\smash{\widetilde{D}} _{\PP ^{\sharp}} ^{(\lambda (0))} (T)} -$
de la catégorie des 
$\smash{\widetilde{D}} _{\PP ^{\sharp}} ^{(\lambda (0))} (T)$-modules de type fini
dans celle 
des 
$\lambda ^{*}\smash{\widetilde{\D}} _{\PP ^\sharp} ^{(\bullet)} (T ) $-modules localement de présentation finie 
est une équivalence de catégories dont un foncteur quasi-inverse est donné par 
$\G ^{(\bullet)} \mapsto \Gamma (\PP, \G ^{(0)})$.
\end{itemize}
\end{rema}

\begin{lemm}
\label{ind-isog-coh-Q}
Soit 
$f ^{(\bullet)}\colon 
\E ^{(\bullet)} \to \FF ^{(\bullet)}$ 
un morphisme de 
$M (\lambda ^{*}\smash{\widetilde{\D}} _{\PP ^\sharp} ^{(\bullet)} (T))$.
On suppose que, pour tout $m \in \N$, 
les $\smash{\widetilde{\D}} _{\PP ^\sharp} ^{(\lambda (m))} (T)$-modules
$\E ^{(m)}$
et 
$\FF ^{(m)}$ 
sont cohérents. 
Le morphisme 
$f ^{(\bullet)}$
est alors une ind-isogénie de 
$M (\lambda ^{*}\smash{\widetilde{\D}} _{\PP ^\sharp} ^{(\bullet)} (T))$ 
si et seulement si
le morphisme 
$f ^{(\bullet)} _\Q
\colon \E ^{(\bullet)} _\Q \to \FF ^{(\bullet)}_\Q $ 
induit après tensorisation par $\Q$
est un isomorphisme de 
$\lambda ^{*}\smash{\widetilde{\D}} _{\PP ^\sharp} ^{(\bullet)} (T) _\Q$-modules.
\end{lemm}

\begin{proof}
0) Comme les morphismes canoniques de la forme
$\E ^{(\bullet)}  \to \chi ^{*} \E ^{(\bullet)} $ avec $\chi \in M$ 
deviennent des isomorphismes après application du foncteur $\Q \otimes _{\Z} -$, 
la première implication est évidente. 
Réciproquement supposons que 
$f ^{(\bullet)} _\Q$ 
soit un isomorphisme. 

1) D'après \cite[3.4.4]{Be1}, 
si on note $\E ^{(m)} _t$
le sous-faisceau des sections de $p$-torsion de $\E ^{(m)}$, alors
$\E ^{(m)} _t$ est un sous-$\smash{\widetilde{\D}} _{\PP ^\sharp} ^{(\lambda (m))} (T)$-module cohérent 
de $\E ^{(m)}$.
Notons
$\alpha \colon \E ^{(\bullet)}  \to \E ^{(\bullet)} / \E ^{(\bullet)} _t $, la projection canonique. Vérifions que c'est une ind-isogénie 
de $M (\lambda ^{*}\smash{\widetilde{\D}} _{\PP ^\sharp} ^{(\bullet)} (T))$. 
Comme $\E ^{(m)} _t$ est un $\smash{\widetilde{\D}} _{\PP ^\sharp} ^{(\lambda (m))} (T)$-module cohérent , 
il existe alors $\chi \in M$ tel que la multiplication
$p ^{\chi (m)}\colon \E ^{(m)} \to \E ^{(m)} $ se factorise (de manière unique) 
en 
$\beta ^{(m)}\colon \E ^{(m)} / \E ^{(m)} _t \to \E ^{(m)} $.
On en déduit que le morphisme
$\beta ^{(\bullet)}\colon \E ^{(\bullet)} / \E ^{(\bullet)} _t \to \chi ^{*}\E ^{(\bullet)} $
de $M (\lambda ^{*}\smash{\widetilde{\D}} _{\PP ^\sharp} ^{(\bullet)} (T))$ 
est tel que 
$\beta ^{(\bullet)} \circ \alpha ^{(\bullet)}$
et 
$\chi ^* \alpha ^{(\bullet)} \circ \beta ^{(\bullet)}$ sont les morphismes canoniques. 

2) D'après l'étape $1)$ quitte à quotienter par les sous-modules de $p$-torsion, on peut supposer 
que, pour tout $m \in \N$, 
$\E ^{(m)}$ et $\FF ^{(m)}$ sont sans $p$-torsion.
Il existe alors $\chi \in M$ tel que 
pour tout $m \in M$, l'isomorphisme canonique 
$h ^{(m)}:= p ^{\chi (m)} ( f ^{(m)} _\Q ) ^{-1} \colon \FF ^{(m)} _\Q \to \E ^{(m)} _\Q$ 
se factorise (de manière unique) par un morphisme de la forme
$b ^{(m)}\colon \FF ^{(m)} \to \E ^{(m)}$.
Comme $\chi ^{*}\E ^{(\bullet)} \subset \chi ^{*}\E ^{(\bullet)} _{\Q}$
et
$\chi ^{*}\FF ^{(\bullet)} \subset \chi ^{*}\FF ^{(\bullet)} _{\Q}$,
on calcule en fait que 
le morphisme 
$h ^{(\bullet)} _{\Q}\colon \FF ^{(\bullet)}  _{\Q}
\to \chi ^{*}\E ^{(\bullet)} _{\Q} $
se factorise en le morphisme
$b ^{(\bullet)}\colon \FF ^{(\bullet)}  \to \chi ^{*}\E ^{(\bullet)} $
(i.e., les $b ^{(m)}$ commutent aux morphismes de transition)
qui est tel que 
$g ^{(\bullet)} \circ f ^{(\bullet)}$
et 
$\chi ^*( f ^{(\bullet)}) \circ g ^{(\bullet)}$ sont les morphismes canoniques.

\end{proof}

\begin{defi}
\label{coh-loc-pres-fini-ind-iso}
Soit 
$\E ^{(\bullet)}$ un 
$\lambda ^{*}\smash{\widetilde{\D}} _{\PP ^\sharp} ^{(\bullet)} (T)$-module.
Le module 
$\E ^{(\bullet)}$
est un 
{\og $\lambda ^{*}\smash{\widetilde{\D}} _{\PP ^\sharp} ^{(\bullet)} (T)$-module 
de type fini à ind-isogénie près\fg} 
s'il existe
un recouvrement ouvert 
$(\PP _i ) _{i\in I}$ de $\PP$ tel que, pour tout $i \in I$, on dispose  
d'une suite exacte dans 
$\underrightarrow{M} _{\Q}  (\lambda ^*\smash{\widetilde{\D}} _{\PP ^\sharp} ^{(\bullet)} (T))$ de la forme: 
$\left ( \lambda ^{*}\smash{\widetilde{\D}} _{\PP ^{\sharp} _i} ^{(\bullet)} (T  \cap P _i) \right) ^{r _i}
\to 
\E ^{(\bullet)} | \PP _i
\to 0$,
où $r _i\in \N$.
De même, on définie la notion de 
{\og $\lambda ^{*}\smash{\widetilde{\D}} _{\PP ^\sharp} ^{(\bullet)} (T)$-module 
localement (resp. globalement) de présentation finie à ind-isogénie près\fg}.
 
 Le module 
$\E ^{(\bullet)}$
est un 
{\og $\lambda ^{*}\smash{\widetilde{\D}} _{\PP ^\sharp} ^{(\bullet)} (T)$-module 
cohérent à ind-isogénie près\fg} 
s'il est de type fini à ind-isogénie près et si 
pour tout ouvert $\PP'$ de $\PP$, tout homomorphisme de $\underrightarrow{M} _{\Q}  (\lambda ^*\smash{\widetilde{\D}} _{\PP ^{\prime \sharp}} ^{(\bullet)} (T))$ 
de la forme
$u ^{(\bullet)}\colon \left ( \lambda ^{*}\smash{\widetilde{\D}} _{\PP ^{\prime \sharp}} ^{(\bullet)} (T  \cap P ' ) \right) ^{r}
\to 
\E ^{(\bullet)} | \PP '$,
le noyau de $u ^{(\bullet)}$ est de type fini à ind-isogénie près. 
\end{defi}

\begin{lemm}
\label{caract-coh-isog}
Soit 
$\E ^{(\bullet)}$ un 
$\lambda ^{*}\smash{\widetilde{\D}} _{\PP ^\sharp} ^{(\bullet)} (T)$-module.
Le $\lambda ^{*}\smash{\widetilde{\D}} _{\PP ^\sharp} ^{(\bullet)} (T)$-module 
$\E ^{(\bullet)}$ est globalement de présentation finie à ind-isogénie près si et seulement s'il est isomorphe dans
$\underrightarrow{M} _{\Q}  (\lambda ^{*}\smash{\widetilde{\D}} _{\PP ^\sharp} ^{(\bullet)} (T))$ à un 
$\lambda ^{*}\smash{\widetilde{\D}} _{\PP ^\sharp} ^{(\bullet)} (T)$-module globalement de présentation finie.
\end{lemm}

\begin{proof}
La suffisance est triviale. Réciproquement, 
supposons 
que 
$\E ^{(\bullet)}$
soit le conoyau dans 
$\underrightarrow{M} _{\Q}  (\lambda ^{*}\smash{\widetilde{\D}} _{\PP ^\sharp} ^{(\bullet)} (T))$
d'une flèche de la forme
$\left ( \lambda ^{*}\smash{\widetilde{\D}} _{\PP ^\sharp} ^{(\bullet)} (T ) \right) ^{r }
\to 
\left ( \lambda ^{*}\smash{\widetilde{\D}} _{\PP ^\sharp} ^{(\bullet)} (T ) \right) ^{s}$.
Soit 
$\phi ^{(\bullet)}
 \colon  
\left ( \lambda ^{*}\smash{\widetilde{\D}} _{\PP ^\sharp} ^{(\bullet)} (T ) \right) ^{r }
\to 
\chi ^* \left ( \lambda ^{*}\smash{\widetilde{\D}} _{\PP ^\sharp} ^{(\bullet)} (T ) \right) ^{s}$
un représentant dans 
$M(\lambda ^{*}\smash{\widetilde{\D}} _{\PP ^\sharp} ^{(\bullet)} (T))$
de ce morphisme pour un certain $\chi \in M$.
Ainsi, on dispose du diagramme commutatif 
\begin{equation}
\notag
\xymatrix @R=0,3cm{
{\left ( \smash{\widetilde{\D}} _{\PP ^\sharp} ^{(\lambda(0))} (T ) \right) ^{r }} 
\ar[r] ^-{\phi ^{(0)}}
\ar[d] ^-{p ^{\chi (m) -\chi (0)} \mathrm{can}}
& 
{\left ( \smash{\widetilde{\D}} _{\PP ^\sharp} ^{(\lambda(0))} (T ) \right) ^{s} } 
\ar[r] ^-{}
\ar[d] ^-{p ^{\chi (m) -\chi (0)} \mathrm{can}}
& 
{\mathrm{coker} ~\phi ^{(0)} } 
\ar[r] ^-{}
\ar[d] ^-{}
& 
{0}
\\ 
{\left ( \smash{\widetilde{\D}} _{\PP ^\sharp} ^{(\lambda(m))} (T ) \right) ^{r }} 
\ar[r] ^-{\phi ^{(m)}}
& 
{\left ( \smash{\widetilde{\D}} _{\PP ^\sharp} ^{(\lambda(m))} (T ) \right) ^{s} } 
\ar[r] ^-{}
& 
{\mathrm{coker} ~\phi ^{(m)} } 
\ar[r] ^-{}
& 
{0}
}
\end{equation}
dont les suites horizontales sont exactes.
On en déduit le diagramme commutatif
\begin{equation}
\label{2.1.12.1}
\xymatrix @C=2cm  @R=0,3cm{
{ \left (\smash{\widetilde{\D}} _{\PP ^\sharp} ^{(\lambda (m))} (T)
\right) ^{r }} 
\ar[r] ^-{\smash{\widetilde{\D}} _{\PP ^\sharp} ^{(\lambda (m))} (T) 
\otimes 
\phi ^{(0)}}
\ar[d] ^-{p ^{\chi (m) -\chi (0)}}
& 
{\left ( \smash{\widetilde{\D}} _{\PP ^\sharp} ^{(\lambda(m))} (T ) \right) ^{s} } 
\ar[r] ^-{}
\ar[d] ^-{p ^{\chi (m) -\chi (0)}}
& 
{ \smash{\widetilde{\D}} _{\PP ^\sharp} ^{(\lambda (m))} (T) 
\otimes _{\smash{\widetilde{\D}} _{\PP ^\sharp} ^{(\lambda (0))} (T)}\mathrm{coker} ~\phi ^{(0)}} 
\ar[r] ^-{}
\ar[d] ^-{}
& 
{0}
\\ 
{\left ( \smash{\widetilde{\D}} _{\PP ^\sharp} ^{(\lambda(m))} (T ) \right) ^{r }} 
\ar[r] ^-{\phi ^{(m)}}
& 
{\left ( \smash{\widetilde{\D}} _{\PP ^\sharp} ^{(\lambda(m))} (T ) \right) ^{s} } 
\ar[r] ^-{}
& 
{\mathrm{coker} ~\phi ^{(m)} } 
\ar[r] ^-{}
& 
{0}
}
\end{equation}
dont les suites horizontales sont exactes.
La famille des morphismes verticaux de droite de \ref{2.1.12.1}
signifie que l'on dispose du morphisme canonique 
$\lambda ^{*}\smash{\widetilde{\D}} _{\PP ^\sharp} ^{(\bullet)} (T) 
\otimes _{\smash{\widetilde{\D}} _{\PP ^\sharp} ^{(\lambda (0))} (T)}\mathrm{coker} ~\phi ^{(0)} 
\to 
\mathrm{coker} ~\phi ^{(\bullet)}$ de $M  (\lambda ^{*}\smash{\widetilde{\D}} _{\PP ^\sharp} ^{(\bullet)} (T))$.
Grâce au lemme \ref{ind-isog-coh-Q}, comme les deux morphismes verticaux de gauche de \ref{2.1.12.1} sont des isomorphismes
après tensorisation par $\Q$, ce dernier
est une ind-isogénie de 
$M  (\lambda ^{*}\smash{\widetilde{\D}} _{\PP ^\sharp} ^{(\bullet)} (T))$. 
Comme 
$\lambda ^{*}\smash{\widetilde{\D}} _{\PP ^\sharp} ^{(\bullet)} (T) 
\otimes _{\smash{\widetilde{\D}} _{\PP ^\sharp} ^{(\lambda (0))} (T)}\mathrm{coker} ~\phi ^{(0)} $
est un $\lambda ^{*}\smash{\widetilde{\D}} _{\PP ^\sharp} ^{(\bullet)} (T)$-module globalement de présentation finie,
on en déduit le résultat demandé.
\end{proof}

\begin{lemm}
\label{ind-isog-cohérence}
Les notions de $\lambda ^{*}\smash{\widetilde{\D}} _{\PP ^\sharp} ^{(\bullet)} (T)$-module 
localement de présentation finie à ind-isogénie près 
et de 
$\lambda ^{*}\smash{\widetilde{\D}} _{\PP ^\sharp} ^{(\bullet)} (T)$-module cohérent à ind-isogénie près sont 
 équivalentes. 
En particulier, le faisceau d'anneaux
$\lambda ^{*}\smash{\widetilde{\D}} _{\PP ^\sharp} ^{(\bullet)} (T)$
est alors un
$\lambda ^{*}\smash{\widetilde{\D}} _{\PP ^\sharp} ^{(\bullet)} (T)$-module cohérent à ind-isogénie près.

\end{lemm}

\begin{proof}
La cohérence à ind-isogénie près entraîne clairement la locale présentation finie à ind-isogénie près. 
Réciproquement, soit 
$\E ^{(\bullet)}$
un $\lambda ^{*}\smash{\widetilde{\D}} _{\PP ^\sharp} ^{(\bullet)} (T)$-module 
localement de présentation finie à ind-isogénie près. 
Comme le lemme est locale et comme les propriétés équivalentes à établir sont stables par isomorphisme
dans $\underrightarrow{M} _{\Q}  (\lambda ^{*}\smash{\widetilde{\D}} _{\PP ^\sharp} ^{(\bullet)} (T))$, 
via les lemmes \ref{eqcat-locpfotimes} et \ref{caract-coh-isog},
on peut supposer qu'il existe un 
$\smash{\widetilde{\D}} _{\PP ^\sharp} ^{(\lambda (0))} (T)$-module globalement de présentation finie
$\FF  ^{(0)}$ tel que 
$\E ^{(\bullet)} = \lambda ^{*}\smash{\widetilde{\D}} _{\PP ^\sharp} ^{(\bullet)} (T ) 
\otimes _{\smash{\widetilde{\D}} _{\PP ^\sharp} ^{(\lambda (0))} (T)} \FF  ^{(0)}$.
Soit 
$f ^{(\bullet)}
\colon \left ( \lambda ^{*}\smash{\widetilde{\D}} _{\PP ^\sharp} ^{(\bullet)} (T) \right) ^r 
\to 
\E ^{(\bullet)}$ 
un morphisme de 
$\underrightarrow{M} _{\Q}  (\lambda ^{*}\smash{\widetilde{\D}} _{\PP ^\sharp} ^{(\bullet)} (T))$.
Il s'agit de vérifier que 
$\ker f ^{(\bullet)}$ est de type fini à ind-isogénie près. 
Soient 
$\chi \in M$, 
$a ^{(\bullet)}\colon \left ( \lambda ^{*}\smash{\widetilde{\D}} _{\PP ^\sharp} ^{(\bullet)} (T) \right) ^r \to 
\chi ^*\E ^{(\bullet)}$ 
un morphisme de 
$M  (\lambda ^{*}\smash{\widetilde{\D}} _{\PP ^\sharp} ^{(\bullet)} (T))$ 
qui soit un représentant de 
$f ^{(\bullet)}$.
Comme $\smash{\widetilde{\D}} _{\PP ^\sharp} ^{(\lambda (m))} (T)$ est un anneau cohérent, 
comme les extensions
$\smash{\widetilde{\D}} _{\PP ^\sharp} ^{(\lambda (m))} (T) _\Q
\to 
\smash{\widetilde{\D}} _{\PP ^\sharp} ^{(\lambda (m +1))} (T) _\Q$
sont plates, il découle alors du lemme \ref{ind-isog-coh-Q} que 
le morphisme canonique 
$\lambda ^{*}\smash{\widetilde{\D}} _{\PP ^\sharp} ^{(\bullet)} (T) 
\otimes _{\smash{\widetilde{\D}} _{\PP ^\sharp} ^{(\lambda (0))} (T)}\ker a ^{(0)} 
\to 
\ker a ^{(\bullet)}$
est une ind-isogénie
de $M  (\lambda ^{*}\smash{\widetilde{\D}} _{\PP ^\sharp} ^{(\bullet)} (T))$. 
D'où le résultat.
\end{proof}

\begin{rema}
Avec les notations du  lemme \ref{ind-isog-cohérence}, 
comme les extensions
$\smash{\widetilde{\D}} _{\PP ^\sharp} ^{(\lambda (m))} (T) 
\to 
\smash{\widetilde{\D}} _{\PP ^\sharp} ^{(\lambda (m +1))} (T)$
ne sont pas plates, 
il semble faux que 
le faisceau d'anneaux
$\lambda ^{*}\smash{\widetilde{\D}} _{\PP ^\sharp} ^{(\bullet)} (T)$
soit cohérent (comme objet de $M  (\lambda ^*\smash{\widetilde{\D}} _{\PP ^\sharp} ^{(\bullet)} (T))$).
\end{rema}

\begin{nota}
Notons
$\underrightarrow{M} _{\Q, \mathrm{coh}} (\lambda ^* \smash{\widetilde{\D}} _{\PP ^\sharp} ^{(\bullet)} (T))$
la sous-catégorie pleine de
$\underrightarrow{M} _{\Q} (\lambda ^*  \smash{\widetilde{\D}} _{\PP ^\sharp} ^{(\bullet)} (T))$
des 
$\smash{\widetilde{\D}} _{\PP ^\sharp} ^{(\bullet)} (T)$-modules cohérents à ind-isogénie près. 
De même, lorsque $\PP$ est affine, en reprenant tout ce qui précède dans ce contexte, 
on définit de manière identique 
la sous-catégorie pleine 
$\underrightarrow{M} _{\Q, \mathrm{coh}} (\lambda ^* \smash{\widetilde{D}} _{\PP ^{\sharp}} ^{(\bullet)} (T))$
de
$\underrightarrow{M} _{\Q} (\lambda ^*  \smash{\widetilde{D}} _{\PP ^{\sharp}} ^{(\bullet)} (T))$
des 
$\smash{\widetilde{D}} _{\PP ^{\sharp}} ^{(\bullet)} (T)$-modules cohérents à ind-isogénie près. 

\end{nota}

\begin{lemm}
\label{proj-local}
Si $\PP ^{(\bullet)}$ est un objet projectif de $M (\lambda ^* \smash{\widetilde{\D}} _{\PP ^\sharp} ^{(\bullet)} (T))$,
alors $\PP ^{(\bullet)}$ est un objet projectif de 
$\underrightarrow{M} _{\Q} (\lambda ^* \smash{\widetilde{\D}} _{\PP ^\sharp} ^{(\bullet)} (T))$.
\end{lemm}

\begin{proof}
Soit $g ^{(\bullet)}\colon \E ^{(\bullet)} \to \FF ^{(\bullet)}$ un épimorphisme de 
$\underrightarrow{M} _{\Q} (\lambda ^* \smash{\widetilde{\D}} _{\PP ^\sharp} ^{(\bullet)} (T))$
et
$f ^{(\bullet)} \colon \PP ^{(\bullet)} \to \FF ^{(\bullet)}$ un morphisme de 
$\underrightarrow{M} _{\Q} (\lambda ^* \smash{\widetilde{\D}} _{\PP ^\sharp} ^{(\bullet)} (T))$.
Soient $\chi \in M$, 
$\phi ^{(\bullet)} \colon \PP ^{(\bullet)} \to \chi ^{*} \FF ^{(\bullet)}$ 
et
$\psi ^{(\bullet)}\colon \E ^{(\bullet)} \to \chi ^{*} \FF ^{(\bullet)}$  
des morphismes de $M (\lambda ^* \smash{\widetilde{\D}} _{\PP ^\sharp} ^{(\bullet)} (T))$
représentant
respectivement 
$f ^{(\bullet)}$ et $g ^{(\bullet)}$.
Notons $\G ^{(\bullet)}:=  \chi ^{*} \FF ^{(\bullet)}$,
$\H ^{(\bullet)}:= \mathrm{im} ~\psi ^{(\bullet)}$
et 
$\alpha ^{(\bullet)}\colon \H ^{(\bullet)} \to \G ^{(\bullet)}$ l'inclusion canonique.
Comme $g ^{(\bullet)}$ est un épimorphisme, il existe $\chi _{1} \in M$ et
un morphisme $\beta ^{(\bullet)} \colon \G ^{(\bullet)} \to \chi _{1} ^{*} \H ^{(\bullet)}$
tels que $\beta ^{(\bullet)} \circ \alpha ^{(\bullet)} $ et $\chi _{1} ^{*}  (\alpha ^{(\bullet)}  )\circ \beta ^{(\bullet)} $
sont les morphismes canoniques.
Comme $\PP ^{(\bullet)} $ est projectif, 
il existe un unique morphisme 
$\theta ^{(\bullet)} \colon \PP ^{(\bullet)} \to \chi _{1} ^{*} \E ^{(\bullet)} $ dont le composé avec 
la surjection $\chi _{1} ^{*} \E ^{(\bullet)} \to \chi _{1} ^{*} \H ^{(\bullet)}$ induite par $\chi _{1} ^{*} \psi ^{(\bullet)}$
donne $\beta ^{(\bullet)} \circ \phi ^{(\bullet)}$.
On remarque de plus que 
$\chi _{1} ^{*} (\psi ^{(\bullet)}) \circ \theta ^{(\bullet)} $
est égal au composé de $\phi ^{(\bullet)}$ avec le morphisme canonique
$\G ^{(\bullet)} \to \chi _{1} ^{*} \G ^{(\bullet)}$.
Si on note $h ^{(\bullet)}\colon \PP ^{(\bullet)} \to  \E ^{(\bullet)} $ le morphisme de
$\underrightarrow{M} _{\Q} (\lambda ^* \smash{\widetilde{\D}} _{\PP ^\sharp} ^{(\bullet)} (T))$
dont $\theta ^{(\bullet)}$ est un représentant, on a donc 
$g ^{(\bullet)} \circ h ^{(\bullet)} = f ^{(\bullet)}$. 
\end{proof}

\begin{rema}
Je ne sais pas s'il est vrai que si $\I ^{(\bullet)}$ est un objet injectif de 
$M (\lambda ^* \smash{\widetilde{\D}} _{\PP ^\sharp} ^{(\bullet)} (T))$,
alors $\I ^{(\bullet)}$ est un objet injectif de 
$\underrightarrow{M} _{\Q} (\lambda ^* \smash{\widetilde{\D}} _{\PP ^\sharp} ^{(\bullet)} (T))$.
\end{rema}

\begin{prop}
\label{Q-coh-stab}
La sous catégorie pleine 
$\underrightarrow{M} _{\Q, \mathrm{coh}} (\lambda ^* \smash{\widetilde{\D}} _{\PP ^\sharp} ^{(\bullet)} (T))$
de 
$\underrightarrow{M} _{\Q}  (\lambda ^* \smash{\widetilde{\D}} _{\PP ^\sharp} ^{(\bullet)} (T))$
est stable par isomorphisme,
noyau, conoyau, extension.
\end{prop}

\begin{proof}
La stabilité par isomorphisme est triviale.
Vérifions à présent la stabilité par noyau et conoyau. 
Soit $f ^{(\bullet)}\colon \E ^{(\bullet)} \to \FF ^{(\bullet)}$ un morphisme
de
$\underrightarrow{M} _{\Q, \mathrm{coh}} (\lambda ^* \smash{\widetilde{\D}} _{\PP ^\sharp} ^{(\bullet)} (T))$.
La cohérence à ind-isogénie étant une propriété locale et stable par isomorphisme de 
$\underrightarrow{M} _{\Q}  (\lambda ^* \smash{\widetilde{\D}} _{\PP ^\sharp} ^{(\bullet)} (T))$, 
grâce à \ref{caract-coh-isog}, on peut supposer que 
$\E ^{(\bullet)}$ et $\FF ^{(\bullet)}$ sont des 
$\lambda ^* \smash{\widetilde{\D}} _{\PP ^\sharp} ^{(\bullet)} (T)$-modules 
globalement de présentation finie.
Soient $\chi \in M$ et 
$\phi ^{(\bullet)}\colon
\E ^{(\bullet)} \to \chi ^{*}\FF ^{(\bullet)}$ un morphisme
de
$M (\lambda ^* \smash{\widetilde{\D}} _{\PP ^\sharp} ^{(\bullet)} (T))$
qui représente $f ^{(\bullet)}$.
Via le lemme \ref{ind-isog-coh-Q},
on vérifie alors que
les morphismes canoniques 
$\lambda ^{*}\smash{\widetilde{\D}} _{\PP ^\sharp} ^{(\bullet)} (T) 
\otimes _{\smash{\widetilde{\D}} _{\PP ^\sharp} ^{(\lambda (0))} (T)}\ker \phi ^{(0)} 
\to 
\ker \phi ^{(\bullet)}$
et 
$\lambda ^{*}\smash{\widetilde{\D}} _{\PP ^\sharp} ^{(\bullet)} (T) 
\otimes _{\smash{\widetilde{\D}} _{\PP ^\sharp} ^{(\lambda (0))} (T)}\mathrm{coker} \phi ^{(0)} 
\to 
\mathrm{coker} \phi ^{(\bullet)}$
sont des ind-isogénies
de 
$M  (\lambda ^* \smash{\widetilde{\D}} _{\PP ^\sharp} ^{(\bullet)} (T))$. 

Traitons maintenant la stabilité par extension.
Soit 
$0 \to \E ^{(\bullet)} \underset{f ^{(\bullet)}}{\longrightarrow}
\FF ^{(\bullet)} \underset{g ^{(\bullet)}}{\longrightarrow}
\G ^{(\bullet)} \to 0$
une suite exacte dans la catégorie 
$\underrightarrow{M} _{\Q} (\lambda ^* \smash{\widetilde{\D}} _{\PP ^\sharp} ^{(\bullet)} (T))$
avec 
$\E ^{(\bullet)},~ \G ^{(\bullet)} \in 
\underrightarrow{M} _{\Q, \mathrm{coh}} (\lambda ^* \smash{\widetilde{\D}} _{\PP ^\sharp} ^{(\bullet)} (T))$.
Comme la cohérence de 
$\FF ^{(\bullet)}$ est locale, on peut supposer qu'il existe 
des épimorphismes de la forme
$a ^{(\bullet)}\colon 
\left ( \lambda ^* \smash{\widetilde{\D}} _{\PP ^\sharp} ^{(\bullet)} (T) \right ) ^{r}
\to 
 \E ^{(\bullet)}$
 et
 $b ^{(\bullet)}\colon  \left ( \lambda ^* \smash{\widetilde{\D}} _{\PP ^\sharp} ^{(\bullet)} (T) \right ) ^{s}
\to 
 \G ^{(\bullet)}$.
 D'après \ref{proj-local}, il existe $h ^{(\bullet)}\colon  \left ( \lambda ^* \smash{\widetilde{\D}} _{\PP ^\sharp} ^{(\bullet)} (T) \right ) ^{s}
\to 
 \FF ^{(\bullet)}$ tel que 
 $g ^{(\bullet)} \circ h ^{(\bullet)} = b ^{(\bullet)}$.
On en déduit que 
$\left ( \lambda ^* \smash{\widetilde{\D}} _{\PP ^\sharp} ^{(\bullet)} (T) \right ) ^{r} \oplus \left ( \lambda ^* \smash{\widetilde{\D}} _{\PP ^\sharp} ^{(\bullet)} (T) \right ) ^{s}
\to 
 \FF ^{(\bullet)}$
 défini par 
$ f ^{(\bullet)} \circ a ^{(\bullet)} + ~h ^{(\bullet)} $
est un épimorphisme.
On a donc validé que 
$\FF ^{(\bullet)}$ est de type fini à ind-isogénie près. 
Soit 
$\alpha ^{(\bullet)}\colon 
\left ( \lambda ^* \smash{\widetilde{\D}} _{\PP ^\sharp} ^{(\bullet)} (T) \right ) ^{t}
\to 
 \FF ^{(\bullet)}$ un morphisme de 
 $\underrightarrow{M} _{\Q} (\lambda ^* \smash{\widetilde{\D}} _{\PP ^\sharp} ^{(\bullet)} (T))$.
 Il reste à vérifier que $\ker \alpha ^{(\bullet)}$ est de type fini à ind-isogénie près.
 Comme $\G ^{(\bullet)}$ est cohérent à ind-isogénie près, 
 $\ker (g ^{(\bullet)} \circ \alpha ^{(\bullet)}) $ est de type fini à ind-isogénie près. Comme ce que l'on doit prouver est local, 
 on peut supposer qu'il existe 
 un épimorphisme de 
  $\underrightarrow{M} _{\Q} (\lambda ^* \smash{\widetilde{\D}} _{\PP ^\sharp} ^{(\bullet)} (T))$
  de la forme
$\beta ^{(\bullet)}\colon 
\left ( \lambda ^* \smash{\widetilde{\D}} _{\PP ^\sharp} ^{(\bullet)} (T) \right ) ^{u}
\to  
  \ker (g ^{(\bullet)} \circ \alpha ^{(\bullet)})$.
Comme   $\underrightarrow{M} _{\Q} (\lambda ^* \smash{\widetilde{\D}} _{\PP ^\sharp} ^{(\bullet)} (T))$ est une catégorie abélienne,
on dispose de l'isomorphisme canonique 
$ \alpha ^{(\bullet)} (   \ker (g ^{(\bullet)} \circ \alpha ^{(\bullet)})) 
\riso 
\mathrm{im} (\alpha ^{(\bullet)}) \cap   \ker g ^{(\bullet)}$.
En particulier, si on note 
$\delta ^{(\bullet)} \colon 
   \ker (g ^{(\bullet)} \circ \alpha ^{(\bullet)}) \to \FF ^{(\bullet)}$ le composé de 
   $\alpha ^{(\bullet)}$ avec le monomorphisme  $   \ker (g ^{(\bullet)} \circ \alpha ^{(\bullet)})  \subset 
   \left ( \lambda ^* \smash{\widetilde{\D}} _{\PP ^\sharp} ^{(\bullet)} (T) \right ) ^{t}$, on obtient la factorisation 
   canonique $\epsilon ^{(\bullet)} \colon 
   \ker (g ^{(\bullet)} \circ \alpha ^{(\bullet)}) \to \E ^{(\bullet)}$ de $\delta ^{(\bullet)}$ par $f ^{(\bullet)}$.
Comme $\E ^{(\bullet)}$ est cohérent à ind-isogénie près,
$\ker ( \epsilon ^{(\bullet)} \circ \beta  ^{(\bullet)})$ est de type fini à ind-isogénie près. 
Comme $f ^{(\bullet)}$ est un monomorphisme, il en résulte que 
$\ker (f ^{(\bullet)} \circ \epsilon ^{(\bullet)} \circ \beta  ^{(\bullet)})$ est de type fini à ind-isogénie près. 
Comme $f ^{(\bullet)} \circ \epsilon ^{(\bullet)} \circ \beta  ^{(\bullet)}=\delta ^{(\bullet)} \circ \beta ^{(\bullet)}$, 
alors $\ker (\delta ^{(\bullet)} \circ \beta ^{(\bullet)} )$ est de type fini à ind-isogénie près.
Or, on dispose des isomorphismes 
$\beta ^{(\bullet)} (\ker (\delta ^{(\bullet)} \circ \beta ^{(\bullet)} ))
\riso 
\mathrm{im} (\beta ^{(\bullet)}) \cap \ker (\delta ^{(\bullet)} )
\riso 
\ker (\alpha ^{(\bullet)})$, qui est donc de type fini à ind-isogénie près.

\end{proof}

\begin{lemm}
\label{lemm-ind-iso-loc-pf}
Soit 
$\E ^{(\bullet)} $ un $\lambda ^* \smash{\widetilde{\D}} _{\PP ^\sharp} ^{(\bullet)} (T)$-module.
Le module $\E ^{(\bullet)}$ est 
isomorphe dans 
$\underrightarrow{M} _{\Q} (\lambda ^* \smash{\widetilde{\D}} _{\PP ^\sharp} ^{(\bullet)} (T))$
à un 
$\lambda ^* \smash{\widetilde{\D}} _{\PP ^\sharp} ^{(\bullet)} (T)$-module localement de présentation finie
si et seulement si $\E ^{(0)} $ est isogène à un $\smash{\widetilde{\D}} _{\PP ^\sharp} ^{(\lambda (0))} (T)$-module cohérent 
et le morphisme canonique
$\lambda ^{*}\smash{\widetilde{\D}} _{\PP ^\sharp} ^{(\bullet)} (T) 
\otimes _{\smash{\widetilde{\D}} _{\PP ^\sharp} ^{(\lambda (0))} (T)} \E ^{(0)} 
\to  \E ^{(\bullet)}$
est une ind-isogénie de 
$M (\lambda ^* \smash{\widetilde{\D}} _{\PP ^\sharp} ^{(\bullet)} (T))$. 

De plus, deux $\lambda ^* \smash{\widetilde{\D}} _{\PP ^\sharp} ^{(\bullet)} (T)$-modules localement de présentation finie
sont isomorphes dans
$\underrightarrow{M} _{\Q} (\lambda ^* \smash{\widetilde{\D}} _{\PP ^\sharp} ^{(\bullet)} (T))$
si et seulement s'ils sont ind-isogènes dans $M (\lambda ^* \smash{\widetilde{\D}} _{\PP ^\sharp} ^{(\bullet)} (T))$. 
\end{lemm}

\begin{proof}
Supposons qu'il existe $\FF ^{(0)}$ 
un $\smash{\widetilde{\D}} _{\PP ^\sharp} ^{(\lambda (0))} (T)$-module cohérent, 
$f ^{(0)}\colon \E ^{(0)} \to \FF ^{(0)}$ et 
$g ^{(0)}\colon \FF ^{(0)} \to \E ^{(0)}$
deux morphismes $\smash{\widetilde{\D}} _{\PP ^\sharp} ^{(\lambda (0))} (T)$-linéaires tels que
$f ^{(0)} \circ g ^{(0)}$ et $g ^{(0)} \circ f ^{(0)}$ sont les multiplications par $p ^{n}$.
On obtient par extension
les morphismes 
$f ^{(\bullet)} \colon 
\lambda ^{*}\smash{\widetilde{\D}} _{\PP ^\sharp} ^{(\bullet)} (T) 
\otimes _{\smash{\widetilde{\D}} _{\PP ^\sharp} ^{(\lambda (0))} (T)} \E ^{(0)} 
\to 
\lambda ^{*}\smash{\widetilde{\D}} _{\PP ^\sharp} ^{(\bullet)} (T) 
\otimes _{\smash{\widetilde{\D}} _{\PP ^\sharp} ^{(\lambda (0))} (T)} \FF ^{(0)} $
et $g ^{(\bullet)} \colon 
\lambda ^{*}\smash{\widetilde{\D}} _{\PP ^\sharp} ^{(\bullet)} (T) 
\otimes _{\smash{\widetilde{\D}} _{\PP ^\sharp} ^{(\lambda (0))} (T)} \FF ^{(0)} 
\to 
\lambda ^{*}\smash{\widetilde{\D}} _{\PP ^\sharp} ^{(\bullet)} (T) 
\otimes _{\smash{\widetilde{\D}} _{\PP ^\sharp} ^{(\lambda (0))} (T)} \E ^{(0)} $.
En prenant $\chi \in M$ la fonction constante égale à $n$, en composant 
$g ^{(\bullet)}$ avec le morphisme canonique $id \to \chi ^*$, on obtient la flèche 
$h ^{(\bullet)} \colon 
\lambda ^{*}\smash{\widetilde{\D}} _{\PP ^\sharp} ^{(\bullet)} (T) 
\otimes _{\smash{\widetilde{\D}} _{\PP ^\sharp} ^{(\lambda (0))} (T)} \FF ^{(0)} 
\to 
\chi ^* (\lambda ^{*}\smash{\widetilde{\D}} _{\PP ^\sharp} ^{(\bullet)} (T) 
\otimes _{\smash{\widetilde{\D}} _{\PP ^\sharp} ^{(\lambda (0))} (T)} \E ^{(0)} )$.
Alors $\chi ^* (f ^{(\bullet)}) \circ h ^{(\bullet)}$ et $h ^{(\bullet)} \circ f ^{(\bullet)}$
sont les morphismes canoniques (i.e. les multiplications par $p ^{n}$ à chaque niveau).
D'où la suffisance de la première assertion à valider.
Réciproquement, supposons que 
$\E ^{(\bullet)}$ est 
isomorphe dans 
$\underrightarrow{M} _{\Q} (\lambda ^* \smash{\widetilde{\D}} _{\PP ^\sharp} ^{(\bullet)} (T))$
à un 
$\lambda ^* \smash{\widetilde{\D}} _{\PP ^\sharp} ^{(\bullet)} (T)$-module $\FF ^{(\bullet)}$ localement de présentation finie.
Dans ce cas, il existe $\chi \in M$ et une ind-isogénie de 
$M (\lambda ^* \smash{\widetilde{\D}} _{\PP ^\sharp} ^{(\bullet)} (T))$
de la forme
$\E ^{(\bullet)} \to \chi ^{*} \FF ^{(\bullet)}$.
On en déduit que $\E ^{(0)} $ et $\FF ^{(0)}$ sont isogènes.
Avec \ref{eqcat-locpfotimes}, il en résulte la première 
 ind-isogénie de $M (\lambda ^* \smash{\widetilde{\D}} _{\PP ^\sharp} ^{(\bullet)} (T))$ de la forme
$\lambda ^{*}\smash{\widetilde{\D}} _{\PP ^\sharp} ^{(\bullet)} (T) 
\otimes _{\smash{\widetilde{\D}} _{\PP ^\sharp} ^{(\lambda (0))} (T)} \E ^{(0)} 
\to  \FF ^{(\bullet)}
\to \chi ^{*} \FF ^{(\bullet)}$.
Comme ce composé est égal au composé
$\lambda ^{*}\smash{\widetilde{\D}} _{\PP ^\sharp} ^{(\bullet)} (T) 
\otimes _{\smash{\widetilde{\D}} _{\PP ^\sharp} ^{(\lambda (0))} (T)} \E ^{(0)} 
\to  \E ^{(\bullet)} \to \chi ^{*} \FF ^{(\bullet)}$, on en déduit la première assertion.
La seconde assertion se traite de manière analogue.

\end{proof}

\subsection{Cohérence à lim-ind-isogénie près}

\begin{defi}
\label{coh-loc-pres-fini-lim-ind-iso}
Soit $\E ^{(\bullet)}$ un 
$\smash{\widetilde{\D}} _{\PP ^\sharp} ^{(\bullet)} (T)$-module.
Le module 
$\E ^{(\bullet)}$
est un 
$\smash{\widetilde{\D}} _{\PP ^\sharp} ^{(\bullet)} (T)$-module 
de type fini à lim-ind-isogénie près 
s'il existe
un recouvrement ouvert 
$(\PP _i ) _{i\in I}$ de $\PP$ tel que, pour tout $i \in I$, on dispose  
d'une suite exacte dans 
$\underrightarrow{LM} _{\Q}  (\smash{\widetilde{\D}} _{\PP ^\sharp} ^{(\bullet)} (T))$ de la forme: 
$\left ( \smash{\widetilde{\D}} _{\PP ^{\sharp} _i} ^{(\bullet)} (T  \cap P _i) \right) ^{r _i}
\to 
\E ^{(\bullet)} | \PP _i
\to 0$,
où $r _i\in \N$.
De même, on définie la notion de 
$\smash{\widetilde{\D}} _{\PP ^\sharp} ^{(\bullet)} (T)$-module 
localement de présentation finie à lim-ind-isogénie près (resp. 
globalement de présentation finie à lim-ind-isogénie près, resp. cohérent à lim-ind-isogénie près). 
\end{defi}

Le lemme qui suit sera amélioré via \ref{strict-m0} (mais ce dernier utilise \ref{caract-coh-lim-isog} qui est une première étape). 
\begin{lemm}
\label{caract-coh-lim-isog}
Soit $\E ^{(\bullet)}$ un 
$\smash{\widetilde{\D}} _{\PP ^\sharp} ^{(\bullet)} (T)$-module.
Le $\smash{\widetilde{\D}} _{\PP ^\sharp} ^{(\bullet)} (T)$-module 
$\E ^{(\bullet)}$ est globalement de présentation finie à lim-ind-isogénie près si et seulement s'il existe $\lambda \in L$ et un 
$\lambda ^{*}\smash{\widetilde{\D}} _{\PP ^\sharp} ^{(\bullet)} (T)$-module $\FF ^{(\bullet)}$ globalement de présentation finie
tel que 
$\E ^{(\bullet)}$ et $\FF ^{(\bullet)}$ soit isomorphe dans 
$\underrightarrow{LM} _{\Q}  (\smash{\widetilde{\D}} _{\PP ^\sharp} ^{(\bullet)} (T))$.
\end{lemm}

\begin{proof}
La suffisance est triviale. 
Réciproquement, supposons 
que 
$\E ^{(\bullet)}$
soit le conoyau dans 
$\underrightarrow{LM} _{\Q}  (\smash{\widetilde{\D}} _{\PP ^\sharp} ^{(\bullet)} (T))$
d'une flèche de la forme
$\left (\smash{\widetilde{\D}} _{\PP ^\sharp} ^{(\bullet)} (T ) \right) ^{r }
\to 
\left (\smash{\widetilde{\D}} _{\PP ^\sharp} ^{(\bullet)} (T ) \right) ^{s}$.
Soient $\chi _0\in M$, $\lambda\in L$ et 
$\phi ^{(\bullet)} \colon  \left (\smash{\widetilde{\D}} _{\PP ^\sharp} ^{(\bullet)} (T ) \right) ^{r }
\to
 \lambda ^{*} \chi  _0 ^* \left (\smash{\widetilde{\D}} _{\PP ^\sharp} ^{(\bullet)} (T ) \right) ^{s} 
=
\chi  ^* \left ( \lambda ^{*}\smash{\widetilde{\D}} _{\PP ^\sharp} ^{(\bullet)} (T ) \right) ^{s}$
un représentant de ce morphisme avec $\chi := \chi _0 \circ \lambda\in M$ (l'égalité provient
de $ \lambda ^{*} \chi  _0 ^* = \chi  ^* \lambda ^{*} $).
Comme $\chi ^* \left ( \lambda ^{*}\smash{\widetilde{\D}} _{\PP ^\sharp} ^{(\bullet)} (T ) \right) ^{s}$
est muni d'une structure canonique 
de $ \lambda ^{*}\smash{\widetilde{\D}} _{\PP ^\sharp} ^{(\bullet)} (T ) $-module qui prolonge celle de
$\smash{\widetilde{\D}} _{\PP ^\sharp} ^{(\bullet)} (T ) $-module, on en déduit par adjonction (voir \ref{adjonction-oubotimes})
que le morphisme 
$\phi ^{(\bullet)}$ se factorise de manière unique
en un morphisme de 
$M  (\lambda ^{*} \smash{\widetilde{\D}} _{\PP ^\sharp} ^{(\bullet)} (T))$
de la forme
$\psi ^{(\bullet)} \colon  \left (\lambda ^{*} \smash{\widetilde{\D}} _{\PP ^\sharp} ^{(\bullet)} (T ) \right) ^{r }
\to 
\chi ^* \left ( \lambda ^{*}\smash{\widetilde{\D}} _{\PP ^\sharp} ^{(\bullet)} (T ) \right) ^{s}$.
Comme 
$\mathrm{coker} ~\psi ^{(\bullet)}$ (calculé dans 
$\underrightarrow{M} _{\Q}  (\lambda ^{*}\smash{\widetilde{\D}} _{\PP ^\sharp} ^{(\bullet)} (T))$)
est un $\lambda ^{*}\smash{\widetilde{\D}} _{\PP ^\sharp} ^{(\bullet)} (T )$-module
globalement de présentation finie à ind-isogénie près, 
d'après le lemme \ref{caract-coh-isog},
le module $\mathrm{coker} ~\psi ^{(\bullet)}$ est 
 isomorphe dans
$\underrightarrow{M} _{\Q}  (\lambda ^{*}\smash{\widetilde{\D}} _{\PP ^\sharp} ^{(\bullet)} (T))$
(et donc dans $\underrightarrow{LM} _{\Q}  (\smash{\widetilde{\D}} _{\PP ^\sharp} ^{(\bullet)} (T))$ 
via le foncteur oubli de \ref{MQlambda2LMQ})
à un 
$\lambda ^{*}\smash{\widetilde{\D}} _{\PP ^\sharp} ^{(\bullet)} (T )$-module globalement de présentation finie.
Comme $\mathrm{coker} ~\psi ^{(\bullet)}$ est isomorphe à $\E ^{(\bullet)}$ dans 
$\underrightarrow{LM} _{\Q}  (\smash{\widetilde{\D}} _{\PP ^\sharp} ^{(\bullet)} (T))$, on a donc terminé la preuve.
\end{proof}

\begin{lemm}
\label{lim-ind-isog-cohérence}
Les notions de $\smash{\widetilde{\D}} _{\PP ^\sharp} ^{(\bullet)} (T)$-module 
localement de présentation finie à lim-ind-isogénie près 
et 
de $\smash{\widetilde{\D}} _{\PP ^\sharp} ^{(\bullet)} (T)$-module 
cohérent à lim-ind-isogénie près 
sont équivalentes. 
En particulier, 
le faisceau d'anneaux
$\smash{\widetilde{\D}} _{\PP ^\sharp} ^{(\bullet)} (T)$
est alors un
$\smash{\widetilde{\D}} _{\PP ^\sharp} ^{(\bullet)} (T)$-module cohérent à lim-ind-isogénie près.

\end{lemm}

\begin{proof}
La cohérence à lim-ind-isogénie près entraîne clairement la locale présentation finie à lim-ind-isogénie près. 
Réciproquement, 
soit $\E ^{(\bullet)}$
un $\smash{\widetilde{\D}} _{\PP ^\sharp} ^{(\bullet)} (T)$-module 
localement de présentation finie à lim-ind-isogénie.
Comme la cohérence à lim-ind-isogénie est locale, on peut supposer
que
$\E ^{(\bullet)}$
un $\smash{\widetilde{\D}} _{\PP ^\sharp} ^{(\bullet)} (T)$-module 
globalement de présentation finie à lim-ind-isogénie. 
D'après
\ref{caract-coh-lim-isog}, 
par stabilité par isomorphismes de 
$\underrightarrow{LM} _{\Q}  (\smash{\widetilde{\D}} _{\PP ^\sharp} ^{(\bullet)} (T))$
de la propriété de cohérence à lim-ind-isogénie,
on se ramène au cas où 
$\E ^{(\bullet)}$
est un $\lambda ^{*}\smash{\widetilde{\D}} _{\PP ^\sharp} ^{(\bullet)} (T)$-module 
globalement de présentation finie.
Soit
$f ^{(\bullet)}
\colon \left ( \smash{\widetilde{\D}} _{\PP ^\sharp} ^{(\bullet)} (T) \right) ^r 
\to 
\E ^{(\bullet)}$ 
un morphisme de 
$\underrightarrow{LM} _{\Q}  (\smash{\widetilde{\D}} _{\PP ^\sharp} ^{(\bullet)} (T))$.
Il s'agit de vérifier que 
$\ker f ^{(\bullet)}$ est de type fini à lim-ind-isogénie près. 
Soient 
$\chi \in M$, $\lambda \in L$, 
$a ^{(\bullet)}\colon \left (\smash{\widetilde{\D}} _{\PP ^\sharp} ^{(\bullet)} (T) \right) ^r \to 
\chi ^*\lambda ^{*}\E ^{(\bullet)}$ 
un morphisme de 
$M  (\smash{\widetilde{\D}} _{\PP ^\sharp} ^{(\bullet)} (T))$ 
qui soit un représentant de 
$f ^{(\bullet)}$.
Par adjonction (voir \ref{adjonction-oubotimes}),
le morphisme 
$a ^{(\bullet)}$ se factorise de manière unique 
en un morphisme 
$b ^{(\bullet)}\colon \left (\lambda ^{*} \smash{\widetilde{\D}} _{\PP ^\sharp} ^{(\bullet)} (T) \right) ^r \to 
\chi ^*\lambda ^{*}\E^{(\bullet)} $ 
de 
$M  (\lambda ^{*} \smash{\widetilde{\D}} _{\PP ^\sharp} ^{(\bullet)} (T))$. 
Il résulte des propositions \ref{ind-isog-cohérence} et \ref{Q-coh-stab} que 
$\ker b ^{(\bullet)}$ est un 
$\lambda ^{*} \smash{\widetilde{\D}} _{\PP ^\sharp} ^{(\bullet)} (T)$-module de type fini à ind-isogénie près, 
et par conséquent un 
$ \smash{\widetilde{\D}} _{\PP ^\sharp} ^{(\bullet)} (T)$-module de type fini à lim-ind-isogénie près (en effet, on dispose du foncteur
oubli de \ref{MQlambda2LMQ}). 
Comme $\ker b ^{(\bullet)}$ et
$\ker f ^{(\bullet)}$ sont isomorphes dans
$\underrightarrow{LM} _{\Q}  (\smash{\widetilde{\D}} _{\PP ^\sharp} ^{(\bullet)} (T))$, 
on conclut alors la preuve. 
\end{proof}

\begin{nota}
\label{nota-(L)Mcoh}
On note
 $\underrightarrow{LM} _{\Q, \mathrm{coh}} (\smash{\widetilde{\D}} _{\PP ^\sharp} ^{(\bullet)} (T))$
la sous-catégorie pleine de
$\underrightarrow{LM} _{\Q} (\smash{\widetilde{\D}} _{\PP ^\sharp} ^{(\bullet)} (T))$
des $\smash{\widetilde{\D}} _{\PP ^\sharp} ^{(\bullet)} (T)$-modules cohérents à lim-ind-isogénie près.
De même, lorsque $\PP$ 
est affine,
on définit la sous-catégorie pleine 
 $\underrightarrow{LM} _{\Q, \mathrm{coh}} (\smash{\widetilde{D}} _{\PP ^{\sharp}} ^{(\bullet)} (T))$
de
$\underrightarrow{LM} _{\Q} (\smash{\widetilde{D}} _{\PP ^{\sharp}} ^{(\bullet)} (T))$
des $\smash{\widetilde{D}} _{\PP ^{\sharp}} ^{(\bullet)} (T)$-modules cohérents à lim-ind-isogénie près
(en remplaçant partout  {\og $\D$\fg} par {\og $D$\fg}).
\end{nota}

\begin{rema}
\label{rema-lem-MQlambda2LMQ}
Soit $\lambda \in L$.
De manière analogue, 
quitte à remplacer {\og $\smash{\widetilde{\D}}$ \fg}
par
{\og $\lambda ^* \smash{\widetilde{\D}}$ \fg},
on peut définir la sous-catégorie pleine
 $\underrightarrow{LM} _{\Q, \mathrm{coh}} (\lambda ^* \smash{\widetilde{\D}} _{\PP ^\sharp} ^{(\bullet)} (T))$
 de  $\underrightarrow{LM} _{\Q} (\lambda ^* \smash{\widetilde{\D}} _{\PP ^\sharp} ^{(\bullet)} (T))$
des $\lambda ^* \smash{\widetilde{\D}} _{\PP ^\sharp} ^{(\bullet)} (T)$-modules cohérents à lim-ind-isogénie près.
Avec le lemme \ref{lem-MQlambda2LMQ} et avec la caractérisation de la globale présentation finie donnée dans 
\ref{caract-coh-lim-isog}, 
les foncteurs 
$\mathrm{oub} _{\lambda}$ et 
$\lambda ^{*}\smash{\widetilde{\D}} _{\PP ^\sharp} ^{(\bullet)}(T) \otimes _{\smash{\widetilde{\D}} _{\PP ^\sharp} ^{(\bullet)}(T)}-$ 
induisent des équivalences de catégories quasi-inverses entre 
$\smash{\underrightarrow{LM}} _{\Q,\mathrm{coh}}  (\lambda ^{*} \smash{\widetilde{\D}} _{\PP ^\sharp} ^{(\bullet)}(T))$
et
$\smash{\underrightarrow{LM}} _{\Q,\mathrm{coh}}  ( \smash{\widetilde{\D}} _{\PP ^\sharp} ^{(\bullet)}(T))$.
\end{rema}

\begin{prop}
\label{prop-MQlambda2LMQ}
Les foncteurs 
oublis et 
$\widetilde{\D} _{\PP ^{\sharp}} ^{(\bullet)}(T) \otimes _{\widehat{\D} _{\PP ^{\sharp}} ^{(\bullet)}(T)}-$ 
sont des équivalences de catégories quasi-inverses entre 
${\underrightarrow{LM}} _{\Q}  (\widehat{\D} _{\PP ^{\sharp}} ^{(\bullet)}(T))$
(resp. ${\underrightarrow{LM}} _{\Q,\mathrm{coh}}  (\widehat{\D} _{\PP ^{\sharp}} ^{(\bullet)}(T))$)
et
${\underrightarrow{LM}} _{\Q}  ( \widetilde{\D} _{\PP ^{\sharp}} ^{(\bullet)}(T))$
(resp. ${\underrightarrow{LM}} _{\Q,\mathrm{coh}}  ( \widetilde{\D} _{\PP ^{\sharp}} ^{(\bullet)}(T))$).
\end{prop}

\begin{proof}
On vérifie de manière analogue à \ref{lem-MQlambda2LMQ} 
le cas non respectif. 
De même, on établit que 
les foncteurs oubli 
et
$\lambda _0 ^* \widehat{\D} _{\PP} ^{(\bullet)}(T) \otimes _{\widetilde{\D} _{\PP ^{\sharp}} ^{(\bullet)}(T)}-$
induisent des équivalences quasi-inverses entre 
$\underrightarrow{LM} _{\Q}  (\widetilde{\D} _{\PP ^{\sharp}} ^{(\bullet)}(T))$
et
$\underrightarrow{LM} _{\Q}  ({\lambda _0 ^* \widehat{\D}} _{\PP} ^{(\bullet)}(T))$.
Traitons à présent le cas respectif.
Les factorisations
$\widetilde{\D} _{\PP ^{\sharp}} ^{(\bullet)}(T) \otimes _{\widehat{\D} _{\PP ^{\sharp}} ^{(\bullet)}(T)}-
\colon 
{\underrightarrow{LM}} _{\Q,\mathrm{coh}}  (\widehat{\D} _{\PP ^{\sharp}} ^{(\bullet)}(T))
\to 
{\underrightarrow{LM}} _{\Q,\mathrm{coh}}  (\widetilde{\D} _{\PP ^{\sharp}} ^{(\bullet)}(T))$
et
$\lambda _0 ^* \widehat{\D} _{\PP} ^{(\bullet)}(T) \otimes _{\widetilde{\D} _{\PP ^{\sharp}} ^{(\bullet)}(T)}-
\colon 
\underrightarrow{LM} _{\Q,\mathrm{coh}}  (\widetilde{\D} _{\PP ^{\sharp}} ^{(\bullet)}(T))
\to 
\underrightarrow{LM} _{\Q,\mathrm{coh}}  ({\lambda _0 ^* \widehat{\D}} _{\PP} ^{(\bullet)}(T))$
sont aisées.
Or, d'après \ref{rema-lem-MQlambda2LMQ}, le foncteur oubli se factorise par 
$\underrightarrow{LM} _{\Q,\mathrm{coh}}  ({\lambda _0 ^* \widehat{\D}} _{\PP} ^{(\bullet)}(T)) 
\to 
\underrightarrow{LM} _{\Q,\mathrm{coh}}  (\widehat{\D} _{\PP ^{\sharp}} ^{(\bullet)}(T))$.
D'où le résultat.
\end{proof}

\begin{lemm}
\label{proj-local-LQ}
Si $\PP ^{(\bullet)}$ est un objet projectif de $\underrightarrow{M} _{\Q} (\smash{\widetilde{\D}} _{\PP ^\sharp} ^{(\bullet)} (T))$,
alors $\PP ^{(\bullet)}$ est un objet projectif de 
$\underrightarrow{LM} _{\Q} ( \smash{\widetilde{\D}} _{\PP ^\sharp} ^{(\bullet)} (T))$.
\end{lemm}

\begin{proof}
La preuve est identique à celle de \ref{proj-local} (on remplace les $\chi$ par des $\lambda$).
\end{proof}

\begin{prop}
\label{LQ-coh-stab}
La sous catégorie pleine 
$\underrightarrow{LM} _{\Q, \mathrm{coh}} ( \smash{\widetilde{\D}} _{\PP ^\sharp} ^{(\bullet)} (T))$
de 
$\underrightarrow{LM} _{\Q}  (\smash{\widetilde{\D}} _{\PP ^\sharp} ^{(\bullet)} (T))$
est stable par isomorphisme,
noyau, conoyau, extension.
\end{prop}

\begin{proof}
Il s'agit de reprendre la preuve de \ref{Q-coh-stab} (on utilise \ref{caract-coh-lim-isog}, \ref{proj-local-LQ} et le fait que 
la catégorie $\underrightarrow{LM} _{\Q}  (\smash{\widetilde{\D}} _{\PP ^\sharp} ^{(\bullet)} (T))$ est abélienne).
\end{proof}

\subsection{Cohérence au sens de Berthelot}

Rappelons d'abord la notion de cohérence au sens de Berthelot (voir \cite[4.2.3]{Beintro2}):

\begin{defi}
\label{defi-LDQ0coh}
Soit $\sharp \in \{\emptyset, +,-, \mathrm{b}\}$.
Soit 
$\E ^{(\bullet)} \in \smash{\underrightarrow{LD}} _{\Q} ^{\sharp} ( \smash{\widetilde{\D}} _{\PP ^\sharp} ^{(\bullet)}(T))$.
Le complexe 
$\E ^{(\bullet)}$
est cohérent 
si et seulement s'il existe
$\lambda \in L$
et
$\FF ^{(\bullet)}\in \smash{\underrightarrow{LD}} _{\Q} ^{\sharp} (\lambda ^{*} \smash{\widetilde{\D}} _{\PP ^\sharp} ^{(\bullet)}(T))$ 
et un isomorphisme de
$\E ^{(\bullet)} \riso \FF ^{(\bullet) }$ dans 
$ \smash{\underrightarrow{LD}} _{\Q} ^{\sharp} ( \smash{\widetilde{\D}} _{\PP ^\sharp} ^{(\bullet)}(T))$,
tels que $\FF ^{(\bullet)}$ vérifie les conditions suivantes:
\begin{enumerate}
\item Pour tout $m \in \N$, $\FF ^{(m)} \in D ^{\sharp} _{\mathrm{coh}} (\smash{\widetilde{\D}} _{\PP ^\sharp} ^{(\lambda (m))} (T))$ ; 
\item Pour tous entiers $0\leq m \leq m'$,  
le morphisme canonique
\begin{equation}
\label{Beintro-4.2.3M}
\smash{\widetilde{\D}} _{\PP ^\sharp} ^{(\lambda (m'))} (T) \otimes ^\L _{\smash{\widetilde{\D}} _{\PP ^\sharp} ^{(\lambda (m))} (T)}
\FF ^{(m)} \to \FF ^{(m')} 
\end{equation}
est un isomorphisme.
\end{enumerate}
\end{defi}

\begin{nota}
\label{nota-LDQ0coh}
Soit $\sharp \in \{\emptyset, +,-, \mathrm{b}\}$.
On note $\underrightarrow{LD} ^{\sharp} _{\Q, \mathrm{coh}} (\smash{\widetilde{\D}} _{\PP ^\sharp} ^{(\bullet)} (T))$
la sous-catégorie strictement pleine de 
$\underrightarrow{LD} ^{\sharp} _{\Q} (\smash{\widetilde{\D}} _{\PP ^\sharp} ^{(\bullet)} (T))$
des complexes cohérents.
On note 
$\underrightarrow{LD} ^{0} _{\Q, \mathrm{coh}} (\smash{\widetilde{\D}} _{\PP ^\sharp} ^{(\bullet)} (T))$
la sous-catégorie strictement pleine de 
$\underrightarrow{LD} ^{0} _{\Q} (\smash{\widetilde{\D}} _{\PP ^\sharp} ^{(\bullet)} (T))$
des complexes $\E ^{(\bullet)} \in\underrightarrow{LD} ^{\mathrm{b}} _{\Q, \mathrm{coh}} (\smash{\widetilde{\D}} _{\PP ^\sharp} ^{(\bullet)} (T))$.

\end{nota}

\begin{vide}
\label{HomcalBertcohbij}
Berthelot a de plus vérifié que, pour tout
$\E ^{(\bullet)}
\in \smash{\underrightarrow{LD}} ^{-} _{\Q,\mathrm{coh}} ( \smash{\widetilde{\D}} _{\PP ^\sharp} ^{(\bullet)}(T))$
et pour tout 
$\FF ^{(\bullet)}
\in 
\smash{\underrightarrow{LD}} ^{+} _{\Q} ( \smash{\widetilde{\D}} _{\PP ^\sharp} ^{(\bullet)}(T))$,
l'application
\begin{equation}
\label{HomcalBertcohbij1}
\underrightarrow{\lim} 
\colon
\R \mathcal{H}om _{\smash{\underrightarrow{LD}}  _{\Q} ( \smash{\widetilde{\D}} _{\PP ^\sharp} ^{(\bullet)}(T))}
(\E ^{(\bullet)}, \FF ^{(\bullet)} )
\to 
\R \mathcal{H}om _{D ( \smash{\D} ^\dag _{\PP ^\sharp} (\hdag T) _{\Q} )}
( \underrightarrow{\lim}\, \E ^{(\bullet)}  , 
\underrightarrow{\lim} \, \FF ^{(\bullet)} ) 
\end{equation}
est un isomorphisme.
\end{vide}

\begin{prop}
\label{bis-prop-MQlambda2LMQ}
Les foncteurs oubli et 
$\widetilde{\D} _{\PP ^{\sharp}} ^{(\bullet)}(T) \widehat{\otimes} ^{\L} _{\widehat{\D} _{\PP ^{\sharp}} ^{(\bullet)}(T)}-$
(resp. $\widetilde{\D} _{\PP ^{\sharp}} ^{(\bullet)}(T) \otimes ^{\L}_{\widehat{\D} _{\PP ^{\sharp}} ^{(\bullet)}(T)}-$) 
induisent des équivalences de catégories quasi-inverses entre 
${\underrightarrow{LD}} ^{\mathrm{b}} _{\Q,\mathrm{qc}}  (\widehat{\D} _{\PP ^{\sharp}} ^{(\bullet)}(T))$
(resp. ${\underrightarrow{LD}} ^{\mathrm{b}} _{\Q,\mathrm{coh}}  (\widehat{\D} _{\PP ^{\sharp}} ^{(\bullet)}(T))$)
et
${\underrightarrow{LD}} ^{\mathrm{b}} _{\Q,\mathrm{qc}}  ( \widetilde{\D} _{\PP ^{\sharp}} ^{(\bullet)}(T))$
(resp. ${\underrightarrow{LD}} ^{\mathrm{b}} _{\Q,\mathrm{coh}}  ( \widetilde{\D} _{\PP ^{\sharp}} ^{(\bullet)}(T))$).

\end{prop}

\begin{proof}
On remarque que les foncteurs sont bien définis grâce au corollaire
\ref{rema-dim-coh-finie}, ce dernier n'utilisant pas \ref{bis-prop-MQlambda2LMQ}.
On procède de manière identique à \ref{prop-MQlambda2LMQ}.
\end{proof}

\subsection{Sur une équivalence de catégories des complexes à cohomologie cohérente}

\begin{nota}
Pour $\sharp \in \{ 0,+,-, \mathrm{b}, \emptyset\}$, 
on notera 
$D ^{\sharp} _{\mathrm{coh}}
(\underrightarrow{LM} _{\Q} (\smash{\widetilde{\D}} _{\PP ^\sharp} ^{(\bullet)} (T)))$
la sous-catégorie pleine de
$D ^{\sharp}
(\underrightarrow{LM} _{\Q} (\smash{\widetilde{\D}} _{\PP ^\sharp} ^{(\bullet)} (T)))$
des complexes $\E ^{(\bullet)}$ tels que, pour $n \in \Z$, 
$\mathcal{H} ^{n} (\E ^{(\bullet)}) \in 
\underrightarrow{LM} _{\Q, \mathrm{coh}} (\smash{\widetilde{\D}} _{\PP ^\sharp} ^{(\bullet)} (T))$
(voir les notations de \ref{nota-(L)Mcoh}).
On remarquera que cette notion de cohérence est locale en $\PP$.
\end{nota}

Le lemme qui suit peut se voir comme un cas particulier de 
\ref{HomcalBertcohbij1}. 
Par soucis d'exhaustivité pour le lecteur, nous donnons une vérification.
\begin{lemm}
\label{lemm-fleche-Rhomcoh}
Soit 
$\FF ^{(\bullet)}
\in
\underrightarrow{LM} _{\Q} (\smash{\widetilde{\D}} _{\PP ^\sharp} ^{(\bullet)} (T))$.
Le morphisme de faisceaux en groupes abéliens de \ref{pre-fleche-RhomLM}
\begin{equation}
\label{pre-fleche-RhomLM2}
\mathcal{H}om _{\underrightarrow{LM} _{\Q} (\smash{\widetilde{\D}} _{\PP ^\sharp} ^{(\bullet)} (T))}
(\smash{\widetilde{\D}} _{\PP ^\sharp} ^{(\bullet)} (T ),~\FF ^{(\bullet)})
\to 
\mathcal{H}om _{\smash{\D} ^\dag _{\PP ^\sharp} (\hdag T) _{\Q}}
(\smash{\D} ^\dag _{\PP ^\sharp} (\hdag T) _{\Q},\underrightarrow{\lim} ~\FF ^{(\bullet)})
\end{equation}
est un isomorphisme.
\end{lemm}

\begin{proof}
Pour tout ouvert $\U$ de $\PP$, on dispose par définition des égalités
\begin{gather}
\Gamma (\U , \mathcal{H}om _{\underrightarrow{LM} _{\Q} (\smash{\widetilde{\D}} _{\PP ^\sharp} ^{(\bullet)} (T))}
(\smash{\widetilde{\D}} _{\PP ^\sharp} ^{(\bullet)} (T ),~\FF ^{(\bullet)})
=
\mathrm{Hom} _{\underrightarrow{LM} _{\Q} (\smash{\widetilde{\D}} _{\U ^{\sharp}} ^{(\bullet)} (T \cap U))}
(\smash{\widetilde{\D}} _{\U ^{\sharp}} ^{(\bullet)} (T \cap U),~\FF ^{(\bullet)}|\U)
\\
=
\underset{\lambda \in L}{\underrightarrow{\lim}}~
\underset{\chi \in M}{\underrightarrow{\lim}}~
\mathrm{Hom} _{\smash{\widetilde{\D}} _{\U ^{\sharp}} ^{(\bullet)} (T \cap U)}
(\smash{\widetilde{\D}} _{\U ^{\sharp}} ^{(\bullet)} (T \cap U),~ \lambda ^* \chi ^*\FF ^{(\bullet)}|\U)
=
\underset{\lambda \in L}{\underrightarrow{\lim}}~
\underset{\chi \in M}{\underrightarrow{\lim}}~
\Gamma ( \U, \lambda ^* \chi ^* \FF ^{(\bullet)}).
\end{gather}
En appliquant le foncteur 
$\Gamma ( \U, -)$
au morphisme \ref{pre-fleche-RhomLM2}, on obtient le morphisme canonique
$\underset{\lambda \in L}{\underrightarrow{\lim}}~
\underset{\chi \in M}{\underrightarrow{\lim}}~
\Gamma ( \U, \lambda ^* \chi ^*\FF ^{(\bullet)})
\to \Gamma ( \U, \underrightarrow{\lim} ~\FF ^{(\bullet)})$, qui est un isomorphisme.
\end{proof}

\begin{theo}
\label{theo-fleche-Rhomcoh}
Pour tous $\E ^{(\bullet)}\in 
D ^{\mathrm{b}} _{\mathrm{coh}} (\underrightarrow{LM} _{\Q} (\smash{\widetilde{\D}} _{\PP ^\sharp} ^{(\bullet)} (T)))$,
$\FF ^{(\bullet)} \in D ^{\mathrm{b}} (\underrightarrow{LM} _{\Q} (\smash{\widetilde{\D}} _{\PP ^\sharp} ^{(\bullet)} (T)))$,
le morphisme de $D (\mathrm{Ab} _\PP)$ défini en \ref{fleche-RhomLM}:
\begin{equation}
\label{fleche-Rhomcoh}
\R \mathcal{H}om _{D (\underrightarrow{LM} _{\Q} (\smash{\widetilde{\D}} _{\PP ^\sharp} ^{(\bullet)} (T)))} 
(\E ^{(\bullet)},~\FF ^{(\bullet)})
\to
\R \mathcal{H}om _{ \smash{\D} ^\dag _{\PP ^\sharp} (\hdag T) _{\Q} } 
(\underrightarrow{\lim} ~\E ^{(\bullet)},\underrightarrow{\lim} ~\FF ^{(\bullet)}).
\end{equation}
est un isomorphisme.
\end{theo}

\begin{proof}
Grâce au lemme \cite[I.7.1.(ii)]{HaRD} (que l'on peut utiliser grâce à \ref{LQ-coh-stab}), 
comme pour $\FF ^{(\bullet)}$ fixé les foncteurs sont way-out à droite, 
on se ramène au cas où 
$\E ^{(\bullet)}
\in 
\underrightarrow{LM} _{\Q,\mathrm{coh}} (\smash{\widetilde{\D}} _{\PP ^\sharp} ^{(\bullet)} (T))$.
Comme le théorème est local en $\PP$, d'après le lemme 
\ref{caract-coh-lim-isog}, on peut supposer que $\PP$ est affine et 
qu'il existe $\lambda \in L$ et un 
$\lambda ^{*}\smash{\widetilde{\D}} _{\PP ^\sharp} ^{(\bullet)} (T)$-module $\G ^{(\bullet)}$ globalement de présentation finie
tel que 
$\E ^{(\bullet)}$ et $\G ^{(\bullet)}$ soient isomorphe dans 
$\underrightarrow{LM} _{\Q}  (\smash{\widetilde{\D}} _{\PP ^\sharp} ^{(\bullet)} (T))$.
On se ramène ainsi à vérifier le théorème pour $\G ^{(\bullet)}$ à la place de $\E ^{(\bullet)}$.
D'après la deuxième remarque de \ref{rema-otimes-eqcat}, 
$G ^{(0)}$ est un $\smash{\widetilde{D}} _{\PP ^{\sharp}} ^{(\lambda (0))} (T)$-module cohérent
et le morphisme canonique 
$ \lambda ^{*}\smash{\widetilde{\D}} _{\PP ^\sharp} ^{(\bullet)} (T ) 
\otimes _{\smash{\widetilde{D}} _{\PP ^{\sharp}} ^{(\lambda (0))} (T)} G ^{(0)}
\to 
\G ^{(\bullet)}$
est isomorphisme.
On obtient une résolution gauche de 
$G ^{(0)}$ par des $\smash{\widetilde{D}} _{\PP ^{\sharp}} ^{(\lambda (0))} (T)$-modules libres de type fini.
Celle-ci induit une résolution gauche de 
$\G ^{(\bullet)}$
par des
$ \lambda ^{*}\smash{\widetilde{\D}} _{\PP ^\sharp} ^{(\bullet)} (T ) $-modules libres de type fini. 
En calquant la preuve de \cite[I.7.1.(iv)]{HaRD} (on utilise les suites exactes de troncation), 
on se ramène à vérifier le théorème pour 
$\G ^{(\bullet)}$ de la forme $\left (  \lambda ^{*}\smash{\widetilde{\D}} _{\PP ^\sharp} ^{(\bullet)} (T ) \right ) ^N$.
Par additivité et comme $\smash{\widetilde{\D}} _{\PP ^\sharp} ^{(\bullet)} (T )
\to \lambda ^{*}\smash{\widetilde{\D}} _{\PP ^\sharp} ^{(\bullet)} (T )$ est une lim-ind-isogénie, 
on peut supposer $\G ^{(\bullet)}= \smash{\widetilde{\D}} _{\PP ^\sharp} ^{(\bullet)} (T )$.
On conclut par construction de notre morphisme et de \ref{lemm-fleche-Rhomcoh} (pour dériver le bifoncteur de gauche de \ref{fleche-Rhomcoh}
dans le cas particulier où $\E ^{(\bullet)}= \smash{\widetilde{\D}} _{\PP ^\sharp} ^{(\bullet)} (T )$, 
on prend une résolution par des injectifs de 
$K ^{\mathrm{b}}( \smash{\widetilde{\D}} _{\PP ^\sharp} ^{(\bullet)}(T))$ de l'objet associé à $\FF ^{(\bullet)}$
via l'équivalence de catégories \ref{eqcatLD=DSM-fonct})
ou alors en remarquant que l'on est dans la situation de \ref{HomcalBertcohbij}.

\end{proof}

\begin{coro}
\begin{enumerate}
\item Pour tous $\E ^{(\bullet)}\in 
D ^{\mathrm{b}} _{\mathrm{coh}} (\underrightarrow{LM} _{\Q} (\smash{\widetilde{\D}} _{\PP ^\sharp} ^{(\bullet)} (T)))$,
$\FF ^{(\bullet)} \in D ^{\mathrm{b}} (\underrightarrow{LM} _{\Q} (\smash{\widetilde{\D}} _{\PP ^\sharp} ^{(\bullet)} (T)))$,
le morphisme canonique de $D (\mathrm{Ab})$ de \ref{fleche-mathrmRhomLM}
\begin{equation}
\label{fleche-Rhomcohrm}
\R \mathrm{Hom} _{D (\underrightarrow{LM} _{\Q} (\smash{\widetilde{\D}} _{\PP ^\sharp} ^{(\bullet)} (T)))} 
(\E ^{(\bullet)},~\FF ^{(\bullet)})
\to
\R \mathrm{Hom} _{ \smash{\D} ^\dag _{\PP ^\sharp} (\hdag T) _{\Q} } 
(\underrightarrow{\lim} ~\E ^{(\bullet)},\underrightarrow{\lim} ~\FF ^{(\bullet)})
\end{equation}
est un isomorphisme.

\item Le foncteur \ref{M-eq-lim} se factorise en l'équivalence de catégories
\begin{equation}
\label{M-eq-coh-lim}
\underrightarrow{\lim} 
\colon
\smash{\underrightarrow{LM}}  _{\Q, \mathrm{coh}}
(\smash{\widetilde{\D}} _{\PP ^\sharp} ^{(\bullet)}(T))
\cong
\mathrm{Coh} ( \smash{\D} ^\dag _{\PP ^\sharp} (\hdag T) _{\Q} ).
\end{equation}
où $\mathrm{Coh} ( \smash{\D} ^\dag _{\PP ^\sharp} (\hdag T) _{\Q} )$ 
désigne la catégorie des 
$\smash{\D} ^\dag _{\PP ^\sharp} (\hdag T) _{\Q}$-modules cohérents.

\item Le foncteur $\underrightarrow{\lim}$ induit l'équivalence de catégories 
\begin{equation}
\label{eqcatcoh}
\underrightarrow{\lim} 
\colon 
D ^{\mathrm{b}} _{\mathrm{coh}} (\underrightarrow{LM} _{\Q} (\smash{\widetilde{\D}} _{\PP ^\sharp} ^{(\bullet)} (T)))
\cong
D ^{\mathrm{b}} _{\mathrm{coh}}( \smash{\D} ^\dag _{\PP ^\sharp} (\hdag T) _{\Q} ).
\end{equation}
\end{enumerate}
\end{coro}

\begin{proof}
On vérifie la première assertion de manière analogue à 
\ref{theo-fleche-Rhomcoh} (on peut aussi 
remarquer qu'en appliquant le foncteur $\R \Gamma (\PP, -)$ au morphisme de 
\ref{fleche-Rhomcoh}, on obtient grâce à \ref{Rhomcal-vs-rm} 
la flèche \ref{fleche-Rhomcohrm}). 
Avec \ref{H0Homrm-DLM}, 
on en déduit la pleine fidélité des foncteurs \ref{eqcatcoh} et \ref{M-eq-coh-lim}.
Vérifions à présent que ces deux foncteurs sont essentiellement surjectifs.
Soit $\E \in D ^{\mathrm{b}} _{\mathrm{coh}}( \smash{\D} ^\dag _{\PP ^\sharp} (\hdag T) _{\Q} )$.
Quitte à utiliser des triangles distingués de troncation et à procéder par récurrence sur le nombre d'espaces de cohomologie non nuls, 
avec \cite[I.2.18]{borel}, par pleine fidélité de \ref{eqcatcoh}, on se ramène au cas où 
$\E$ est un $ \smash{\D} ^\dag _{\PP ^\sharp} (\hdag T) _{\Q} $-module cohérent.
Or, grâce à \cite[3.4.5 et 3.6.2]{Be1}, 
il existe 
$m _0\in \N$, un $\smash{\widetilde{\D}} _{\PP ^\sharp} ^{(m _0)} (T)$-module cohérent $\E ^{(0)}$ 
et un isomorphisme 
$ \smash{\D} ^\dag _{\PP ^\sharp} (\hdag T) _{\Q} $-linéaire de la forme 
$ \smash{\D} ^\dag _{\PP ^\sharp} (\hdag T) _{\Q} 
\otimes _{\smash{\widetilde{\D}} _{\PP ^\sharp} ^{(m _0)} (T)} 
\E ^{(0)} \riso \E$.
En posant 
$\E ^{(\bullet)}:= 
\smash{\widetilde{\D}} _{\PP ^\sharp} ^{(\bullet + m_0)} (T)
\otimes _{\smash{\widetilde{\D}} _{\PP ^\sharp} ^{(m _0)} (T)} 
\E ^{(0)} $, on obtient alors 
l'isomorphisme
$\underrightarrow{\lim} (\E ^{(\bullet)}) \riso \E$.

\end{proof}

La proposition qui suit améliore \ref{caract-coh-lim-isog}:
\begin{lemm}
\label{strict-m0}
Soit $\E ^{(\bullet)} \in \underrightarrow{LM} _{\Q, \mathrm{coh}} (\smash{\widetilde{\D}} _{\PP ^\sharp} ^{(\bullet)} (T))$.
Il existe alors $m_0 \in \N$ tel que
$\E ^{(\bullet)}$ 
soit 
isomorphe dans $\underrightarrow{LM} _{\Q} (\smash{\widetilde{\D}} _{\PP ^\sharp} ^{(\bullet)} (T))$
à un 
$\smash{\widetilde{\D}} _{\PP ^\sharp} ^{(\bullet + m_0)} (T)$-module localement de présentation finie.
\end{lemm}

\begin{proof}
Cela a été prouvé au cours de la preuve de \ref{M-eq-coh-lim}.
\end{proof}

\subsection{Comparaison entre les notions de cohérence}

\begin{lemm}
\label{Hnstabcoh}
Pour tout entier $n \in \Z$, 
on dispose de la factorisation
$\mathcal{H} ^{n}
\colon \underrightarrow{LD} ^{\mathrm{b}} _{\Q, \mathrm{coh}} (\smash{\widetilde{\D}} _{\PP ^\sharp} ^{(\bullet)} (T))
\to
\underrightarrow{LM}  _{\Q, \mathrm{coh}} (\smash{\widetilde{\D}} _{\PP ^\sharp} ^{(\bullet)} (T))$.
\end{lemm}

\begin{proof}
Soient $n \in \Z$ fixé et $\FF ^{(\bullet)} 
\in 
\underrightarrow{LD} ^{\mathrm{b}} _{\Q, \mathrm{coh}} (\smash{\widetilde{\D}} _{\PP ^\sharp} ^{(\bullet)} (T))$.
Quitte à remplacer $\FF ^{(\bullet)} $ par un objet isomorphe dans 
$\underrightarrow{LD} ^{\mathrm{b}} _{\Q, \mathrm{coh}} (\smash{\widetilde{\D}} _{\PP ^\sharp} ^{(\bullet)} (T))$, 
on peut supposer qu'il existe
$\lambda \in L$
tel que 
$\FF ^{(\bullet)} 
\in 
\underrightarrow{LD} ^{\mathrm{b}} _{\Q} (\lambda ^{*}\smash{\widetilde{\D}} _{\PP ^\sharp} ^{(\bullet)} (T))$, pour tout entier $m$, 
$\FF ^{(m)}\in D _{\mathrm{coh}} ^{\mathrm{b}} ( \smash{\widetilde{\D}} _{\PP ^\sharp} ^{(\lambda (m))} (T) )$
et tel que, pour tous $m \leq m'$,
le morphisme canonique 
\begin{equation}
\label{Beintro-4.2.3}
\smash{\widetilde{\D}} _{\PP ^\sharp} ^{(\lambda (m'))} (T) \otimes _{\smash{\widetilde{\D}} _{\PP ^\sharp} ^{(\lambda (m))} (T)} ^{\L}
\FF ^{(m)} 
\to \FF ^{(m')} 
\end{equation}
soit un isomorphisme.
On pose
$\G ^{(\bullet)} := \mathcal{H} ^{n} (\FF ^{(\bullet)})$, 
$\E ^{(0)} := 
\mathcal{H} ^{n} (\FF ^{(0)})
$
et
$\E ^{(\bullet)} := \lambda ^{*} \smash{\widetilde{\D}} ^{(\bullet)} _{\PP ^{\sharp}} (T) \otimes _{\smash{\widetilde{\D}} _{\PP ^\sharp} ^{(\lambda (0))} (T)}
\E ^{(0)}$.
Ainsi, d'après le lemme \ref{eqcat-locpfotimes}, 
le $\lambda ^{*}\smash{\widetilde{\D}} _{\PP ^\sharp} ^{(\bullet)} (T ) $-module $\E ^{(\bullet)}$
est localement de présentation finie.
On dispose du morphisme canonique
$\mathcal{H} ^{n} (\FF ^{(0)})
\to
\mathcal{H} ^{n} (\FF ^{(\bullet)})$ qui induit par extension 
le morphisme
$\E ^{(\bullet)} 
\to 
\G ^{(\bullet)}$
de $M( \lambda ^{*}\smash{\widetilde{\D}} _{\PP ^\sharp} ^{(\bullet)} (T))$.
Comme les homomorphismes
$\smash{\widetilde{\D}} _{\PP ^\sharp} ^{(\lambda (0))} (T) _\Q
\to 
\smash{\widetilde{\D}} _{\PP ^\sharp} ^{(\lambda (m))} (T) _\Q$ sont plats,
on vérifie que cette flèche devient un isomorphisme après application du foncteur 
$\Q \otimes _{\Z} -$.
Pour tout entier $m$,
 $\E ^{(m)} , \G ^{(m)}$ sont des $\smash{\widetilde{\D}} _{\PP ^\sharp} ^{(\lambda (m))} (T)$-modules cohérents.
Grâce au lemme \ref{ind-isog-coh-Q},
il en résulte que 
$\E ^{(\bullet)} 
\to 
\G ^{(\bullet)}$
est une ind-isogénie. 
On en déduit le résultat.
\end{proof}

\begin{lemm}
\label{cubecomm1}
On dispose du diagramme canonique commutatif
\begin{equation}
\label{pre-lim-MD2square}
\xymatrix @R=0,2cm @C=0,2cm{
{} 
& 
{D (\smash{\widetilde{\D}} _{\PP ^\sharp} ^{(\bullet)} (T))}  
\ar[rr] ^-(0.7){}
\ar[dd] ^-(0.3){\underrightarrow{\lim} } _-{}  |\hole
\ar[dl] ^-{\mathcal{H} ^{n}}
&& 
{\underrightarrow{LD} _{\Q} (\smash{\widetilde{\D}} _{\PP ^\sharp} ^{(\bullet)} (T))} 
\ar[rr] ^-(0.7){}
\ar[dd] ^-(0.3){\underrightarrow{\lim} } _-{}  |(0,48)\hole |(0,52)\hole
\ar[dl] ^-{\mathcal{H} ^{n}}
&& 
{D (\underrightarrow{LM} _{\Q} (\smash{\widetilde{\D}} _{\PP ^\sharp} ^{(\bullet)} (T)))} 
\ar[dd] ^-(0.35){\underrightarrow{\lim} } _-{} 
\ar[dl] ^-{\mathcal{H} ^{n}}
\\ 
{M (\smash{\widetilde{\D}} _{\PP ^\sharp} ^{(\bullet)} (T))} 
\ar[rr] ^-{} 
\ar[dd] ^-(0.3){\underrightarrow{\lim} } _-{} 
&& 
{\underrightarrow{LM} _{\Q} (\smash{\widetilde{\D}} _{\PP ^\sharp} ^{(\bullet)} (T))} 
\ar@{=}[rr] ^-{}
\ar[dd] ^-(0.3){\underrightarrow{\lim} } _-{} 
&& 
{\underrightarrow{LM} _{\Q} (\smash{\widetilde{\D}} _{\PP ^\sharp} ^{(\bullet)} (T))} 
\ar[dd] ^-(0.3){\underrightarrow{\lim} } _-{} 
\\
{} 
& 
{ D (\smash{\D} ^\dag _{\PP ^\sharp} (\hdag T) _{\Q})}
\ar@{=}[rr] ^-{} |(0,495) \hole
\ar[dl] ^-{\mathcal{H} ^{n}}
&& 
{ D (\smash{\D} ^\dag _{\PP ^\sharp} (\hdag T) _{\Q})}
\ar@{=}[rr] ^-{}  |(0,48) \hole 
\ar[dl] ^-{\mathcal{H} ^{n}}
&& 
{ D (\smash{\D} ^\dag _{\PP ^\sharp} (\hdag T) _{\Q})}
\ar[dl] ^-{\mathcal{H} ^{n}}
\\
{M ( \smash{\D} ^\dag _{\PP ^\sharp} (\hdag T) _{\Q} )} 
\ar@{=}[rr] ^-{}
&& 
{M ( \smash{\D} ^\dag _{\PP ^\sharp} (\hdag T) _{\Q} )} 
\ar@{=}[rr] ^-{}
&& 
{M ( \smash{\D} ^\dag _{\PP ^\sharp} (\hdag T) _{\Q} ),}  
 }
\end{equation}
dont la flèche horizontale du milieu du haut est la factorisation \ref{defi-Hn-LMQ}, 
dont les foncteurs verticaux du carré du milieu sont les factorisations canoniques (voir \ref{M-eq-lim} pour le cas des modules).

\end{lemm}

\begin{proof}
Comme $\Q$ est plat sur $\Z$ et que le foncteur limite inductive sur le niveau est exact, 
les foncteurs verticaux des deux carrés de la face de devant sont exacts (pour celui de droite, cela découle de plus 
de
\cite[2.3.4--5]{Bockle-Pink-cohomological-theory}). 
Cela implique que les trois carrés parallèles à la face de droite sont commutatifs. 
De plus, comme le foncteur canonique
$M(\smash{\widetilde{\D}} _{\PP ^\sharp} ^{(\bullet)} (T ))
\to 
\underrightarrow{LM} _{\Q} (\smash{\widetilde{\D}} _{\PP ^\sharp} ^{(\bullet)} (T ))$
est exact (e.g. voir \cite[2.3.4]{Bockle-Pink-cohomological-theory}),
le contour du grand rectangle de la face du haut est commutatif. 
Comme $\underrightarrow{LD} _{\Q} (\smash{\widetilde{\D}} _{\PP ^\sharp} ^{(\bullet)} (T))$ 
est une catégorie
localisée de $D (\smash{\widetilde{\D}} _{\PP ^\sharp} ^{(\bullet)} (T))$, 
la commutativité du carré de droite de la face du haut en résulte. 
De la même façon, on vérifie que le carré de droite de la face du fond est commutatif. 

\end{proof}

\begin{vide}
On déduit alors des lemmes \ref{Hnstabcoh} et \ref{cubecomm1} que l'équivalence de catégories 
$\underrightarrow{LD} ^{\mathrm{b}} _{\Q} (\smash{\widetilde{\D}} _{\PP ^\sharp} ^{(\bullet)} (T))
\cong 
D ^{\mathrm{b}}
(\underrightarrow{LM} _{\Q} (\smash{\widetilde{\D}} _{\PP ^\sharp} ^{(\bullet)} (T)))$
de
\ref{eqcatLD=DSM-fonct}
se factorise en le foncteur pleinement fidèle
\begin{equation}
\label{eqcatLD=DSM-fonct-coh}
\underrightarrow{LD} ^{\mathrm{b}} _{\Q, \mathrm{coh}} (\smash{\widetilde{\D}} _{\PP ^\sharp} ^{(\bullet)} (T))
\to 
D ^{\mathrm{b}} _{\mathrm{coh}}
(\underrightarrow{LM} _{\Q} (\smash{\widetilde{\D}} _{\PP ^\sharp} ^{(\bullet)} (T))).
\end{equation}
\end{vide}

La proposition qui suit est la version logarithmique de 
\cite[2.2.7]{Be1}
\begin{prop}
\label{227-Be1}
Pour tout $m' >m$,
le noyau de l'homomorphisme canonique
$\rho ^{(m,m')}
\colon 
\smash{\D} _{P ^\sharp} ^{(m)} 
\to
\smash{\D} _{P ^\sharp} ^{(m')} $
est l'idéal bilatère de $\smash{\D} _{P ^\sharp} ^{(m)} $
localement engendré par les opérateurs 
$\partial _{\sharp,i} ^{<p ^{m+1}> _{(m)}}$.
De plus, $\smash{\D} _{P ^\sharp} ^{(m')} $
est localement libre sur l'image de $\rho ^{(m,m')}$, de base les opérateurs de la forme
$\underline{\partial} _{\sharp} ^{<p ^{m+1}\underline{n} > _{(m')}}$, 
avec $\underline{n} \in \N ^{d}$.
\end{prop}

\begin{proof}
Supposons que $P ^{\sharp}$ admette des coordonnées locales $t _1, \dots, t _d$ telles que 
$Z= V(t _1 \cdots t _r)$. 
On déduit des calculs de la preuve de \cite[2.2.7]{Be1} (en particulier, les trois dernières lignes) 
que le noyau de 
$\rho ^{(m,m')}$ est l'ensemble des sommes finies  
de la forme 
$\sum _{\underline{k} \not < p ^{m+1}}
a _{\underline{k}} 
\underline{\partial} _{\sharp} ^{<\underline{k} > _{(m)}}$, 
avec $a _{\underline{k}} \in \smash{\O} _{P}$ et $\underline{k} \not <p ^{m+1}$ signifiant que 
$k _i \geq p ^{m+1}$ pour au moins un  $1\leq i\leq d$.
Grâce à \cite[2.2.4.(ii),(iv)]{Be1}, on vérifie que 
$t _i ^{p ^{m+1}}$ est dans le centre de 
$\smash{\D} _{P } ^{(m)}$
et donc de
$\smash{\D} _{P ^\sharp} ^{(m)} $.
Avec la formule 
\cite[2.2.5.1]{Be1},
il en résulte que 
$\partial _{\sharp,i} ^{<p ^{m+1}> _{(m)}}
=
t _i ^{p ^{m+1}}\partial _{i} ^{<p ^{m+1}> _{(m)}}$
engendre 
$\partial _{\sharp,i} ^{<k> _{(m)}}=t _i ^{k}\partial _{i} ^{<k> _{(m)}}$ pour $k \geq p ^{m+1}$.
D'où le résultat pour le noyau. 
On en déduit, que l'image de 
$\rho ^{(m,m')}$ sont les sommes finies  
de la forme 
$\sum _{\underline{k} <p ^{m+1}}
a _{\underline{k}} 
\underline{\partial} _{\sharp} ^{<\underline{k} > _{(m')}}$,
où $\underline{k} <p ^{m+1}$ signifie $k _i <p ^{m+1}$
quelque soit $i=1,\dots,d$
(on utilise que 
$\underline{\partial} _{\sharp} ^{<\underline{k} > _{(m)}}=
\underline{\partial} _{\sharp} ^{<\underline{k} > _{(m')}}$
pour $\underline{k} <p ^{m+1}$).
Pour $\underline{r} < p ^{m+1}$,
$\underline{q} \geq 0$ 
(i.e. $q _i \geq 0$  quelque soit $i=1,\dots,d$)
et $\underline{k}= p ^{m+1} \underline{q} +\underline{r}$,
grâce à \cite[2.2.4.(iii)]{Be1}, 
on a $\underline{\partial} ^{<\underline{k} > _{(m')}}
=
\underline{\partial} ^{<\underline{r} > _{(m')}}
\cdot
\underline{\partial} ^{< p ^{m+1} \underline{q}  > _{(m')}}$.
Comme 
$t _i ^{p ^{m+1}}$ est dans le centre de 
$\smash{\D} _{P } ^{(m)} $,
alors $t _i ^{p ^{m+1}}$ est dans le centre de 
l'image de $\rho ^{(m,m')}$.
Comme $\underline{\partial} ^{<\underline{r} > _{(m')}}$ est dans cette image, il en résulte que 
$\underline{\partial} _{\sharp} ^{<\underline{k} > _{(m')}}
=
\underline{\partial} _{\sharp} ^{<\underline{r} > _{(m')}}
\cdot
\underline{\partial} _{\sharp} ^{< p ^{m+1} \underline{q}  > _{(m')}}$.
D'où le résultat. 
\end{proof}

\begin{coro}
\label{126Beintro}
\begin{enumerate}
\item Pour tout $m'>m$, 
$\smash{\D} _{P _i ^\sharp} ^{(m')} $
est de Tor-dimension au plus $d+1$ à droite et à gauche sur 
$\smash{\D} _{P _i ^\sharp} ^{(m)} $.

\item Pour tous $m''\geq m'\geq m$, le foncteur 
$(\widetilde{\B} ^{(m'')} _{\PP} ( T)  \smash{\widehat{\otimes}} _{\O _{\PP}} \smash{\widehat{\D}} _{\PP ^{\sharp}} ^{(m')})
\widehat{\otimes} ^{\L} _{(\widetilde{\B} ^{(m'')} _{\PP} ( T)  \smash{\widehat{\otimes}} _{\O _{\PP}} \smash{\widehat{\D}} _{\PP ^{\sharp}} ^{(m)})}
-
\colon 
D  ^{-} (\widetilde{\B} ^{(m'')} _{\PP} ( T)  \smash{\widehat{\otimes}} _{\O _{\PP}} \smash{\widehat{\D}} _{\PP ^{\sharp}} ^{(m)})
\to 
D  ^{-} (\widetilde{\B} ^{(m'')} _{\PP} (T)  \smash{\widehat{\otimes}} _{\O _{\PP}} \smash{\widehat{\D}} _{\PP ^{\sharp}} ^{(m')})$
est way-out avec une amplitude bornée indépendemment  de $m$, $m'$ et $m''$.

\item Soit $\lambda \in L$. On dispose (avec les notations sur la quasi-cohérence du chapitre \ref{chap3}) de la factorisation: 
$$\lambda ^* \smash{\widetilde{\D}} _{\PP ^\sharp} ^{(\bullet)} (T) \widehat{\otimes} ^{\L} _{\smash{\widetilde{\D}} _{\PP ^\sharp} ^{(\lambda (0))} (T)}
-
\colon 
D _{\Q,\mathrm{qc}} ^{\mathrm{b}}
(\smash{\widetilde{\D}} _{\PP ^\sharp} ^{(\lambda (0))} (T))
\to 
\underrightarrow{LD} ^{\mathrm{b}} _{\Q, \mathrm{qc}} (\lambda ^*\smash{\widetilde{\D}} _{\PP ^\sharp} ^{(\bullet)} (T)).$$
\end{enumerate}

\end{coro}

\begin{proof}
Comme pour la version non logarithmique (voir  \cite[1.2.6]{Beintro2}), 
la propriété 1) se déduit
de la proposition \ref{227-Be1}.
Pour tout $\E ^{(m)} \in 
D ^{-} (\widetilde{\B} ^{(m'')} _{\PP} ( T)  \smash{\widehat{\otimes}} _{\O _{\PP}} \smash{\widehat{\D}} _{\PP ^{\sharp}} ^{(m)})$, 
on vérifie que le morphisme canonique 
\begin{equation}
\notag
 \smash{\widehat{\D}} _{\PP ^{\sharp}} ^{(m')}
\widehat{\otimes} ^{\L} _{ \smash{\widehat{\D}} _{\PP ^{\sharp}} ^{(m)}}
\E ^{(m)} 
\to 
(\widetilde{\B} ^{(m'')} _{\PP} ( T)  \smash{\widehat{\otimes}} _{\O _{\PP}} \smash{\widehat{\D}} _{\PP ^{\sharp}} ^{(m')})
\widehat{\otimes} ^{\L} _{(\widetilde{\B} ^{(m'')} _{\PP} ( T)  \smash{\widehat{\otimes}} _{\O _{\PP}} \smash{\widehat{\D}} _{\PP ^{\sharp}} ^{(m)})}
\E ^{(m)} 
\end{equation}
est un isomorphisme, ce qui nous ramène au cas où $T$ est vide. La propriété 2) résulte alors aussitôt de 1).
La propriété 3) résulte de 2) et du corollaire \ref{rema-dim-coh-finie} (ce corollaire a été prouvé sans rien utiliser de cette section).

\end{proof}

Ce corollaire \ref{126Beintro} nous permet d'obtenir le lemme suivant:

\begin{lemm}
\label{LMeqLD0}
Les équivalences du lemme \ref{eq-cat-LM-LD0} se factorisent en les 
équivalences de catégories quasi-inverses de la forme
$\underrightarrow{LM}  _{\Q, \mathrm{coh}} (\smash{\widetilde{\D}} _{\PP ^\sharp} ^{(\bullet)} (T))
\cong 
\underrightarrow{LD} ^{0} _{\Q, \mathrm{coh}} (\smash{\widetilde{\D}} _{\PP ^\sharp} ^{(\bullet)} (T))$
et
$\mathcal{H} ^{0}
\colon \underrightarrow{LD} ^{0} _{\Q, \mathrm{coh}} (\smash{\widetilde{\D}} _{\PP ^\sharp} ^{(\bullet)} (T))
\cong 
\underrightarrow{LM}  _{\Q, \mathrm{coh}} (\smash{\widetilde{\D}} _{\PP ^\sharp} ^{(\bullet)} (T))$.
\end{lemm}

\begin{proof}
Grâce au lemme \ref{eq-cat-LM-LD0}, il suffit 
de vérifier que les foncteurs sont bien définis. 
Grâce au lemme \ref{Hnstabcoh}, on le sait déjà pour le foncteur $\mathcal{H} ^0$.
Traitons à présent le cas du foncteur oubli. Soient $\lambda \in L$
et $\E ^{(\bullet)}$ un $\lambda ^{*}\smash{\widetilde{\D}} _{\PP ^\sharp} ^{(\bullet)} (T)$-module localement de présentation finie
(comme d'habitude, grâce au lemme \ref{caract-coh-lim-isog}, il suffit de traiter ce cas).
On pose $\FF ^{(\bullet)} := \lambda ^* \smash{\widetilde{\D}} _{\PP ^\sharp} ^{(\bullet)} (T) \otimes ^{\L} _{\smash{\widetilde{\D}} _{\PP ^\sharp} ^{(\lambda (0))} (T)}
\E ^{(0)}$.
D'après le corollaire \ref{126Beintro}.3,
le  complexe $\FF ^{(\bullet)} $ est  à cohomologie bornée.
On obtient alors que 
$\FF ^{(\bullet)} \in \underrightarrow{LD} ^{\mathrm{b}} _{\Q, \mathrm{coh}} (\smash{\widetilde{\D}} _{\PP ^\sharp} ^{(\bullet)} (T))$.
Soit $n \in \Z \setminus \{ 0\}$. 
Or, 
comme l'extension $\smash{\widetilde{\D}} _{\PP ^\sharp} ^{(\lambda (0))} (T) _\Q
\to 
\lambda ^* \smash{\widetilde{\D}} _{\PP ^\sharp} ^{(\bullet)} (T)_\Q$ est plate,
on obtient 
$(\mathcal{H} ^{n} (\FF ^{(\bullet)}) ) _\Q\riso 
\mathcal{H} ^{n} (\FF ^{(\bullet)} _\Q) \riso 0$.
Comme $\mathcal{H} ^{n} (\FF ^{(m)})$ est un 
$\smash{\widetilde{\D}} _{\PP ^\sharp} ^{(\lambda (m))} (T)$-module cohérent pour tout $m$,
on déduit de \ref{ind-isog-coh-Q} que 
$\mathcal{H} ^{n} (\FF ^{(m)})$ est ind-isogène à $0$.
En particulier, 
$\FF ^{(\bullet)} \in \underrightarrow{LD} ^{0} _{\Q, \mathrm{coh}} (\smash{\widetilde{\D}} _{\PP ^\sharp} ^{(\bullet)} (T))$. 
Avec le lemme \ref{eq-cat-LM-LD0}, 
cela entraîne l'isomorphisme
$\FF ^{(\bullet)}  \riso \mathcal{H} ^{0} (\FF ^{(\bullet)} )$
dans $\underrightarrow{LD} ^{0} _{\Q, \mathrm{coh}} (\smash{\widetilde{\D}} _{\PP ^\sharp} ^{(\bullet)} (T))$.
Comme $\mathcal{H} ^{0} (\FF ^{(\bullet)} ) \riso \E ^{(\bullet)} $, on en déduit
$\E ^{(\bullet)} \in \underrightarrow{LD} ^{0} _{\Q, \mathrm{coh}} (\smash{\widetilde{\D}} _{\PP ^\sharp} ^{(\bullet)} (T))$. 
\end{proof}

\begin{theo}
\label{theo-eq-coh-lim}
Le foncteur de \ref{eqcatLD=DSM-fonct-coh} induit les équivalences de catégories:
\begin{equation}
\label{eq-coh-lim}
\underrightarrow{LD} ^0  _{\Q, \mathrm{coh}} (\smash{\widetilde{\D}} _{\PP ^\sharp} ^{(\bullet)} (T))
\cong 
D ^{0}  _{\mathrm{coh}}
(\underrightarrow{LM} _{\Q} (\smash{\widetilde{\D}} _{\PP ^\sharp} ^{(\bullet)} (T))),~
\underrightarrow{LD} ^{\mathrm{b}}  _{\Q, \mathrm{coh}} (\smash{\widetilde{\D}} _{\PP ^\sharp} ^{(\bullet)} (T))
\cong 
D ^{\mathrm{b}} _{\mathrm{coh}}
(\underrightarrow{LM} _{\Q} (\smash{\widetilde{\D}} _{\PP ^\sharp} ^{(\bullet)} (T))).
\end{equation}
\end{theo}

\begin{proof}
Comme $D ^{0}  _{\mathrm{coh}}
(\underrightarrow{LM} _{\Q} (\smash{\widetilde{\D}} _{\PP ^\sharp} ^{(\bullet)} (T))) \cong
\underrightarrow{LM} _{\Q, \mathrm{coh}} (\smash{\widetilde{\D}} _{\PP ^\sharp} ^{(\bullet)} (T))$, le lemme \ref{LMeqLD0} nous permet
de conclure que tel est le cas du premier foncteur. 
Traitons à présent l'essentielle surjectivité du second foncteur. 
Soit $\E \in D ^{\mathrm{b}}  _{\mathrm{coh}}
(\underrightarrow{LM} _{\Q} (\smash{\widetilde{\D}} _{\PP ^\sharp} ^{(\bullet)} (T)))$.
Quitte à utiliser des triangles distingués de troncation et à procéder par récurrence sur le nombre d'espaces de cohomologie non nuls, 
avec \cite[I.2.18]{borel}, par pleine fidélité de notre foncteur, on se ramène au cas où 
$\E \in \underrightarrow{LM} _{\Q,\mathrm{coh}} (\smash{\widetilde{\D}} _{\PP ^\sharp} ^{(\bullet)} (T))$, ce qui résulte encore du lemme
\ref{LMeqLD0}.

\end{proof}

\begin{rema}
\label{QCoh-local}
Soit $\E  ^{(\bullet)} \in \underrightarrow{LD}  ^{\mathrm{b}} _{\Q} (\smash{\widetilde{\D}} _{\PP ^\sharp} ^{(\bullet)} (T))$.
Il résulte du théorème \ref{theo-eq-coh-lim}, 
que le fait que 
$\E  ^{(\bullet)} \in \underrightarrow{LD} ^{0} _{\Q, \mathrm{coh}} (\smash{\widetilde{\D}} _{\PP ^\sharp} ^{(\bullet)} (T))$ est local en $P$.
\end{rema}

\begin{lemm}
\label{lemm-lim-MD2square}
Notons $D ^{0} _{\mathrm{coh}} (\smash{\D} ^\dag _{\PP ^\sharp} (\hdag T) _{\Q})$ la sous-catégorie pleine
de 
$D ^{\mathrm{b}} _{\mathrm{coh}} (\smash{\D} ^\dag _{\PP ^\sharp} (\hdag T) _{\Q})$
des complexes $\E$ tels que, pour tout entier $n \not = 0$, 
on ait $\mathcal{H} ^{n} (\E) =0$.

\begin{enumerate}
\item \label{lemm1-lim-MD2square} 
Soit $\E  ^{(\bullet)} \in \underrightarrow{LD}  ^{\mathrm{b}} _{\Q, \mathrm{coh}} (\smash{\widetilde{\D}} _{\PP ^\sharp} ^{(\bullet)} (T))$.
La propriété $\E  ^{(\bullet)} \in \underrightarrow{LD}  ^0 _{\Q,\mathrm{coh}} (\smash{\widetilde{\D}} _{\PP ^\sharp} ^{(\bullet)} (T))$
équivaut à celle $\underrightarrow{\lim}~\E  ^{(\bullet)} \in   
D ^{0} _{\mathrm{coh}} (\smash{\D} ^\dag _{\PP ^\sharp} (\hdag T) _{\Q})$.

\item On dispose du diagramme de foncteurs commutatif à équivalence canonique près
\begin{equation}
\label{lim-MD2square}
\xymatrix @ R=0,4cm{
{\underrightarrow{LM}  _{\Q, \mathrm{coh}} (\smash{\widetilde{\D}} _{\PP ^\sharp} ^{(\bullet)} (T))}
\ar[r] ^-{\cong} 
\ar[d] ^-{\underrightarrow{\lim} } _-{\cong} 
& 
{ \underrightarrow{LD} ^{0} _{\Q, \mathrm{coh}} (\smash{\widetilde{\D}} _{\PP ^\sharp} ^{(\bullet)} (T))}
\ar[r] ^-{\cong} _-{\mathcal{H} ^{0}}
\ar[d] ^-{\underrightarrow{\lim} } _-{\cong} 
& 
{\underrightarrow{LM}  _{\Q, \mathrm{coh}} (\smash{\widetilde{\D}} _{\PP ^\sharp} ^{(\bullet)} (T)) } 
\ar[d] ^-{\underrightarrow{\lim} } _-{\cong} 
\\ 
{\mathrm{Coh} ( \smash{\D} ^\dag _{\PP ^\sharp} (\hdag T) _{\Q} )}
\ar[r] ^-{\cong} 
& 
{ D ^{0} _{\mathrm{coh}} (\smash{\D} ^\dag _{\PP ^\sharp} (\hdag T) _{\Q})}
\ar[r] ^-{\cong} _-{\mathcal{H} ^{0}}
& 
{\mathrm{Coh} ( \smash{\D} ^\dag _{\PP ^\sharp} (\hdag T) _{\Q} ),} 
}
\end{equation}
dont tous les foncteurs sont des équivalences de catégories.
\end{enumerate}
\end{lemm}

\begin{proof}
On déduit de \ref{pre-lim-MD2square} que l'on bénéficie pour tout $n \in \Z$
du diagramme de foncteur commutatif à équivalence canonique près 
\begin{equation}
\label{pre1-lim-MD2square}
\xymatrix @ R=0,4cm{
{ \underrightarrow{LD}  ^{\mathrm{b}} _{\Q ,\mathrm{coh} } (\smash{\widetilde{\D}} _{\PP ^\sharp} ^{(\bullet)} (T))}
\ar[r] ^-{} _-{\mathcal{H} ^{n}}
\ar[d] ^-{\underrightarrow{\lim} } _-{\cong} 
& 
{\underrightarrow{LM}  _{\Q,\mathrm{coh}} (\smash{\widetilde{\D}} _{\PP ^\sharp} ^{(\bullet)} (T)) } 
\ar[d] ^-{\underrightarrow{\lim} } _-{\cong} 
\\ 
{ D ^{\mathrm{b}} _{\mathrm{coh}} (\smash{\D} ^\dag _{\PP ^\sharp} (\hdag T) _{\Q})}
\ar[r] ^-{} _-{\mathcal{H} ^{n}}
& 
{\mathrm{Coh} ( \smash{\D} ^\dag _{\PP ^\sharp} (\hdag T) _{\Q} ),} 
}
\end{equation}
dont les foncteurs verticaux sont des équivalences de catégories (voir \ref{M-eq-coh-lim} pour celle de droite).
On en déduit la première assertion du lemme.
La seconde assertion découle aussitôt de la première équivalence de catégories du théorème \ref{theo-eq-coh-lim}.
\end{proof}

\section{Foncteur de localisation en dehors d'un diviseur}
\label{chap3}
\subsection{Produits tensoriels, quasi-cohérence et foncteur oubli}

\begin{vide}
Pour tous $\E,\FF \in D  ^-
(\overset{^\mathrm{g}}{} \smash{\widetilde{\D}} _{\PP ^\sharp} ^{(m)} (T ))$
et $\M \in D  ^-
(\smash{\widetilde{\D}} _{\PP ^\sharp} ^{(m)} (T ) \overset{^\mathrm{d}}{})$,
on pose:

\begin{gather} \notag
\M _i := \M \otimes ^\L _{\smash{\widetilde{\D}} _{\PP ^\sharp} ^{(m)} (T )} \smash{\widetilde{\D}} _{P ^{\sharp} _i} ^{(m)} (T ),\
\E _i := \smash{\widetilde{\D}} _{P ^{\sharp} _i} ^{(m)} (T ) \otimes ^\L _{\smash{\widetilde{\D}} _{\PP ^\sharp} ^{(m)} ( T )} \E,\\
\notag
\M \smash{\widehat{\otimes}} ^\L _{\smash{\widetilde{\B}} _{\PP} ^{(m)} (T )} \E :=
\R \underset{\underset{i}{\longleftarrow}}{\lim}\, ( \M _i \otimes ^\L  _{\widetilde{\B} _{P _i} ^{(m)} (T )} \E _i)
,\,
\E \smash{\widehat{\otimes}} ^\L _{\smash{\widetilde{\B}} _{\PP} ^{(m)} (T )} \FF :=
\R \underset{\underset{i}{\longleftarrow}}{\lim}\, ( \E _i \otimes ^\L  _{\widetilde{\B} _{P _i} ^{(m)} (T )} \FF _i),
\\
\M \smash{\widehat{\otimes}} ^\L _{\smash{\widetilde{\D}} _{\PP ^\sharp} ^{(m)} (T )} \E :=
\R \underset{\underset{i}{\longleftarrow}}{\lim}\, ( \M _i \otimes ^\L  _{\widetilde{\D} _{P ^{\sharp} _i} ^{(m)} (T )} \E _i).
\end{gather}
\end{vide}

\begin{vide}
Pour tous
$\E ^{(\bullet)} \in 
D ^{-}
(\overset{^\mathrm{g}}{} \smash{\widetilde{\D}} _{\PP ^\sharp} ^{(\bullet)} (T ))$,
$\M ^{(\bullet)} \in 
D ^{-}
( \smash{\widetilde{\D}} _{\PP ^\sharp} ^{(\bullet)} (T ) \overset{^\mathrm{d}}{} )$,
on pose
\begin{gather}
  \M ^{(\bullet)}
\smash{\overset{\L}{\otimes}}   ^{\dag}
_{\D ^\dag  _{\PP ^{\sharp}} (\hdag T) _\Q}\E ^{(\bullet)}
:=
(\M ^{(m)}  \smash{\widehat{\otimes}} ^\L _{\widetilde{\D} ^{(m)} _{\PP ^{\sharp}} ( T) } \E ^{(m)}) _{m\in \N}.
\end{gather}
Pour $? = d$ ou $? = g$, 
on définit les bifoncteurs produits tensoriels 
\begin{align}
\label{predef-otimes-coh1}
-
 \smash{\overset{\L}{\otimes}}^{\dag} _{\O _{\PP } ( \hdag T ) _{\Q}}
 -
 \colon
D ^-
(\overset{^\mathrm{?}}{} \smash{\widetilde{\D}} _{\PP ^\sharp} ^{(\bullet)} (T ))
\times 
D ^-
(\overset{^\mathrm{g}}{} \smash{\widetilde{\D}} _{\PP ^\sharp} ^{(\bullet)} (T ))
&
\to 
D ^-
(\overset{^\mathrm{?}}{} \smash{\widetilde{\D}} _{\PP ^\sharp} ^{(\bullet)} (T )),
\end{align}
en posant, pour tous
$\E ^{(\bullet)}\in 
D ^{-}
(\overset{^\mathrm{?}}{} \smash{\widetilde{\D}} _{\PP ^\sharp} ^{(\bullet)} (T ))$,
$\FF ^{(\bullet)} \in 
D ^{-}
(\overset{^\mathrm{g}}{} \smash{\widetilde{\D}} _{\PP ^\sharp} ^{(\bullet)} (T ))$,
\begin{equation}
\notag
\E ^{(\bullet)}
\smash{\overset{\L}{\otimes}}   ^{\dag}
_{\O  _{\PP} (\hdag T) _\Q}\FF ^{(\bullet)}
:=
(\E ^{(m)}  \smash{\widehat{\otimes}}
^\L _{\widetilde{\B} ^{(m)}  _{\PP} ( T) }
\FF ^{(m)}) _{m\in \N}.
\end{equation}
\end{vide}

\begin{nota}
[Quasi-cohérence et foncteur oubli partiel du diviseur]
Soit $D \subset T$ un second diviseur. 

\begin{itemize}
\item Soit $\E ^{(m)} \in 
D ^{\mathrm{b}}
(\overset{^\mathrm{g}}{} \smash{\widetilde{\D}} _{\PP ^\sharp} ^{(m)} (T ))$.
Comme $\smash{\widetilde{\D}} _{\PP ^\sharp} ^{(m)}(T)$ (resp. $\smash{\widetilde{\B}} _{\PP} ^{(m)}(T)$)
n'a pas de $p$-torsion,
il résulte du théorème \cite[3.2.2]{Beintro2} de Berthelot
que 
$\E ^{(m)} $ est quasi-cohérent au sens de Berthelot 
comme objet de
$D ^{\mathrm{b}}
(\overset{^\mathrm{g}}{} \smash{\D} _{\PP} ^{(m)})$ (voir la définition \cite[3.2.1]{Beintro2}) 
si et seulement si 
$\E ^{(m)} _0 \in D ^{\mathrm{b}} _{\mathrm{qc}}
(\O  _{P})$
et 
le morphisme canonique
$\E ^{(m)} \to 
\smash{\widetilde{\D}} _{\PP ^\sharp} ^{(m)} (T ) 
\smash{\widehat{\otimes}} ^\L _{\smash{\widetilde{\D}} _{\PP ^\sharp} ^{(m)} (T )} \E ^{(m)} $
(resp. $\E ^{(m)} \to 
\smash{\widetilde{\B}} _{\PP} ^{(m)} (T ) 
\smash{\widehat{\otimes}} ^\L _{\smash{\widetilde{\B}} _{\PP} ^{(m)} (T )} \E ^{(m)} $)
est un isomorphisme. En particulier, cela ne dépend pas de $T$.
En notant 
$D ^{\mathrm{b}} _{\mathrm{qc}}
(\overset{^\mathrm{g}}{} \smash{\widetilde{\D}} _{\PP ^\sharp} ^{(m)} (T ))$, 
la sous-catégorie pleine de 
$D ^{\mathrm{b}}
(\overset{^\mathrm{g}}{} \smash{\widetilde{\D}} _{\PP ^\sharp} ^{(m)} (T ))$
des complexes quasi-cohérents, on dispose alors du foncteur 
{\og oubli partiel du diviseur\fg}
$\mathrm{oub} _{D,T}\colon 
D ^{\mathrm{b}} _{\mathrm{qc}}
(\overset{^\mathrm{g}}{} \smash{\widetilde{\D}} _{\PP ^\sharp} ^{(m)} (T))
\to
D ^{\mathrm{b}} _{\mathrm{qc}}
(\overset{^\mathrm{g}}{} \smash{\widetilde{\D}} _{\PP ^\sharp} ^{(m)} (D ))$
qui est une factorisation du foncteur oubli canonique
$\mathrm{oub} _{D,T}\colon 
D ^{\mathrm{b}} 
(\overset{^\mathrm{g}}{} \smash{\widetilde{\D}} _{\PP ^\sharp} ^{(m)} (T))
\to
D ^{\mathrm{b}} 
(\overset{^\mathrm{g}}{} \smash{\widetilde{\D}} _{\PP ^\sharp} ^{(m)} (D ))$.

\item De manière analogue, on note
$D ^{\mathrm{b}} _{\mathrm{qc}}
(\overset{^\mathrm{g}}{} \smash{\widetilde{\D}} _{\PP ^\sharp} ^{(\bullet)} (T ))$
la sous-catégorie pleine de
$D ^{\mathrm{b}}
(\overset{^\mathrm{g}}{} \smash{\widetilde{\D}} _{\PP ^\sharp} ^{(\bullet)} (T ))$
des complexes $\E ^{(\bullet)} $ tels que, pour tout $m\in \Z$,  
$\E ^{(m)} _0 \in D ^{\mathrm{b}} _{\mathrm{qc}}
(\O  _{P})$ 
et le morphisme canonique
$\E ^{(\bullet)}  \to 
\smash{\widetilde{\D}} _{\PP ^\sharp} ^{(\bullet)} (T )
\smash{\widehat{\otimes}} ^\L _{\smash{\widetilde{\D}} _{\PP ^\sharp} ^{(\bullet)} (T )} \E ^{(\bullet)} $
soit un isomorphisme de $D ^{\mathrm{b}}
(\overset{^\mathrm{g}}{} \smash{\widetilde{\D}} _{\PP ^\sharp} ^{(\bullet)} (T ))$. 
On bénéficie encore du foncteur {\og oubli partiel du diviseur\fg}
$\mathrm{oub} _{D,T}\colon 
D ^{\mathrm{b}} _{\mathrm{qc}}
(\overset{^\mathrm{g}}{} \smash{\widetilde{\D}} _{\PP ^\sharp} ^{(\bullet)} (T))
\to
D ^{\mathrm{b}}  _{\mathrm{qc}}
(\overset{^\mathrm{g}}{} \smash{\widetilde{\D}} _{\PP ^\sharp} ^{(\bullet)} (D ))$.

\item On note
$\smash{\underrightarrow{LD}}  ^{\mathrm{b}} _{\Q,\mathrm{qc}}
( \smash{\widetilde{\D}} _{\PP ^\sharp} ^{(\bullet)} (T ))$ 
la sous-catégorie strictement pleine de 
$\smash{\underrightarrow{LD}}  ^{\mathrm{b}} _{\Q}
( \smash{\widetilde{\D}} _{\PP ^\sharp} ^{(\bullet)} (T ))$
des complexes isomorphes dans 
$\smash{\underrightarrow{LD}}  ^{\mathrm{b}} _{\Q}
( \smash{\widetilde{\D}} _{\PP ^\sharp} ^{(\bullet)} (T ))$ à un complexe appartenant
$D ^{\mathrm{b}}  _{\mathrm{qc}}
(\overset{^\mathrm{g}}{} \smash{\widetilde{\D}} _{\PP ^\sharp} ^{(\bullet)} (T ))$.
Comme le foncteur $\mathrm{oub} _{D,T}$
(du second point) envoie une lim-ind-isogénie sur une lim-ind-isogénie, 
on en déduit sa factorisation de la forme:
\begin{equation}
\label{def-oubDT}
\mathrm{oub} _{D,T}
\colon 
\smash{\underrightarrow{LD}}  ^{\mathrm{b}} _{\Q,\mathrm{qc}}
( \smash{\widetilde{\D}} _{\PP ^\sharp} ^{(\bullet)} (T ))
\to
\smash{\underrightarrow{LD}}  ^{\mathrm{b}} _{\Q,\mathrm{qc}}
( \smash{\widetilde{\D}} _{\PP ^\sharp} ^{(\bullet)} (D )).
\end{equation}

\item On note encore
$\mathrm{oub} _{D, T}\colon 
D ^{\mathrm{b}} (\D ^\dag _{\PP ^\sharp} (\hdag T) _\Q) \to 
D ^{\mathrm{b}} (\D ^\dag _{\PP ^\sharp} (\hdag D) _\Q)$
le foncteur oubli partiel du diviseur. 

\end{itemize}
\end{nota}

\begin{rema}
Soit 
$\E ^{(\bullet)}
\in 
D ^{\mathrm{b}}
(\overset{^\mathrm{g}}{} \smash{\widetilde{\D}} _{\PP ^\sharp} ^{(\bullet)} (T ))$.
La propriété $\E ^{(\bullet)}
\in 
D ^{\mathrm{b}} _{\mathrm{qc}}
(\overset{^\mathrm{g}}{} \smash{\widetilde{\D}} _{\PP ^\sharp} ^{(\bullet)} (T ))$
équivaut à, 
pour tout $m \in \Z$, 
$\E ^{(m)}
\in 
D ^{\mathrm{b}} _{\mathrm{qc}}
(\overset{^\mathrm{g}}{} \smash{\widetilde{\D}} _{\PP ^\sharp} ^{(m)} (T ))$.
Ainsi, la catégorie 
$\smash{\underrightarrow{LD}}  ^{\mathrm{b}} _{\Q,\mathrm{qc}}
( \smash{\widetilde{\D}} _{\PP ^\sharp} ^{(\bullet)} (T ))$ 
ci-dessus 
correspond à celle donnée par Berthelot dans \cite[4.2.3]{Beintro2} sans pôle logarithmique ni singularités surconvergentes.

\end{rema}

\begin{lemm}
\label{lemm-def-otimes-coh1}
Le bifoncteur \ref{predef-otimes-coh1} se factorise en
\begin{align}
\label{def-otimes-coh1}
-
 \smash{\overset{\L}{\otimes}}^{\dag} _{\O _{\PP } ( \hdag T ) _{\Q}}
 -
 \colon
 \smash{\underrightarrow{LD}}  ^- _{\Q}
(\overset{^\mathrm{?}}{} \smash{\widetilde{\D}} _{\PP ^\sharp} ^{(\bullet)} (T ))
\times 
\smash{\underrightarrow{LD}}  ^- _{\Q}
(\overset{^\mathrm{g}}{} \smash{\widetilde{\D}} _{\PP ^\sharp} ^{(\bullet)} (T ))
&
\to 
\smash{\underrightarrow{LD}}  ^- _{\Q}
(\overset{^\mathrm{?}}{} \smash{\widetilde{\D}} _{\PP ^\sharp} ^{(\bullet)} (T )).
\end{align}
\end{lemm}

\begin{proof}
Soient $\E ^{(\bullet)}\in 
D ^{-}
(\overset{^\mathrm{?}}{} \smash{\widetilde{\D}} _{\PP ^\sharp} ^{(\bullet)} (T ))$,
$\FF ^{(\bullet)} \in 
D ^{-}
(\overset{^\mathrm{g}}{} \smash{\widetilde{\D}} _{\PP ^\sharp} ^{(\bullet)} (T ))$.
Soient $\chi \in M$, $\lambda \in L$.
Il existe alors un morphisme 
$
\E ^{(\bullet)}
\smash{\overset{\L}{\otimes}}   ^{\dag}
_{\O  _{\PP} (\hdag T) _\Q}
\lambda ^{*} \chi ^{*}  \FF ^{(\bullet)}
\to 
\lambda ^{*} \chi ^{*}  \left (\E ^{(\bullet)}
\smash{\overset{\L}{\otimes}}   ^{\dag}
_{\O  _{\PP} (\hdag T) _\Q}
 \FF ^{(\bullet)}\right) $ (on résout $\FF ^{(\bullet)}$ par des  $ \smash{\widetilde{\D}} _{\PP ^\sharp} ^{(\bullet)} (T )$-modules plats, 
 on utilise qu'un 
 $ \smash{\widetilde{\D}} _{\PP ^\sharp} ^{(\bullet)} (T )$-module $\PP ^{(\bullet)}$ est plat
 si et seulement si les $ \smash{\widetilde{\D}} _{\PP ^\sharp} ^{(m)} (T )$-module $\PP ^{(m)}$
 sont plats)
induisant le diagramme canonique commutatif
\begin{equation}
\notag
\xymatrix @ R=0,3cm{
{\E ^{(\bullet)}
\smash{\overset{\L}{\otimes}}   ^{\dag}
_{\O  _{\PP} (\hdag T) _\Q}
\FF ^{(\bullet)}
} 
\ar[r] ^-{}
\ar[d] ^-{}
& 
{\E ^{(\bullet)}
\smash{\overset{\L}{\otimes}}   ^{\dag}
_{\O  _{\PP} (\hdag T) _\Q}
\lambda ^{*} \chi ^{*}  \FF ^{(\bullet)} } 
\ar[d] ^-{}
\ar@{.>}[dl] ^-{}
\\ 
{\lambda ^{*} \chi ^{*}  \left (\E ^{(\bullet)}
\smash{\overset{\L}{\otimes}}   ^{\dag}
_{\O  _{\PP} (\hdag T) _\Q}
 \FF ^{(\bullet)}\right) } 
 \ar[r] ^-{}
& 
{ \lambda ^{*} \chi ^{*}  \left (\E ^{(\bullet)}
\smash{\overset{\L}{\otimes}}   ^{\dag}
_{\O  _{\PP} (\hdag T) _\Q}
\lambda ^{*} \chi ^{*} \FF ^{(\bullet)}\right) .} 
}
\end{equation}
On en déduit la factorisation
$$-
 \smash{\overset{\L}{\otimes}}^{\dag} _{\O _{\PP } ( \hdag T ) _{\Q}}
 -
 \colon
 D  ^- 
(\overset{^\mathrm{?}}{} \smash{\widetilde{\D}} _{\PP ^\sharp} ^{(\bullet)} (T ))
\times 
\smash{\underrightarrow{LD}}  ^- _{\Q}
(\overset{^\mathrm{g}}{} \smash{\widetilde{\D}} _{\PP ^\sharp} ^{(\bullet)} (T ))
\to 
\smash{\underrightarrow{LD}}  ^- _{\Q}
(\overset{^\mathrm{?}}{} \smash{\widetilde{\D}} _{\PP ^\sharp} ^{(\bullet)} (T )).$$
De même, on obtient la factorisation sur le premier facteur. D'où le résultat.
\end{proof}

\subsection{Définition et propriétés dans le cas des complexes}
\label{section3.2}
Soient $m' \geq m \geq 0$ deux entiers,
$D ' \subset D \subset T$ trois diviseurs (réduits) de $P$. 
On dispose alors des morphismes canoniques
$\B _{P _i} ^{(m)} (D ') \to \B _{P _i} ^{(m)} (D )\to \B _{P _i} ^{(m')} (T)$ (voir la construction dans \cite[4.2.3]{Be1}).
Conformément aux notations de \cite{Beintro2},
notons $D _{\Q,\mathrm{qc}} ^{-} (\smash{\widetilde{\B}} _{\PP} ^{(m)} (D))$ 
(resp. 
$D _{\Q,\mathrm{qc}} ^{-} (\widetilde{\B} ^{(m')} _{\PP} ( D)  \smash{\widehat{\otimes}} _{\O _{\PP}} \smash{\widehat{\D}} _{\PP ^{\sharp}} ^{(m)})$)
la localisation de la catégorie
$D ^{-} _{\mathrm{qc}} (\smash{\widetilde{\B}} _{\PP} ^{(m)} (D))$ 
($D _{\mathrm{qc}} ^{-} (\widetilde{\B} ^{(m')} _{\PP} ( D)  \smash{\widehat{\otimes}} _{\O _{\PP}} \smash{\widehat{\D}} _{\PP ^{\sharp}} ^{(m)})$)
par les isogénies.

\begin{lemm}
\label{lem1-hdagDT}
\begin{enumerate}
\item Le noyau de l'épimorphisme canonique 
$\smash{\widetilde{\B}} _{\PP} ^{(m)} (D ) \widehat{\otimes} _{\O _{\PP} }
\smash{\widetilde{\B}} _{\PP} ^{(m')} (T )
\to 
\smash{\widetilde{\B}} _{\PP} ^{(m')} (T )$
est un $\O _{P}$-module quasi-cohérent. 

\item Le morphisme canonique 
$\smash{\widetilde{\B}} _{\PP} ^{(m)} (D ) 
\widehat{\otimes} ^{\L}  _{\O _{\PP} }
\smash{\widetilde{\B}} _{\PP} ^{(m')} (T )
\to 
\smash{\widetilde{\B}} _{\PP} ^{(m)} (D ) \widehat{\otimes} _{\O _{\PP} }
\smash{\widetilde{\B}} _{\PP} ^{(m')} (T )$
est un isomorphisme.

\end{enumerate}
\end{lemm}

\begin{proof}
Il suffit d'établir le lemme pour $\lambda _0=id$.
Comme le lemme est locale,
on peut supposer $\PP$ affine, intègre et qu'il existe $f \in \O _{\PP}$ (resp. $g \in \O _{\PP}$) relevant une équation locale
de $T \subset P$ (resp $D \subset P$).
Comme les diviseurs sont réduits, $f$ est un multiple de $g$ modulo $\pi \O _{\PP}$.
Comme $\smash{\widehat{\B}} _{\PP} ^{(m')} (T )$ ne dépend que de $f$ modulo $\pi \O _{\PP}$ (voir \cite[4.2.3]{Be1}),
une fois $g$ fixé, on peut choisir $f$ égal à un multiple de $g$.
Posons 
$\B _{\PP} ^{(m')} (f ):=
\O _{\PP} [X] / (f ^{p ^{m'+1}} X -p)$
et
$\B _{\PP} ^{(m)} (g ):=
\O _{\PP} [X] / (g ^{p ^{m+1}} X -p)$.
On dispose du morphisme canonique de la forme
$\B _{\PP} ^{(m)} (g )
\to 
\B _{\PP} ^{(m')} (f)$
(voir \cite[4.2.1.(ii)]{Be1}).

1) Prouvons d'abord que, 
pour tout $j\leq  -1$,
$\mathcal{H} ^{j} (\B _{\PP} ^{(m)} (g ) \otimes ^{\L}_{\O _{\PP}}
\B _{\PP} ^{(m')} (f )) 
=0.$

On dispose de la suite exacte courte
$0\to \O _{\PP} [X] 
\underset{g ^{p ^{m+1}}X-p}{\longrightarrow}
\O _{\PP} [X] 
\to 
\B _{\PP} ^{(m)} (g)
\to 0$.
Comme cette suite exacte donne une résolution canonique 
de $\B _{\PP} ^{(m)} (g)$ par des 
$\O _{\PP}$-modules plats, 
en appliquant le foncteur 
$-\otimes_{\O _{\PP}}
\B _{\PP} ^{(m')} (f) $
à cette suite exacte, 
on constate qu'il s'agit alors d'établir 
que $g ^{p ^{m+1}}$
n'est pas un diviseur de zéro de $B _{\PP} ^{(m')} (f)$.
Comme $B _{\PP} ^{(m')} (f)$ est intègre (voir \cite[4.3.3]{Be1}), on conclut.

2) En notant abusivement $\frac{p}{ g ^{p ^{m+1}}}$ la classe de $X$ dans 
$B _{\PP} ^{(m)} (g)$ et encore $\frac{p}{ g ^{p ^{m+1}}}$ son image dans 
$B _{\PP} ^{(m')} (f)$, on obtient 
le diagramme canonique de morphismes
$B _{\PP} ^{(m')} (f)$-linéaires:
\begin{equation}
\notag
\xymatrix @ R=0,4cm @ C=2cm { 
{\O _{\PP} [X] \otimes _{\O _{\PP}}  \B _{\PP} ^{(m')} (f) } 
\ar[r] ^-{(g ^{p ^{m+1}}X -p )\otimes id}
& 
{\O _{\PP} [X] \otimes _{\O _{\PP}}  \B _{\PP} ^{(m')} (f) } 
\ar[r] ^-{}
\ar@{=}[d] ^-{}
&
{\B _{\PP} ^{(m)} (g)  \otimes _{\O _{\PP}} \B _{\PP} ^{(m')} (f)} 
\ar[d] ^-{}
\ar[r] ^-{}
&
{0}
\\
{\B _{\PP} ^{(m')} (f)[X]} 
\ar[r] ^-{X -\frac{p}{ g ^{p ^{m+1}}}}
& 
{\B _{\PP} ^{(m')} (f )[X]} 
\ar[r] ^-{}
&
{\B _{\PP} ^{(m')} (f )} 
\ar[r] ^-{}
&
{0,}
}
\end{equation}
dont les suites horizontales forment des suites exactes et dont le carré, ayant pour flèches horizontales les morphismes
de $B _{\PP} ^{(m')} (f)$-algèbres
induit par respectivement $X\mapsto \frac{p}{ g ^{p ^{m+1}}}\otimes 1$
et
$X\mapsto \frac{p}{ g ^{p ^{m+1}}}$,
est commutatif. 
On en déduit que le noyau, noté $\NN _{0}$, 
de la surjection canonique
$\B _{\PP} ^{(m)} (g)  \otimes _{\O _{\PP}} \B _{\PP} ^{(m')} (f)
\to 
\B _{\PP} ^{(m')} (f)$
est annulé par $g ^{p ^{m+1}}\otimes 1$ et donc  par $p$.
Comme $\B _{\PP} ^{(m')} (f)$ est sans $p$-torsion, 
comme 
$\B _{P _i} ^{(m)} (D) \otimes _{\O _{P _i}}
\B _{P _i} ^{(m')} (T )
\riso
\V / \pi ^{i+1} \V \otimes _{\V} 
\B _{\PP} ^{(m)} (g) \otimes _{\O _{\PP} }
\B _{\PP} ^{(m')} (f)$,
on en déduit la suite exacte courte
$0 \to \NN _{0} 
\to 
\B _{P _i} ^{(m)} (D ) 
\otimes _{O _{P _i}}
\B _{P _i} ^{(m')} (T )
\to 
\B _{P _i} ^{(m')} (T ) 
\to 
0$.
Pour $i=0$, cela entraîne que
$ \NN _{0} $ est un $\O _P$-module quasi-cohérent.
D'après \cite[3.2.1.b)]{Beintro2}, il en résulte que
$ \NN _{0} \in D ^{\mathrm{b}} _{\mathrm{qc}} ( \O _{\PP})$.
En passant à la limite projectif, on obtient de plus
la suite exacte:
$0 
\to 
\NN _{0} 
\to 
\smash{\widehat{\B}} _{\PP} ^{(m)} (D ) \widehat{\otimes} _{\O _{\PP} }
\smash{\widehat{\B}} _{\PP} ^{(m')} (T )
\to 
\smash{\widehat{\B}} _{\PP} ^{(m')} (T )
\to 
0$.
Comme $\smash{\widehat{\B}} _{\PP} ^{(m')} (T )$
est sans $p$-torsion et séparé complet, 
alors $\smash{\widehat{\B}} _{\PP} ^{(m')} (T )
\in D ^{\mathrm{b}} _{\mathrm{qc}} ( \O _{\PP})$.
Il en est donc de même de 
$\smash{\widehat{\B}} _{\PP} ^{(m)} (D ) \widehat{\otimes} _{\O _{\PP} }
\smash{\widehat{\B}} _{\PP} ^{(m')} (T )$.
D'après \cite[3.2.2]{Beintro2},
$\smash{\widehat{\B}} _{\PP} ^{(m)} (D ) 
\widehat{\otimes} ^{\L}  _{\O _{\PP} }
\smash{\widehat{\B}} _{\PP} ^{(m')} (T )
 \in D ^{\mathrm{b}} _{\mathrm{qc}} ( \O _{\PP})$
et 
$ \O _{P _{i}} \otimes ^{\L} _{\O _{\PP}}
 \left (\smash{\widehat{\B}} _{\PP} ^{(m)} (D ) 
\widehat{\otimes} ^{\L}  _{\O _{\PP} }
\smash{\widehat{\B}} _{\PP} ^{(m')} (T ) \right ) 
\riso
\B _{P _i} ^{(m)} (D ) \otimes ^{\L}_{\O _{P _i} }
\B _{P _i} ^{(m)} (T )$.
Il reste à établir l'isomorphisme canonique
$ \O _{P _{i}} \otimes ^{\L} _{\O _{\PP}}
 \left (\smash{\widehat{\B}} _{\PP} ^{(m)} (D ) 
\widehat{\otimes}   _{\O _{\PP} }
\smash{\widehat{\B}} _{\PP} ^{(m')} (T ) \right ) 
\riso
\B _{P _i} ^{(m)} (D ) \otimes ^{\L}_{\O _{P _i} }
\B _{P _i} ^{(m)} (T )$.
Or, en appliquant le foncteur 
$ \O _{P _{i}} \otimes ^{\L} _{\O _{\PP}}-$ au diagramme
\begin{equation}
\notag
\xymatrix @R=0,3cm {
{0} 
\ar[r] ^-{}
& 
{\NN _{0}} 
\ar[r] ^-{}
\ar@{=}[d] ^-{}
& 
{\B _{\PP} ^{(m)} (g) 
\otimes  _{\O _{\PP} }
\smash{\B} _{\PP} ^{(m')} (f) } 
\ar[r] ^-{}
\ar[d] ^-{}
& 
{\B _{\PP} ^{(m')} (f) } 
\ar[r] ^-{}
\ar[d] ^-{}
& 
{0} 
\\
{0} 
\ar[r] ^-{}
& 
{\NN _{0}} 
\ar[r] ^-{}
& 
{\smash{\widehat{\B}} _{\PP} ^{(m)} (D ) 
\widehat{\otimes}   _{\O _{\PP} }
\smash{\widehat{\B}} _{\PP} ^{(m')} (T ) } 
\ar[r] ^-{}
& 
{\smash{\widehat{\B}} _{\PP} ^{(m')} (T) } 
\ar[r] ^-{}
& 
{0,} 
 }
\end{equation}
on obtient 
$ \O _{P _{i}} \otimes ^{\L} _{\O _{\PP}}
 \left (\B _{\PP} ^{(m)} (g) 
\otimes  _{\O _{\PP} }
\smash{\B} _{\PP} ^{(m')} (f)  \right ) 
\riso
 \O _{P _{i}} \otimes ^{\L} _{\O _{\PP}}
 \left (\smash{\widehat{\B}} _{\PP} ^{(m)} (D ) 
\widehat{\otimes}   _{\O _{\PP} }
\smash{\widehat{\B}} _{\PP} ^{(m')} (T ) \right ) 
$.
Enfin, d'après l'étape 1), on obtient
$ \O _{P _{i}} \otimes ^{\L} _{\O _{\PP}}
 \left (\B _{\PP} ^{(m)} (g) 
\otimes  _{\O _{\PP} }
\smash{\B} _{\PP} ^{(m')} (f)  \right ) 
\riso
( \O _{P _{i}} \otimes ^{\L} _{\O _{\PP}}
\B _{\PP} ^{(m)} (g) )
\otimes  ^\L _{\O _{P _{i}} }
( \O _{P _{i}} \otimes ^{\L} _{\O _{\PP}} \smash{\B} _{\PP} ^{(m')} (f) ) 
\riso 
\B _{P _i} ^{(m)} (D ) \otimes ^{\L}_{\O _{P _i} }
\B _{P _i} ^{(m')} (T )$.
 
\end{proof}

\begin{prop}
\label{lem3-hdagDT}
Les morphismes canoniques
$\smash{\widetilde{\B}} _{\PP} ^{(m')} (T ) \to
\smash{\widetilde{\B}} _{\PP} ^{(m)} (D ) \widehat{\otimes} ^{\L}_{\smash{\widetilde{\B}} _{\PP} ^{(m)} (D' )}
\smash{\widetilde{\B}} _{\PP} ^{(m')} (T )
\to 
\smash{\widetilde{\B}} _{\PP} ^{(m)} (D ) \widehat{\otimes} _{\smash{\widetilde{\B}} _{\PP} ^{(m)} (D' )}
\smash{\widetilde{\B}} _{\PP} ^{(m')} (T )
\to 
\smash{\widetilde{\B}} _{\PP} ^{(m')} (T )$
sont des isogénies. 
\end{prop}

\begin{proof}
D'après le lemme \ref{lem1-hdagDT}, le morphisme canonique
$
\smash{\widetilde{\B}} _{\PP} ^{(m)} (D ') \widehat{\otimes} ^{\L}_{\O _{\PP} }
\smash{\widetilde{\B}} _{\PP} ^{(m')} (T )
\to 
\smash{\widetilde{\B}} _{\PP} ^{(m')} (T )$
est une isogénie.
Comme le foncteur 
$\smash{\widetilde{\B}} _{\PP} ^{(m)} (D) \widehat{\otimes} ^{\L}_{\smash{\widetilde{\B}} _{\PP} ^{(m)} (D ')}-$
envoie une isogénie sur une isogénie, 
il en est donc de même du morphisme induit
$
\smash{\widetilde{\B}} _{\PP} ^{(m)} (D) \widehat{\otimes} ^{\L}_{\smash{\widetilde{\B}} _{\PP} ^{(m)} (D ')}
(\smash{\widetilde{\B}} _{\PP} ^{(m)} (D ') \widehat{\otimes} ^{\L}_{\O _{\PP} }
\smash{\widetilde{\B}} _{\PP} ^{(m')} (T ))
\to 
\smash{\widetilde{\B}} _{\PP} ^{(m)} (D) \widehat{\otimes} ^{\L}_{\smash{\widetilde{\B}} _{\PP} ^{(m)} (D ')}
\smash{\widetilde{\B}} _{\PP} ^{(m')} (T )$.
Comme la composition
$\smash{\widetilde{\B}} _{\PP} ^{(m)} (D ) \widehat{\otimes} ^{\L}_{\O _{\PP} }
\smash{\widetilde{\B}} _{\PP} ^{(m')} (T )
\riso
\smash{\widetilde{\B}} _{\PP} ^{(m)} (D ) \widehat{\otimes} ^{\L}_{\smash{\widetilde{\B}} _{\PP} ^{(m)} (D ')}
(\smash{\widetilde{\B}} _{\PP} ^{(m)} (D ') \widehat{\otimes} ^{\L}_{\O _{\PP} }
\smash{\widetilde{\B}} _{\PP} ^{(m')} (T ))
\to
\smash{\widetilde{\B}} _{\PP} ^{(m)} (D ) \widehat{\otimes} ^{\L}_{\smash{\widetilde{\B}} _{\PP} ^{(m)} (D ')}
\smash{\widetilde{\B}} _{\PP} ^{(m')} (T )
\to 
\smash{\widetilde{\B}} _{\PP} ^{(m')} (T ) 
$
est égal au morphisme canonique de la forme apparaissant dans le lemme \ref{lem1-hdagDT},
ce composé 
est donc une isogénie. 
Il en résulte que le morphisme canonique
$
\smash{\widetilde{\B}} _{\PP} ^{(m)} (D ) \widehat{\otimes} ^{\L}_{\smash{\widetilde{\B}} _{\PP} ^{(m)} (D ')}
\smash{\widetilde{\B}} _{\PP} ^{(m')} (T )
\to 
\smash{\widetilde{\B}} _{\PP} ^{(m')} (T ) 
$
est une isogénie.
De même (on enlève les $\L$) on vérifie que  
le morphisme canonique
$
\smash{\widetilde{\B}} _{\PP} ^{(m)} (D ) \widehat{\otimes} _{\smash{\widetilde{\B}} _{\PP} ^{(m)} (D ')}
\smash{\widetilde{\B}} _{\PP} ^{(m')} (T )
\to 
\smash{\widetilde{\B}} _{\PP} ^{(m')} (T ) $
est une isogénie. 
Comme le composé des deux morphismes canoniques
$\smash{\widetilde{\B}} _{\PP} ^{(m')} (T ) 
\to
\smash{\widetilde{\B}} _{\PP} ^{(m)} (D ) \widehat{\otimes} ^{\L} _{\smash{\widetilde{\B}} _{\PP} ^{(m)} (D ')}
\smash{\widetilde{\B}} _{\PP} ^{(m')} (T )
\riso 
\smash{\widetilde{\B}} _{\PP} ^{(m')} (T ) $
est l'identité, il en résulte que le premier morphisme est aussi une isogénie.
\end{proof}

Le corollaire qui suit n'utilise que les résultats élémentaires de cette section et a été utilisé pour valider le corollaire \ref{126Beintro}. 
\begin{coro}
\label{rema-dim-coh-finie}
\begin{enumerate}
\item Les foncteurs de la forme 
$\B _{P _i} ^{(m')} (T) \otimes ^{\L}_{\O _{P _i} }- $ sont de dimension cohomologique
égale à $1$. 
Le foncteur $\smash{\widetilde{\B}} _{\PP} ^{(m')} (T) 
\widehat{\otimes} ^{\L} _{\O _{\PP} }-$
est way-out sur
$D  ^{-} (\smash{\widetilde{\B}} _{\PP} ^{(m)} (D))$
avec une amplitude bornée indépendemment  de  $m'$ et $m$.

\item 
Le foncteur 
$\smash{\widetilde{\B}} _{\PP} ^{(m')} (T) \widehat{\otimes} ^{\L} _{\smash{\widetilde{\B}} _{\PP} ^{(m)} (D)}-
\colon 
D _{\Q,\mathrm{qc}} ^{\mathrm{b}} (\smash{\widetilde{\B}} _{\PP} ^{(m)} (D))
\to 
D _{\Q,\mathrm{qc}} ^{\mathrm{b}} (\smash{\widetilde{\B}} _{\PP} ^{(m')} (T))$ 
est way-out
avec une amplitude bornée indépendemment  de  $m'$ et $m$.
On dispose de la factorisation
$\smash{\widetilde{\B}} _{\PP} ^{(m+\bullet)} (T) \widehat{\otimes} ^{\L} _{\smash{\widetilde{\B}} _{\PP} ^{(m)} (D)}-
\colon 
D _{\Q,\mathrm{qc}} ^{\mathrm{b}} (\smash{\widetilde{\B}} _{\PP} ^{(m)} (D))
\to
\smash{\underrightarrow{LD}} _{\Q,\mathrm{qc}} ^{\mathrm{b}} (\smash{\widetilde{\B}} _{\PP} ^{(m+\bullet)} (T ))$.

\item 
Le foncteur 
$(\widetilde{\B} ^{(m')} _{\PP} ( T)  \smash{\widehat{\otimes}} _{\O _{\PP}} \smash{\widehat{\D}} _{\PP ^{\sharp}} ^{(m)})
\widehat{\otimes} ^{\L} _{(\widetilde{\B} ^{(m)} _{\PP} ( D)  \smash{\widehat{\otimes}} _{\O _{\PP}} \smash{\widehat{\D}} _{\PP ^{\sharp}} ^{(m)})}
-
\colon 
D _{\Q,\mathrm{qc}} ^{\mathrm{b}} (\widetilde{\B} ^{(m)} _{\PP} ( D)  \smash{\widehat{\otimes}} _{\O _{\PP}} \smash{\widehat{\D}} _{\PP ^{\sharp}} ^{(m)})
\to 
D _{\Q,\mathrm{qc}} ^{\mathrm{b}} (\widetilde{\B} ^{(m')} _{\PP} (T)  \smash{\widehat{\otimes}} _{\O _{\PP}} \smash{\widehat{\D}} _{\PP ^{\sharp}} ^{(m)})$
est way-out avec une amplitude bornée indépendemment  de  $m'$ et $m$.
\end{enumerate}

\end{coro}

\begin{proof}
Vérifions 1). Comme le reste en découle, établissons seulement
que $\B _{P _i} ^{(m')} (T) \otimes ^{\L}_{\O _{P _i} }- $ est de dimension cohomologique
égale à $1$.
Comme cela est locale,
on peut supposer $\PP$ affine, intègre et qu'il existe $f \in \O _{\PP}$
 relevant une équation locale
$T \subset P$.
Il suffit alors de considérer la
suite exacte courte
$0\to \O _{P _i} [X] 
\underset{\overline{f} ^{p ^{m'+1}}X-p}{\longrightarrow}
\O _{P _i} [X] 
\to 
\B _{P _i} ^{(m')} (T)
\to 0$ 
où
$\overline{f}$ est la réduction de $f$ modulo $\pi ^{i+1}$, 
qui donne une résolution canonique à la longueur souhaitée 
de $\B _{P _i} ^{(m')} (T)$ par des 
$\O _{P _i}$-modules plats.
Traitons à présent 2). 
Pour tout $\E ^{(m)} \in D _{\Q, \mathrm{qc}} ^{\mathrm{b}} (\smash{\widetilde{\B}} _{\PP} ^{(m)} (D))$, 
on dispose des  isomorphismes canoniques 
dans $D _{\Q}^{\mathrm{b}} (\smash{\widetilde{\B}} _{\PP} ^{(m')} (T ))$
(on utilise 1) pour garantir la bornitude): 
\begin{equation}
\label{iso11-dim-coh-finie}
\smash{\widetilde{\B}} _{\PP} ^{(m')} (T) \widehat{\otimes} ^{\L} _{\smash{\widetilde{\B}} _{\PP} ^{(m)} (D)}~\E ^{(m)}
\underset{\ref{lem1-hdagDT}}{\riso} 
(\smash{\widetilde{\B}} _{\PP} ^{(m')} (T) 
\widehat{\otimes} ^{\L} _{\O _{\PP} }
\smash{\widetilde{\B}} _{\PP} ^{(m)} (D)
)\widehat{\otimes} ^{\L} _{\smash{\widetilde{\B}} _{\PP} ^{(m)} (D)}~\E ^{(m)}
\riso 
\smash{\widetilde{\B}} _{\PP} ^{(m')} (T) 
\widehat{\otimes} ^{\L} _{\O _{\PP} }
~\E ^{(m)}.
\end{equation}
Grâce au théorème de Berthelot  \cite[3.2.2]{Beintro2} 
(on remarque qu'il n'est valable que pour les complexes à cohomologie bornée), 
les isomorphismes de \ref{iso11-dim-coh-finie}
sont en fait dans $D _{\Q,\mathrm{qc}}^{\mathrm{b}} (\smash{\widetilde{\B}} _{\PP} ^{(m')} (T ))$.
On en déduit 2) grâce à 1). 
Traitons à présent  3).
De manière analogue à \cite[1.1.8]{caro_courbe-nouveau}, 
on vérifie que,
pour tout $\E ^{(m)} \in 
D ^{-} (
\widetilde{\B} ^{(m)} _{\PP} ( D)  \smash{\widehat{\otimes}} _{\O _{\PP}} \smash{\widehat{\D}} _{\PP ^{\sharp}} ^{(m)})$,
le morphisme canonique 
$$\smash{\widetilde{\B}} _{\PP} ^{(m')} (T) \widehat{\otimes} ^{\L} _{\smash{\widetilde{\B}} _{\PP} ^{(m)} (D)}
\E ^{(m)}
\to 
(\widetilde{\B} ^{(m')} _{\PP} ( T)  \smash{\widehat{\otimes}} _{\O _{\PP}} \smash{\widehat{\D}} _{\PP ^{\sharp}} ^{(m)})
\widehat{\otimes} ^{\L} _{(\widetilde{\B} ^{(m)} _{\PP} ( D)  \smash{\widehat{\otimes}} _{\O _{\PP}} \smash{\widehat{\D}} _{\PP ^{\sharp}} ^{(m)})}
\E ^{(m)}
$$
est un isomorphisme. 
D'où le résultat.

\end{proof}

\begin{vide}
\label{hdagT-nota}
Pour tout
$\E ^{(\bullet)} \in 
D ^{-}
(\overset{^\mathrm{g}}{} \smash{\widetilde{\D}} _{\PP ^\sharp} ^{(\bullet)} (D))$, 
comme pour \cite[1.1.8]{caro_courbe-nouveau},
(i.e. il est inutile de localiser pour obtenir la commutativité du diagramme ci-dessous et cela reste valable avec des pôles logarithmiques),
on bénéficie du diagramme commutatif dans $D ^{-}
(\overset{^\mathrm{g}}{} \smash{\widetilde{\D}} _{\PP ^\sharp} ^{(\bullet)} (T))$:
\begin{equation}
\label{ODdivcohe}
  \xymatrix @ R=0,4cm {
{(\widetilde{\B} ^{(m)} _{\PP} ( T)  \smash{\widehat{\otimes}} ^\L
_{\widetilde{\B} ^{(m)}  _{\PP} ( D) } \E ^{(m)}) _{m\in \N}}
\ar@{=}[r] ^-{\mathrm{def}} \ar[d] _\sim
&
{\O _{\PP} (\hdag T) _\Q \smash{\overset{\L}{\otimes}}   ^{\dag}_{\O _{\PP} (\hdag D) _\Q}\E ^{(\bullet)}}
\ar@{.>}[d] _\sim
\\
{(\widetilde{\D} ^{(m)} _{\PP ^{\sharp}} ( T)  \smash{\widehat{\otimes}} ^\L
_{\widetilde{\D} ^{(m)} _{\PP ^{\sharp}} ( D) } \E ^{(m)}) _{m\in \N}}
\ar@{=}[r] ^-{\mathrm{def}}
&
{\D ^\dag _{\PP ^\sharp} (\hdag T) _\Q
\smash{\overset{\L}{\otimes}}   ^{\dag}
_{\D ^\dag  _{\PP ^{\sharp}} (\hdag D) _\Q}\E ^{(\bullet)}=: (\hdag T, D) (\E ^{(\bullet)})}.
}
\end{equation}
Comme pour le lemme \ref{lemm-def-otimes-coh1}, grâce au corollaire \ref{rema-dim-coh-finie} qui entraîne la stabilité d'être à cohomologie bornée, 
on vérifie la factorisation de ce foncteur sous la forme :
\begin{gather}
\label{hdag-def-qc}
 (\hdag T ,\,D) 
 :=\D ^\dag _{\PP ^\sharp} (\hdag T) _\Q
\smash{\overset{\L}{\otimes}}   ^{\dag}
_{\D ^\dag  _{\PP ^{\sharp}} (\hdag D) _\Q} -
\colon
\smash{\underrightarrow{LD}} ^{\mathrm{b}} _{\Q,\mathrm{qc}} ( \smash{\widetilde{\D}} _{\PP ^\sharp} ^{(\bullet)}(D))
\to
\smash{\underrightarrow{LD}} ^{\mathrm{b}} _{\Q,\mathrm{qc}} ( \smash{\widetilde{\D}} _{\PP ^\sharp} ^{(\bullet)}(T)).
\end{gather}
On écrit aussi
$ \E ^{(\bullet)} (\hdag D ,\,T) :=(\hdag T ,\,D) (\E ^{(\bullet)})$. Ce foncteur $(\hdag T ,\,D)$ est {\og le foncteur extension (du diviseur
de $D$ à $T$)\fg}. On dit aussi que $(\hdag T ,\,D)$ est le {\og foncteur de localisation en dehors de $T$\fg}. 
Lorsque $D=\emptyset $ où plus généralement lorsqu'il n'y a aucune confusion à craindre sur $D$, 
on omet de l'indiquer. 
On note de manière analogue pour les complexes cohérents:
\begin{equation}
\label{hdag-def-coh}
(\hdag T, D) := 
\D ^\dag _{\PP ^\sharp} (\hdag T) _\Q \otimes _{ \D ^\dag _{\PP ^\sharp} (\hdag D) _\Q} - \colon 
D ^\mathrm{b} _\mathrm{coh} ( \D ^\dag _{\PP ^\sharp} (\hdag D) _\Q)
\to 
D ^\mathrm{b} _\mathrm{coh} (\D ^\dag _{\PP ^\sharp} (\hdag T) _\Q).
\end{equation}
Le foncteur \ref{hdag-def-coh} restreint aux modules cohérents est  exact, ce qui justifie l'absence du symbole
$\L$. 
\end{vide}

\begin{vide}
\label{nota-oub-sharp}
 Notons 
 $\mathrm{oub} _{\sharp}$ les foncteurs oublis canoniques de la forme
 $\mathrm{oub} _{\sharp} \colon 
\smash{\underrightarrow{LD}} ^{\mathrm{b}} _{\Q,\mathrm{qc}} ( \smash{\widetilde{\D}} _{\PP } ^{(\bullet)}(D))
\to
\smash{\underrightarrow{LD}} ^{\mathrm{b}} _{\Q,\mathrm{qc}} ( \smash{\widetilde{\D}} _{\PP ^\sharp} ^{(\bullet)}(D))$
ou de la forme 
$\mathrm{oub} _{\sharp} \colon 
 D ^\mathrm{b} _\mathrm{coh} ( \D ^\dag _{\PP} (\hdag D) _\Q)
\to 
D ^\mathrm{b} _\mathrm{coh} (\D ^\dag _{\PP ^\sharp} (\hdag D) _\Q)$.
On déduit de \ref{ODdivcohe}
que les foncteurs de localisation 
$(\hdag T, D)$
de \ref{hdag-def-qc} ou de \ref{hdag-def-coh} ne dépendent pas de la log-structure, i.e., 
on dispose de l'isomorphisme canonique 
\begin{equation}
\label{nota-oub-sharp-iso} (\hdag T, D)  \circ \mathrm{oub} _{\sharp} 
\riso \mathrm{oub} _{\sharp} \circ (\hdag T, D)
\end{equation}
de foncteurs de 
$\smash{\underrightarrow{LD}} ^{\mathrm{b}} _{\Q,\mathrm{qc}} ( \smash{\widetilde{\D}} _{\PP } ^{(\bullet)}(D))
\to
\smash{\underrightarrow{LD}} ^{\mathrm{b}} _{\Q,\mathrm{qc}} ( \smash{\widetilde{\D}} _{\PP ^\sharp} ^{(\bullet)}(T))$
ou de 
$ D ^\mathrm{b} _\mathrm{coh} ( \D ^\dag _{\PP} (\hdag D) _\Q)
\to 
D ^\mathrm{b} _\mathrm{coh} (\D ^\dag _{\PP ^\sharp} (\hdag T) _\Q)$. 

\end{vide}

\begin{prop}
\label{oub-pl-fid}
Soit $\E ^{(\bullet)}\in 
\smash{\underrightarrow{LD}} ^{\mathrm{b}} _{\Q,\mathrm{qc}}
 ( \smash{\widetilde{\D}} _{\PP ^\sharp} ^{(\bullet)}(T))$.
\begin{enumerate}
\item Le morphisme canonique fonctoriel en $\E ^{(\bullet)}$:
\begin{equation}
\label{oub-pl-fid-iso1}
(\hdag T ,\,D) \circ \mathrm{oub} _{D,T} (\E ^{(\bullet)})
\to 
\E ^{(\bullet)}
\end{equation}
est un isomorphisme
de $\smash{\underrightarrow{LD}} ^{\mathrm{b}} _{\Q,\mathrm{qc}} ( \smash{\widetilde{\D}} _{\PP ^\sharp} ^{(\bullet)}(T))$.
\item Le morphisme canonique fonctoriel en $\E ^{(\bullet)}$:
\begin{equation}
\label{oub-pl-fid-iso2}
\mathrm{oub}_{D, T}  (\E ^{ (\bullet)} )
\to
\mathrm{oub}_{D, T} \circ (\hdag T, D) \circ \mathrm{oub}_{D, T}(\E ^{ (\bullet)} ) 
\end{equation}
est un isomorphisme
de $\underrightarrow{LD} ^{\mathrm{b}}  _{\Q, \mathrm{qc}}  ( \smash{\widetilde{\D}} _{\PP ^\sharp} ^{(\bullet)}(D))$.
\item Le foncteur
$\mathrm{oub} _{D,T}\colon 
\underrightarrow{LD} ^{\mathrm{b}}  _{\Q, \mathrm{qc}} ( \smash{\widetilde{\D}} _{\PP ^\sharp} ^{(\bullet)}(T))
\to 
\underrightarrow{LD} ^{\mathrm{b}}  _{\Q, \mathrm{qc}} ( \smash{\widetilde{\D}} _{\PP ^\sharp} ^{(\bullet)}(D))$
est pleinement fidèle.
\end{enumerate}

\end{prop}

\begin{proof}
Avec les notations \ref{ODdivcohe}, 
on dispose des isomorphismes canoniques
de $\smash{\underrightarrow{LD}} ^{\mathrm{b}} _{\Q,\mathrm{qc}} ( \smash{\widetilde{\D}} _{\PP ^\sharp} ^{(\bullet)}(T))$:
$$\widetilde{\B} ^{(\bullet)} _{\PP} ( T)  \smash{\widehat{\otimes}} ^\L
_{\widetilde{\B} ^{(\bullet)}  _{\PP} ( D) } \E ^{(\bullet)}
\riso
\left (\widetilde{\B} ^{(\bullet)} _{\PP} ( T)  
\smash{\widehat{\otimes}} ^\L _{\widetilde{\B} ^{(\bullet)}  _{\PP} ( D) } 
\widetilde{\B} ^{(\bullet)}  _{\PP} ( T) \right )
\smash{\widehat{\otimes}} ^\L _{\widetilde{\B} ^{(\bullet)}  _{\PP} ( T) } 
\E ^{(\bullet)}
\underset{\ref{lem3-hdagDT}}{\riso}
\widetilde{\B} ^{(\bullet)}  _{\PP} ( T) 
\smash{\widehat{\otimes}} ^\L _{\widetilde{\B} ^{(\bullet)}  _{\PP} ( T) } 
\E ^{(\bullet)}
\liso 
\E ^{(\bullet)}.
$$
Cette composition est le morphisme canonique 
$(\hdag T ,\,D) \circ \mathrm{oub} _{D,T} (\E ^{(\bullet)})
\to 
\E ^{(\bullet)}$
qui est donc un isomorphisme.
Il en résulte l'isomorphisme canonique
$\mathrm{oub}_{D, T} \circ (\hdag T, D) \circ \mathrm{oub}_{D, T}(\E ^{ (\bullet)} ) 
\riso 
\mathrm{oub}_{D, T}  (\E ^{ (\bullet)} )$
de
$\underrightarrow{LD} ^{\mathrm{b}}  _{\Q, \mathrm{qc}} ( \smash{\widetilde{\D}} _{\PP ^\sharp} ^{(\bullet)}(D))$.
Comme 
le composé 
$\mathrm{oub}_{T, D}  (\E ^{ (\bullet)} )
\to
\mathrm{oub}_{D, T} \circ (\hdag T, D) \circ \mathrm{oub}_{D, T}(\E ^{ (\bullet)} ) 
\riso 
\mathrm{oub}_{D, T}  (\E ^{ (\bullet)} )$
est l'identité, 
on en déduit que le morphisme canonique
\ref{oub-pl-fid-iso2} est un isomorphisme. 
La pleine fidélité de $\mathrm{oub} _{D,T}$ découle des deux premières assertions.
\end{proof}

\begin{coro}
\label{gen-oub-pl-fid}
Soit $\E ^{(\bullet)}\in \smash{\underrightarrow{LD}} ^{\mathrm{b}} _{\Q,\mathrm{qc}} ( \smash{\widetilde{\D}} _{\PP ^\sharp} ^{(\bullet)}(D))$.
Le morphisme canonique fonctoriel en $\E ^{(\bullet)}$:
\begin{equation}
\label{gen-oub-pl-fid-iso1}
(\hdag T ,\,D ') \circ \mathrm{oub} _{D',D} (\E ^{(\bullet)})
\to 
(\hdag T ,~D) (\E ^{(\bullet)})
\end{equation}
est un isomorphisme
de $\smash{\underrightarrow{LD}} ^{\mathrm{b}} _{\Q,\mathrm{qc}} ( \smash{\widetilde{\D}} _{\PP ^\sharp} ^{(\bullet)}(T))$.
\end{coro}

\begin{proof}
D'après \ref{oub-pl-fid-iso1}, on bénéficie de l'isomorphisme canonique
$(\hdag D ,~D') \circ \mathrm{oub} _{D',D} (\E ^{(\bullet)})
\riso 
\E ^{(\bullet)}$.
En lui appliquant le foncteur 
$(\hdag T ,~D)$, l'isomorphisme
$(\hdag T ,~D')
\riso 
(\hdag T ,~D) \circ (\hdag D ,~D')$
nous permet alors de conclure.
\end{proof}

\begin{nota}
\label{nota-tdf}
On note
$D  ^{\mathrm{b}} _{\mathrm{tdf}}
( \smash{\widetilde{\D}} _{\PP ^\sharp} ^{(\bullet)} (T ))$
la sous-catégorie pleine de
$D  ^{\mathrm{b}} 
(\smash{\widetilde{\D}} _{\PP ^\sharp} ^{(\bullet)} (T ))$
des complexes de Tor-dimension finie. 
On note 
$\smash{\underrightarrow{LD}}  ^{\mathrm{b}} _{\Q, \mathrm{qc}, \mathrm{tdf}}
(\smash{\widetilde{\D}} _{\PP ^\sharp} ^{(\bullet)} (T ))$
la sous-catégorie strictement pleine de
$\smash{\underrightarrow{LD}}  ^{\mathrm{b}} _{\Q, \mathrm{qc}}
(\smash{\widetilde{\D}} _{\PP ^\sharp} ^{(\bullet)} (T ))$
des objets isomorphes dans 
$\smash{\underrightarrow{LD}}  ^{\mathrm{b}} _{\Q, \mathrm{qc}}
(\smash{\widetilde{\D}} _{\PP ^\sharp} ^{(\bullet)} (T ))$
à un objet de $D  ^{\mathrm{b}} _{\mathrm{tdf}}
( \smash{\widetilde{\D}} _{\PP ^\sharp} ^{(\bullet)} (T ))$.

\end{nota}

\begin{coro}
\begin{enumerate}
\item Avec les notations de \ref{nota-tdf}, 
on bénéficie de l'égalité
$\smash{\underrightarrow{LD}}  ^{\mathrm{b}} _{\Q, \mathrm{qc}}
(\smash{\widetilde{\D}} _{\PP ^\sharp} ^{(\bullet)} (T ))
=\smash{\underrightarrow{LD}}  ^{\mathrm{b}} _{\Q, \mathrm{qc}, \mathrm{tdf}}
(\smash{\widetilde{\D}} _{\PP ^\sharp} ^{(\bullet)} (T ))$.

\item Le bifoncteur \ref{def-otimes-coh1} se factorise en le bifoncteur
\begin{align}
\label{def-otimes-coh1qc}
-
 \smash{\overset{\L}{\otimes}}^{\dag} _{\O _{\PP } ( \hdag T ) _{\Q}}
 -
\colon
\smash{\underrightarrow{LD}}  ^{\mathrm{b}} _{\Q, \mathrm{qc}}
(\overset{^\mathrm{?}}{} \smash{\widetilde{\D}} _{\PP ^\sharp} ^{(\bullet)} (T ))
\times 
\smash{\underrightarrow{LD}}  ^{\mathrm{b}} _{\Q, \mathrm{qc}}
(\overset{^\mathrm{g}}{} \smash{\widetilde{\D}} _{\PP ^\sharp} ^{(\bullet)} (T ))
&
\to 
\smash{\underrightarrow{LD}}  ^{\mathrm{b}}  _{\Q, \mathrm{qc}}
(\overset{^\mathrm{?}}{} \smash{\widetilde{\D}} _{\PP ^\sharp} ^{(\bullet)} (T )).
\end{align}

\end{enumerate}
\end{coro}

\begin{proof}
Comme la seconde propriété résulte de la première, établissons 1).
Dans un premier temps, supposons $T$ vide. 
Grâce au théorème  \cite[3.2.3]{Beintro2} de Berthelot qui reste valable avec des structures logarithmiques, 
on se ramène à établir 
$D  ^{\mathrm{b}} _{ \mathrm{qc}}
(\D  _{P ^\sharp} ^{(m)})
=D  ^{\mathrm{b}} _{ \mathrm{qc}, \mathrm{tdf}}
(\D _{P ^\sharp} ^{(m)})$. 
Comme cela est local ($P$ est supposé quasi-compact), 
on peut supposer $P$ affine.
Avec les théorèmes de type $A$ pour les faisceaux quasi-cohérents,
on se ramène alors à vérifier 
$D  ^{\mathrm{b}} 
(D  _{P ^\sharp} ^{(m)})
=D  ^{\mathrm{b}} _{\mathrm{tdf}}
(D _{P ^\sharp} ^{(m)})$. 
Cette dernière égalité résulte du fait que 
$D  _{P ^\sharp} ^{(m)}$ est un anneau de dimension cohomologique finie (voir \cite[Proposition 5.3.1]{these_montagnon})
et de \cite[I.5.9]{sga6}.
Revenons à présent au cas général. 
Soit 
$\E ^{(\bullet)} \in \smash{\underrightarrow{LD}}  ^{\mathrm{b}} _{\Q, \mathrm{qc}}
(\smash{\widetilde{\D}} _{\PP ^\sharp} ^{(\bullet)} (T ))$.
Il découle de \ref{oub-pl-fid-iso1} utilisé pour $D$ vide, que l'on dispose de l'isomorphisme
$(\hdag T) \circ \mathrm{oub} _{T} (\E ^{(\bullet)})
\riso
\E ^{(\bullet)}$.
D'après le premier cas traité, on obtient
$\mathrm{oub} _{T} (\E ^{(\bullet)})\in 
\smash{\underrightarrow{LD}}  ^{\mathrm{b}} _{\Q, \mathrm{qc}, \mathrm{tdf}}
(\smash{\widetilde{\D}} _{\PP ^\sharp} ^{(\bullet)})$.
Comme le foncteur 
$(\hdag T)$ préserve la Tor-dimension finie, on conclut la preuve. 
\end{proof}

Après avoir traiter le cas d'un diviseur inclus dans un autre,
terminons cette section par considérer le cas de deux diviseurs dont les composantes irréductibles
sont distinctes. 

\begin{lemm}
\label{lem1-hdagT1T2}
Soient $T, ~T'$ deux diviseurs de $P$ dont les composantes irréductibles sont distinctes,
$\U''$ l'ouvert de $\PP$ complémentaire de $T \cup T'$.

\begin{enumerate}
\item Pour tout $i \in \N$, 
le morphisme canonique
$\smash{\widetilde{\B}} _{P _i} ^{(m)} (T ) \otimes ^{\L}_{\O _{P _i} }
\smash{\widetilde{\B}} _{P _i} ^{(m)} (T' ) \to 
\smash{\widetilde{\B}} _{P _i} ^{(m)} (T ) \otimes _{\O _{P _i} }
\smash{\widetilde{\B}} _{P _i} ^{(m)} (T' )$
est un isomorphisme. 

\item Le morphisme canonique 
$\smash{\widetilde{\B}} _{\PP} ^{(m)} (T ) \widehat{\otimes} ^{\L} _{\O _{\PP} }
\smash{\widetilde{\B}} _{\PP} ^{(m)} (T ')
\to
\smash{\widetilde{\B}} _{\PP} ^{(m)} (T ) \widehat{\otimes} _{\O _{\PP} }
\smash{\widetilde{\B}} _{\PP} ^{(m)} (T ')$
est un isomorphisme
et
la $\O _{\PP} $-algèbre
$\smash{\widetilde{\B}} _{\PP} ^{(m)} (T ) \widehat{\otimes} _{\O _{\PP} }
\smash{\widetilde{\B}} _{\PP} ^{(m)} (T ')$
est sans $p$-torsion.

\item Le morphisme canonique de $\O _{\PP}$-algèbres
$\smash{\widetilde{\B}} _{\PP} ^{(m)} (T ) \widehat{\otimes} _{\O _{\PP} }
\smash{\widetilde{\B}} _{\PP} ^{(m)} (T ')
\to 
j _* \O _{\U ''} $,
où $j\colon \U '' \hookrightarrow \PP$ est l'inclusion,
est un monomorphisme. 

\item Soient $\chi,~\lambda \colon \N \to \N$ définis respectivement pour tout entier $m\in \N$ en posant 
$\chi (m) := p ^{p-1}$ et $\lambda (m) := m +1$.
On dispose des monomorphismes canoniques de la forme
$\alpha ^{(\bullet)}
\colon
\smash{\widetilde{\B}} _{\PP} ^{(\bullet)} (T ) \widehat{\otimes} _{\O _{\PP} }
\smash{\widetilde{\B}} _{\PP} ^{(\bullet)} (T ')
\to 
\smash{\widetilde{\B}} _{\PP} ^{(\bullet)} (T \cup T')$
et
$\beta ^{(\bullet)}\colon 
\smash{\widetilde{\B}} _{\PP} ^{(\bullet)} (T \cup T')
\to 
\lambda ^{*} \chi ^{*} (
\smash{\widetilde{\B}} _{\PP} ^{(\bullet)} (T ) \widehat{\otimes} _{\O _{\PP} }
\smash{\widetilde{\B}} _{\PP} ^{(\bullet)} (T '))$
tels que 
$\lambda ^{*} \chi ^{*} (\alpha  ^{(\bullet)})
\circ \beta  ^{(\bullet)}$ et $\beta  ^{(\bullet)} \circ \alpha  ^{(\bullet)}$ 
sont les morphismes canoniques.
\end{enumerate}

\end{lemm}

\begin{proof}
Il ne coûte pas cher de supposer $\lambda _0 = id$.
Comme les assertions sont locales,
on peut supposer $\PP$ affine, intègre et qu'il existe $f _1,\dots, f _r \in \O _{\PP}$ (resp. $f '_1,\dots, f ' _{r'}  \in \O _{\PP}$) 
relevant une équation locale des composantes irréductibles
de $T \subset P$ (resp $T '\subset P$). 
Notons $f := \prod _{s=1} ^{r} f _s$
et $f ':= \prod _{s=1} ^{r'} f '_s$.
On dispose alors des égalités
$\smash{\widetilde{\B}} _{\PP} ^{(m)} (T )
\riso
\O _{\PP} \{X\} /(f ^{p ^{m+1}}X-p)$
et
$\smash{\widetilde{\B}} _{\PP} ^{(m)} (T ')
\riso
\O _{\PP} \{X'\}/ ((f') ^{p ^{m+1}}X'-p)$, 
où on a pris soin de distinguer les variables $X$ et $X'$.

1) Il s'agit d'établir que, 
pour tout $j\leq  -1$,
$\mathcal{H} ^{j} (\B _{P _i} ^{(m)} (T ) \otimes ^{\L}_{\O _{P _i}}
\B _{P _i} ^{(m)} (T')) 
=0.$

On dispose de la suite exacte courte
$0\to \O _{P _i} [X] 
\underset{\overline{f} ^{p ^{m+1}}X-p}{\longrightarrow}
\O _{P _i} [X] 
\to 
\B _{P _i} ^{(m)} (T)
\to 0$ 
où
$\overline{f}$ est la réduction de $f$ modulo $\pi ^{i+1}$.
Comme cette suite exacte donne une résolution canonique 
de $\B _{P _i} ^{(m)} (T )$ par des 
$\O _{P _i}$-modules plats, 
en appliquant le foncteur 
$-\otimes_{\O _{P _i}}
\B _{P _i} ^{(m)} (T' ) $
à cette suite exacte, 
on constate qu'il s'agit alors d'établir 
que l'image de
$\overline{f} $
dans 
$B _{P _i} ^{(m)} (T')$
est non nul 
et
n'est pas un diviseur de zéro.
Comme $\widehat{B} _{\PP} ^{(m)} (T')$
est sans-$p$-torsion
(voir \cite[4.3.3]{Be1}),
on se ramène au cas $i=0$.
On a alors
$B _{P} ^{(m)} (T')
\riso
O _{P} [X'] / (\overline{f'} ^{p ^{m+1}} X')$
où
$\overline{f'}$ est la réduction de $f'$ modulo $\pi$.
Soient 
$P(X'), Q (X') \in \O _{P} [X'] $
tels que 
$\overline{f} \cdot P(X')= (\overline{f'} ^{p ^{m+1}} X' )  \cdot  Q (X')$.
Comme les composantes irréductibles de 
$T$ et $T'$ sont distinctes,
il en résulte que 
$\overline{f'} ^{p ^{m+1}}X'$ divise $P(X')$.
D'où le résultat.

2) En appliquant le foncteur 
$\R \underset{\underset{i}{\longleftarrow}}{\lim}$
aux isomorphismes canoniques de $1)$, 
via Mittag-Leffler, on obtient l'isomorphisme de $2)$ voulu.
De plus, 
d'après \cite[3.2.2]{Beintro2},
$\smash{\widehat{\B}} _{\PP} ^{(m)} (T ) 
\widehat{\otimes} ^{\L}  _{\O _{\PP} }
\smash{\widehat{\B}} _{\PP} ^{(m)} (T ')
 \in D ^{\mathrm{b}} _{\mathrm{qc}} ( \O _{\PP})$
et 
$k \otimes ^{\L} _{\V}
 \left (\smash{\widehat{\B}} _{\PP} ^{(m)} (T) 
\widehat{\otimes} ^{\L}  _{\O _{\PP} }
\smash{\widehat{\B}} _{\PP} ^{(m)} (T ') \right ) 
\riso
\B _{P} ^{(m)} (T) \otimes ^{\L}_{\O _{P} }
\B _{P} ^{(m)} (T')$.
Il en résulte d'après 1) et l'isomorphisme de 2)
que 
$\mathcal{H} ^{-1} 
(k \otimes ^{\L} _{\V} (\smash{\widehat{\B}} _{\PP} ^{(m)} (T) 
\widehat{\otimes} _{\O _{\PP} }
\smash{\widehat{\B}} _{\PP} ^{(m)} (T ')))
\riso 
0$.

3) Il suffit de reprendre le fil de la preuve de 
\cite[4.3.3.(ii)]{Be1}: notons 
$\widehat{B}: = 
\smash{\widehat{B}} _{\PP} ^{(m)} (T) 
\widehat{\otimes}  _{O _{\PP} }
\smash{\widehat{B}} _{\PP} ^{(m)} (T ')$.
Comme $\widehat{B}$ est sans $p$-torsion, 
il est aussi sans $f$-torsion
et s'injecte donc 
$\widehat{B} _{f}$. 
Notons $\U$ l'ouvert de $\PP$ complémentaire de $T$.
On obtient :
$(\widehat{B} _{f} ) ^{\widehat{}}
\riso
O _{\U} 
\widehat{\otimes}  _{O _{\PP} }
\smash{\widehat{B}} _{\PP} ^{(m)} (T ')
\riso 
\smash{\widehat{B}} _{\U} ^{(m)} (T '\cap U)
\subset
O _{\U''} $ (cette dernière inclusion est \cite[4.3.3.(ii)]{Be1}).
Il suffit donc de prouver que 
$\widehat{B} _{f}$ est $p$-adiquement séparé. 
On note encore $X$, l'image de $X$ dans $\widehat{B}$.
Comme $X$ est un diviseur de $p$, il suffit d'établir
que $\widehat{B} _{f}$ est $X$-adiquement séparé. 
Soient $b, x \in \widehat{B} _{f}$ tels que 
$(1 -bX )x=0$. Il s'agit de prouver que $x =0$.
Quitte à multiplier par une puissance de $f$ assez grande, 
on peut supposer $b, x \in \widehat{B} $
et qu'il existe $s \in \N$ assez grand tel que 
$(f ^{s} -b X) x =0$.
Or, 
$\widehat{B} /X \widehat{B} \riso 
O _{\PP} /p O _{\PP} \otimes _{O _{\PP}} \O _{\PP} \{X'\}/ ((f') ^{p ^{m+1}}X'-p)
\riso
(O _{\PP} /p O _{\PP}) [X '] / (\overline{f'} ^{p ^{m+1}}X')
\riso 
B _{P} ^{(m)} (T')$.
Comme l'image de $f ^{s}$ dans $B _{P} ^{(m)} (T')$ est non nul et n'est pas un diviseur de zéro (voir l'étape 1),
on en déduit que l'égalité 
$(f ^{s} -b X) x =0$ implique que 
$x \in X  \widehat{B}$.
Par itération du raisonnement,
on obtient 
$x \in \cap _{n \in \N} X ^{n} \widehat{B}$.
Il existe donc $c \in \widehat{B}$ tel que 
$(1-cX) x=0$.

Or, 
$\widehat{B} /p \widehat{B} \riso  
B _{P} ^{(m)} (T') [X] /(\overline{f} ^{p ^{m+1}}X)$.
La relation $(1-cX) x=0$ dans 
$\widehat{B}$ induit dans 
$\widehat{B} /p \widehat{B} $ l'égalité
$(1 -\overline{c}X) \overline{x}=0$.
Soient $\widetilde{c}(X)$, $\widetilde{x}(X)$ 
des éléments de $B _{P} ^{(m)} (T') [X]$ 
induisant modulo $\overline{f} ^{p ^{m+1}}X$
les éléments 
$\overline{c}$, $\overline{x}$.
Ainsi, $(1 -\widetilde{c}(X) X) \widetilde{x}(X) =\overline{f} ^{p ^{m+1}}X Q (X)$,
avec $ Q(X)\in B _{P} ^{(m)} (T') [X]$.
Il en résulte qu'il existe $R(X) \in B _{P} ^{(m)} (T') [X]$
tel que $\widetilde{x}(X) =X R(X)$.
D'où $(1 -\widetilde{c}(X) X) R(X) =\overline{f} ^{p ^{m+1}}Q (X)$, 
ce qui entraîne que $\overline{f} ^{p ^{m+1}}$ divise $R(X) $.
Ainsi, on a vérifié que $\overline{x}=0$, i.e. $x \in p \widehat{B} $.
Comme $\widehat{B} $ est sans $p$-torsion, en réitérant le procédé on obtient 
$x \in \cap _{n \in \N} p ^{n}  \widehat{B} = \{ 0\}$, ce qu'il fallait démontrer.

4) La définition du morphisme 
$\alpha ^{(m)}$ est triviale. 
Le fait que cela soit un monomorphisme résulte aussitôt de $3$.
De plus, 
on note 
$\gamma ^{(m)} 
\colon 
\O _{\PP} [X]
\to 
( \smash{\widetilde{\B}} _{\PP} ^{(m+1)} (T ) \widehat{\otimes} _{\O _{\PP} }
\smash{\widetilde{\B}} _{\PP} ^{(m+1)} (T ')) _{\Q}$
le morphisme de $\O _{\PP}$-algèbres
défini en posant 
$\gamma ^{(m)}  (X) := \frac{1}{p}( \frac{p}{f ^{p ^{m+1}}} \widehat{\otimes} \frac{p}{(f ') ^{ p ^{m+1}}})$.
Soient $n \in \N$, $q \in \N$, $r \in \N$ tels que $0\leq r<p$ et $n=pq +r$. 
On calcule 
$\gamma ^{(m)}  (X ^n)=
\frac{1}{p ^{r}}( \frac{p}{f ^{p ^{m+1}}} \widehat{\otimes} \frac{p}{(f ') ^{ p ^{m+1}}}) ^{r}
\left ( p ^{p-2}( \frac{p}{f ^{p ^{m+2}}} \widehat{\otimes} \frac{p}{(f ') ^{ p ^{m+2}}}) \right ) ^q$.
Le morphisme $p ^{p-1}\gamma ^{(m)}$ se factorise alors canoniquement en le morphisme de la forme $\beta ^{(m)}$ voulu. 

\end{proof}

\begin{prop}
\label{hdagT'T=cup}
Soient $T', T $ deux diviseurs de $P$.
Pour tout $\E ^{(\bullet)}
\in 
\underrightarrow{LD} ^{\mathrm{b}} _{\Q,\mathrm{qc}} (\smash{\widetilde{\D}} _{\PP ^\sharp} ^{(\bullet)})$, 
on dispose de l'isomorphisme 
$(\hdag T ') \circ (\hdag T) (\E ^{(\bullet)})
\to
(T '\cup T) (\E ^{(\bullet)})$
fonctoriel en $T,~T',~\E ^{(\bullet)}$.
\end{prop}

\begin{proof}
Soient $T _1,T _2, T''$ des diviseurs dont les composantes irréductibles sont distinctes deux à deux
et tels que 
$T = T _1 \cup T''$ et $T' = T  _2 \cup T''$.
Par associativité du produit tensoriel,
il résulte de \ref{lem1-hdagT1T2} que 
$(\hdag T ') \circ (\hdag T ) (\E ^{(\bullet)})
\riso 
(\hdag T _2 ) \circ  (\hdag T '' ) \circ (\hdag T '' ) \circ  (\hdag T _1) 
(\E ^{(\bullet)})$.
Grâce à \ref{oub-pl-fid}, on obtient 
$(\hdag T '' ) \circ (\hdag T '' )  \riso (\hdag T '' )$. 
Avec encore \ref{lem1-hdagT1T2}, 
comme 
$T \cup T' = T _1 \cup T'' \cup T _2$,
on en déduit le résultat.
\end{proof}

\subsection{Définition et propriétés dans le cas des modules}

Soit $T' \supset T$ un second diviseur de $P$.

\begin{vide}
Soit $\E ^{(\bullet)} \in M (\smash{\widetilde{\D}} _{\PP ^\sharp} ^{(\bullet)} (T))$.
On définit le foncteur extension du diviseur
$(\hdag T', T) ^0\colon M (\smash{\widetilde{\D}} _{\PP ^\sharp} ^{(\bullet)} (T))
\to 
 M (\smash{\widetilde{\D}} _{\PP ^\sharp} ^{(\bullet)} (T'))$
en posant 
$(\hdag T', T) ^0 (\E ^{(\bullet)} ):= 
\widetilde{\D} ^{(\bullet)} _{\PP ^{\sharp}} ( T')  \smash{\widehat{\otimes}} 
_{\widetilde{\D} ^{(\bullet)} _{\PP ^{\sharp}} ( T) } \E ^{(\bullet)}$.
On dit que $\E ^{(\bullet)} $ est {\og pseudo quasi-cohérent\fg} si 
le morphisme canonique 
$\E ^{(\bullet)} \to
(\hdag T, T) ^0 (\E ^{(\bullet)})$ est un isomorphisme. 
Si $\E ^{(\bullet)} $ est sans $p$-torsion, 
il est clair que 
$\E ^{(\bullet)} \in D ^{\mathrm{b}} _{\mathrm{qc}}  (\smash{\widetilde{\D}} _{\PP ^\sharp} ^{(\bullet)} (T))$, 
si et seulement si $\E ^{(\bullet)}$ est pseudo quasi-cohérent.
On note $M _{\mathrm{pqc}}  (\smash{\widetilde{\D}} _{\PP ^\sharp} ^{(\bullet)} (T))$,
la sous-catégorie pleine de 
$M (\smash{\widetilde{\D}} _{\PP ^\sharp} ^{(\bullet)} (T))$
des modules pseudo quasi-cohérents.
\end{vide}

\begin{vide}
Comme la pseudo quasi-cohérence est une notion indépendante du choix des diviseurs, 
le foncteur oubli se factorise en un foncteur de la forme 
$\mathrm{oub}_{T, T'}
\colon 
M _{\mathrm{pqc}}  (\smash{\widetilde{\D}} _{\PP ^\sharp} ^{(\bullet)} (T'))
\to M _{\mathrm{pqc}}  (\smash{\widetilde{\D}} _{\PP ^\sharp} ^{(\bullet)} (T))$.
La factorisation 
$(\hdag T', T) ^0\colon 
M _{\mathrm{pqc}}  (\smash{\widetilde{\D}} _{\PP ^\sharp} ^{(\bullet)} (T))
\to M _{\mathrm{pqc}}  (\smash{\widetilde{\D}} _{\PP ^\sharp} ^{(\bullet)} (T'))$
est encore plus évidente.
Comme pour \ref{oub-pl-fid} (i.e., il suffit d'enlever les {\og $\L$\fg} dans la preuve) , on déduit de \ref{lem3-hdagDT} que, 
pour tout  
$\E ^{\prime (\bullet)} 
\in 
M_{\mathrm{pqc}} (\smash{\widetilde{\D}} _{\PP ^\sharp} ^{(\bullet)} (T'))$, 
les morphismes canoniques fonctoriels en $\E ^{\prime (\bullet)} $:
\begin{gather}
\notag
\mathrm{oub}_{T, T'}  (\E ^{\prime (\bullet)} )
\to
\mathrm{oub}_{T, T'} \circ (\hdag T', T) ^0 \circ \mathrm{oub}_{T, T'}(\E ^{\prime (\bullet)} ),
\\
\label{id->hdagT'Toub}
(\hdag T', T) ^0 \circ \mathrm{oub}_{T, T'}(\E ^{\prime (\bullet)} )
\to
\E ^{\prime (\bullet)} 
\end{gather}
sont des isomorphismes.
Le foncteur 
$\mathrm{oub}_{T, T'}
\colon 
M _{\mathrm{pqc}}  (\smash{\widetilde{\D}} _{\PP ^\sharp} ^{(\bullet)} (T'))
\to M _{\mathrm{pqc}}  (\smash{\widetilde{\D}} _{\PP ^\sharp} ^{(\bullet)} (T))$
est donc pleinement fidèle.

\end{vide}

\begin{vide}
Comme pour \ref{def-otimes-coh1}, on vérifie que 
le foncteur $(\hdag T', T) ^0$ se factorise en
\begin{equation}
\label{hdagT'T}
(\hdag T', T) ^0 \colon 
\underrightarrow{LM} _{\Q} (\smash{\widetilde{\D}} _{\PP ^\sharp} ^{(\bullet)} (T))
\to 
\underrightarrow{LM} _{\Q} (\smash{\widetilde{\D}} _{\PP ^\sharp} ^{(\bullet)} (T')).
\end{equation}

Le foncteur oubli canonique
$M (\smash{\widetilde{\D}} _{\PP ^\sharp} ^{(\bullet)} (T') )
\to M (\smash{\widetilde{\D}} _{\PP ^\sharp} ^{(\bullet)} (T))$
se factorise en le foncteur noté
\begin{equation}
\label{oubTT'}
\mathrm{oub}_{T, T'} \colon 
\underrightarrow{LM} _{\Q} (\smash{\widetilde{\D}} _{\PP ^\sharp} ^{(\bullet)} (T'))
\to 
\underrightarrow{LM} _{\Q} (\smash{\widetilde{\D}} _{\PP ^\sharp} ^{(\bullet)} (T)).
\end{equation}
Il semble faux que le foncteur \ref{oubTT'} soit pleinement fidèle sans hypothèse de finitude. 
Néanmoins, nous verrons que tel est le cas pour les modules cohérents à lim-ind-isogénie près (voir
\ref{hdagT'T-MD2square}).
\end{vide}

\subsection{Théorème de type $A$}

\begin{vide}
Supposons $\PP$ affine.

\begin{itemize}
\item On dispose du foncteur section globale
$\Gamma (\PP, -)
\colon 
M(\smash{\widetilde{\D}} _{\PP ^\sharp} ^{(\bullet)} (T))
\to
M(\smash{\widetilde{D}} _{\PP ^{\sharp}} ^{(\bullet)} (T))$
défini en posant 
$\Gamma (\PP, \E ^{(\bullet)}):=
 E ^{(\bullet)}$, où 
 $E ^{(m)}\to E ^{(m+1)}$ est l'image par le foncteur $\Gamma (\PP,-)$ 
 de la flèche $\E ^{(m)}\to \E ^{(m+1)}$.
Comme le foncteur $\Gamma (\PP, -)$ commute canoniquement 
au foncteur $\chi ^{*}$, 
le foncteur $\Gamma (\PP, -)$ transforme les ind-isogénies en ind-isogénies.  
Il induit donc le foncteur 
$\Gamma (\PP, -)
\colon 
\smash{\underrightarrow{M}} _{\Q}  ( \smash{\widetilde{\D}} _{\PP ^\sharp} ^{(\bullet)}(T))
\to
\smash{\underrightarrow{M}} _{\Q}  ( \smash{\widetilde{D}} _{\PP ^{\sharp}} ^{(\bullet)}(T))$.
De même, comme 
$\Gamma (\PP, -)$ 
envoie une lim-ind-isogénie sur une lim-ind-isogénie,
on obtient la factorisation
$\Gamma (\PP, -)
\colon 
\underrightarrow{LM} _{\Q} (\smash{\widetilde{\D}} _{\PP ^\sharp} ^{(\bullet)} (T))
\to
\underrightarrow{LM} _{\Q} (\smash{\widetilde{D}} _{\PP ^{\sharp}} ^{(\bullet)} (T))$.

\item On dispose du foncteur
$\smash{\widetilde{\D}} _{\PP ^\sharp} ^{(\bullet)}(T) \otimes _{\smash{\widetilde{D}} _{\PP ^{\sharp}} ^{(\bullet)}(T)} -
\colon
M(\smash{\widetilde{D}} _{\PP ^{\sharp}} ^{(\bullet)} (T))
\to
M(\smash{\widetilde{\D}} _{\PP ^\sharp} ^{(\bullet)} (T))$.
On vérifie de même que l'on obtient la factorisation 
$\smash{\widetilde{\D}} _{\PP ^\sharp} ^{(\bullet)}(T) \otimes _{\smash{\widetilde{D}} _{\PP ^{\sharp}} ^{(\bullet)}(T)} -
\colon
\underrightarrow{LM} _{\Q} (\smash{\widetilde{D}} _{\PP ^{\sharp}} ^{(\bullet)} (T))
\to
\underrightarrow{LM} _{\Q} (\smash{\widetilde{\D}} _{\PP ^\sharp} ^{(\bullet)} (T))$.
\end{itemize}

\end{vide}

\begin{lemm}
\label{lambda-id-iso}
Donnons-nous $\lambda \in L$, $T' \supset T$ un second diviseur.

\begin{itemize}
\item Pour tout 
$\lambda ^{*}\smash{\widetilde{\D}} _{\PP ^\sharp} ^{(\bullet)} (T)$-module $\E ^{(\bullet)}$,
les morphismes canoniques 
$\widetilde{\D} ^{(\bullet)} _{\PP ^{\sharp}} ( T') 
 \otimes 
_{\widetilde{\D} ^{(\bullet)} _{\PP ^{\sharp}} ( T) } \E ^{(\bullet)}
\to 
\lambda ^{*}\widetilde{\D} ^{(\bullet)} _{\PP ^{\sharp}} ( T')  
 \otimes 
_{\lambda ^{*}\widetilde{\D} ^{(\bullet)} _{\PP ^{\sharp}} ( T) } \E ^{(\bullet)}$
et
$\widetilde{\D} ^{(\bullet)} _{\PP ^{\sharp}} ( T')  \smash{\widehat{\otimes}} 
_{\widetilde{\D} ^{(\bullet)} _{\PP ^{\sharp}} ( T) } \E ^{(\bullet)}
\to 
\lambda ^{*}\widetilde{\D} ^{(\bullet)} _{\PP ^{\sharp}} ( T')  \smash{\widehat{\otimes}} 
_{\lambda ^{*}\widetilde{\D} ^{(\bullet)} _{\PP ^{\sharp}} ( T) } \E ^{(\bullet)}$
sont des isomorphismes de
$\underrightarrow{LM} _{\Q} (\smash{\widetilde{\D}} _{\PP ^\sharp} ^{(\bullet)} (T'))$.

\item Pour tout $\lambda ^{*}\smash{\widetilde{D}} _{\PP ^{\sharp}} ^{(\bullet)} (T)$-module $F ^{(\bullet)}$, 
les morphismes canoniques
$\widetilde{\D} ^{(\bullet)} _{\PP ^{\sharp}} ( T)  \otimes 
_{\widetilde{D} ^{(\bullet)} _{\PP ^{\sharp}} ( T) } F ^{(\bullet)}
\to 
\lambda ^{*}\widetilde{\D} ^{(\bullet)} _{\PP ^{\sharp}} ( T)  \otimes 
_{\lambda ^{*}\widetilde{D} ^{(\bullet)} _{\PP ^{\sharp}} ( T) } F ^{(\bullet)}$
et
$\widetilde{\D} ^{(\bullet)} _{\PP ^{\sharp}} ( T)  \smash{\widehat{\otimes}} 
_{\widetilde{D} ^{(\bullet)} _{\PP ^{\sharp}} ( T) } F ^{(\bullet)}
\to 
\lambda ^{*}\widetilde{\D} ^{(\bullet)} _{\PP ^{\sharp}} ( T)  \smash{\widehat{\otimes}} 
_{\lambda ^{*}\widetilde{D} ^{(\bullet)} _{\PP ^{\sharp}} ( T) } F ^{(\bullet)}$
sont des isomorphismes de
$\underrightarrow{LM} _{\Q} (\smash{\widetilde{\D}} _{\PP ^\sharp} ^{(\bullet)} (T))$.
\end{itemize}

\end{lemm}

\begin{proof}
On dispose de la factorisation 
$\lambda ^{*}\widetilde{\D} ^{(\bullet)} _{\PP ^{\sharp}} ( T')  
 \otimes 
_{\lambda ^{*}\widetilde{\D} ^{(\bullet)} _{\PP ^{\sharp}} ( T) } \E ^{(\bullet)}
\to 
\lambda ^{*} \left ( \widetilde{\D} ^{(\bullet)} _{\PP ^{\sharp}} ( T') 
 \otimes 
_{\widetilde{\D} ^{(\bullet)} _{\PP ^{\sharp}} ( T) } \E ^{(\bullet)} \right) $
s'inscrivant dans le diagramme canonique commutatif
\begin{equation}
\label{diag-lambda-id-iso}
\xymatrix @ R=0,4cm{
{\widetilde{\D} ^{(\bullet)} _{\PP ^{\sharp}} ( T') 
 \otimes 
_{\widetilde{\D} ^{(\bullet)} _{\PP ^{\sharp}} ( T) } \E ^{(\bullet)}} 
\ar[r] ^-{}
\ar[d] ^-{}
& 
{\lambda ^{*}\widetilde{\D} ^{(\bullet)} _{\PP ^{\sharp}} ( T')  
 \otimes 
_{\lambda ^{*}\widetilde{\D} ^{(\bullet)} _{\PP ^{\sharp}} ( T) } \E ^{(\bullet)}} 
\ar[d] ^-{}
\ar@{.>}[ld] ^-{}
\\ 
{\lambda ^{*} \left ( \widetilde{\D} ^{(\bullet)} _{\PP ^{\sharp}} ( T') 
 \otimes 
_{\widetilde{\D} ^{(\bullet)} _{\PP ^{\sharp}} ( T) } \E ^{(\bullet)} \right) } 
\ar[r] ^-{}
& 
{\lambda ^{*} \left (  \lambda ^{*}\widetilde{\D} ^{(\bullet)} _{\PP ^{\sharp}} ( T')  
 \otimes 
_{\lambda ^{*}\widetilde{\D} ^{(\bullet)} _{\PP ^{\sharp}} ( T) } \E ^{(\bullet)}
\right ),} 
}
\end{equation}
et de même en remplaçant 
{\og $ \otimes $ \fg}
par 
{\og $ \widehat{\otimes} $ \fg}.
On en déduit la première assertion. 
On procède de manière analogue pour la seconde. 
\end{proof}

\begin{lemm}
\label{lemm-coh-otimes}
Soit $T' \supset T$ un second diviseur. 
\begin{enumerate}
\item 
Le foncteur $(\hdag T', T) ^0 $ se factorise  en
\begin{equation}
\label{hdagT'Tcoh}
(\hdag T', T) ^0 \colon 
\underrightarrow{LM} _{\Q, \mathrm{coh}} (\smash{\widetilde{\D}} _{\PP ^\sharp} ^{(\bullet)} (T))
\to 
\underrightarrow{LM} _{\Q, \mathrm{coh}} (\smash{\widetilde{\D}} _{\PP ^\sharp} ^{(\bullet)} (T')).
\end{equation}

\item Pour tout $\E ^{(\bullet)}  \in \underrightarrow{LM} _{\Q, \mathrm{coh}} (\smash{\widetilde{\D}} _{\PP ^\sharp} ^{(\bullet)} (T))$,
le morphisme canonique 
$\smash{\widetilde{\D}} _{\PP ^\sharp} ^{ (\bullet)} (T')
\otimes _{\smash{\widetilde{\D}} _{\PP ^\sharp} ^{ (\bullet)} (T)}
\E ^{(\bullet)}
\to
\smash{\widetilde{\D}} _{\PP ^\sharp} ^{ (\bullet)} (T')
\widehat{\otimes} _{\smash{\widetilde{\D}} _{\PP ^\sharp} ^{ (\bullet)} (T)}
\E ^{(\bullet)}$
est un isomorphisme
de
$\underrightarrow{LM} _{\Q, \mathrm{coh}} (\smash{\widetilde{\D}} _{\PP ^\sharp} ^{(\bullet)} (T'))$.

\item Supposons que $\PP$ soit affine. 
Pour tout $F ^{(\bullet)} \in \underrightarrow{LM} _{\Q, \mathrm{coh}} (\smash{\widetilde{D}} _{\PP ^{\sharp}} ^{(\bullet)} (T))$,
le morphisme canonique 
$\smash{\widetilde{\D}} _{\PP ^\sharp} ^{ (\bullet)} (T)
\otimes _{\smash{\widetilde{D}} _{\PP ^{\sharp}} ^{ (\bullet)} (T)}
F ^{(\bullet)}
\to
\smash{\widetilde{\D}} _{\PP ^\sharp} ^{ (\bullet)} (T)
\widehat{\otimes} _{\smash{\widetilde{D}} _{\PP ^{\sharp}} ^{ (\bullet)} (T)}
F ^{(\bullet)}$
est alors un isomorphisme
de $\underrightarrow{LM} _{\Q, \mathrm{coh}} (\smash{\widetilde{\D}} _{\PP ^\sharp} ^{(\bullet)} (T))$.
\end{enumerate}

\end{lemm}

\begin{proof}
D'après le lemme \ref{strict-m0},
pour vérifier la préservation de la cohérence,
on peut supposer qu'il existe 
$m_0 , n _0\in \N$ tels que
$\E ^{(\bullet)}$ (resp. $F ^{(\bullet)}$)
soit un
$\smash{\widetilde{\D}} _{\PP ^\sharp} ^{(\bullet + m_0)} (T)$-module localement de présentation finie
(resp. un $\smash{\widetilde{D}} _{\PP ^{\sharp}} ^{(\bullet + n_0)} (T)$-module de présentation finie).
On obtient ainsi les égalités (à isomorphisme canonique près):
$\E ^{(\bullet)}= 
\smash{\widetilde{\D}} _{\PP ^\sharp} ^{(\bullet+ m_0)} (T)
\otimes _{\smash{\widetilde{\D}} _{\PP ^\sharp} ^{(m_0)} (T)}
\E ^{(m _0)}$
et
$F ^{(\bullet)}= 
\smash{\widetilde{D}} _{\PP ^{\sharp}} ^{(\bullet+ n_0)} (T)
\otimes _{\smash{\widetilde{D}} _{\PP ^{\sharp}} ^{(n_0)} (T)}
F ^{(n _0)}$.
Le lemme \ref{lambda-id-iso} nous permet de conclure. 
\end{proof}

\begin{prop}
\label{thA-LMQ} 
Supposons $\PP$ affine. On dispose alors des factorisations canoniques
$\Gamma (\PP, -)
\colon 
\underrightarrow{LM} _{\Q,\mathrm{coh}} (\smash{\widetilde{\D}} _{\PP ^\sharp} ^{(\bullet)} (T))
\to
\underrightarrow{LM} _{\Q,\mathrm{coh}} (\smash{\widetilde{D}} _{\PP ^{\sharp}} ^{(\bullet)} (T))$
et
$\smash{\widetilde{\D}} _{\PP ^\sharp} ^{(\bullet)}(T) \otimes _{\smash{\widetilde{D}} _{\PP ^{\sharp}} ^{(\bullet)}(T)} -
\colon
\underrightarrow{LM} _{\Q,\mathrm{coh}} (\smash{\widetilde{D}} _{\PP ^{\sharp}} ^{(\bullet)} (T))
\to
\underrightarrow{LM} _{\Q,\mathrm{coh}} (\smash{\widetilde{\D}} _{\PP ^\sharp} ^{(\bullet)} (T))$
qui induisent 
des équivalences quasi-inverses de catégories.
\end{prop}

\begin{proof}
On procède de manière analogue à la preuve du lemme \ref{lemm-coh-otimes}
en utilisant en plus 
les théorèmes 
de type $A$ du deuxième point de la remarque 
\ref{rema-otimes-eqcat}.
\end{proof}

\begin{vide}
\label{parag-hdagT'T-MD2square}
On dispose du diagramme canonique
\begin{equation}
\label{hdagT'T-MD2square}
\xymatrix @ R=0,4cm{
{\underrightarrow{LM}  _{\Q, \mathrm{coh}} (\smash{\widetilde{\D}} _{\PP ^\sharp} ^{(\bullet)} (T))}
\ar[r] ^-{\cong} 
\ar[d] ^-{(\hdag T', T) ^0}
& 
{ \underrightarrow{LD} ^{0} _{\Q, \mathrm{coh}} (\smash{\widetilde{\D}} _{\PP ^\sharp} ^{(\bullet)} (T))}
\ar[r] ^-{\cong} _-{\mathcal{H} ^{0}}
\ar[d] ^-{(\hdag T', T) }
& 
{\underrightarrow{LM}  _{\Q, \mathrm{coh}} (\smash{\widetilde{\D}} _{\PP ^\sharp} ^{(\bullet)} (T)) } 
\ar[d] ^-{(\hdag T', T) ^0}
\\ 
{\underrightarrow{LM}  _{\Q, \mathrm{coh}} (\smash{\widetilde{\D}} _{\PP ^\sharp} ^{(\bullet)} (T'))}
\ar[r] ^-{\cong} 
\ar[d] ^-{\mathrm{oub}_{T, T'} }
& 
{ \underrightarrow{LD} ^{0} _{\Q, \mathrm{coh}} (\smash{\widetilde{\D}} _{\PP ^\sharp} ^{(\bullet)} (T'))}
\ar[r] ^-{\cong} _-{\mathcal{H} ^{0}}
\ar[d] ^-{\mathrm{oub}_{T, T'} }
& 
{\underrightarrow{LM}  _{\Q, \mathrm{coh}} (\smash{\widetilde{\D}} _{\PP ^\sharp} ^{(\bullet)} (T')) } 
\ar[d] ^-{\mathrm{oub}_{T, T'} }
\\
{\underrightarrow{LM}  _{\Q} (\smash{\widetilde{\D}} _{\PP ^\sharp} ^{(\bullet)} (T))}
\ar[r] ^-{\cong} 
& 
{ \underrightarrow{LD} ^{0} _{\Q} (\smash{\widetilde{\D}} _{\PP ^\sharp} ^{(\bullet)} (T))}
\ar[r] ^-{\cong} _-{\mathcal{H} ^{0}}
& 
{\underrightarrow{LM}  _{\Q} (\smash{\widetilde{\D}} _{\PP ^\sharp} ^{(\bullet)} (T)) } 
}
\end{equation}
commutatif à isomorphisme canoniques près. 
En effet, il suffit de reprendre les arguments de la preuve
du lemme 
\ref{LMeqLD0}.
Comme le foncteur oubli $\mathrm{oub}_{T, T'}$ du milieu est pleinement fidèle (voir \ref{oub-pl-fid}),
on en déduit qu'il en est de même 
$\mathrm{oub}_{T, T'}
\colon 
\underrightarrow{LM}  _{\Q, \mathrm{coh}} (\smash{\widetilde{\D}} _{\PP ^\sharp} ^{(\bullet)} (T')) 
\to 
\underrightarrow{LM}  _{\Q} (\smash{\widetilde{\D}} _{\PP ^\sharp} ^{(\bullet)} (T)) $.
\end{vide}

\subsection{Un critère de stabilité de la cohérence par localisation en dehors d'un diviseur}

\begin{theo}
\label{limTouD}
Soient  $T' \supset T$ un diviseur,
$\E ^{(\bullet)}
\in 
\underrightarrow{LD} ^{\mathrm{b}} _{\Q, \mathrm{coh}} (\smash{\widetilde{\D}} _{\PP ^\sharp} ^{(\bullet)} (T))$
et
$\E := 
\underrightarrow{\lim}
\E ^{(\bullet)}
\in D ^{\mathrm{b}} _{\mathrm{coh}} (\smash{\D} ^\dag _{\PP ^\sharp} (\hdag T) _{\Q})$.
On suppose de plus que 
le morphisme $\E \to (\hdag T',T) (\E)$
est un isomorphisme dans 
$D ^{\mathrm{b}}  (\smash{\D} ^\dag _{\PP ^\sharp} (\hdag T) _{\Q})$.
Le morphisme canonique 
$\E ^{(\bullet)} \to 
(\hdag T',T) (\E ^{(\bullet)})$
est alors un isomorphisme 
dans 
$\underrightarrow{LD} ^{\mathrm{b}} _{\Q, \mathrm{coh}} (\smash{\widetilde{\D}} _{\PP ^\sharp} ^{(\bullet)}(T))$.
\end{theo}

\begin{proof}
0) {\it On se ramène au cas où $\E ^{(\bullet)}
\in 
\underrightarrow{LM}  _{\Q, \mathrm{coh}} (\smash{\widetilde{\D}} _{\PP ^\sharp} ^{(\bullet)}(T))$.}

 Pour tout entier $n \in \Z$, 
d'après le lemme \ref{Hnstabcoh}, 
on a $\mathcal{H} ^{n} (\E ^{(\bullet)} )
\in 
\underrightarrow{LM}  _{\Q, \mathrm{coh}} (\smash{\widetilde{\D}} _{\PP ^\sharp} ^{(\bullet)}(T))$.
De plus, d'après le corollaire \ref{LDLMiso=isoHn},
 le morphisme canonique 
$\phi \colon \E ^{(\bullet)} \to 
(\hdag T',T) (\E ^{(\bullet)})$
est un isomorphisme 
dans 
$\underrightarrow{LD} ^{\mathrm{b}} _{\Q} 
(\smash{\widetilde{\D}} _{\PP ^\sharp} ^{(\bullet)}(T))$,
si et seulement si,
pour tout entier $n \in \Z$, 
le morphisme 
$\mathcal{H} ^{n} (\phi) \colon 
\mathcal{H} ^{n}  (\E ^{(\bullet)})
\to 
\mathcal{H} ^{n} ((\hdag T',T) (\E^{(\bullet)}))$
est un isomorphisme 
de 
$\underrightarrow{LM}  _{\Q} (\smash{\widetilde{\D}} _{\PP ^\sharp} ^{(\bullet)} (T))$.
Notons $\psi _{n}\colon (\hdag T',T) ^{0} (\mathcal{H} ^{n}  (\E ^{(\bullet)}))
\to 
\mathcal{H} ^{n} ((\hdag T',T) (\E^{(\bullet)}))$
l'unique morphisme de 
$\underrightarrow{LM}  _{\Q,\mathrm{coh}} (\smash{\widetilde{\D}} _{\PP ^\sharp} ^{(\bullet)} (T'))$
factorisant canoniquement
$\mathcal{H} ^{n} (\phi)$.

a) Vérifions que $\psi _{n}$ est un isomorphisme. 
Comme le foncteur $\underrightarrow{\lim}$ est pleinement fidèle sur 
$\underrightarrow{LM}  _{\Q,\mathrm{coh}} (\smash{\widetilde{\D}} _{\PP ^\sharp} ^{(\bullet)} (T'))$, 
il suffit de prouver que 
$\underrightarrow{\lim}~ (\psi _{n})$
 est un isomorphisme.
Or, comme le foncteur $\underrightarrow{\lim}$ commute canoniquement à $\mathcal{H} ^{n}$
à isomorphisme canonique près 
(voir le diagramme \ref{pre1-lim-MD2square}),
on obtient que 
$\underrightarrow{\lim}~ (\psi _{n})$ est canoniquement isomorphe
au morphisme canonique
$(\hdag T',T) (\mathcal{H} ^{n} \E)
\to 
\mathcal{H} ^{n} ((\hdag T',T) (\E))$.
Comme le foncteur $(\hdag T',T) \colon 
D ^{\mathrm{b}} _{\mathrm{coh}} (\smash{\D} ^\dag _{\PP ^\sharp} (\hdag T) _{\Q})
\to 
D ^{\mathrm{b}} _{\mathrm{coh}} (\smash{\D} ^\dag _{\PP ^\sharp} (\hdag T') _{\Q})$
est exact, 
le morphisme
$(\hdag T',T) (\mathcal{H} ^{n} \E)
\to 
\mathcal{H} ^{n} ((\hdag T',T) (\E))$
est donc un isomorphisme. D'où le résultat. 

b) Si le théorème est vérifié pour $\underrightarrow{LM}  _{\Q, \mathrm{coh}} (\smash{\widetilde{\D}} _{\PP ^\sharp} ^{(\bullet)}(T))$, 
alors le morphisme canonique
$\mathcal{H} ^{n}  (\E ^{(\bullet)})
\to 
(\hdag T',T) ^{0} (\mathcal{H} ^{n}  (\E ^{(\bullet)}))$
est un isomorphisme. Comme en le composant avec 
$\psi _{n}$ on obtient $\mathcal{H} ^{n} (\phi)$, 
il en résulte la réduction demandée.

1)  Grâce au lemme \ref{strict-m0}, on peut supposer qu'il existe 
$m _0 \geq 0$ 
assez grand tel que
$\E ^{(\bullet)}$ 
soit un 
$\smash{\widetilde{\D}} _{\PP ^\sharp} ^{(\bullet +m _0)} (T)$-module localement de présentation finie. 
Comme le théorème est local (voir \ref{LMQ-iso:local}), on peut alors supposer
$\PP$ affine. 
Posons $E ^{(\bullet)}=
\Gamma (\PP, \E ^{(\bullet)})$,
$\FF ^{(\bullet)}:= (\hdag T',T) ^0(\E ^{(\bullet)})$,
$F ^{(\bullet)}=
\Gamma (\PP, \FF ^{(\bullet)})$.
Vérifions 
dans cette étape $1)$ que 
le morphisme canonique
$E ^{(\bullet)} \to 
F ^{(\bullet)}$
est un isomorphisme 
dans 
$\underrightarrow{LM}  _{\Q} (\smash{\widetilde{D}} _{\PP ^{\sharp}} ^{(\bullet)}(T))$, i.e., est une lim-ind-isogénie
(de $M (\smash{\widetilde{D}} _{\PP ^{\sharp}} ^{(\bullet + m _0)}(T))$).

a) Soient $E:=  \underrightarrow{\lim} _{m} \, E ^{(m)}$ et $N ^{(m)}$ le noyau du morphisme 
canonique
$E ^{(m)} \to E$.
On déduit des théorèmes de type $A$ (voir par exemple le deuxième point de la remarque \ref{rema-otimes-eqcat}) 
que, pour tout entier $m \in \N$, 
$E ^{(m)}$ est un $\smash{\widetilde{D}} _{\PP ^{\sharp}} ^{(m+m _0)}(T)$-module de type fini.
Par noethérianité de $\smash{\widetilde{D}} _{\PP ^{\sharp}} ^{(m+m _0)}(T)$,
le $\smash{\widetilde{D}} _{\PP ^{\sharp}} ^{(m+m _0)}(T)$-module $N ^{(m)}$ est alors aussi de type fini.
On en déduit qu'il existe $\lambda (m)\geq m$ tel que 
$E ^{(m)} \to E ^{(\lambda (m))}$ se factorise par 
$E ^{(m)} / N ^{(m)}\to E ^{(\lambda (m))}$.
On peut choisir $\lambda\colon \N \to \N$ tel que $\lambda \in L$.
On en déduit que $E ^{(\bullet)} \to E ^{(\bullet)} / N ^{(\bullet)}$
est en particulier une lim-ind-isogénie. 
Quitte à remplacer $E ^{(\bullet)}$ par $E ^{(\bullet)} / N ^{(\bullet)}$, 
on peut donc supposer que les morphismes de transition
$E ^{(m)}\to E ^{(m+1)}$ sont injectifs. 
De même, en utilisant \cite[3.4.4]{Be1}, 
on se ramène au cas où les $E ^{(m)}$ sont de plus sans $p$-torsion.

De même, posons $F:=  \underrightarrow{\lim} _{m} \, F ^{(m)}$ et,
pour tout entier $m\geq 0$,
soient $G  ^{(m)}$ le quotient de $F ^{(m)}$ par le noyau du morphisme canonique
$F ^{(m)} \to F _\Q$. Ainsi, les morphismes de transition
$G ^{(m)}\to G ^{(m+1)}$ sont injectifs et
les $G ^{(m)}$ sont des $\smash{\widetilde{D}} _{\PP ^{\sharp}} ^{(m+m _0)}(T')$-modules de type fini sans $p$-torsion.
Comme, pour les raisons analogues à ci-dessus, la flèche canonique $F ^{(\bullet)}\to G ^{(\bullet)}$
est une lim-ind-isogénie 
(de $M (\smash{\widetilde{D}} _{\PP ^{\sharp}} ^{(\bullet + m _0)}(T'))$ et donc de 
$M (\smash{\widetilde{D}} _{\PP ^{\sharp}} ^{(\bullet + m _0)}(T))$), 
on se ramène à établir que le morphisme 
canonique 
$E ^{(\bullet)} \to 
G ^{(\bullet)}$
est une lim-ind-isogénie de $M (\smash{\widetilde{D}} _{\PP ^{\sharp}} ^{(\bullet + m _0)}(T))$.

b) 
Comme le foncteur 
$\Gamma (\PP, -)$ commute aux limites inductives filtrantes et 
au foncteur $-\otimes _{\Z} \Q$, 
le fait que le morphisme canonique $\underrightarrow{\lim} ~\E ^{(\bullet)}
\to 
\underrightarrow{\lim} ~\FF ^{(\bullet)}$ 
soit un isomorphisme
entraîne que le morphisme canonique 
$E _\Q \to F _\Q$ est un isomorphisme.

c) 
Les $K$-espaces $E ^{(m)} _{\Q}$ et $G ^{(m)} _{\Q}$ sont munis d'une structure canoniques 
de $K$-espaces de Banach induites respectivement par leur structure
de $\smash{\widetilde{D}} _{\PP ^{\sharp}} ^{(m+m _0)}(T)_{\Q}$-module de type fini
et de $\smash{\widetilde{D}} _{\PP ^{\sharp}} ^{(m+m _0)}(T')_{\Q}$-module de type fini.
Pour ces topologies, les morphismes canoniques
$E ^{(m)} _{\Q}\to G ^{(m)} _{\Q}$ sont des morphismes continus
de $K$-espaces de Banach. 
Notons $W := \underrightarrow{\lim} _{m}\, G ^{(m)} _{\Q}$ muni de la topologie limite inductive
de $K$-espace localement convexe.
Notons $i _{m}\colon E ^{(m)} _{\Q}\to W$ le composé
$E ^{(m)} _{\Q}\to G ^{(m)} _{\Q}\to W$.
Alors $i _m$ est continu pour les topologies décrites ci-dessus. 
De plus, comme d'après la partie $b)$ on a 
$E _\Q \riso F _\Q$, 
on en déduit que le morphisme
$E _\Q \to W$ est bijectif.
Comme
$E ^{(m)} _\Q \to E _\Q$ est injectif, 
il en est alors de même de $i _m$.
 De plus, on a l'égalité d'ensembles
$W = \cup _{m\in \N} i _{m} (E ^{(m)} _{\Q})$.
D'après \cite[8.9]{Schneider-NonarchFuncAn}, pour tout entier $m$, il en résulte qu'il existe $\lambda (m)\geq m$  et un morphisme continu
$\beta ^{(m)}\colon G ^{(m)}  _{\Q} \to E ^{(\lambda (m))} _{\Q}$ tel que en le composant avec $i _{\lambda (m)}$ on obtienne
le morphisme canonique continue $G ^{(m)} _{\Q}\to W$. 
On choisit un tel $\lambda (m)\geq m$ le plus petit possible.
Comme $i _{\lambda (m)}$ est injective, 
une telle factorisation $\beta ^{(m)}$ du morphisme canonique
$G ^{(m)} _{\Q}\to W$
est unique. 
On vérifie de même que les morphismes $\beta ^{(m)}$ sont compatibles aux morphismes de transition,
i.e. que l'on obtient en fait le morphisme de
$\smash{\widetilde{\D}} _{\PP ^\sharp} ^{(\bullet+ m _0)} (T) _\Q$-modules
$\beta ^{(\bullet)} \colon 
G ^{(\bullet)} _\Q
\to 
\lambda ^{*}
E ^{(\bullet)} _\Q$.

Comme la famille $(p ^n G ^{(m)}) _{n\in \N}$  forme une base de voisinage de $0$, 
comme $E ^{(\lambda (m))}$ est un ouvert de $E ^{(\lambda (m))} _\Q$, 
(on se rappelle que $E ^{(\lambda (m))}$ et $G ^{(m)}$ sont sans $p$-torsion), 
le fait que $\beta ^{(m)}$ soit continue implique alors qu'il existe $\chi (m)\in \N$ assez grand tel que 
$\beta ^{(m)} (p ^{\chi (m)} G ^{(m)})
\subset
E ^{(\lambda (m))}$.
On peut choisir les $\chi (m)$ tels que 
l'application $\chi \colon \N \to \N$ induite soit croissante. 
Notons $\gamma ^{(\bullet)}$ le composé de 
$\beta ^{(\bullet)}
\colon G ^{(\bullet)}  _{\Q} \to\lambda ^* E ^{(\bullet)} _{\Q} $
avec le morphisme canonique
$\lambda ^* E ^{(\bullet)} _{\Q} 
\to \chi ^{*} \lambda ^* E ^{(\bullet)} _{\Q} $.
D'après ce que l'on vient de voir, 
$\gamma ^{(\bullet)}$ se factorise alors (de manière unique)
en 
le morphisme de la forme 
$g  ^{(\bullet)}
\colon G ^{(\bullet)}  \to \chi ^* \lambda ^* E ^{(\bullet)} $.
Notons 
$f ^{(\bullet)} 
\colon E ^{(\bullet)} 
\to 
G ^{(\bullet)}$
le morphisme canonique.
Comme pour tout $m\in \N$, les morphismes canoniques
$E ^{(m)} \to W$
et
$G ^{(m)} \to W$
sont injectifs, 
on calcule que 
$g  ^{(\bullet)} \circ f  ^{(\bullet)}$
et $\chi ^* \lambda ^* (f  ^{(\bullet)}) \circ g  ^{(\bullet)}$
sont les morphismes canoniques.

2) 
On déduit de l'étape $1$
que le morphisme canonique 
$\smash{\widetilde{\D}} _{\PP ^\sharp} ^{ (\bullet+m _0)} (T)
\widehat{\otimes} _{\smash{\widetilde{D}} _{\PP ^{\sharp}} ^{ (\bullet+m _0)} (T)}
E ^{(\bullet)}
\to 
\smash{\widetilde{\D}} _{\PP ^\sharp} ^{ (\bullet+m _0)} (T)
\widehat{\otimes} _{\smash{\widetilde{D}} _{\PP ^{\sharp}} ^{ (\bullet+m _0)} (T)}
F ^{(\bullet)}$
est une lim-ind-isogénie de $M (\smash{\widetilde{\D}} _{\PP ^\sharp} ^{(\bullet + m _0)}(T))$.

3) Vérifions à présent que le morphisme canonique
$\smash{\widetilde{\D}} _{\PP ^\sharp} ^{ (\bullet+m _0)} (T)
\widehat{\otimes} _{\smash{\widetilde{D}} _{\PP ^{\sharp}} ^{ (\bullet+m _0)} (T)}
F ^{(\bullet)}
\to 
\smash{\widetilde{\D}} _{\PP ^\sharp} ^{ (\bullet+m _0)} (T')
\widehat{\otimes} _{\smash{\widetilde{D}} _{\PP ^{\sharp}} ^{ (\bullet+m _0)} (T')}
F ^{(\bullet)}$
est un isomorphisme de $M (\smash{\widetilde{\D}} _{\PP ^\sharp} ^{(\bullet + m _0)}(T))$.

Par quasi-cohérence des faisceaux $\widetilde{\D} _{P ^{\sharp} _i} ^{ (m)} (T)$ et
$\widetilde{\D} _{P ^{\sharp} _i} ^{ (m)} (T')$,
le morphisme canonique 
$\widetilde{\D} _{P ^{\sharp} _i} ^{ (m)} (T)
\otimes _{\widetilde{D} _{P _i} ^{ (m)} (T)}
\widetilde{D} _{P _i} ^{ (m)} (T')
\to
\widetilde{\D} _{P ^{\sharp} _i} ^{ (m)} (T')$
est un isomorphisme.
On en déduit que le morphisme canonique
$\widetilde{\D} _{P ^{\sharp} _i} ^{ (m)} (T)
\otimes _{\widetilde{D} _{P _i} ^{ (m)} (T)}
F _i ^{(m)}
\to
\widetilde{\D} _{P ^{\sharp} _i} ^{ (m)} (T')
\otimes _{\widetilde{D} _{P _i} ^{ (m)} (T')}
F _i ^{(m)}$
est un isomorphisme.
En passant à la limite projective,
on en déduit le résultat. 

4) Il résulte des étapes 2) et 3) 
que le morphisme canonique 
$\smash{\widetilde{\D}} _{\PP ^\sharp} ^{ (\bullet+m _0)} (T)
\widehat{\otimes} _{\smash{\widetilde{D}} _{\PP ^{\sharp}} ^{ (\bullet+m _0)} (T)}
E ^{(\bullet)}
\to 
\smash{\widetilde{\D}} _{\PP ^\sharp} ^{ (\bullet+m _0)} (T')
\widehat{\otimes} _{\smash{\widetilde{D}} _{\PP ^{\sharp}} ^{ (\bullet+m _0)} (T')}
F ^{(\bullet)}$
est une lim-ind-isogénie de $M (\smash{\widetilde{\D}} _{\PP ^\sharp} ^{(\bullet + m _0)}(T))$.
Il découle alors de la deuxième partie du lemme \ref{lemm-coh-otimes} (et aussi du lemme \ref{lambda-id-iso})
que le morphisme du haut du diagramme canonique commutatif:
\begin{equation}
\notag
\xymatrix @ R=0,4cm{
{\smash{\widetilde{\D}} _{\PP ^\sharp} ^{ (\bullet+m _0)} (T)
\otimes _{\smash{\widetilde{D}} _{\PP ^{\sharp}} ^{ (\bullet+m _0)} (T)}
E ^{(\bullet)}} 
\ar[r] ^-{}
\ar[d] ^-{}
& 
{\smash{\widetilde{\D}} _{\PP ^\sharp} ^{ (\bullet+m _0)} (T')
\otimes _{\smash{\widetilde{D}} _{\PP ^{\sharp}} ^{ (\bullet+m _0)} (T')}
F ^{(\bullet)}} 
\ar[d] ^-{}
\\ 
{\E ^{(\bullet)}} 
\ar[r] ^-{}
& 
{\FF ^{(\bullet)}} 
}
\end{equation}
est une lim-ind-isogénie de $M (\smash{\widetilde{\D}} _{\PP ^\sharp} ^{(\bullet + m _0)}(T))$.
Or, on déduit des théorèmes de type $A$ du deuxième point de la remarque 
\ref{rema-otimes-eqcat}
 que les morphismes verticaux sont des isomorphismes de $M (\smash{\widetilde{\D}} _{\PP ^\sharp} ^{(\bullet + m _0)}(T))$.
 D'où le théorème. 
\end{proof}

\begin{coro}
\label{coro1limTouD}
Soient  $T' \supset T$ un diviseur,
$\E ^{\prime (\bullet)}
\in 
\underrightarrow{LD} ^{\mathrm{b}} _{\Q, \mathrm{coh}} (\smash{\widetilde{\D}} _{\PP ^\sharp} ^{(\bullet)} (T'))$
et
$\E ':= 
\underrightarrow{\lim}
\E ^{\prime (\bullet)}
\in D ^{\mathrm{b}} _{\mathrm{coh}} (\smash{\D} ^\dag _{\PP ^\sharp} (\hdag T ') _{\Q})$.
Si  
$\E '\in D ^{\mathrm{b}}  _{\mathrm{coh}} (\smash{\D} ^\dag _{\PP ^\sharp} (\hdag T) _{\Q})$,
alors 
$\E ^{\prime (\bullet)}
\in 
\underrightarrow{LD} ^{\mathrm{b}} _{\Q, \mathrm{coh}} (\smash{\widetilde{\D}} _{\PP ^\sharp} ^{(\bullet)}(T))$.
\end{coro}

\begin{proof}
Comme 
$\E '\in D ^{\mathrm{b}}  _{\mathrm{coh}} (\smash{\D} ^\dag _{\PP ^\sharp} (\hdag T) _{\Q})$, 
il existe
$\E ^{(\bullet)}
\in 
\underrightarrow{LD} ^{\mathrm{b}} _{\Q, \mathrm{coh}} (\smash{\widetilde{\D}} _{\PP ^\sharp} ^{(\bullet)} (T))$
et un isomorphisme 
$\smash{\D} ^\dag _{\PP ^\sharp} (\hdag T) _{\Q}$-linéaire de la forme
$\E '\riso
\underrightarrow{\lim}
\E ^{(\bullet)}$.
Or, comme le morphisme canonique
$(\hdag T' , T) (\E ') \to \E'$ est un morphisme de 
$D ^{\mathrm{b}} _{\mathrm{coh}} (\smash{\D} ^\dag _{\PP ^\sharp} (\hdag T ') _{\Q})$
qui est un isomorphisme en dehors de $T'$, il en résulte d'après \cite[4.8]{caro_log-iso-hol}
que celui-ci est un isomorphisme.
On en déduit que le morphisme canonique
$\E' \to (\hdag T' , T) (\E ')$ de 
$D ^{\mathrm{b}}  (\smash{\D} ^\dag _{\PP ^\sharp} (\hdag T ) _{\Q})$
est un isomorphisme. 
En appliquant le théorème \ref{limTouD}
à 
$\E ^{(\bullet)}$, 
on en déduit que le morphisme canonique 
$\E ^{(\bullet)} \to 
(\hdag T',T) (\E ^{(\bullet)})$
est alors un isomorphisme 
dans 
$\underrightarrow{LD} ^{\mathrm{b}} _{\Q, \mathrm{coh}} (\smash{\widetilde{\D}} _{\PP ^\sharp} ^{(\bullet)}(T))$.
Comme $(\hdag T' , T) (\E ^{(\bullet)}) ,~\E ^{\prime(\bullet)}
\in 
\underrightarrow{LD} ^{\mathrm{b}} _{\Q, \mathrm{coh}} (\smash{\widetilde{\D}} _{\PP ^\sharp} ^{(\bullet)} (T'))$
et
comme 
$$\underrightarrow{\lim}~
(\hdag T' , T) (\E ^{(\bullet)})
\riso 
(\hdag T' , T) (\underrightarrow{\lim}~ \E ^{(\bullet)}) 
\riso 
(\hdag T' , T) (\E ') 
\riso \E'
\riso 
\underrightarrow{\lim}~
\E ^{\prime(\bullet)},$$ 
par pleine fidélité du foncteur 
$\underrightarrow{\lim}$ sur 
$\underrightarrow{LD} ^{\mathrm{b}} _{\Q, \mathrm{coh}} (\smash{\widetilde{\D}} _{\PP ^\sharp} ^{(\bullet)} (T'))$
on obtient que 
$(\hdag T' , T) (\E ^{(\bullet)})
\riso 
\E ^{\prime(\bullet)}$.
\end{proof}

\begin{coro}
\label{coh-Bbullet}
On a
$\smash{\widetilde{\B}} _{\PP} ^{(\bullet)} (T ) \in 
\underrightarrow{LM}  _{\Q, \mathrm{coh}} (\smash{\widehat{\D}} _{\PP } ^{(\bullet)})
\cap 
\underrightarrow{LM}  _{\Q, \mathrm{coh}} (\smash{\widetilde{\D}} _{\PP ^{\sharp}} ^{(\bullet)} (T))$.
\end{coro}

\begin{proof}
On sait déjà que 
$\smash{\widetilde{\B}} _{\PP} ^{(\bullet)} (T ) 
\in \underrightarrow{LM}  _{\Q, \mathrm{coh}} (\smash{\widetilde{\D}} _{\PP ^{\sharp}} ^{(\bullet)} (T))
\cap \underrightarrow{LM}  _{\Q, \mathrm{coh}} (\smash{\widetilde{\D}} _{\PP} ^{(\bullet)} (T))$.
Or, Berthelot a vérifié dans \cite{Becohdiff}
que 
$\O _{\PP} (\hdag T) _{\Q}=
\underrightarrow{\lim}
\smash{\widetilde{\B}} _{\PP} ^{(\bullet)} (T ) $
est un 
$\smash{\D} ^\dag _{\PP,\Q}$-module cohérent.
On conclut grâce au lemme \ref{coro1limTouD} (utilisé sans structure logarithmique).
\end{proof}

\begin{coro}
\label{coro2limTouD}
Soient  $T' \supset T$ un diviseur,
$\E \in D ^{\mathrm{b}}  _{\mathrm{coh}} (\smash{\D} ^\dag _{\PP ^\sharp} (\hdag T) _{\Q}) \cap 
D ^{\mathrm{b}}  _{\mathrm{coh}} (\smash{\D} ^\dag _{\PP ^\sharp} (\hdag T') _{\Q})$.
Soient $\E ^{(\bullet)}
\in 
\underrightarrow{LD} ^{\mathrm{b}} _{\Q, \mathrm{coh}} (\smash{\widetilde{\D}} _{\PP ^\sharp} ^{(\bullet)} (T))$
et
$\E ^{\prime (\bullet)}
\in 
\underrightarrow{LD} ^{\mathrm{b}} _{\Q, \mathrm{coh}} (\smash{\widetilde{\D}} _{\PP ^\sharp} ^{(\bullet)} (T'))$
tels que 
l'on ait les isomorphismes $\smash{\D} ^\dag _{\PP ^\sharp} (\hdag T) _{\Q}$-linéaires de la forme
$\underrightarrow{\lim}
\E ^{(\bullet)}
\riso 
\E$
et
$\underrightarrow{\lim}
\E ^{\prime (\bullet)}
\riso 
\E$.
Alors 
$\E ^{(\bullet)} \riso \E ^{\prime (\bullet)}$
dans 
$\underrightarrow{LD} ^{\mathrm{b}} _{\Q, \mathrm{coh}} (\smash{\widetilde{\D}} _{\PP ^\sharp} ^{(\bullet)} (T))$.
\end{coro}

\begin{proof}
Cela découle de \ref{coro1limTouD} et de la pleine fidélité 
du foncteur 
$\underrightarrow{\lim}$ sur 
$\underrightarrow{LD} ^{\mathrm{b}} _{\Q, \mathrm{coh}} (\smash{\widetilde{\D}} _{\PP ^\sharp} ^{(\bullet)} (T))$.
\end{proof}

\begin{nota}
\label{fct-qcoh2coh}
Soient $\mathcal{Q}$ un $\V$-schéma formel lisse,
$\ZZ' $ un diviseur à croisements normaux stricts de $\mathcal{Q}$, 
$\mathcal{Q} ^{\sharp} := (\mathcal{Q}, \ZZ')$, 
$U$ un diviseur de $Q$
et
$\phi ^{(\bullet)}
\colon 
\smash{\underrightarrow{LD}} ^{\mathrm{b}} _{\Q,\mathrm{qc}}
( \smash{\widetilde{\D}} _{\PP ^\sharp} ^{(\bullet)}(T))
\to 
\smash{\underrightarrow{LD}} ^{\mathrm{b}} _{\Q,\mathrm{qc}}
( \smash{\widetilde{\D}} _{\mathcal{Q} ^{\sharp}} ^{(\bullet)}(U))$
un foncteur.
On en déduit un foncteur 
$\mathrm{Coh} _{T} ( \phi ^{(\bullet)})
\colon 
D ^\mathrm{b} _\mathrm{coh} ( \smash{\D} ^\dag _{\PP ^\sharp} (\hdag T) _{\Q} ) \to
D ^\mathrm{b}  ( \smash{\D} ^\dag _{\mathcal{Q} ^{\sharp}} (\hdag U) _{\Q} )$
en posant 
$\mathrm{Coh} _{T} ( \phi ^{(\bullet)}) := \underrightarrow{\lim} \circ \phi ^{(\bullet)}\circ (\underrightarrow{\lim} _T ) ^{-1}$,
où $(\underrightarrow{\lim} _T ) ^{-1}$ désigne un foncteur quasi-inverse de 
l'équivalence de catégories de 
\begin{equation}
\label{eq-catLDBer-LD-D}
\underrightarrow{\lim} 
\colon 
\underrightarrow{LD} ^{\mathrm{b}}  _{\Q, \mathrm{coh}} (\smash{\widetilde{\D}} _{\PP ^\sharp} ^{(\bullet)} (T))
\overset{\ref{eq-coh-lim}}{\cong} 
D ^{\mathrm{b}} _{\mathrm{coh}}
(\underrightarrow{LM} _{\Q} (\smash{\widetilde{\D}} _{\PP ^\sharp} ^{(\bullet)} (T)))
\overset{\ref{eqcatcoh}}{\cong}
D ^{\mathrm{b}} _{\mathrm{coh}}( \smash{\D} ^\dag _{\PP ^\sharp} (\hdag T) _{\Q} ).
\end{equation}

\end{nota}

\begin{rema}
\label{rema-fct-qcoh2coh}
Soit $T \subset T'$ un second diviseur. 
Lorsque $\E \in 
D ^\mathrm{b} _\mathrm{coh} ( \smash{\D} ^\dag _{\PP ^\sharp} (\hdag T') _{\Q} ) 
\cap D ^\mathrm{b} _\mathrm{coh} ( \smash{\D} ^\dag _{\PP ^\sharp} (\hdag T) _{\Q} )$, 
grâce au corollaire \ref{coro2limTouD},  
les objets  de
$\smash{\underrightarrow{LD}} ^{\mathrm{b}} _{\Q ,\mathrm{coh}}
(\smash{\widetilde{\D}} _{\PP ^\sharp} ^{(\bullet)}(T))$
et
$\smash{\underrightarrow{LD}} ^{\mathrm{b}} _{\Q ,\mathrm{coh}}
(\smash{\widetilde{\D}} _{\PP ^\sharp} ^{(\bullet)}(T'))$
correspondant via les équivalences de catégories respectives de \ref{eq-coh-lim}
soient isomorphes. 
Avec les notations de \ref{fct-qcoh2coh}, les foncteurs 
$\mathrm{Coh} _{T} ( \phi ^{(\bullet)}) $
et
$\mathrm{Coh} _{T'} ( \phi ^{(\bullet)}) $
sont donc isomorphes sur 
$D ^\mathrm{b} _\mathrm{coh} ( \smash{\D} ^\dag _{\PP ^\sharp} (\hdag T') _{\Q} ) 
\cap D ^\mathrm{b} _\mathrm{coh} ( \smash{\D} ^\dag _{\PP ^\sharp} (\hdag T) _{\Q} )$.
\end{rema}

Terminons la section par d'autres applications du théorème.

\begin{prop}
\label{stab-coh-oub-DT}
Soient $T \subset D \subset T' $ des diviseurs de $P$.
\begin{enumerate}
\item Soit 
$\E ^{(\bullet)}
\in 
\underrightarrow{LD} ^{\mathrm{b}} _{\Q, \mathrm{coh}} (\smash{\widetilde{\D}} _{\PP ^\sharp} ^{(\bullet)}(T))
\cap \underrightarrow{LD} ^{\mathrm{b}} _{\Q, \mathrm{coh}} (\smash{\widetilde{\D}} _{\PP ^\sharp} ^{(\bullet)} (T'))$.
Alors 
$\E ^{(\bullet)}
\in 
 \underrightarrow{LD} ^{\mathrm{b}} _{\Q, \mathrm{coh}} (\smash{\widetilde{\D}} _{\PP ^\sharp} ^{(\bullet)} (D))$.
\item Soit 
$\E \in 
D ^{\mathrm{b}} _{\mathrm{coh}}( \smash{\D} ^\dag _{\PP ^\sharp} (\hdag T) _{\Q} )
\cap 
D ^{\mathrm{b}} _{\mathrm{coh}}( \smash{\D} ^\dag _{\PP ^\sharp} (\hdag T') _{\Q} )$.
Alors 
$\E 
\in 
D ^{\mathrm{b}} _{\mathrm{coh}}( \smash{\D} ^\dag _{\PP ^\sharp} (\hdag D) _{\Q} )$.
\end{enumerate}
\end{prop}

\begin{proof}
Il découle de \ref{oub-pl-fid-iso1} que le morphisme canonique
$(\hdag D,\,T) \circ 
\mathrm{oub} _{T ,D} 
( \mathrm{oub} _{D ,T'} (\E ^{(\bullet)}))
\to 
\mathrm{oub} _{D ,T'} (\E ^{(\bullet)})$
de $ \underrightarrow{LD} ^{\mathrm{b}} _{\Q, \mathrm{qc}} (\smash{\widetilde{\D}} _{\PP ^\sharp} ^{(\bullet)} (D))$
est un isomorphisme.
D'où la première assertion.
Grâce à \ref{coro1limTouD}, la seconde assertion découle de la première.
\end{proof}

\begin{nota}
\label{nota-hag-sansrisque}

\begin{itemize}
\item 
Soient $D \subset T \subset T' $ des diviseurs de $P$.
D'après \ref{gen-oub-pl-fid},
en omettant d'indiquer les foncteurs oublis respectifs, 
les foncteur $(\hdag T',~D) $ et
$(\hdag T', ~T) $ sont canoniquement isomorphes sur 
$\smash{\underrightarrow{LD}} ^{\mathrm{b}} _{\Q ,\mathrm{qc}}
(\smash{\widetilde{\D}} _{\PP ^\sharp} ^{(\bullet)}(T))$.
On pourra donc les noter sans risque 
$(\hdag T') $. 

\item Pour tous diviseurs $D \subset T$, on dispose de l'isomorphisme de foncteurs 
$\mathrm{Coh} _{D} ((\hdag T' ,~D) ) \riso (\hdag T' ,~D) $, le dernier foncteur étant celui défini en
\ref{hdag-def-coh}. Les notations respectives sont donc compatibles. 

\item Soient $T$ et $D \subset D' $ des diviseurs de $P$.
Grâce à \ref{hdagT'T=cup} et \ref{oub-pl-fid}, 
les foncteurs $(\hdag T)$ et $(\hdag T \cup D)$ sont canoniquement isomorphes
sur $\underrightarrow{LD} ^{\mathrm{b}} _{\Q} (\smash{\widetilde{\D}} _{\PP ^\sharp} ^{(\bullet)} (D))$.
On dispose alors du foncteur 
$(\hdag T' ) := \mathrm{Coh} _{D} ((\hdag T) )
\colon D ^\mathrm{b} _\mathrm{coh} ( \smash{\D} ^\dag _{\PP ^\sharp} (\hdag D) _{\Q} ) 
\to D ^\mathrm{b} _\mathrm{coh} ( \smash{\D} ^\dag _{\PP ^\sharp} (\hdag T \cup D) _{\Q} ) $.
Avec la remarque \ref{rema-fct-qcoh2coh}, 
comme les foncteurs  
$\mathrm{Coh} _{D} ((\hdag T) )$ et
$\mathrm{Coh} _{D'} ((\hdag T) )$
sont isomorphes sur 
 $D ^\mathrm{b} _\mathrm{coh} ( \smash{\D} ^\dag _{\PP ^\sharp} (\hdag D) _{\Q} ) 
\cap D ^\mathrm{b} _\mathrm{coh} ( \smash{\D} ^\dag _{\PP ^\sharp} (\hdag T) _{\Q} )$,
il n'est donc pas nécessaire de préciser $D$. 
\end{itemize}

\end{nota}

\section{Foncteur cohomologique local à support strict dans un fermé,
foncteur de localisation en dehors d'un fermé}

\subsection{Triangle distingué de localisation relativement à un diviseur}

\begin{vide}
[Triangle distingué de localisation]
Soit $\E ^{(\bullet)}
\in 
\smash{\underrightarrow{LD}} ^{\mathrm{b}} _{\Q,\mathrm{qc}} ( \smash{\widetilde{\D}} _{\PP ^\sharp} ^{(\bullet)})$.
Berthelot (dans le cas non logarithmique, mais on peut reprendre sa définition avec des 
pôles logarithmiques)
a défini le foncteur cohomologique local à support strict dans $T$ noté
$ \R \underline{\Gamma} ^\dag _{T}
\colon 
 \smash{\underrightarrow{LD}} ^{\mathrm{b}} _{\Q,\mathrm{qc}} ( \smash{\widetilde{\D}} _{\PP ^\sharp} ^{(\bullet)})
 \to 
 \smash{\underrightarrow{LD}} ^{\mathrm{b}} _{\Q,\mathrm{qc}} ( \smash{\widetilde{\D}} _{\PP ^\sharp} ^{(\bullet)})$.
De plus, il a vérifié le triangle distingué de localisation 
\begin{equation}
\label{tri-loc-berthelot}
 \R \underline{\Gamma} ^\dag _{T} (\E ^{(\bullet)})
\to 
\E ^{(\bullet)}
\to 
(\hdag T) (\E ^{(\bullet)})
\to 
 \R \underline{\Gamma} ^\dag _{T} (\E ^{(\bullet)})
 [1].
\end{equation}
On pourra désigner par la suite le triangle distingué \ref{tri-loc-berthelot} 
par $\Delta _{T} (\E ^{(\bullet)})$.
On peut étendre ce triangle au cas où $T =P$ de la manière suivante:
$ \R \underline{\Gamma} ^\dag _{T} (\E ^{(\bullet)})= \E ^{(\bullet)}$
et
$(\hdag T) (\E ^{(\bullet)})=0$.
Ce dernier est fonctoriel en $T$ et $\E ^{(\bullet)}$, i.e.
si $T \subset T'$ avec $T'$ un diviseur ou $T' =P$ et 
si $\E ^{(\bullet)} \to \E ^{\prime (\bullet)}$ est un morphisme de 
$ \smash{\underrightarrow{LD}} ^{\mathrm{b}} _{\Q,\mathrm{qc}} ( \smash{\widetilde{\D}} _{\PP ^\sharp} ^{(\bullet)})$
on bénéficie du morphisme canonique de triangles distingués de la forme
$\Delta _{T} (\E ^{(\bullet)})
\to 
\Delta _{T'} (\E ^{\prime (\bullet)})$

En fait, d'après ce qui suit, le foncteur $\R \underline{\Gamma} ^\dag _{T}$ est canoniquement déterminé à partir du foncteur 
$(\hdag T) $ et du triangle distingué \ref{tri-loc-berthelot}.
On pourrait donc le prendre comme définition.
\end{vide}

\begin{lemm}
\label{annulationHom-hdag}
Soient 
$\E ^{(\bullet)}
\in \smash{\underrightarrow{LD}} ^\mathrm{b} _{\Q,\mathrm{qc}} ( \smash{\widetilde{\D}} _{\PP ^\sharp} ^{(\bullet)})$
et 
$\FF ^{(\bullet)}
\in \smash{\underrightarrow{LD}} ^\mathrm{b} _{\Q,\mathrm{qc}} ( \smash{\widetilde{\D}} _{\PP ^\sharp} ^{(\bullet) } (T))$.
On suppose de plus que l'on dispose 
dans 
$\smash{\underrightarrow{LD}} ^\mathrm{b} _{\Q,\mathrm{qc}} ( \smash{\widetilde{\D}} _{\PP ^\sharp} ^{(\bullet)})$
de l'isomorphisme
$(\hdag T) (\E ^{(\bullet)} )\riso 0$.
Alors 
$\mathrm{Hom} _{\smash{\underrightarrow{LD}}  _{\Q} ( \smash{\widetilde{\D}} _{\PP ^\sharp} ^{(\bullet)})}
(\E ^{(\bullet)} , \FF ^{(\bullet)}) =0.$
\end{lemm}

\begin{proof}
Soit $\phi \colon 
\E ^{(\bullet)}\to  \FF ^{(\bullet)}$
un morphisme de $\smash{\underrightarrow{LD}} ^{\mathrm{b}} _{\Q,\mathrm{qc}} ( \smash{\widetilde{\D}} _{\PP ^\sharp} ^{(\bullet)}(T))$. 
Comme le morphisme canonique 
$\FF ^{(\bullet)} 
\to
(\hdag T)  (\FF ^{(\bullet)} )$
est un isomorphisme,
le morphisme $\phi$ se factorise canoniquement par 
$(\hdag T) (\phi )$.
Or, comme 
$(\hdag T) (\E ^{(\bullet)} )\riso 0$, 
alors 
$(\hdag T) (\phi )=0$.
On en déduit que $\phi =0$.
D'où le résultat. 
\end{proof}

\begin{lemm}
\label{lemm-FtoEtohdagE}
Soit un triangle distingué dans 
$\smash{\underrightarrow{LD}} ^\mathrm{b} _{\Q,\mathrm{qc}} ( \smash{\widetilde{\D}} _{\PP ^\sharp} ^{(\bullet)})$
de la forme
\begin{equation}
\label{FtoEtohdagE-to}
\FF ^{(\bullet)}
\to 
\E ^{(\bullet)} 
\to 
(\hdag T) (\E ^{(\bullet)} )  
 \to 
\FF ^{(\bullet)} [1] ,
\end{equation}
où le second morphisme est le morphisme canonique. 
Pour tout diviseur $T \subset T'$, 
on dispose alors de l'isomorphisme dans 
$\smash{\underrightarrow{LD}} ^\mathrm{b} _{\Q,\mathrm{qc}} ( \smash{\widetilde{\D}} _{\PP ^\sharp} ^{(\bullet)})$ de la forme 
$(\hdag T') (\FF ^{(\bullet)} )\riso 0$.
\end{lemm}

\begin{proof}
Comme d'après \ref{hdagT'T=cup} le morphisme canonique 
$(\hdag T') (\E ^{(\bullet)} )
\to
(\hdag T') ((\hdag T) (\E ^{(\bullet)} ))$
est un isomorphisme,
en appliquant le foncteur 
$(\hdag T')$ au triangle distingué \ref{FtoEtohdagE-to},
l'un des axiomes sur les catégories triangulées nous permet de conclure. 
\end{proof}

\begin{vide}
[Une définition alternative équivalente du foncteur cohomologique local à support strict dans un diviseur]
\label{GammaT}
Soit $T \subset T'$ un second diviseur. 
Supposons donné le diagramme commutatif dans 
$\smash{\underrightarrow{LD}} ^\mathrm{b} _{\Q,\mathrm{qc}} ( \smash{\widetilde{\D}} _{\PP ^\sharp} ^{(\bullet)})$
de la forme
\begin{equation}
\label{prefonct-GammaT}
\xymatrix @=0,4cm{
{\FF ^{(\bullet)}  } 
\ar[r] ^-{}
& 
{\E ^{(\bullet)}  } 
\ar[r] ^-{}
\ar[d] ^-{\phi}
& 
{(\hdag T) (\E ^{(\bullet)} )} 
\ar[d] ^-{(\hdag T)(\phi)}
\ar[r] ^-{}
&
{\FF ^{(\bullet)}  [1]} 
\\ 
{\FF ^{\prime (\bullet)}  } 
\ar[r] ^-{}
& 
{\E ^{\prime (\bullet)}  } 
\ar[r] ^-{}
& 
{(\hdag T) (\E ^{\prime (\bullet)} )} 
\ar[r] ^-{}
&
{\FF ^{\prime (\bullet)}   [1]} 
}
\end{equation}
dont les flèches horizontales du milieu sont les morphismes canoniques
et dont les deux triangles horizontaux sont distingués. 
D'après les lemmes \ref{annulationHom-hdag} et \ref{lemm-FtoEtohdagE}, 
on obtient alors 
$$H ^{-1} (\R \mathrm{Hom} _{D (\underrightarrow{LM} _{\Q} (\smash{\widetilde{\D}} _{\PP ^\sharp} ^{(\bullet)} (T)))}
( \FF ^{(\bullet)}  ,(\hdag T) (\E ^{\prime (\bullet)} )))
\underset{\ref{H0Homrm-DLM}}{\riso}
\mathrm{Hom} _{D (\underrightarrow{LM} _{\Q} (\smash{\widetilde{\D}} _{\PP ^\sharp} ^{(\bullet)} (T)))}
( \FF ^{(\bullet)}  ,(\hdag T) (\E ^{\prime (\bullet)}) [-1]) =0.$$
On en déduit, grâce à 
\cite[1.1.9]{BBD},
qu'il existe donc un unique morphisme 
$\FF ^{(\bullet)}\to \FF ^{\prime (\bullet)} $
induisant dans 
$\smash{\underrightarrow{LD}} ^\mathrm{b} _{\Q,\mathrm{qc}} ( \smash{\widetilde{\D}} _{\PP ^\sharp} ^{(\bullet)})$
le diagramme commutatif:
\begin{equation}
\label{fonct-hdagT}
\xymatrix @=0,4cm{
{\FF ^{(\bullet)}  } 
\ar[r] ^-{}
\ar@{.>}[d] ^-{\exists !}
& 
{\E ^{(\bullet)}  } 
\ar[r] ^-{}
\ar[d] ^-{\phi}
& 
{(\hdag T) (\E ^{(\bullet)} )} 
\ar[d] ^-{(\hdag T)(\phi)}
\ar[r] ^-{}
& 
{\FF ^{(\bullet)}  [1]} 
\ar@{.>}[d] ^-{\exists !}
\\ 
{\FF ^{\prime (\bullet)}  } 
\ar[r] ^-{}
& 
{\E ^{\prime (\bullet)}  } 
\ar[r] ^-{}
& 
{(\hdag T) (\E ^{\prime (\bullet)} )} 
\ar[r] ^-{}
& 
{\FF ^{\prime (\bullet)} [1].} 
}
\end{equation}
Comme pour \cite[1.1.10]{BBD}, 
cela implique que le cône de 
$\E ^{(\bullet)}
\to  
(\hdag T) (\E ^{(\bullet)} ) $
est unique à isomorphisme canonique près. 
Un tel complexe $\FF ^{ (\bullet)}$ est donc unique à isomorphisme unique près et se
notera alors
$ \R \underline{\Gamma} ^\dag _{T} (\E ^{(\bullet)})$.
En considérant le diagramme \ref{fonct-hdagT}, 
la formation du complexe
$\R \underline{\Gamma} ^\dag _{T} (\E ^{(\bullet)}) $ est fonctoriel en 
$\E ^{(\bullet)}$.

On dispose aussi
du morphisme fonctoriel en $\E ^{(\bullet)} $ de forme
$\R \underline{\Gamma} ^\dag _{T} (\E ^{(\bullet)}) \to \E ^{(\bullet)} $.
De la même manière, en utilisant 
les lemmes  \ref{annulationHom-hdag} et \ref{lemm-FtoEtohdagE}
et la proposition \cite[1.1.9]{BBD}, on vérifie qu'il existe une unique flèche 
$\R \underline{\Gamma} ^\dag _{T} (\E ^{(\bullet)}) 
\to
\R \underline{\Gamma} ^\dag _{T'} (\E ^{(\bullet)}) $
induisant le diagramme canonique 
\begin{equation}
\label{fonct-hdagX2}
\xymatrix @=0,4cm{
{\R \underline{\Gamma} ^\dag _{T} (\E ^{(\bullet)}) } 
\ar[r] ^-{}
\ar@{.>}[d] ^-{\exists !}
& 
{\E ^{(\bullet)}  } 
\ar[r] ^-{}
\ar@{=}[d] ^-{}
& 
{(\hdag T) (\E ^{(\bullet)} )} 
\ar[d] ^-{}
\ar[r] ^-{}
& 
{\R \underline{\Gamma} ^\dag _{T} (\E ^{(\bullet)})[1]} 
\ar@{.>}[d] ^-{\exists !}
\\ 
{\R \underline{\Gamma} ^\dag _{T'} (\E ^{ (\bullet)})  } 
\ar[r] ^-{}
& 
{\E ^{ (\bullet)}  } 
\ar[r] ^-{}
& 
{(\hdag T') (\E ^{ (\bullet)} )} 
\ar[r] ^-{}
& 
{\R \underline{\Gamma} ^\dag _{T'} (\E ^{ (\bullet)})[1]} 
}
\end{equation}
 commutatif.
Ces propriétés sont synthétisées en disant  que le complexe
$\R \underline{\Gamma} ^\dag _{T} (\E ^{(\bullet)}) $
est fonctoriel en $T$.
\end{vide}

\begin{rema}
Si $T$ est vide
et si
$\E ^{(\bullet)}
\in 
\smash{\underrightarrow{LD}} ^{\mathrm{b}} _{\Q,\mathrm{qc}} ( \smash{\widetilde{\D}} _{\PP ^\sharp} ^{(\bullet)})$,
alors
le morphisme canonique 
$\E ^{(\bullet)}
\to 
(\hdag \emptyset) (\E ^{(\bullet)})$
est un isomorphisme.
Cela n'est plus forcément le cas sans l'hypothèse de quasi-cohérence de
$\E ^{(\bullet)}$.
On pourrait sûrement construire le foncteur cohomologique local
$ \R \underline{\Gamma} ^\dag _{T}
\colon 
 \smash{\underrightarrow{LD}} ^{\mathrm{b}} _{\Q} ( \smash{\widetilde{\D}} _{\PP ^\sharp} ^{(\bullet)})
 \to 
 \smash{\underrightarrow{LD}} ^{\mathrm{b}} _{\Q} ( \smash{\widetilde{\D}} _{\PP ^\sharp} ^{(\bullet)})$
 comme ci-dessus, cependant on se restreindra exclusivement aux complexes au moins quasi-cohérents pour éviter cette pathologie.
 \end{rema}

\begin{vide}
\label{GammaT-oub-sharp}
On déduit de \ref{nota-oub-sharp}
que le foncteur 
$ \R \underline{\Gamma} ^\dag _{T}$
ne dépend pas de la log-structure, i.e., 
on dispose de l'isomorphisme canonique 
$$\R \underline{\Gamma} ^\dag _{T}\circ \mathrm{oub} _{\sharp} 
\riso \mathrm{oub} _{\sharp} \circ  \R \underline{\Gamma} ^\dag _{T}$$ 
de foncteurs de 
$ \smash{\underrightarrow{LD}} ^\mathrm{b} _{\Q,\mathrm{qc}} ( \smash{\widetilde{\D}} _{\PP } ^{(\bullet)})
\to 
\smash{\underrightarrow{LD}} ^\mathrm{b} _{\Q,\mathrm{qc}} ( \smash{\widetilde{\D}} _{\PP ^\sharp} ^{(\bullet)})$. 

\end{vide}

\subsection{Commutations et compatibilités des foncteurs locaux et de localisation dans le cas d'un diviseur}

\begin{vide}
[Commutation des foncteurs locaux et de localisation au produit tensoriel]
\label{iso-comm-locaux-prod-tens}
Soient $\E ^{(\bullet)},~\FF ^{(\bullet)}
\in \smash{\underrightarrow{LD}} ^\mathrm{b} _{\Q,\mathrm{qc}} ( \smash{\widetilde{\D}} _{\PP ^\sharp} ^{(\bullet)})$.
Comme le foncteur $(\hdag T)$ commute canoniquement au produit tensoriel, 
il existe alors un unique isomorphisme de la forme
$\R \underline{\Gamma} ^\dag _{T} 
(\E ^{(\bullet)}
\smash{\overset{\L}{\otimes}}   ^{\dag}
_{\O  _{\PP, \Q}} 
\FF ^{(\bullet)})
\riso 
\R \underline{\Gamma} ^\dag _{T}( \E ^{(\bullet)})
\smash{\overset{\L}{\otimes}}   ^{\dag}
_{\O  _{\PP, \Q}} 
\FF ^{(\bullet)}$
(resp. $\R \underline{\Gamma} ^\dag _{T} 
(\E ^{(\bullet)}
\smash{\overset{\L}{\otimes}}   ^{\dag}
_{\O  _{\PP, \Q}} 
\FF ^{(\bullet)})
\riso 
\E ^{(\bullet)}
\smash{\overset{\L}{\otimes}}   ^{\dag}
_{\O  _{\PP, \Q}} 
\R \underline{\Gamma} ^\dag _{T}
(\FF ^{(\bullet)})$)
et s'inscrivant dans le diagramme commutatif
\begin{equation}
\label{fonct-hdagXbis}
\xymatrix @=0,4cm{
{\R \underline{\Gamma} ^\dag _{T} 
(\E ^{(\bullet)})
\smash{\overset{\L}{\otimes}}   ^{\dag}
_{\O  _{\PP, \Q}} 
\FF ^{(\bullet)}} 
\ar[r] ^-{}
& 
{\E ^{(\bullet)}
\smash{\overset{\L}{\otimes}}   ^{\dag}
_{\O  _{\PP, \Q}} 
\FF ^{(\bullet)} } 
\ar@{=}[d] ^-{}
\ar[r] ^-{}
& 
{(\hdag T) (\E ^{(\bullet)})
\smash{\overset{\L}{\otimes}}   ^{\dag}
_{\O  _{\PP, \Q}} 
\FF ^{(\bullet)} } 
\ar[r] ^-{}
& 
{\R \underline{\Gamma} ^\dag _{T} 
(\E ^{(\bullet)} )
\smash{\overset{\L}{\otimes}}   ^{\dag}
_{\O  _{\PP, \Q}} 
\FF ^{(\bullet)} [1]}
\\
{\R \underline{\Gamma} ^\dag _{T} 
(\E ^{(\bullet)}
\smash{\overset{\L}{\otimes}}   ^{\dag}
_{\O  _{\PP, \Q}} 
\FF ^{(\bullet)})} 
\ar[r] ^-{}
\ar@{.>}[d] ^-{\exists !}
\ar@{.>}[u] ^-{\exists !}
& 
{\E ^{(\bullet)}
\smash{\overset{\L}{\otimes}}   ^{\dag}
_{\O  _{\PP, \Q}} 
\FF ^{(\bullet)} } 
\ar@{=}[d] ^-{}
\ar[r] ^-{}
& 
{(\hdag T) (\E ^{(\bullet)}
\smash{\overset{\L}{\otimes}}   ^{\dag}
_{\O  _{\PP, \Q}} 
\FF ^{(\bullet)})  } 
\ar[d] ^-{\sim}
\ar[r] ^-{}
\ar[u] ^-{\sim}
& 
{\R \underline{\Gamma} ^\dag _{T} 
(\E ^{(\bullet)}
\smash{\overset{\L}{\otimes}}   ^{\dag}
_{\O  _{\PP, \Q}} 
\FF ^{(\bullet)}) [1]}
\ar@{.>}[d] ^-{\exists !}
\ar@{.>}[u] ^-{\exists !}
\\ 
 {\E ^{(\bullet)}
\smash{\overset{\L}{\otimes}}   ^{\dag}
_{\O  _{\PP, \Q}} 
\R \underline{\Gamma} ^\dag _{T}
(\FF ^{(\bullet)})}
\ar[r] ^-{}
& 
{\E ^{(\bullet)}
\smash{\overset{\L}{\otimes}}   ^{\dag}
_{\O  _{\PP, \Q}} 
\FF ^{(\bullet)} } 
\ar[r] ^-{}
& 
{ \E ^{(\bullet)}
\smash{\overset{\L}{\otimes}}   ^{\dag}
_{\O  _{\PP, \Q}} 
(\hdag T) (\FF ^{(\bullet)})  }\ar[r] ^-{}
& 
{\E ^{(\bullet)}
\smash{\overset{\L}{\otimes}}   ^{\dag}
_{\O  _{\PP, \Q}} 
\R \underline{\Gamma} ^\dag _{T}
(\FF ^{(\bullet)})[1].} 
}
\end{equation}
Ces isomorphismes sont fonctoriel en $\E ^{(\bullet)},~\FF ^{(\bullet)},~ T$ (pour la signification de la fonctorialité en $T$, on pourra par exemple
regarder \ref{fonct-hdagX2}).
Pour vérifier la fonctorialité en $\E ^{(\bullet)},~\FF ^{(\bullet)},~ T$, il s'agit d'écrire les diagramme correspondants en trois dimensions, 
ce qui est laissé au lecteur.

\end{vide}

\begin{vide}
[Commutation des foncteurs locaux et de localisation entre eux]
\label{iso-comm-locaux}
Soient $T _1, T _2$ des diviseurs de $P$, 
$\E ^{(\bullet)}
\in \smash{\underrightarrow{LD}} ^\mathrm{b} _{\Q,\mathrm{qc}} ( \smash{\widetilde{\D}} _{\PP ^\sharp} ^{(\bullet)})$.

$\bullet$ Par commutativité du produit tensoriel, on dispose alors 
de l'isomorphisme canonique 
\begin{equation}
\label{commhdagT1T2}
(\hdag T _2) \circ (\hdag T _1) (\E ^{(\bullet)})
\riso 
(\hdag T _1) \circ (\hdag T _2) (\E ^{(\bullet)})
\end{equation}
fonctoriel en $T _1$, $T _2$ et $\E ^{(\bullet)}$.

$\bullet$ Il existe alors un unique isomorphisme 
$(\hdag T _2) \circ \R \underline{\Gamma} ^\dag _{T _1}(\E ^{(\bullet)} )
\riso 
\R \underline{\Gamma} ^\dag _{T _1}
\circ
(\hdag T _2)(\E ^{(\bullet)} )$
induisant le morphisme canonique de triangles distingués 
$ (\hdag T _2) ( \Delta _{T _1} (\E ^{(\bullet)} ))
\to 
\Delta _{T _1} ((\hdag T _2) (\E ^{(\bullet)} ))$,
i.e. de la forme:
\begin{equation}
\label{commGammahdag}
\xymatrix @=0,3cm{
{(\hdag T _2) \circ \R \underline{\Gamma} ^\dag _{T _1} (\E ^{(\bullet)}) }
\ar[r] ^-{}
\ar@{.>}[d] ^-{\exists !}
& 
{(\hdag T _2) (\E ^{(\bullet)} )} 
\ar[r] ^-{}
\ar@{=}[d] ^-{}
& 
{(\hdag T _2) \circ (\hdag T _1)(\E ^{(\bullet)} )} 
\ar[d] ^-{\sim}
\ar[r] ^-{}
& 
{(\hdag T _2) \circ \R \underline{\Gamma} ^\dag _{T _1} (\E ^{(\bullet)})[1]} 
\ar@{.>}[d] ^-{\exists !}
\\ 
{\R \underline{\Gamma} ^\dag _{T _1} \circ (\hdag T _2) (\E ^{(\bullet)})}  
\ar[r] ^-{}
& 
{(\hdag T _2) (\E ^{(\bullet)} ) } 
\ar[r] ^-{}
& 
{(\hdag T _1) \circ (\hdag T _2) (\E ^{(\bullet)})} 
\ar[r] ^-{}
& 
{\R \underline{\Gamma} ^\dag _{T _1} \circ (\hdag T _2) (\E ^{(\bullet)})[1],} 
}
\end{equation}
dont le carré du milieu est bien commutatif par fonctorialité en $T _1$ de l'isomorphisme 
\ref{commhdagT1T2}.
En écrivant les diagrammes en trois dimensions (i.e., on écrit le parallélépipède dont la face de devant est \ref{commGammahdag},
la face de derrière est \ref{commGammahdag} avec $T' _1$ remplaçant $T _1$, les morphismes de l'avant vers l'arrière sont les morphismes de fonctorialité induit par $T _1\subset T '_1$ ; idem pour valider la fonctorialité en $T _2$ ou $\E ^{(\bullet)}$), 
on vérifie que l'isomorphisme 
$(\hdag T _2) \circ \R \underline{\Gamma} ^\dag _{T _1}(\E ^{(\bullet)} )
\riso 
\R \underline{\Gamma} ^\dag _{T _1}
\circ
(\hdag T _2)(\E ^{(\bullet)} )$
est fonctoriel en $T _1$, $T _2$, $\E ^{(\bullet)}$.

$\bullet$ De même, il existe alors un unique isomorphisme 
$\R \underline{\Gamma} ^\dag _{T _2} \circ \R \underline{\Gamma} ^\dag _{T _1} (\E ^{(\bullet)})
\riso 
\R \underline{\Gamma} ^\dag _{T _1}\circ \R \underline{\Gamma} ^\dag _{T _2} (\E ^{(\bullet)})$
fonctoriel en $T _1$, $T _2$, $\E ^{(\bullet)}$
induisant le morphisme canonique de triangles distingués 
$\Delta _{T _2} (\R \underline{\Gamma} ^\dag _{T _1} (\E ^{(\bullet)} ))
\to 
\R \underline{\Gamma} ^\dag _{T _1} ( \Delta _{T _2} (\E ^{(\bullet)} ))$,
i.e. de la forme:
\begin{equation}
\label{commGammaT1T2}
\xymatrix @=0,4cm{
{\R \underline{\Gamma} ^\dag _{T _2} \circ \R \underline{\Gamma} ^\dag _{T _1} (\E ^{(\bullet)})} 
\ar[r] ^-{}
\ar@{.>}[d] ^-{\exists !}
& 
{\R \underline{\Gamma} ^\dag _{T _1} (\E ^{(\bullet)} )} 
\ar[r] ^-{}
\ar@{=}[d] ^-{}
& 
{(\hdag T _2) \circ \R \underline{\Gamma} ^\dag _{T _1}(\E ^{(\bullet)} )} 
\ar[d] ^-{\sim}
\ar[r] ^-{}
& 
{\R \underline{\Gamma} ^\dag _{T _2} \circ \R \underline{\Gamma} ^\dag _{T _1} (\E ^{(\bullet)})[1]} 
\ar@{.>}[d] ^-{\exists !}
\\ 
{\R \underline{\Gamma} ^\dag _{T _1} \circ \R \underline{\Gamma} ^\dag _{T _2} (\E ^{(\bullet)}) } 
\ar[r] ^-{}
& 
{\R \underline{\Gamma} ^\dag _{T _1} (\E ^{(\bullet)} ) } 
\ar[r] ^-{}
& 
{\R \underline{\Gamma} ^\dag _{T _1} \circ (\hdag T _2) (\E ^{(\bullet)})} 
\ar[r] ^-{}
& 
{(\hdag T _1) \circ \R \underline{\Gamma} ^\dag _{T _2} (\E ^{(\bullet)})[1],} 
}
\end{equation}
dont le carré du milieu est bien commutatif par fonctorialité en $T _2$ de l'isomorphisme 
$(\hdag T _2) \circ \R \underline{\Gamma} ^\dag _{T _1}(\E ^{(\bullet)} )
\riso 
\R \underline{\Gamma} ^\dag _{T _1}
\circ
(\hdag T _2)(\E ^{(\bullet)} )$.

\end{vide}

\begin{vide}
[Compatibilité des isomorphismes de commutation au produit tensoriel]
\label{comp-prod-tens}

Soient $T, T _1, T _2$ des diviseurs de $P$, 
$\E ^{(\bullet)},~\FF ^{(\bullet)}
\in \smash{\underrightarrow{LD}} ^\mathrm{b} _{\Q,\mathrm{qc}} ( \smash{\widetilde{\D}} _{\PP ^\sharp} ^{(\bullet)})$.
En composant 
les isomorphismes verticaux du diagramme \ref{fonct-hdagXbis}, 
on obtient les isomorphismes canoniques 
$(\hdag T)  (\E ^{(\bullet)})
\smash{\overset{\L}{\otimes}}   ^{\dag}
_{\O  _{\PP, \Q}} 
\FF ^{(\bullet)}
\riso 
\E ^{(\bullet)}
\smash{\overset{\L}{\otimes}}   ^{\dag}
_{\O  _{\PP, \Q}} 
(\hdag T)  (\FF ^{(\bullet)})$ 
et
$\R \underline{\Gamma} ^\dag _{T}  (\E ^{(\bullet)})
\smash{\overset{\L}{\otimes}}   ^{\dag}
_{\O  _{\PP, \Q}} 
\FF ^{(\bullet)}
\riso 
\E ^{(\bullet)}
\smash{\overset{\L}{\otimes}}   ^{\dag}
_{\O  _{\PP, \Q}} 
\R \underline{\Gamma} ^\dag _{T}  (\FF ^{(\bullet)})$ 
fonctoriels en 
$T,~\E ^{(\bullet)},~\FF ^{(\bullet)}$.
Nous vérifions dans ce paragraphe que ces isomorphismes 
sont compatibles avec 
les trois isomorphismes de commutation des foncteurs locaux et de localisation de \ref{iso-comm-locaux}.

$\bullet$ 
La compatibilité avec les isomorphismes \ref{commhdagT1T2} signifie que le diagramme 
\begin{equation}
\label{comp-prod-tens-hdag-hdag}
\xymatrix @=0,4cm{
{(\hdag T _2) \circ (\hdag T _1) (\E ^{(\bullet)})
\smash{\overset{\L}{\otimes}}   ^{\dag}
_{\O  _{\PP, \Q}} 
\FF ^{(\bullet)}} 
\ar[r] ^-{\sim}
\ar[d] ^-{\sim} _-{\ref{commhdagT1T2}}
& 
{(\hdag T _1) (\E ^{(\bullet)})
\smash{\overset{\L}{\otimes}}   ^{\dag}
_{\O  _{\PP, \Q}} 
(\hdag T _2)  (\FF ^{(\bullet)})} 
\ar[r] ^-{\sim}
& 
{\E ^{(\bullet)}
\smash{\overset{\L}{\otimes}}   ^{\dag}
_{\O  _{\PP, \Q}} 
(\hdag T _1) \circ (\hdag T _2) (\FF ^{(\bullet)})} 
\\ 
{(\hdag T _1) \circ (\hdag T _2) (\E ^{(\bullet)})
\smash{\overset{\L}{\otimes}}   ^{\dag}
_{\O  _{\PP, \Q}} 
\FF ^{(\bullet)}  } 
\ar[r] ^-{\sim}
&
{(\hdag T _2) (\E ^{(\bullet)})
\smash{\overset{\L}{\otimes}}   ^{\dag}
_{\O  _{\PP, \Q}} 
(\hdag T _1)  (\FF ^{(\bullet)})} \ar[r] ^-{\sim}
& 
{\E ^{(\bullet)}
\smash{\overset{\L}{\otimes}}   ^{\dag}
_{\O  _{\PP, \Q}} 
(\hdag T _2) \circ (\hdag T _1) (\FF ^{(\bullet)}),} 
\ar[u] ^-{\sim} _-{\ref{commhdagT1T2}}
} 
\end{equation}
est commutatif, ce qui est aisé.

$\bullet$ Vérifions à présent que le diagramme
\begin{equation}
\label{comp-prod-tens-hdag-gamma}
\xymatrix @=0,4cm{ 
{(\hdag T _2) \circ \R \underline{\Gamma} ^\dag _{T _1}  (\E ^{(\bullet)})
\smash{\overset{\L}{\otimes}}   ^{\dag}
_{\O  _{\PP, \Q}} 
\FF ^{(\bullet)}  } 
\ar[r] ^-{\sim}
\ar[d] ^-{\sim} 
&
{\R \underline{\Gamma} ^\dag _{T _1}  (\E ^{(\bullet)})
\smash{\overset{\L}{\otimes}}   ^{\dag}
_{\O  _{\PP, \Q}} 
(\hdag T _2)  (\FF ^{(\bullet)})} \ar[r] ^-{\sim}
& 
{\E ^{(\bullet)}
\smash{\overset{\L}{\otimes}}   ^{\dag}
_{\O  _{\PP, \Q}} 
\R \underline{\Gamma} ^\dag _{T _1}  \circ (\hdag T _2) (\FF ^{(\bullet)})} 
\\
{\R \underline{\Gamma} ^\dag _{T _1}  \circ (\hdag T _2) (\E ^{(\bullet)})
\smash{\overset{\L}{\otimes}}   ^{\dag}
_{\O  _{\PP, \Q}} 
\FF ^{(\bullet)}} 
\ar[r] ^-{\sim}
& 
{(\hdag T _2) (\E ^{(\bullet)})
\smash{\overset{\L}{\otimes}}   ^{\dag}
_{\O  _{\PP, \Q}} 
\R \underline{\Gamma} ^\dag _{T _1}   (\FF ^{(\bullet)})} 
\ar[r] ^-{\sim}
& 
{\E ^{(\bullet)}
\smash{\overset{\L}{\otimes}}   ^{\dag}
_{\O  _{\PP, \Q}} 
(\hdag T _2) \circ \R \underline{\Gamma} ^\dag _{T _1}  (\FF ^{(\bullet)}),}
\ar[u] ^-{\sim} 
} 
\end{equation}
est commutatif.
Pour cela considérons le diagramme:
\begin{equation}
\label{pre-3morphDelta}
\xymatrix @=0,4cm{
{(\hdag T _2) \circ \R \underline{\Gamma} ^\dag _{T _1} (\E ^{(\bullet)}) 
\smash{\overset{\L}{\otimes}}   ^{\dag}
_{\O  _{\PP, \Q}} 
\FF ^{(\bullet)}}
\ar[r] ^-{}
\ar[d] ^-{}
& 
{(\hdag T _2) (\E ^{(\bullet)} )
\smash{\overset{\L}{\otimes}}   ^{\dag}
_{\O  _{\PP, \Q}} 
\FF ^{(\bullet)}} 
\ar[r] ^-{}
\ar[d] ^-{}
& 
{(\hdag T _2) \circ (\hdag T _1)(\E ^{(\bullet)} )
\smash{\overset{\L}{\otimes}}   ^{\dag}
_{\O  _{\PP, \Q}} 
\FF ^{(\bullet)}} 
\ar[d] ^-{\sim}
\ar[r] ^-{}
& 
{+1} 
\\ 
{\E ^{(\bullet)}
\smash{\overset{\L}{\otimes}}   ^{\dag}
_{\O  _{\PP, \Q}} 
\R \underline{\Gamma} ^\dag _{T _1} \circ (\hdag T _2)  (\FF ^{(\bullet)})}  
\ar[r] ^-{}
& 
{\E ^{(\bullet)}
\smash{\overset{\L}{\otimes}}   ^{\dag}
_{\O  _{\PP, \Q}} 
 (\hdag T _2)  (\FF ^{(\bullet)})}  
 \ar[r] ^-{}
& 
{\E ^{(\bullet)}
\smash{\overset{\L}{\otimes}}   ^{\dag}
_{\O  _{\PP, \Q}} 
(\hdag T _1) \circ (\hdag T _2)  (\FF ^{(\bullet)})}  \ar[r] ^-{}
& 
{+1,}
} 
\end{equation}
dont la flèche verticale de droite (resp. de gauche) 
est le composé du haut de \ref{comp-prod-tens-hdag-hdag} (resp. \ref{comp-prod-tens-hdag-gamma}).
Comme ces morphismes sont fonctoriels en $T _1$, 
le carré de gauche (resp de droite) du diagramme \ref{pre-3morphDelta}
est commutatif. Ce diagramme \ref{pre-3morphDelta} correspond donc 
à un morphisme de triangles distingués de la forme
$(\hdag T _2) (  \Delta _{T _1} (\E ^{(\bullet)} ) )
\smash{\overset{\L}{\otimes}}   ^{\dag}
_{\O  _{\PP, \Q}} 
\FF ^{(\bullet)}
\to 
\E ^{(\bullet)}
\smash{\overset{\L}{\otimes}}   ^{\dag}
_{\O  _{\PP, \Q}} 
\Delta _{T _1} ((\hdag T _2)(\E ^{(\bullet)} ) )$.
De la même manière, on vérifie la commutativité des carrés du milieu du diagramme
\begin{equation}
\label{3morphDelta}
\xymatrix @=0,4cm{
{(\hdag T _2) \circ \R \underline{\Gamma} ^\dag _{T _1} (\E ^{(\bullet)}) 
\smash{\overset{\L}{\otimes}}   ^{\dag}
_{\O  _{\PP, \Q}} 
\FF ^{(\bullet)}}
\ar[r] ^-{}
\ar[d] ^-{}
& 
{(\hdag T _2) (\E ^{(\bullet)} )
\smash{\overset{\L}{\otimes}}   ^{\dag}
_{\O  _{\PP, \Q}} 
\FF ^{(\bullet)}} 
\ar[r] ^-{}
\ar@{=}[d] ^-{}
& 
{(\hdag T _2) \circ (\hdag T _1)(\E ^{(\bullet)} )
\smash{\overset{\L}{\otimes}}   ^{\dag}
_{\O  _{\PP, \Q}} 
\FF ^{(\bullet)}} 
\ar[d] ^-{\sim}
\ar[r] ^-{}
& 
{+1} 
\\ 
{\R \underline{\Gamma} ^\dag _{T _1} \circ (\hdag T _2) (\E ^{(\bullet)})
\smash{\overset{\L}{\otimes}}   ^{\dag}
_{\O  _{\PP, \Q}} 
\FF ^{(\bullet)}}  
\ar[r] ^-{}
\ar[d] ^-{}
& 
{(\hdag T _2) (\E ^{(\bullet)} ) 
\smash{\overset{\L}{\otimes}}   ^{\dag}
_{\O  _{\PP, \Q}} 
\FF ^{(\bullet)}} 
\ar[r] ^-{}
\ar[d] ^-{}
& 
{(\hdag T _1) \circ (\hdag T _2) (\E ^{(\bullet)}) \smash{\overset{\L}{\otimes}}   ^{\dag}
_{\O  _{\PP, \Q}} 
\FF ^{(\bullet)}} 
\ar[r] ^-{}
\ar[d] ^-{}
& 
{+1}
\\
{\E ^{(\bullet)}
\smash{\overset{\L}{\otimes}}   ^{\dag}
_{\O  _{\PP, \Q}} 
(\hdag T _2) \circ \R \underline{\Gamma} ^\dag _{T _1}  (\FF ^{(\bullet)})}  
\ar[r] ^-{}
\ar[d] ^-{}
& 
{\E ^{(\bullet)}
\smash{\overset{\L}{\otimes}}   ^{\dag}
_{\O  _{\PP, \Q}} 
(\hdag T _2)  (\FF ^{(\bullet)})}  \ar[r] ^-{}
\ar@{=}[d] ^-{}
& 
{\E ^{(\bullet)}
\smash{\overset{\L}{\otimes}}   ^{\dag}
_{\O  _{\PP, \Q}} 
(\hdag T _2) \circ (\hdag T _1) (\FF ^{(\bullet)})}  
\ar[r] ^-{}
\ar[d] ^-{}
& 
{+1}
\\
{\E ^{(\bullet)}
\smash{\overset{\L}{\otimes}}   ^{\dag}
_{\O  _{\PP, \Q}} 
\R \underline{\Gamma} ^\dag _{T _1} \circ (\hdag T _2)  (\FF ^{(\bullet)})}  
\ar[r] ^-{}
& 
{\E ^{(\bullet)}
\smash{\overset{\L}{\otimes}}   ^{\dag}
_{\O  _{\PP, \Q}} 
 (\hdag T _2)  (\FF ^{(\bullet)})}  
 \ar[r] ^-{}
& 
{\E ^{(\bullet)}
\smash{\overset{\L}{\otimes}}   ^{\dag}
_{\O  _{\PP, \Q}} 
(\hdag T _1) \circ (\hdag T _2)  (\FF ^{(\bullet)})}  \ar[r] ^-{}
& 
{+1,}
} 
\end{equation}
dont le morphisme de triangle du haut (resp. du bas) est l'image par le foncteur 
$- \smash{\overset{\L}{\otimes}}   ^{\dag}
_{\O  _{\PP, \Q}} 
\FF ^{(\bullet)}$ 
(resp. 
$\E ^{(\bullet)}
\smash{\overset{\L}{\otimes}}   ^{\dag}
_{\O  _{\PP, \Q}}-$)
du morphisme de triangles distingués \ref{commGammahdag}.
La commutativité du diagramme \ref{comp-prod-tens-hdag-hdag} signifie
que la flèche composée verticale de droite (resp. du milieu) du diagramme \ref{3morphDelta}
est la flèche verticale de droite (resp. du milieu) de \ref{pre-3morphDelta}.
Par unicité (grâce à \cite[1.1.9]{BBD}),
il en est de même des flèches verticales de gauche.  
D'où le résultat.

$\bullet$ En calquant la preuve de la commutativité de
\ref{comp-prod-tens-hdag-gamma} à partir de \ref{comp-prod-tens-hdag-hdag}, 
on vérifie à partir de la commutativité de  \ref{comp-prod-tens-hdag-gamma} celle du diagramme ci-dessous :
\begin{equation}
\label{comp-prod-tens-gamma-gamma}
\xymatrix @=0,4cm{ 
{\R \underline{\Gamma} ^\dag _{T _2}  \circ \R \underline{\Gamma} ^\dag _{T _1}  (\E ^{(\bullet)})
\smash{\overset{\L}{\otimes}}   ^{\dag}
_{\O  _{\PP, \Q}} 
\FF ^{(\bullet)}  } 
\ar[r] ^-{\sim}
\ar[d] ^-{\sim} 
&
{\R \underline{\Gamma} ^\dag _{T _1}  (\E ^{(\bullet)})
\smash{\overset{\L}{\otimes}}   ^{\dag}
_{\O  _{\PP, \Q}} 
\R \underline{\Gamma} ^\dag _{T _2}  
(\FF ^{(\bullet)})} \ar[r] ^-{\sim}
& 
{\E ^{(\bullet)}
\smash{\overset{\L}{\otimes}}   ^{\dag}
_{\O  _{\PP, \Q}} 
\R \underline{\Gamma} ^\dag _{T _1}  \circ 
\R \underline{\Gamma} ^\dag _{T _2}   (\FF ^{(\bullet)})} 
\\
{\R \underline{\Gamma} ^\dag _{T _1}  \circ 
\R \underline{\Gamma} ^\dag _{T _2}  (\E ^{(\bullet)})
\smash{\overset{\L}{\otimes}}   ^{\dag}
_{\O  _{\PP, \Q}} 
\FF ^{(\bullet)}} 
\ar[r] ^-{\sim}
& 
{\R \underline{\Gamma} ^\dag _{T _2}   (\E ^{(\bullet)})
\smash{\overset{\L}{\otimes}}   ^{\dag}
_{\O  _{\PP, \Q}} 
\R \underline{\Gamma} ^\dag _{T _1}   (\FF ^{(\bullet)})} 
\ar[r] ^-{\sim}
& 
{\E ^{(\bullet)}
\smash{\overset{\L}{\otimes}}   ^{\dag}
_{\O  _{\PP, \Q}} 
\R \underline{\Gamma} ^\dag _{T _2}   \circ \R \underline{\Gamma} ^\dag _{T _1}  (\FF ^{(\bullet)}).}
\ar[u] ^-{\sim} 
} 
\end{equation}

\end{vide}

\begin{vide}
\label{comp-prod-tensbis}

Soient $T$ un diviseur $P$, 
$\E ^{(\bullet)},~\FF ^{(\bullet)}
\in \smash{\underrightarrow{LD}} ^\mathrm{b} _{\Q,\mathrm{qc}} ( \smash{\widetilde{\D}} _{\PP ^\sharp} ^{(\bullet)})$.
On dispose 
des isomorphismes 
$(\hdag T)  (\E ^{(\bullet)}
\smash{\overset{\L}{\otimes}}   ^{\dag}
_{\O  _{\PP, \Q}} 
\FF ^{(\bullet)})
\riso 
\E ^{(\bullet)}
\smash{\overset{\L}{\otimes}}   ^{\dag}
_{\O  _{\PP, \Q}} 
(\hdag T)  (\FF ^{(\bullet)})$,
et
$\R \underline{\Gamma} ^\dag _{T}  (\E ^{(\bullet)}
\smash{\overset{\L}{\otimes}}   ^{\dag}
_{\O  _{\PP, \Q}} 
\FF ^{(\bullet)}) 
\riso 
 \E ^{(\bullet)}
\smash{\overset{\L}{\otimes}}   ^{\dag}
_{\O  _{\PP, \Q}} 
\R \underline{\Gamma} ^\dag _{T}  (\FF ^{(\bullet)})$ 
fonctoriels en 
$T,~\E ^{(\bullet)},~\FF ^{(\bullet)}$.
De manière identique à \ref{comp-prod-tens} (il s'agit de changer la position des parenthèses)
on vérifie que ces isomorphismes 
sont compatibles avec 
les trois isomorphismes de commutation des foncteurs locaux et de  localisation de \ref{iso-comm-locaux}.
Il en résulte par composition qu'il en est de même 
des isomorphismes
$(\hdag T)  (\E ^{(\bullet)}
\smash{\overset{\L}{\otimes}}   ^{\dag}
_{\O  _{\PP, \Q}} 
\FF ^{(\bullet)})
\riso 
(\hdag T)  (\E ^{(\bullet)})
\smash{\overset{\L}{\otimes}}   ^{\dag}
_{\O  _{\PP, \Q}} 
\FF ^{(\bullet)}$
et
$\R \underline{\Gamma} ^\dag _{T}  (\E ^{(\bullet)}
\smash{\overset{\L}{\otimes}}   ^{\dag}
_{\O  _{\PP, \Q}} 
\FF ^{(\bullet)}) 
\riso 
\R \underline{\Gamma} ^\dag _{T}  ( \E ^{(\bullet)})
\smash{\overset{\L}{\otimes}}   ^{\dag}
_{\O  _{\PP, \Q}} 
\FF ^{(\bullet)}$.

\end{vide}

\begin{vide}
Soient $T _{1},\dots, T _{r}$ et 
$T ' _{1},\dots, T' _{s}$ des diviseurs de $P$, 
$\E ^{(\bullet)},~\FF ^{(\bullet)}
\in \smash{\underrightarrow{LD}} ^\mathrm{b} _{\Q,\mathrm{qc}} ( \smash{\widetilde{\D}} _{\PP ^\sharp} ^{(\bullet)})$.
En utilisant $s$-fois le diagramme \ref{fonct-hdagXbis},
on vérifie que l'on dispose des isomorphismes canoniques  
fonctoriels en 
$T _{1},\dots, T _{r}$,
$T ' _{1},\dots, T' _{s}$,
$\E ^{(\bullet)},~\FF ^{(\bullet)}$ de la forme
\begin{gather}
\notag
\R \underline{\Gamma} ^\dag _{T '_1} \circ \cdots \circ 
\R \underline{\Gamma} ^\dag _{T '_s}
\circ 
\R \underline{\Gamma} ^\dag _{T _1} \circ \cdots \circ 
\R \underline{\Gamma} ^\dag _{T _r}
(\E ^{(\bullet)}
\smash{\overset{\L}{\otimes}}   ^{\dag}
_{\O  _{\PP, \Q}} \FF ^{(\bullet)} )
\riso
\R \underline{\Gamma} ^\dag _{T '_1} \circ \cdots \circ 
\R \underline{\Gamma} ^\dag _{T '_s}
(\E ^{(\bullet)}
\smash{\overset{\L}{\otimes}}   ^{\dag}
_{\O  _{\PP, \Q}} 
\R \underline{\Gamma} ^\dag _{T  _1} \circ \cdots \circ 
\R \underline{\Gamma} ^\dag _{T  _{r}}
(\FF ^{(\bullet)} ))
\\
\label{fonctXX'Gammadiag1}
\riso
\R \underline{\Gamma} ^\dag _{T '_1} \circ \cdots \circ 
\R \underline{\Gamma} ^\dag _{T '_s}
(\E ^{(\bullet)}) 
\smash{\overset{\L}{\otimes}}   ^{\dag}
_{\O  _{\PP, \Q}} 
\R \underline{\Gamma} ^\dag _{T  _1} \circ \cdots \circ 
\R \underline{\Gamma} ^\dag _{T  _{r}}
(\FF ^{(\bullet)} ),
\end{gather}
ceux-ci ne dépendant pas, à isomorphisme canonique près, 
de l'ordre des $T _1,\dots, T _r$ ou des $T ' _1,\dots, T ' _s$. 
En particulier, en prenant 
$\FF ^{(\bullet)}  = \O _{\PP} ^{(\bullet)} $ et les 
$T ' _1,\dots, T ' _s$ vides, on obtient l'isomorphisme canonique fonctoriel en 
$T _{1},\dots, T _{r}$,
$\E ^{(\bullet)}$ de la forme: 
\begin{equation}
\label{4.4.5.2}
\R \underline{\Gamma} ^\dag _{T _1} \circ \cdots \circ 
\R \underline{\Gamma} ^\dag _{T _r}
(\E ^{(\bullet)} )
\riso
\E ^{(\bullet)}
\smash{\overset{\L}{\otimes}}   ^{\dag}
_{\O  _{\PP, \Q}} 
\R \underline{\Gamma} ^\dag _{T  _1} \circ \cdots \circ 
\R \underline{\Gamma} ^\dag _{T  _{r}}
(\O _{\PP} ^{(\bullet)} )
\end{equation}
et qui ne dépend pas, à isomorphisme canonique près, de l'ordre des 
$T _1,\dots, T _r$.
\end{vide}

\subsection{Foncteur cohomologique local à support strict dans un sous-schéma fermé}

\begin{rema}
Soient $X$ une sous-variété fermée de $P$, 
$\U$ est l'ouvert de $\PP$ complémentaire de $X$, 
$D$ un diviseur de $X$,
$\E ^{(\bullet)}
\in 
\smash{\underrightarrow{LD}} ^{\mathrm{b}} _{\Q,\mathrm{coh}} ( \smash{\widetilde{\D}} _{\PP ^\sharp} ^{(\bullet)} (D))$.
Alors $\E ^{(\bullet)}|\U \riso 0$
dans 
$\smash{\underrightarrow{LD}} ^{\mathrm{b}} _{\Q,\mathrm{coh}} ( \smash{\widetilde{\D}} _{\U ^{\sharp}} ^{(\bullet)} (D\cap U))$
si et seulement si les espaces de cohomologie de 
$\underrightarrow{\lim} 
\E ^{(\bullet)}$ sont des $\smash{\D} ^\dag _{\PP ^\sharp} (\hdag D) _{\Q}$-cohérents à support dans $X$
au sens classique.
\end{rema}

\begin{lemm}
\label{lemme2.2.3}
Soient $D,~T$ deux diviseurs de $P$, 
$\E ^{(\bullet)}
\in 
\smash{\underrightarrow{LD}} ^{\mathrm{b}} _{\Q,\mathrm{coh}} ( \smash{\widetilde{\D}} _{\PP ^\sharp} ^{(\bullet)} (D))$,
$\U$ l'ouvert de $\PP$ complémentaire du support de $T$.
Les assertions suivantes sont équivalentes:
\begin{enumerate}
\item On dispose dans $\smash{\underrightarrow{LD}} ^{\mathrm{b}} _{\Q,\mathrm{coh}} ( \smash{\widetilde{\D}} _{\U ^{\sharp}} ^{(\bullet)} (D\cap U))$
de l'isomorphisme $\E ^{(\bullet)}|\U \riso 0$.
\item Le morphisme canonique 
$\R \underline{\Gamma} ^\dag _{T} (\E ^{(\bullet)})
\to 
\E ^{(\bullet)}$ est un isomorphisme dans 
$\smash{\underrightarrow{LD}} ^{\mathrm{b}} _{\Q} ( \smash{\widetilde{\D}} _{\PP ^\sharp} ^{(\bullet)} (D))$.
\item On dispose dans $\smash{\underrightarrow{LD}} ^{\mathrm{b}} _{\Q,\mathrm{coh}} ( \smash{\widetilde{\D}} _{\PP ^\sharp} ^{(\bullet)} (D))$
de l'isomorphisme $(\hdag T) (\E ^{(\bullet)} )\riso 0$.
\end{enumerate}
\end{lemm}

\begin{proof}
L'équivalence entre $2$ et $3$ est tautologique (i.e., on dispose du triangle distingué \ref{tri-loc-berthelot}). 
L'assertion $3 \Rightarrow 1$ est triviale. 
Réciproquement, établissons $1 \Rightarrow 3$. 
Comme 
$(\hdag T) (\E ^{(\bullet)} ) 
\riso 
(\hdag T)\circ (\hdag D) (\E ^{(\bullet)} ) 
\riso 
(\hdag T \cup D) (\E ^{(\bullet)} ) 
\in 
\smash{\underrightarrow{LD}} ^{\mathrm{b}} _{\Q,\mathrm{coh}} ( \smash{\widetilde{\D}} _{\PP ^\sharp} ^{(\bullet)} (D \cup T))$, 
comme le foncteur
$\underrightarrow{\lim}$
est pleinement fidèle sur
$\smash{\underrightarrow{LD}} ^{\mathrm{b}} _{\Q,\mathrm{coh}} ( \smash{\widetilde{\D}} _{\PP ^\sharp} ^{(\bullet)} (D \cup T))$,
il est alors équivalent de prouver que 
$\underrightarrow{\lim}  (\hdag T) (\E ^{(\bullet)} ) 
\riso 0$. 
Comme $\underrightarrow{\lim} (\hdag T) (\E ^{(\bullet)} ) $
est un $\smash{\D} ^\dag _{\PP ^\sharp} (\hdag T\cup D) _{\Q}$-cohérent à support dans 
$T$ (et donc dans $T \cup D$), ce dernier est bien nul
grâce à \cite[4.8]{caro_log-iso-hol}.
\end{proof}

\begin{lemm}
\label{induction-div-coh}
Soient $T _1, \dots, T _{r}$ des diviseurs de $P$.
Alors
$\R \underline{\Gamma} ^\dag _{T _r} \circ \cdots \circ 
\R \underline{\Gamma} ^\dag _{T _1} (\smash{\widetilde{\B}} _{\PP} ^{(\bullet)} (T  ))
\in 
\smash{\underrightarrow{LD}} ^{\mathrm{b}} _{\Q,\mathrm{coh}} ( \smash{\widehat{\D}} _{\PP } ^{(\bullet)})$.
\end{lemm}

\begin{proof}
Comme d'après \ref{hdagT'T=cup} on dispose de l'isomorphisme canonique
$(\hdag T _1) ( \smash{\widetilde{\B}} _{\PP} ^{(\bullet)} (T) )
\riso 
\smash{\widetilde{\B}} _{\PP} ^{(\bullet)} (T _1 \cup T) $, 
on obtient le triangle distingué :
\begin{equation}
\label{tri-distT1TB}
\R \underline{\Gamma} ^\dag _{T _1} (\smash{\widetilde{\B}} _{\PP} ^{(\bullet)} (T  ))
\to 
\smash{\widetilde{\B}} _{\PP} ^{(\bullet)} (T ) 
\to 
\smash{\widetilde{\B}} _{\PP} ^{(\bullet)} (T _1 \cup T ) 
\to
\R \underline{\Gamma} ^\dag _{T _1} (\smash{\widetilde{\B}} _{\PP} ^{(\bullet)} (T  )) [1].
\end{equation}
D'après \ref{coh-Bbullet}, 
il en résulte que le triangle distingué \ref{tri-distT1TB}
est un triangle dans
$\smash{\underrightarrow{LD}} ^{\mathrm{b}} _{\Q,\mathrm{coh}} ( \smash{\widehat{\D}} _{\PP} ^{(\bullet)})$.
On en déduit par récurrence sur $r$ (on applique 
le foncteur 
$\R \underline{\Gamma} ^\dag _{T _r} \circ \cdots \circ 
\R \underline{\Gamma} ^\dag _{T _2}$ au triangle distingué
\ref{tri-distT1TB}) le résultat. 

\end{proof}

\begin{vide}
\label{dfn-4.3.4}
Soit $X$ un sous-schéma fermé réduit de $P$.
Comme pour \cite[2.2]{caro_surcoherent}, on définit
le foncteur cohomologique local $\R \underline{\Gamma} ^\dag _{X}$ à support strict dans $X$.
Puisque $P$ est somme de ses composantes irréductibles, il suffit de le définir dans le cas où $P$ est intègre. 
\begin{enumerate}
\item Si $X= P$, alors le foncteur $\R \underline{\Gamma} ^\dag _{X}$ 
est par définition l'identité. 

\item Supposons à présent $X \not = P$. 
D'après \cite[2.2.5]{caro_surcoherent} (il y a une coquille: il faut rajouter l'hypothèse {\og $P$ est intègre\fg})
$X$ est une intersection finie de diviseur de $P$.
Précisons que l'on s'était ramené au cas où $X$ est irréductible en utilisant la formule
de commutation des réunions et intersections de la forme
$\cup _{i=1} ^{r} \cap _{j=1} ^{s} D _{i,j} =
\cap _{1\leq j _1, \dots, j _r\leq s}
\cup _{i=1} ^{r} D _{i,j _i},$
pour une famille de sous-ensembles (e.g. des supports de diviseurs)
$D _{i,j}$ de $P$.
Si 
$T _1, \dots, T _{r}$ sont des diviseurs de $P$ tels que 
$X = \cap _{i =1} ^{r} T _i$, 
alors, pour tout $\E ^{(\bullet)}
\in \smash{\underrightarrow{LD}} ^\mathrm{b} _{\Q,\mathrm{qc}} ( \smash{\widetilde{\D}} _{\PP ^\sharp} ^{(\bullet)})$,
 le complexe
$\R \underline{\Gamma} ^\dag _{X} (\E ^{(\bullet)} ):=
\R \underline{\Gamma} ^\dag _{T _r} \circ \cdots \circ 
\R \underline{\Gamma} ^\dag _{T _1} (\E ^{(\bullet)} )$
ne dépend canoniquement pas du choix de tels diviseurs $T _1,\dots, T _r$.
En effet, grâce à \ref{4.4.5.2}, 
on se ramène au cas où $\E ^{(\bullet)} = \O ^{(\bullet)} _{\PP}$ et par conséquent au cas où $\ZZ$ est vide
;
on vérifie ensuite que 
cela ne dépend pas de l'ordre des diviseurs $T _1, \dots, T _r$
;
puis, grâce aux lemmes \ref{lemme2.2.3} et \ref{induction-div-coh},
il est superflu de rajouter d'autres diviseurs contenant $ X$.

Pour toutes sous-variétés fermées $X$, $X'$ de $P$, 
on dispose alors par construction (de même que \cite[2.2.8]{caro_surcoherent})
de l'isomorphisme canonique:
\begin{equation}
\label{theo2.2.8}
\R \underline{\Gamma} ^\dag _{X} \circ \R \underline{\Gamma} ^\dag _{X'} 
(\E ^{(\bullet)} )
\riso
\R \underline{\Gamma} ^\dag _{X\cap X'}  (\E ^{(\bullet)} ).
\end{equation}

\end{enumerate}

\end{vide}

\begin{vide}
\label{GammaX-oub-sharp}
Soit $X$ un sous-schéma fermé réduit de $P$.
On déduit de \ref{GammaT-oub-sharp}
que le foncteur 
$ \R \underline{\Gamma} ^\dag _{X}$
ne dépend pas de la log-structure, i.e., 
on dispose de l'isomorphisme canonique 
\begin{equation}
\label{GammaX-oub-sharp-iso}
\R \underline{\Gamma} ^\dag _{X}\circ \mathrm{oub} _{\sharp} 
\riso \mathrm{oub} _{\sharp} \circ  \R \underline{\Gamma} ^\dag _{X}
\end{equation}
de foncteurs de 
$ \smash{\underrightarrow{LD}} ^\mathrm{b} _{\Q,\mathrm{qc}} ( \smash{\widetilde{\D}} _{\PP } ^{(\bullet)})
\to 
\smash{\underrightarrow{LD}} ^\mathrm{b} _{\Q,\mathrm{qc}} ( \smash{\widetilde{\D}} _{\PP ^\sharp} ^{(\bullet)})$. 
\end{vide}

\begin{prop}
\label{fonctXX'Gamma}
 Soient $X,~X'$ deux sous-schémas fermés de $P$, 
$\E ^{(\bullet)},~\FF ^{(\bullet)}
\in \smash{\underrightarrow{LD}} ^\mathrm{b} _{\Q,\mathrm{qc}} ( \smash{\widetilde{\D}} _{\PP ^\sharp} ^{(\bullet)})$.
On dispose de l'isomorphisme canonique fonctoriel en $\E ^{(\bullet)},~\FF ^{(\bullet)},~ X,~X'$ :
\begin{equation}
\label{fonctXX'Gamma-iso}
\R \underline{\Gamma} ^\dag _{X \cap X'} (\E ^{(\bullet)}
\smash{\overset{\L}{\otimes}}   ^{\dag}
_{\O  _{\PP, \Q}} \FF ^{(\bullet)} )
\riso 
\R \underline{\Gamma} ^\dag _{X} 
(\E ^{(\bullet)})
\smash{\overset{\L}{\otimes}}   ^{\dag}
_{\O  _{\PP, \Q}} 
\R \underline{\Gamma} ^\dag _{X'}
(\FF ^{(\bullet)}).
\end{equation}
\end{prop}

\begin{proof}
Cela résulte des isomorphismes canoniques \ref{fonctXX'Gammadiag1}.
\end{proof}

\subsection{Foncteur de localisation en dehors d'un sous-schéma fermé}
\begin{lemm}
\label{annulationHom}
Soient $X \subset X'$ deux sous-variétés fermées de $P$,
$\E ^{(\bullet)}, 
\FF ^{(\bullet)}
\in \smash{\underrightarrow{LD}} ^\mathrm{b} _{\Q,\mathrm{qc}} ( \smash{\widetilde{\D}} _{\PP ^\sharp} ^{(\bullet)})$.
On suppose de plus que l'on dispose 
dans 
$\smash{\underrightarrow{LD}} ^\mathrm{b} _{\Q,\mathrm{qc}} ( \smash{\widetilde{\D}} _{\PP ^\sharp} ^{(\bullet)})$
de l'isomorphisme
$\R \underline{\Gamma} ^\dag _{X '} (\FF ^{(\bullet)} )\riso 0$.
Alors 
$\mathrm{Hom} _{\smash{\underrightarrow{LD}}  _{\Q} ( \smash{\widetilde{\D}} _{\PP ^\sharp} ^{(\bullet)})}
(\R \underline{\Gamma} ^\dag _{X} (\E ^{(\bullet)} ), \FF ^{(\bullet)}) =0.$
\end{lemm}

\begin{proof}
Soit $\phi \colon 
\R \underline{\Gamma} ^\dag _{X} (\E ^{(\bullet)} )\to  \FF ^{(\bullet)}$
un morphisme de $\smash{\underrightarrow{LD}} ^{\mathrm{b}} _{\Q,\mathrm{qc}} ( \smash{\widetilde{\D}} _{\PP ^\sharp} ^{(\bullet)})$. 
Comme le morphisme canonique 
$\R \underline{\Gamma} ^\dag _{X} (\R \underline{\Gamma} ^\dag _{X} (\E ^{(\bullet)} ))
\to
\R \underline{\Gamma} ^\dag _{X} (\E ^{(\bullet)} )$
est un isomorphisme de $\smash{\underrightarrow{LD}} ^{\mathrm{b}} _{\Q,\mathrm{qc}} ( \smash{\widetilde{\D}} _{\PP ^\sharp} ^{(\bullet)})$(voir \ref{theo2.2.8}),
le morphisme $\phi$ se factorise canoniquement par 
$\R \underline{\Gamma} ^\dag _{X} (\phi )$.
Or, comme 
$\R \underline{\Gamma} ^\dag _{X} (\FF ^{(\bullet)} )\liso 
\R \underline{\Gamma} ^\dag _{X} \circ \R \underline{\Gamma} ^\dag _{X '} (\FF ^{(\bullet)} )\riso 0$, 
alors 
$\R \underline{\Gamma} ^\dag _{X} (\phi )\riso0$ dans 
$\smash{\underrightarrow{LD}} ^{\mathrm{b}} _{\Q,\mathrm{qc}} ( \smash{\widetilde{\D}} _{\PP ^\sharp} ^{(\bullet)})$.
On en déduit que $\phi =0$.
D'où le résultat. 
\end{proof}

\begin{lemm}
\label{lemm-GammaXEtoEtoF}
Soient $X$ une sous-variété fermée de $P$
et un triangle distingué dans 
$\smash{\underrightarrow{LD}} ^\mathrm{b} _{\Q,\mathrm{qc}} ( \smash{\widetilde{\D}} _{\PP ^\sharp} ^{(\bullet)})$
de la forme
\begin{equation}
\label{GammaXEtoEtoF}
\R \underline{\Gamma} ^\dag _{X} (\E ^{(\bullet)} )
\to 
\E ^{(\bullet)} 
\to 
 \FF ^{(\bullet)}
 \to 
\R \underline{\Gamma} ^\dag _{X} (\E ^{(\bullet)} )[1] ,
\end{equation}
où le premier morphisme est le morphisme canonique. 
On dispose alors de l'isomorphisme 
$\R \underline{\Gamma} ^\dag _{X} (\FF ^{(\bullet)} )\riso 0$
dans 
$\smash{\underrightarrow{LD}} ^\mathrm{b} _{\Q,\mathrm{qc}} ( \smash{\widetilde{\D}} _{\PP ^\sharp} ^{(\bullet)})$.
\end{lemm}

\begin{proof}
Comme le morphisme canonique 
$\R \underline{\Gamma} ^\dag _{X} (\R \underline{\Gamma} ^\dag _{X} (\E ^{(\bullet)} ))
\to
\R \underline{\Gamma} ^\dag _{X} (\E ^{(\bullet)} )$
est un isomorphisme (voir \ref{theo2.2.8}),
en appliquant le foncteur 
$\R \underline{\Gamma} ^\dag _{X}$ au triangle distingué \ref{GammaXEtoEtoF},
l'un des axiomes sur les catégories triangulées nous permet de conclure. 
\end{proof}

\begin{vide}
[Foncteur localisation en dehors d'une sous-variété fermée $X$]
\label{hdagX}
Soient $X \subset X'$ deux sous-variétés de $P$.
Supposons donné le diagramme commutatif dans 
$\smash{\underrightarrow{LD}} ^\mathrm{b} _{\Q,\mathrm{qc}} ( \smash{\widetilde{\D}} _{\PP ^\sharp} ^{(\bullet)})$
de la forme
\begin{equation}
\label{prefonct-hdagX}
\xymatrix @=0,4cm{
{\R \underline{\Gamma} ^\dag _{X} (\E ^{(\bullet)} )} 
\ar[r] ^-{}
\ar[d] ^-{\R \underline{\Gamma} ^\dag _{X}(\phi)}
& 
{\E ^{(\bullet)}  } 
\ar[r] ^-{}
\ar[d] ^-{\phi}
& 
{\FF ^{(\bullet)} } 
\ar[r] ^-{}
&
{\R \underline{\Gamma} ^\dag _{X} (\E ^{(\bullet)} )[1]} 
\ar[d] ^-{\R \underline{\Gamma} ^\dag _{X}(\phi)}
\\ 
{\R \underline{\Gamma} ^\dag _{X'} (\E ^{\prime (\bullet)} )} 
\ar[r] ^-{}
& 
{\E ^{\prime (\bullet)}  } 
\ar[r] ^-{}
& 
{\FF ^{\prime (\bullet)}  } 
\ar[r] ^-{}
&
{\R \underline{\Gamma} ^\dag _{X'} (\E ^{\prime (\bullet)} )[1]} 
}
\end{equation}
dont les flèches horizontales de gauche sont les morphismes canoniques
et dont les deux triangles horizontaux sont distingués. 
D'après les lemmes \ref{annulationHom} et \ref{lemm-GammaXEtoEtoF}, on obtient alors 
$$H ^{-1} (\R \mathrm{Hom} _{D (\underrightarrow{LM} _{\Q} (\smash{\widetilde{\D}} _{\PP ^\sharp} ^{(\bullet)} ))}
(\R \underline{\Gamma} ^\dag _{X} (\E ^{(\bullet)} ), \FF ^{\prime (\bullet)} ) )
=
\mathrm{Hom} _{D (\underrightarrow{LM} _{\Q} (\smash{\widetilde{\D}} _{\PP ^\sharp} ^{(\bullet)} ))}
(\R \underline{\Gamma} ^\dag _{X} (\E ^{(\bullet)} ), \FF ^{\prime (\bullet)}[-1]) =0.$$
On en déduit, grâce à 
\cite[1.1.9]{BBD},
qu'il existe donc un unique morphisme 
$\FF ^{(\bullet)}\to \FF ^{\prime (\bullet)} $
induisant dans 
$\smash{\underrightarrow{LD}} ^\mathrm{b} _{\Q,\mathrm{qc}} ( \smash{\widetilde{\D}} _{\PP ^\sharp} ^{(\bullet)})$
le diagramme commutatif:
\begin{equation}
\label{fonct-hdagX}
\xymatrix @=0,4cm{
{\R \underline{\Gamma} ^\dag _{X} (\E ^{(\bullet)} )} 
\ar[r] ^-{}
\ar[d] ^-{\R \underline{\Gamma} ^\dag _{X}(\phi)}
& 
{\E ^{(\bullet)}  } 
\ar[r] ^-{}
\ar[d] ^-{\phi}
& 
{\FF ^{(\bullet)}  } 
\ar[r] ^-{}
\ar@{.>}[d] ^-{\exists !}
& 
{\R \underline{\Gamma} ^\dag _{X} (\E ^{(\bullet)} )[1]} 
\ar[d] ^-{\R \underline{\Gamma} ^\dag _{X}(\phi)}
\\ 
{\R \underline{\Gamma} ^\dag _{X'} (\E ^{\prime (\bullet)} )} 
\ar[r] ^-{}
& 
{\E ^{\prime (\bullet)}  } 
\ar[r] ^-{}
& 
{\FF ^{\prime (\bullet)}  } 
\ar[r] ^-{}
& 
{\R \underline{\Gamma} ^\dag _{X'} (\E ^{\prime (\bullet)} )[1].} 
}
\end{equation}
Comme pour \cite[1.1.10]{BBD}, 
cela implique que le cône de 
$\R \underline{\Gamma} ^\dag _{X} (\E ^{(\bullet)} )
\to  
\E ^{(\bullet)} $
est unique à isomorphisme canonique près. 
On le notera 
$(\hdag X) (\E ^{(\bullet)} )$.
On vérifie de plus que 
$(\hdag X) (\E ^{(\bullet)} )$
est fonctoriel en $X$ et $\E ^{(\bullet)} $, e.g.
on dispose
du morphisme fonctoriel en $\E ^{(\bullet)} $ de forme
$\E ^{(\bullet)} \to  (\hdag X) (\E ^{(\bullet)} )$.
 
On déduit de \ref{GammaX-oub-sharp}
que le foncteur 
$(\hdag X) $
ne dépend pas de la log-structure, i.e., 
on dispose de l'isomorphisme canonique 
$(\hdag X) \circ \mathrm{oub} _{\sharp} 
\riso \mathrm{oub} _{\sharp} \circ (\hdag X) $ de foncteurs de 
$ \smash{\underrightarrow{LD}} ^\mathrm{b} _{\Q,\mathrm{qc}} ( \smash{\widetilde{\D}} _{\PP } ^{(\bullet)})
\to 
\smash{\underrightarrow{LD}} ^\mathrm{b} _{\Q,\mathrm{qc}} ( \smash{\widetilde{\D}} _{\PP ^\sharp} ^{(\bullet)})$. 

\end{vide}

\begin{vide}
Pour toute
sous-variété fermée $X$ de $P$, 
pour tous
$\E ^{(\bullet)},~\FF ^{(\bullet)}
\in \smash{\underrightarrow{LD}} ^\mathrm{b} _{\Q,\mathrm{qc}} ( \smash{\widetilde{\D}} _{\PP ^\sharp} ^{(\bullet)})$,
il existe un unique isomorphisme de la forme
$(\hdag X) (\E ^{(\bullet)}
\smash{\overset{\L}{\otimes}}   ^{\dag}
_{\O  _{\PP, \Q}} 
\FF ^{(\bullet)}) 
\riso
\E ^{(\bullet)}
\smash{\overset{\L}{\otimes}}   ^{\dag}
_{\O  _{\PP, \Q}} 
(\hdag X) (\FF ^{(\bullet)}) $
s'inscrivant dans le diagramme commutatif
\begin{equation}
\notag
\xymatrix @=0,4cm{
{\R \underline{\Gamma} ^\dag _{X} 
(\E ^{(\bullet)}
\smash{\overset{\L}{\otimes}}   ^{\dag}
_{\O  _{\PP, \Q}} 
\FF ^{(\bullet)})} 
\ar[d] ^-{\sim} _-{\ref{fonctXX'Gamma-iso}}
\ar[r] ^-{}
& 
{\E ^{(\bullet)}
\smash{\overset{\L}{\otimes}}   ^{\dag}
_{\O  _{\PP, \Q}} 
\FF ^{(\bullet)} } 
\ar@{=}[d] ^-{}
\ar[r] ^-{}
& 
{(\hdag X) (\E ^{(\bullet)}
\smash{\overset{\L}{\otimes}}   ^{\dag}
_{\O  _{\PP, \Q}} 
\FF ^{(\bullet)})  } 
 \ar@{.>}[d] ^-{\exists !} 
\ar[r] ^-{}
& 
{\R \underline{\Gamma} ^\dag _{X} 
(\E ^{(\bullet)}
\smash{\overset{\L}{\otimes}}   ^{\dag}
_{\O  _{\PP, \Q}} 
\FF ^{(\bullet)}) [1]}
\ar[d] ^-{\sim} _-{\ref{fonctXX'Gamma-iso}}
\\ 
 {\E ^{(\bullet)}
\smash{\overset{\L}{\otimes}}   ^{\dag}
_{\O  _{\PP, \Q}} 
\R \underline{\Gamma} ^\dag _{X}
(\FF ^{(\bullet)})}
\ar[r] ^-{}
& 
{\E ^{(\bullet)}
\smash{\overset{\L}{\otimes}}   ^{\dag}
_{\O  _{\PP, \Q}} 
\FF ^{(\bullet)} } 
\ar[r] ^-{}
& 
{ \E ^{(\bullet)}
\smash{\overset{\L}{\otimes}}   ^{\dag}
_{\O  _{\PP, \Q}} 
(\hdag X) (\FF ^{(\bullet)})  }\ar[r] ^-{}
& 
{\E ^{(\bullet)}
\smash{\overset{\L}{\otimes}}   ^{\dag}
_{\O  _{\PP, \Q}} 
\R \underline{\Gamma} ^\dag _{X}
(\FF ^{(\bullet)})[1].} 
}
\end{equation}
Comme d'habitude (en écrivant des parallélépipèdes), on vérifie qu'il est fonctoriel en 
$X,~\E ^{(\bullet)},~\FF ^{(\bullet)}$.
\end{vide}

\begin{vide}
Soient $X, X'$ deux sous-variétés fermées de $P$,
$\E ^{(\bullet)}
\in \smash{\underrightarrow{LD}} ^\mathrm{b} _{\Q,\mathrm{qc}} ( \smash{\widetilde{\D}} _{\PP ^\sharp} ^{(\bullet)})$.
Il existe alors un unique isomorphisme 
$(\hdag X') \circ \R \underline{\Gamma} ^\dag _{X}(\E ^{(\bullet)} )
\riso
\R \underline{\Gamma} ^\dag _{X} \circ (\hdag X') (\E ^{(\bullet)})$
fonctoriel en $X,~X ',~\E ^{(\bullet)}$
s'incrivant dans le diagramme commutatif de la forme
\begin{equation}
\notag
\xymatrix @=0,4cm{
{\R \underline{\Gamma} ^\dag _{X'} \circ \R \underline{\Gamma} ^\dag _{X} (\E ^{(\bullet)})} 
\ar[r] ^-{}
\ar[d] ^-{\sim}
& 
{\R \underline{\Gamma} ^\dag _{X} (\E ^{(\bullet)} )} 
\ar[r] ^-{}
\ar@{=}[d] ^-{}
& 
{(\hdag X') \circ \R \underline{\Gamma} ^\dag _{X}(\E ^{(\bullet)} )} 
\ar@{.>}[d] ^-{\exists !}
\ar[r] ^-{}
& 
{\R \underline{\Gamma} ^\dag _{X'} \circ \R \underline{\Gamma} ^\dag _{X} (\E ^{(\bullet)})[1]} 
\ar[d] ^-{\sim}
\\ 
{\R \underline{\Gamma} ^\dag _{X} \circ \R \underline{\Gamma} ^\dag _{X'} (\E ^{(\bullet)}) } 
\ar[r] ^-{}
& 
{\R \underline{\Gamma} ^\dag _{X} (\E ^{(\bullet)} ) } 
\ar[r] ^-{}
& 
{\R \underline{\Gamma} ^\dag _{X} \circ (\hdag X') (\E ^{(\bullet)})} 
\ar[r] ^-{}
& 
{\R \underline{\Gamma} ^\dag _{X} \circ \R \underline{\Gamma} ^\dag _{X'} (\E ^{(\bullet)})[1],} 
}
\end{equation}
De manière identique à \cite[2.2.14]{caro_surcoherent} (cette fois-ci, il n'y a aucune précision à apporter), 
on vérifie que l'on dispose de l'isomorphisme canonique
\begin{equation}
\label{2.2.14-surcoh}
(\hdag X) \circ (\hdag X') (\E ^{(\bullet)})
\riso
(\hdag X \cup X')(\E ^{(\bullet)}),
\end{equation}
fonctoriel en $X,~X ',~\E ^{(\bullet)}$.
On établit comme pour \cite[2.2.16]{caro_surcoherent},
que l'on bénéficie des triangles distingués de localisation de
Mayer-Vietoris
\begin{gather}\label{eq1mayer-vietoris}
  \R \underline{\Gamma} ^\dag _{X \cap X'}(\E ^{(\bullet)}) \rightarrow
  \R \underline{\Gamma} ^\dag _{X }(\E ^{(\bullet)}) \oplus
\R \underline{\Gamma} ^\dag _{X ' }(\E ^{(\bullet)})  \rightarrow
\R \underline{\Gamma} ^\dag _{X \cup X '}(\E^{(\bullet)} ) \rightarrow
\R \underline{\Gamma} ^\dag _{X \cap X'}(\E^{(\bullet)} )[1],\\
 (\hdag X \cap X')(\E ) \rightarrow  (\hdag X )(\E^{(\bullet)} ) \oplus
 (\hdag X ') (\E^{(\bullet)} )  \rightarrow   (\hdag X \cup X ')(\E^{(\bullet)} ) \rightarrow (\hdag X \cap X')(\E ^{(\bullet)})[1].
\end{gather}
\end{vide}

Nous ajouterons dans le prochain chapitre (voir \ref{2.2.18}) 
un détail sur la vérification de la commutativité des foncteurs locaux et de localisation par rapport aux images inverses extraordinaires
de \cite[2.2.18]{caro_surcoherent}.
Finissons enfin le chapitre par une extension du lemme \ref{lemme2.2.3}:
 \begin{prop}
\label{prop2.2.9}
Soient $D$ un diviseur de $P$, 
$X$ une sous-variété fermée de $P$,
$\U$ l'ouvert de $\PP$ complémentaire du support de $X$,
$\E ^{(\bullet)}
\in 
\smash{\underrightarrow{LD}} ^{\mathrm{b}} _{\Q,\mathrm{coh}} ( \smash{\widetilde{\D}} _{\PP ^\sharp} ^{(\bullet)} (D))$.
Les assertions suivantes sont équivalentes:
\begin{enumerate}
\item On dispose dans $\smash{\underrightarrow{LD}} ^{\mathrm{b}} _{\Q,\mathrm{coh}} ( \smash{\widetilde{\D}} _{\U ^{\sharp}} ^{(\bullet)} (D\cap U))$
de l'isomorphisme $\E ^{(\bullet)}|\U \riso 0$.
\item Le morphisme canonique 
$\R \underline{\Gamma} ^\dag _{X} (\E ^{(\bullet)})
\to 
\E ^{(\bullet)}$ est un isomorphisme dans 
$\smash{\underrightarrow{LD}} ^{\mathrm{b}} _{\Q} ( \smash{\widetilde{\D}} _{\PP ^\sharp} ^{(\bullet)} (D))$.
\item On dispose dans $\smash{\underrightarrow{LD}} ^{\mathrm{b}} _{\Q,\mathrm{coh}} ( \smash{\widetilde{\D}} _{\PP ^\sharp} ^{(\bullet)} (D))$
de l'isomorphisme $(\hdag X) (\E ^{(\bullet)} )\riso 0$.
\end{enumerate}
\end{prop}

\begin{proof}
L'équivalence entre $2$ et $3$ est tautologique (voir la définition du foncteur $(\hdag X)$ donnée dans \ref{hdagX}). 
L'assertion $3 \Rightarrow 1$ est triviale. 
Il reste à établir l'implication $1 \Rightarrow 2$. 
Soit $T$ un diviseur contenant $X$. 
Par construction du foncteur 
$\R \underline{\Gamma} ^\dag _{X}$, 
il suffit de vérifier que le morphisme $\R \underline{\Gamma} ^\dag _{T} (\E ^{(\bullet)})
\to 
\E ^{(\bullet)}$ est un isomorphisme, ce qui résulte du lemme \ref{lemme2.2.3}.
\end{proof}

\section{Surcohérence}

\subsection{Image inverse extraordinaire, image directe}

Soient $f \colon \PP' \to \PP$ un morphisme de $\V$-schémas formels lisses, 
$T$ et $T'$ des diviseurs respectifs de $P$ et $P'$ tels que
$f ( P '\setminus T' ) \subset P \setminus T$, 
$\ZZ$ et $\ZZ'$ des diviseurs à croisements normaux stricts de 
respectivement $\PP$ et $\PP'$ tels que $f ^{-1} (\ZZ) \subset \ZZ'$.
On pose $\PP ^{\sharp}:= (\PP, \ZZ)$, $\PP ^{\prime \sharp}:= (\PP', \ZZ')$
et
$f ^{\sharp}
\colon 
\PP ^{\prime \sharp}
\to 
\PP ^{\sharp}$ le morphisme de $\V$-schémas formels logarithmiques induits par $f$. 
Nous rappelons dans cette section les constructions des
images inverses extraordinaires et 
des images directes par $f ^{\sharp}$ 
à singularités surconvergentes le long de $T$ et $T'$.
Nous donnons de plus quelques premières propriétés.

\begin{nota}
Nous aurons besoin de quelques notations préliminaires afin
de définir l'image inverse extraordinaire par $f ^{\sharp}$
à singularités surconvergentes le long de $T$ et $T'$.
\begin{itemize}
\item Comme $f ^{-1} (T) \subset T'$, 
on dispose alors du morphisme canonique
$f ^{-1} \widetilde{\B} _{P _i} ^{(m)} ( T ) \rightarrow \widetilde{\B} _{P' _i} ^{(m)} ( T ')$.
Il en résulte que 
le faisceau 
$\widetilde{\B} _{P' _i} ^{(m)} ( T ') \otimes _{\O _{P ^\prime _i}} f _i ^* \D _{P ^{\sharp} _i} ^{(m)}
\riso 
\widetilde{\B} _{P' _i} ^{(m)} ( T ') \otimes _{f ^{-1} \widetilde{\B} _{P _i} ^{(m)} ( T ) } f ^{-1} \widetilde{\D} _{P ^{\sharp} _i} ^{(m)} (T)$
est muni d'une structure canonique
de ($ \widetilde{\D} _{P  ^{\prime \sharp} _i} ^{(m)} ( T ')$, $ f  ^{-1} \widetilde{\D} _{P ^{\sharp} _i} ^{(m)}(T)$)-bimodule.
Ce bimodule sera noté $\widetilde{\D} ^{(m)} _{P ^{\prime \sharp} _i \rightarrow P ^{\sharp} _i} ( T' , T)$.

\item Il en dérive par complétion $p$-adique  le
$(\smash{\widetilde{\D}} _{\PP ^{\prime \sharp}} ^{(m)}(T') , f ^{-1} \smash{\widetilde{\D}} _{\PP ^{\sharp}} ^{(m)} (T))$-bimodule :
$\smash{\widetilde{\D}} _{\PP ^{\prime \sharp}\rightarrow \PP ^{\sharp}} ^{(m)} ( T' , T):=
\underset{\underset{i}{\longleftarrow}}{\lim}\, \widetilde{\D}_{P ^{\prime \sharp} _i \rightarrow P ^{\sharp} _i} ^{(m)}( T' , T)$.

\item On obtient un $(\D ^{\dag } _{\PP ^{\prime \sharp}} (\hdag T' )_{\Q}, f ^{-1} \D ^{\dag }_{\PP ^{ \sharp}} (\hdag T) _{\Q})$-bimodule
en posant 
$\D ^{\dag} _{\PP ^{\prime \sharp}\rightarrow \PP ^{ \sharp}} (\hdag T' ,T) _{\Q}:=\underset{\underset{m}{\longrightarrow}}{\lim}\,
\smash{\widetilde{\D}} _{\PP ^{\prime \sharp}\rightarrow \PP ^{ \sharp}} ^{(m)} ( T' ,T)_{\Q}$.

\end{itemize}

\end{nota}

\begin{vide}
$\bullet$ On bénéficie du foncteur image inverse extraordinaire par $f$ à singularités surconvergentes le long de $T$ et $T'$
de la forme
$f  ^{\sharp(\bullet)!} _{T',T} \colon
\smash{\underrightarrow{LD}} ^{\mathrm{b}} _{\Q,\mathrm{qc}} ( \smash{\widetilde{\D}} _{\PP ^\sharp} ^{(\bullet)}(T))
\to
\smash{\underrightarrow{LD}} ^{\mathrm{b}} _{\Q,\mathrm{qc}} ( \smash{\widetilde{\D}} _{\PP ^{\prime \sharp}} ^{(\bullet)}(T'))$
en posant, 
pour tout $\E ^{(\bullet)} \in \smash{\underrightarrow{LD}}  ^\mathrm{b} _{\Q, \mathrm{qc}}
( \smash{\widetilde{\D}} _{\PP ^{\sharp}} ^{(\bullet)} (T ))$:
\begin{equation}
\notag
  f _{T',T} ^{\sharp (\bullet)!} ( \E ^{(\bullet)}) :=
\smash{\widetilde{\D}} ^{(\bullet)} _{\PP ^{\prime \sharp} \rightarrow \PP ^{\sharp}} (T',T)
\smash{\widehat{\otimes}} ^\L _{f ^{-1} \smash{\widetilde{\D}} ^{(\bullet)} _{\PP ^{\sharp}} (T)}
f ^{-1} \E ^{(\bullet)} [ d_{\PP ' /\PP}].
\end{equation}
Par abus de notation (ce n'est pas un foncteur exact à droite), 
on pose 
$\L f _{T',T} ^{\sharp (\bullet)*} ( \E ^{(\bullet)}) :=  f _{T',T} ^{\sharp (\bullet)!} ( \E ^{(\bullet)}) [ -d_{\PP ' /\PP}]$.

$\bullet$ On dispose du foncteur image inverse extraordinaire par $f$ à singularités surconvergentes le long de $T$ et $T'$
de la forme
$f  ^{\sharp !} _{T',T} \colon
D ^\mathrm{b} _\mathrm{coh} ( \smash{\D} ^\dag _{\PP ^\sharp} (\hdag T) _{\Q} )
\to 
 D ^\mathrm{b} ( \smash{\D} ^\dag _{\PP ^{\prime \sharp}} (\hdag T') _{\Q} )$
en posant
(comme pour \cite[4.3.2]{Beintro2}), 
pour tout $\E \in D ^{\mathrm{b}} _{\mathrm{coh}} ( \D ^{\dag} _{\PP ^{\sharp} } (\hdag T ) _{\Q})$ :
\begin{equation}
\label{def-image-inv-extr}
f  ^{\sharp !} _{T' , T} (\E ):=\D ^{\dag} _{\PP ^{\prime \sharp}\rightarrow \PP ^{\sharp} }  ( \hdag T' , T ) _{\Q}
\otimes ^{\L} _{ f ^{-1} \D ^{\dag} _{\PP ^{\sharp} } (\hdag T ) _{\Q}} f ^{-1} \E [d _{\PP '/\PP} ].
\end{equation}
\end{vide}

\begin{vide}
\label{f!oubsharp}
Avec les notations de \ref{nota-oub-sharp}, 
pour tout $\E ^{(\bullet)} \in \smash{\underrightarrow{LD}}  ^\mathrm{b} _{\Q, \mathrm{qc}}
(\widetilde{\D} _{\PP } ^{(\bullet)} (T ))$
on vérifie que le morphisme canonique
$$f _{T',T} ^{\sharp (\bullet)!} 
\circ \mathrm{oub} _{\sharp}
(\E ^{(\bullet)})
\to 
\mathrm{oub} _{\sharp} 
\circ 
f _{T',T} ^{ (\bullet)!} 
(\E ^{(\bullet)})$$
est un isomorphisme
(en effet, au niveau des schémas, ces foncteurs sont tous les deux 
au décalage près le foncteur dérivé gauche de l'image inverse en tant que $\O$-modules composé avec le foncteur de localisation).  
\end{vide}

\begin{nota}
Nous aurons besoin
pour définir l'image directe par $f$
à singularités surconvergentes le long de $T$ et $T'$ de considérer les bimodules ci-dessous.

\begin{itemize}
\item Le ($f ^{-1} \widetilde{\D} _{P ^{\sharp} _i} ^{(m)}(T)$, $ \widetilde{\D} _{P ^{\prime \sharp} _i} ^{(m)} ( T ')$)-bimodule
$\widetilde{\B} _{P' _i} ^{(m)} ( T ') \otimes _{\O _{P ^\prime _i}}
(\omega _{P ^{\prime \sharp} _i} \otimes _{\O _{P ' _i}}f ^* _g ( \D _{P ^{\sharp} _i} ^{(m)}(T) \otimes _{\O _{P _i}} \omega ^{-1} _{P ^{\sharp} _i} ))$,
où l'indice $g$ signifie que l'on choisit la structure gauche de $\D _{P _i} ^{(m)} $-module à gauche,
sera noté $\widetilde{\D} ^{(m)} _{P ^{\sharp} _i \leftarrow P ^{\prime \sharp  }_i} ( T , T')$.

\item On dispose alors du
$(f ^{-1} \smash{\widetilde{\D}} _{\PP ^{\sharp}} ^{(m)} (T),~\smash{\widetilde{\D}} _{\PP ^{\prime \sharp}} ^{(m)}(T') )$-bimodule :
$\smash{\widetilde{\D}} _{\PP ^{\sharp} \leftarrow \PP ^{\prime \sharp}} ^{(m)} ( T, T'):=
\underset{\underset{i}{\longleftarrow}}{\lim}\, \widetilde{\D}_{P ^{\sharp} _i \leftarrow P ^{\prime \sharp} _i} ^{(m)}( T , T')$.

\item D'où le
($f ^{-1} \D ^{\dag }_{\PP ^{\sharp} } (\hdag T) _{\Q}$, $\D ^{\dag } _{\PP ^{\prime \sharp}} (\hdag T' )_{\Q}$)-bimodule
$\D ^{\dag} _{\PP ^{\sharp} \leftarrow \PP ^{\prime \sharp}} (\hdag T ,T ')_{\Q}:=
\underset{\underset{m}{\longrightarrow}}{\lim}
\smash{\widetilde{\D}} _{\PP ^{\sharp} \leftarrow \PP ^{\prime \sharp}} ^{(m)} ( T, T') _\Q$.

\end{itemize}

\end{nota}

\begin{vide}
$\bullet$ On dispose du foncteur 
image directe par $f$ à singularités surconvergentes le long de $T$ et $T'$ de la forme
$f  ^{\sharp (\bullet)} _{T,T',+} \colon
\smash{\underrightarrow{LD}} ^{\mathrm{b}} _{\Q,\mathrm{qc}} ( \smash{\widetilde{\D}} _{\PP ^{\prime \sharp}} ^{(\bullet)}(T'))
\to 
\smash{\underrightarrow{LD}} ^{\mathrm{b}} _{\Q,\mathrm{qc}} ( \smash{\widetilde{\D}} _{\PP ^\sharp} ^{(\bullet)}(T))$
en posant, 
pour tout $\E ^{\prime (\bullet)} \in \smash{\underrightarrow{LD}}  ^\mathrm{b} _{\Q, \mathrm{qc}}
( \smash{\widetilde{\D}} _{\PP ^{\prime \sharp}} ^{(\bullet)} (T '))$:
\begin{gather}\notag
f _{T,T',+} ^{\sharp (\bullet)} ( \E ^{\prime (\bullet)} ):= 
\R f _* (
\smash{\widetilde{\D}} ^{(\bullet)} _{\PP ^{\sharp} \leftarrow \PP ^{\prime \sharp}} (T,T')
\smash{\widehat{\otimes}} ^\L _{\smash{\widetilde{\D}} ^{(\bullet)} _{\PP ^{\prime \sharp}} (T')}
\E ^{\prime (\bullet)}).
\end{gather}

$\bullet$ Avec \cite[4.3.7]{Beintro2},
on construit le foncteur 
image directe par $f$ à singularités surconvergentes le long de $T$ et $T'$ de la forme
$f  ^{\sharp} _{T,T',+} \colon
D ^\mathrm{b} _\mathrm{coh} ( \smash{\D} ^\dag _{\PP ^{\prime \sharp}} (\hdag T') _{\Q} )
\to 
 D ^\mathrm{b} ( \smash{\D} ^\dag _{\PP ^\sharp} (\hdag T) _{\Q} )$
en posant,
pour tout $\E '\in D ^{\mathrm{b}} _{\mathrm{coh}} ( \D ^{\dag} _{\PP ^{\prime \sharp}} (\hdag T' ) _{\Q})$ :
\begin{equation}
\label{ftt'+}
f  ^{\sharp} _{T , T ', +}( \E '):=
 \R f_* (\D ^{\dag} _{\PP ^{\sharp} \leftarrow \PP ^{\prime \sharp}}  ( \hdag T , T ') _{\Q}
\otimes ^{\L} _{ \D ^{\dag} _{\PP ^{\prime \sharp}} (\hdag T ') _{\Q}} \E ').
\end{equation}
\end{vide}

\begin{vide}
\label{com-f+-oubsharp}
Dans ce paragraphe, supposons que $f ^{-1} (\ZZ) =\ZZ'$ et $f ^{\sharp}$ est exact.
Pour tout $\E ^{\prime (\bullet)} \in \smash{\underrightarrow{LD}}  ^\mathrm{b} _{\Q, \mathrm{qc}}
( \smash{\widetilde{\D}} _{\PP ^{\prime }} ^{(\bullet)} (T '))$, le morphisme canonique
\begin{equation}
\label{com-f+-oubsharp-iso}
f _{T,T',+} ^{\sharp (\bullet)} \circ \mathrm{oub} _{\sharp} ( \E ^{\prime (\bullet)} )
\to 
\mathrm{oub} _{\sharp}  \circ f _{T,T',+} ^{(\bullet)} ( \E ^{\prime (\bullet)} )
\end{equation}
est alors un isomorphisme. 
En effet, quitte à décomposer $f$ en son graphe suivant de la projection canonique
$\PP '\times \PP \to \PP$, on se ramène au cas où $f $ est une immersion fermée ou un morphisme lisse. 
Dans chacun de ces deux cas, cela se vérifie par un calcul aisé en coordonnées locales. 

Sans l'hypothèse que $f ^{\sharp}$ soit exact, on prendra garde au fait que 
le morphisme \ref{com-f+-oubsharp-iso} n'est plus un isomorphisme (il suffit de considérer $f =id$).

\end{vide}

\begin{vide}
\label{coh-Qcoh}
Avec les notations du paragraphe \ref{fct-qcoh2coh}, on dispose de 
l'isomorphisme de foncteurs 
$\mathrm{Coh} _{T'} (f  ^{\sharp (\bullet)}_{T , T ', +}) \riso f  ^{\sharp} _{T , T ', +}$ 
et $\mathrm{Coh} _{T} ( f _{T',T} ^{\sharp (\bullet)!}) \riso  f _{T',T} ^{\sharp !}$
(cela se vérifie de manière analogue à \cite[4.3.2.2 et 4.3.7.1]{Beintro2}).
\end{vide}

\begin{vide}
\label{stab-coh}
$\bullet$ Soit $\E \in D ^{\mathrm{b}} _{\mathrm{coh}} ( \D ^{\dag} _{\PP ^{\sharp} } (\hdag T ) _{\Q})$.
Si $f$ est lisse et si $f ^{\sharp}$ est exact, alors 
$f  ^{\sharp !} _{T' , T} (\E ) \in
D ^{\mathrm{b}} _{\mathrm{coh}} ( \D ^{\dag} _{\PP ^{\prime \sharp}} (\hdag T' ) _{\Q})$.

$\bullet$ Soit $\E '\in D ^{\mathrm{b}} _{\mathrm{coh}} ( \D ^{\dag} _{\PP ^{\prime \sharp}} (\hdag T' ) _{\Q})$.
Si $f$ est propre et $T ' = f ^{-1}(T)$, alors
$f  ^{\sharp} _{T , T ', +}( \E ')\in D ^{\mathrm{b}} _{\mathrm{coh}} ( \D ^{\dag} _{\PP ^{\sharp} } (\hdag T ) _{\Q})$
(cela se vérifie de manière analogue à \cite[4.3.7-8]{Beintro2}, i.e. cela résulte des isomorphismes de la forme
\ref{f_+-chgb}).

\end{vide}

\begin{vide}
Supposons dans ce paragraphe que $T' = f ^{-1} (T)$.
On écrit alors
$ f _{T} ^{\sharp (\bullet)!}$, 
$f ^{\sharp !} _T$, 
 $f  ^{\sharp (\bullet)}_{T ,  +}$ et 
 $f  ^{\sharp} _{T,+} $ 
 à la place respectivement de 
$ f _{T',T} ^{\sharp (\bullet)!}$, 
$f ^{\sharp !} _{T',T}$,
$f  ^{\sharp (\bullet)}_{T , T ', +} $ et 
$f  ^{\sharp} _{T,T',+}$.
Si $T$ est l'ensemble vide, nous omettrons de l'indiquer dans toutes
les expressions faisant intervenir $T$. 
\end{vide}

\begin{prop}
\label{2.1.4-caro-surcoh}
Soient $\E ^{(\bullet)}
\in \underrightarrow{LD}  ^\mathrm{b} _{\Q, \mathrm{coh}}
(\overset{^\mathrm{g}}{} \smash{\widehat{\D}} _{\PP ^{\sharp}} ^{(\bullet)})$
et
$\E ^{\prime (\bullet)}
\in \underrightarrow{LD}  ^\mathrm{b} _{\Q, \mathrm{coh}}
(\overset{^\mathrm{g}}{} \smash{\widehat{\D}} _{\PP ^{\prime \sharp }} ^{(\bullet)})$.
On dispose de l'isomorphisme :
\begin{equation}
\label{2.1.4-caro-surcoh-iso}
f ^{\sharp (\bullet)} _{+} 
(\E ^{\prime (\bullet)})
\smash{\overset{\L}{\otimes}}   ^{\dag} _{\O _{\PP} }
 \E ^{(\bullet)}
\liso
f ^{\sharp (\bullet)} _{+} \left ( \E ^{\prime (\bullet)}
\smash{\overset{\L}{\otimes}}   ^{\dag} _{\O _{\PP '}}
\L f  ^{\sharp (\bullet)*}  (\E ^{(\bullet)}) \right).
\end{equation}
\end{prop}

\begin{proof}
Il s'agit de reprendre la preuve de la proposition analogue non logarithmique, i.e. la preuve de 
\cite[2.1.4]{caro_surcoherent}.
Plus précisément, l'isomorphisme 
\ref{2.1.4-caro-surcoh-iso} se déduit {\og par complétion et passage à la limite\fg} de l'isomorphisme analogue 
au niveau des log schémas. La vérification de cet isomorphisme au niveau des log schémas est identique à celle non logarithmique, i.e. 
à celle de \cite[1.2.27]{caro_surcoherent}. En d'autres termes, on se ramène à vérifier que les analogues logarithmiques de 
\cite[1.2.21]{caro_surcoherent} et \cite[1.2.25]{caro_surcoherent} sont encore valables
(pour valider l'analogue logarithmique de  \cite[1.2.25]{caro_surcoherent}, on aura besoin d'utiliser l'isomorphisme de transposition logarithmique de
\cite[1.18]{caro_log-iso-hol}).
\end{proof}

\begin{rema}
\begin{itemize}
\item Grâce à \ref{oub-div-opcoh} dont la preuve utilise \ref{2.1.4-caro-surcoh}, 
on dispose d'une version avec singularités surconvergentes de 
\ref{2.1.4-caro-surcoh-iso}. 
\item Lorsque $f=id$ et $\ZZ$ est vide, 
l'isomorphisme \ref{2.1.4-caro-surcoh-iso}
est un avatar du théorème \cite[3.6]{caro_log-iso-hol} 
dont la preuve reprenait celle fournie pour les variétés complexes.

\end{itemize}
\end{rema}

\subsection{Stabilité par image inverse du coefficient constant avec singularités surconvergentes}

 Soient $u \colon \X \hookrightarrow \PP$ une immersion fermée de $\V$-schémas formels lisses,
$\ZZ$ et $\mathfrak{D}$ des diviseurs à croisements normaux stricts de 
respectivement $\PP$ et $\X$ tels que $u ^{-1} (\ZZ) \subset \mathfrak{D}$.
On pose $\PP ^{\sharp}:= (\PP, \ZZ)$, $\X ^{ \sharp}:= (\X, \mathfrak{D})$
et
$u ^{\sharp}
\colon 
\X ^{ \sharp}
\to 
\PP ^{\sharp}$ le morphisme de $\V$-schémas formels logarithmiques induits par $u$. 

\begin{lemm}
\label{u*Odiv}
Soient $f \colon \PP' \to \PP$ un morphisme de $\V$-schémas formels lisses, 
$T$ un diviseur de $P$ tel que $T':= f ^{-1} (T)$ soit un diviseur de $P'$.  
On dispose alors de l'isomorphisme 
canonique
$$\O _{P ' _i} \otimes  ^\L
_{f ^{-1} \O _{P _i}} f ^{-1} \B ^{(m)}  _{P _i} ( T) 
\riso 
\B ^{(m)}  _{P ' _i} ( T') .$$
On bénéficie de plus de l'isomorphisme canonique
$\L f ^{(\bullet)*} (\widetilde{\B} ^{(\bullet)}  _{\PP} ( T) )
\riso 
\widetilde{\B} ^{(\bullet)}  _{\PP '} ( T')$
dans 
$\smash{\underrightarrow{LD}} ^{\mathrm{b}} _{\Q,\mathrm{coh}} ( \smash{\widetilde{\D}} _{\PP'} ^{(\bullet)})$.

\end{lemm}

\begin{proof}
1) Pour vérifier le premier isomorphisme, 
comme 
$\O _{P ' _i} \otimes 
_{f ^{-1} \O _{P _i}} f ^{-1} \B ^{(m)}  _{P _i} ( T) 
\riso 
\B ^{(m)}  _{P ' _i} ( T') $,
il s'agit de prouver que le morphisme canonique
$\O _{P ' _i} \otimes  ^\L
_{f ^{-1} \O _{P _i}} f ^{-1} \B ^{(m)}  _{P _i} ( T) 
\to 
\O _{P ' _i} \otimes 
_{f ^{-1} \O _{P _i}} f ^{-1} \B ^{(m)}  _{P _i} ( T) $
est un isomorphisme. 
Comme $f $ est le composé de son graphe suivi de la projection 
$\PP ' \times \PP \to \PP$, comme le cas où $f $ est lisse est trivial, 
on se ramène au cas où $f $ est une immersion fermée.
Comme cela est local, on peut supposer $\PP$ affine et intègre. 
Comme $\B ^{(m)}  _{P _i} ( T) $ est quasi-cohérent,
on se ramène à établir que 
le morphisme canonique
$O _{P ' _i} \otimes  ^\L
_{O _{P _i}}  B ^{(m)}  _{P _i} ( T) 
\to 
O _{P ' _i} \otimes 
_{O _{P _i}} B ^{(m)}  _{P _i} ( T) $
est un isomorphisme. 

Comme $\widehat{B} ^{(m)}  _{\PP} ( T) $ est sans $p$-torsion, 
on a : 
$O _{P ' _i} \otimes  ^\L
_{O _{P _i}} B ^{(m)}  _{P _i} ( T) 
\riso
O _{P ' _i} 
\otimes  ^\L _{O _{\PP '}} 
(O _{\PP '}
\otimes  ^\L _{O _{\PP}} 
\widehat{B} ^{(m)}  _{\PP} ( T) )$.
Or, comme $f $ est fini, 
on dispose de l'isomorphisme canonique
$O _{\PP '}
\otimes  _{O _{\PP}} 
\widehat{B} ^{(m)}  _{\PP} ( T) 
\riso 
\widehat{B} ^{(m)}  _{\PP'} ( T')$. 
Comme $\widehat{B} ^{(m)}  _{\PP} ( T) $ est intègre,
le morphisme canonique
$O _{\PP '}
\otimes  ^\L _{O _{\PP}} 
\widehat{B} ^{(m)}  _{\PP} ( T) 
\to 
O _{\PP '}
\otimes  _{O _{\PP}} 
\widehat{B} ^{(m)}  _{\PP} ( T) $
est un isomorphisme (pour le voir, 
on utilise la résolution de Koszul de $O _{\PP '}$).
Comme $\widehat{B} ^{(m)}  _{\PP'} ( T')$ est sans $p$-torsion, on conclut.

2) Comme, pour tout entier positif $m$, 
les $\V$-modules 
$\widetilde{\B} ^{(m)}  _{\PP '} ( T')$
et
$\widetilde{\B} ^{(m)}  _{\PP } ( T)$
sont sans $p$-torsion,
le deuxième isomorphisme à prouver résulte aussitôt du premier.
\end{proof}

\begin{vide}
\label{nota-u^!alg}
Grâce à
\cite[3.2.4]{Be1}, on vérifie que la flèche de droite du carré ci-dessous:
\begin{equation}
\label{pre-lemm-ualg-u!nivm}
\xymatrix @=0,4cm{
{\O _{\X} \otimes ^{\L} _{u ^{-1} \O _{\PP}} u ^{-1} \smash{\widetilde{\D}} _{\PP ^\sharp} ^{(m)}} 
\ar[r] ^-{\sim}
\ar[d] ^-{}
& 
{\O _{\X} \otimes _{u ^{-1} \O _{\PP}} u ^{-1} \smash{\widetilde{\D}} _{\PP ^\sharp} ^{(m)}} 
\ar[d] ^-{\sim}
\\ 
{\O _{\X} \widehat{\otimes} ^{\L} _{u ^{-1} \O _{\PP}} u ^{-1} \smash{\widetilde{\D}} _{\PP ^\sharp} ^{(m)}} 
\ar[r] ^-{\sim}
& 
{\O _{\X} \widehat{\otimes} _{u ^{-1} \O _{\PP}} u ^{-1} \smash{\widetilde{\D}} _{\PP ^\sharp} ^{(m)}
=\smash{\widetilde{\D}} _{\X ^{\sharp} \hookrightarrow \PP ^{\sharp}} ^{(m)}} 
}
\end{equation}
est un isomorphisme.
De plus, par platitude de $\smash{\widetilde{\D}} _{\PP ^\sharp} ^{(m)}$ sur 
$\O _{\PP}$, on vérifie 
que la flèche horizontale du haut est un isomorphisme. 
Comme la flèche du bas est un morphisme de
$D ^{\mathrm{b}} _{\mathrm{qc}} (\smash{\widetilde{\D}} _{\X ^{\sharp}} ^{(m)})$ (cela résulte du théorème 
\cite[3.2.3]{Beintro2} de Berthelot), on vérifie de même que celle-ci est un isomorphisme.
Il en est donc de même de la flèche de gauche. 
Pour tout
$\E ^{(m)} \in  D ^{-} ( \smash{\widetilde{\D}} _{\PP ^\sharp} ^{(m)}) $
on en déduit que le morphisme canonique
$\O _{\X} \otimes ^{\L} _{u ^{-1} \O _{\PP}} u ^{-1} \E ^{(m)} 
\to
\smash{\widetilde{\D}} _{\X ^{\sharp} \hookrightarrow \PP ^{\sharp}} ^{(m)}
\otimes ^{\L} _{u ^{-1} \smash{\widetilde{\D}} _{\PP ^\sharp} ^{(m)}} u ^{-1} \E ^{(m)} $
est un isomorphisme.
On dispose alors du foncteur 
$u ^{\sharp (m)!} _{\mathrm{alg}}
\colon 
D ^{-} ( \smash{\widetilde{\D}} _{\PP ^\sharp} ^{(m)}) 
\to D ^{-} ( \smash{\widetilde{\D}} _{\X ^{\sharp}} ^{(m)})$
défini en posant, pour tout $\E ^{(m)} \in  D ^{-} ( \smash{\widetilde{\D}} _{\PP ^\sharp} ^{(m)}) $,
\begin{equation}
\label{defu!alg-m}
u ^{\sharp (m)!} _{\mathrm{alg}} (\E ^{(m)} )
:=
\O _{\X} \otimes ^{\L} _{u ^{-1} \O _{\PP}} u ^{-1} \E ^{(m)} 
[ d _{X/P}].
\end{equation}

\end{vide}

\begin{lemm}
\label{lemm-ualg-u!nivm}
Pour tout $\E ^{(m)} \in  D ^{-} ( \smash{\widetilde{\D}} _{\PP ^\sharp} ^{(m)}) $,
on dispose du morphisme canonique
$u ^{\sharp (m)!} _{\mathrm{alg}} (\E ^{(m)} )
\to 
u ^{\sharp (m)!} (\E ^{(m)} )$.
Si 
$\E ^{(m)} \in D ^{-} _{\mathrm{coh}}( \smash{\widetilde{\D}} _{\PP ^\sharp} ^{(m)}) $,
alors celui-ci est un isomorphisme.

\end{lemm}

\begin{proof}
La construction se construit par adjonction via le morphisme d'espaces annelés
$(X _{\bullet}, \O _{X _{\bullet}}) \to (\X, \O _{\X})$. 
Vérifions à présent que celui-ci est un isomorphisme pour les complexes à cohomologie cohérente. 
Comme cela est local, on peut supposer $\PP$ affine. 
D'après le lemme sur les foncteurs way-out (et le théorème de type $A$ pour les $\smash{\widetilde{\D}} _{\PP ^\sharp} ^{(m)} $-modules cohérents),
on se ramène au cas où $\E ^{(m)}$ est de la forme
$\E ^{(m)} = (\smash{\widetilde{\D}} _{\PP ^\sharp} ^{(m)} )^r$.
Dans ce cas, cela résulte de l'isomorphisme vertical de gauche du carré
\ref{pre-lemm-ualg-u!nivm}.
\end{proof}

\begin{vide}
On déduit de 
\ref{defu!alg-m}
que l'on dispose du foncteur 
$u ^{\sharp (\bullet)!} _{\mathrm{alg}}
\colon 
D ^{-} ( \smash{\widetilde{\D}} _{\PP ^\sharp} ^{(\bullet)}) 
\to D ^{-} ( \smash{\widetilde{\D}} _{\X ^{\sharp}} ^{(\bullet)})$
défini en posant, pour tout $\E ^{(\bullet)} \in  D ^{-} ( \smash{\widetilde{\D}} _{\PP ^\sharp} ^{(\bullet)}) $,
\begin{equation}
\label{defu!alg}
u ^{\sharp (\bullet)!} _{\mathrm{alg}} (\E ^{(\bullet)} )
:=
\O _{\X} \otimes ^{\L} _{u ^{-1} \O _{\PP}} u ^{-1} \E ^{(\bullet)} 
[ d _{X/P}].
\end{equation}
Ce foncteur  envoie les ind-isogénies sur les ind-isogénies
et les lim-ind-isogénies sur les lim-ind-isogénies (voir les définitions de \ref{loc-LM}). 
Ce foncteur se factorise donc en
$u ^{\sharp (\bullet)!} _{\mathrm{alg}}
\colon 
\smash{\underrightarrow{LD}} ^{\mathrm{b}} _{\Q}   ( \smash{\widetilde{\D}} _{\PP ^\sharp} ^{(\bullet)}) 
\to 
\smash{\underrightarrow{LD}} ^{\mathrm{b}} _{\Q}  ( \smash{\widetilde{\D}} _{\X ^{\sharp}} ^{(\bullet)})$.
On prendra garde de constater que ce foncteur ne préserve pas la quasi-cohérence.
On bénéficie, pour tout 
$\E ^{(\bullet)}  \in
\smash{\underrightarrow{LD}} ^{\mathrm{b}} _{\Q,\mathrm{qc}} 
( \smash{\widetilde{\D}} _{\PP ^\sharp} ^{(\bullet)}) $,
du morphisme canonique
$u ^{\sharp (\bullet)!} _{\mathrm{alg}} (\E ^{(\bullet)} )
\to 
u ^{\sharp (\bullet)!} (\E ^{(\bullet)} )$ (ce morphisme se construit même dans
$D  ( \smash{\widetilde{\D}} _{\X ^{\sharp}} ^{(\bullet)})$, i.e. avant de localiser).

\end{vide}

\begin{lemm}
\label{u!alg=u!}
Soit $\E ^{(\bullet)}  \in
\smash{\underrightarrow{LD}} ^{\mathrm{b}} _{\Q,\mathrm{coh}} 
( \smash{\widetilde{\D}} _{\PP ^\sharp} ^{(\bullet)}) $.
Le morphisme canonique
$u ^{\sharp (\bullet)!} _{\mathrm{alg}} (\E ^{(\bullet)} )
\to 
u ^{\sharp (\bullet)!} (\E ^{(\bullet)} )$
est un isomorphisme dans
$\smash{\underrightarrow{LD}} ^{\mathrm{b}} _{\Q} 
( \smash{\widetilde{\D}} _{\X ^{\sharp}} ^{(\bullet)}) $.
\end{lemm}

\begin{proof}
Par définition, il existe
$\lambda \in L$ un isomorphisme 
$\smash{\underrightarrow{LD}} ^{\mathrm{b}} _{\Q}  ( \smash{\widetilde{\D}} _{\PP ^\sharp} ^{(\bullet)})$
de la forme 
$\E ^{(\bullet)} \riso \FF^{(\bullet)}$
tel que $\FF ^{(m)} \in D ^{\mathrm{b}}_{\mathrm{coh}}( \smash{\widetilde{\D}} _{\PP ^\sharp} ^{(\lambda (m))}) $
(plus une autre condition que l'on n'utilisera pas).
On se ramène donc à établir le lemme pour $\FF^{(\bullet)}$ à la place de $\E ^{(\bullet)} $.
Notons $u ^{\sharp (\lambda (\bullet))!} _{\mathrm{alg}} (\FF ^{(\bullet)} ):=
(u ^{\sharp( \lambda (m))!} _{\mathrm{alg}} (\FF ^{(m)} )) _{m\in \N}
\in 
D ^{\mathrm{b}} ( \smash{\widetilde{\D}} _{\X ^{\sharp}} ^{(\lambda (\bullet))})$
et
$u ^{\sharp (\lambda (\bullet))!} (\FF ^{(\bullet)} ):= (u ^{\sharp ( \lambda (m))!} (\FF ^{(m)} )) _{m\in \N}
\in 
D ^{\mathrm{b}}  ( \smash{\widetilde{\D}} _{\X ^{\sharp}} ^{(\lambda (\bullet))})$.
En omettant d'indiquer les foncteurs oublis,
on dispose par définition de l'égalité 
$u ^{\sharp ( \lambda (\bullet))!} _{\mathrm{alg}} (\FF ^{(\bullet)} )=u ^{\sharp  (\bullet)!} _{\mathrm{alg}} (\FF ^{(\bullet)} )$.
De plus, on dispose du morphisme canonique
$u ^{\sharp (\lambda (\bullet))!} (\FF ^{(\bullet)} )
\to 
 \lambda ^{*} \left (u ^{\sharp  (\bullet)!} (\FF ^{(\bullet)} )\right) $
 induisant le diagramme canonique commutatif:
\begin{equation}
\notag
\xymatrix @ R=0,4cm {
{u ^{\sharp (\bullet)!}  (\FF ^{(\bullet)} )} 
\ar[r] ^-{}
\ar[d] ^-{}
& 
{u ^{\sharp (\lambda (\bullet))!} (\FF ^{(\bullet)} )} 
\ar[d] ^-{}
\ar@{.>}[dl] ^-{}
\\ 
{ \lambda ^{*} \left (u ^{\sharp  (\bullet)!} (\FF ^{(\bullet)} )\right) } 
 \ar[r] ^-{}
& 
{ \lambda ^{*} \left (u ^{\sharp (\lambda (\bullet))!} (\FF ^{(\bullet)} )\right) .} 
}
\end{equation}
Il en résulte que 
$u ^{\sharp (\lambda (\bullet))!} (\FF ^{(\bullet)} )$
est canoniquement isomorphisme 
à 
$u ^{\sharp (\bullet)!}  (\FF ^{(\bullet)} )$
dans 
$\smash{\underrightarrow{LD}} ^{\mathrm{b}} _{\Q}  ( \smash{\widetilde{\D}} _{\X ^{\sharp}} ^{(\bullet)})$.
On se ramène à vérifier que le morphisme canonique 
\begin{equation}
u ^{\sharp (\lambda (\bullet))!} _{\mathrm{alg}} (\FF ^{(\bullet)} )
\to
u ^{\sharp (\lambda (\bullet))!}  (\FF ^{(\bullet)} )
\end{equation}
est un isomorphisme de 
$D ^{\mathrm{b}} ( \smash{\widetilde{\D}} _{\X ^{\sharp}} ^{(\lambda (\bullet))})$, ce qui résulte du fait que 
$\FF ^{(m)} \in D ^{\mathrm{b}}_{\mathrm{coh}}( \smash{\widetilde{\D}} _{\PP ^\sharp} ^{(\lambda (m))}) $ 
et du lemme \ref{lemm-ualg-u!nivm}.
\end{proof}

\begin{lemm}
\label{annul-imm-ferm}
Soit $T$ un diviseur de $P$ tel que $u (X) \subset T$.
On 
dispose de l'isomorphisme canonique:
$u ^{(\bullet) !} (\widetilde{\B} ^{(\bullet)}  _{\PP} ( T) )
\riso 
0$
dans 
$\smash{\underrightarrow{LD}} ^{\mathrm{b}} _{\Q} ( \smash{\widetilde{\D}} _{\X} ^{(\bullet)})$.
\end{lemm}

\begin{proof}
0) D'après la proposition \ref{LDiso-local}, 
ce que l'on doit vérifier est local en $\PP$.
On peut supposer $\PP$ affine et muni de coordonnées locales
$t _1, \dots, t _d$ telles que $\X = V (t _1, \dots, t _r)$
et qu'il existe un élément $f$ de $O _{\PP}$
tel que $T = V (\overline{f})$, où $\overline{f}$ est la réduction de $f$ modulo $\pi O _{\PP}$.
Vérifions par récurrence sur $r \geq 1$ le lemme.

1) Traitons d'abord le cas $r =1$. 
Grâce à \ref{u!alg=u!}, 
il s'agit d'établir l'isomorphisme 
$u ^{(\bullet)!} _{\mathrm{alg}} ( \widetilde{\B} ^{(\bullet)}  _{\PP} ( T) ) \riso 0$
dans 
$\smash{\underrightarrow{LD}} ^{\mathrm{b}} _{\Q} ( \smash{\widetilde{\D}} _{\X} ^{(\bullet)})$.
Comme 
$\widetilde{\B} ^{(\bullet)}  _{\PP} ( T) $ est un faisceau d'anneaux intègres, 
en utilisant la suite exacte
$0 \to u ^{-1} \O _{\PP}\overset{t _1}{\longrightarrow} u ^{-1} \O _{\PP}\to \O _{\X} \to 0$
qui permet de résoudre $\O _{\X}$ par des $u ^{-1} \O _{\PP}$-modules plats,
on vérifie que le morphisme canonique 
 de la forme
$\O _{\X} \otimes ^{\L} _{u ^{-1} \O _{\PP}} u ^{-1} \widetilde{\B} ^{(\bullet)}  _{\PP} ( T)  
\to 
\O _{\X} \otimes  _{u ^{-1} \O _{\PP}} u ^{-1} \widetilde{\B} ^{(\bullet)}  _{\PP} ( T) $
est alors un isomorphisme dans $D ( \smash{\widetilde{\D}} _{\X} ^{(\bullet)})$.
Ainsi, on dispose de l'isomorphisme canonique 
$u ^{(\bullet)!} _{\mathrm{alg}} ( \widetilde{\B} ^{(\bullet)}  _{\PP} ( T) )
\riso 
\O _{\X} \otimes  _{u ^{-1} \O _{\PP}} u ^{-1}  \widetilde{\B} ^{(\bullet)}  _{\PP} ( T) 
[ d _{X/P}]$ dans 
$D ( \smash{\widetilde{\D}} _{\X} ^{(\bullet)})$ et donc dans 
$\smash{\underrightarrow{LD}} ^{\mathrm{b}} _{\Q} ( \smash{\widetilde{\D}} _{\X} ^{(\bullet)})$.
Soit $\chi \colon \N \to \N$ l'élément de $M$ (voir la définition dans le paragraphe \ref{defi-M-L})
défini par $\chi (m)=1$.
Il suffit alors de vérifier que le morphisme canonique de $\smash{\widetilde{\D}} _{\X} ^{(\bullet)}$-modules
$$\O _{\X} \otimes  _{u ^{-1} \O _{\PP}} u ^{-1} \widetilde{\B} ^{(\bullet)}  _{\PP} ( T) 
\to
\chi ^{*} 
(\O _{\X} \otimes  _{u ^{-1} \O _{\PP}} u ^{-1} \widetilde{\B} ^{(\bullet)}  _{\PP} ( T) )$$
est le morphisme nul (en d'autres termes, 
$\O _{\X} \otimes  _{u ^{-1} \O _{\PP}} u ^{-1} \widetilde{\B} ^{(\bullet)}  _{\PP} ( T) $ est annulé par la multiplication par $p$).
Comme $ \widetilde{\B} ^{(m)}  _{\PP} ( T) $ ne dépend pas, à isomorphisme canonique près, 
du choix du relèvement $f$ d'une équation locale de $T$ (voir \cite[4.2.3]{Be1}), on peut supposer 
que $t$ divise $f$, i.e. l'image de $f$ sur $\O _{\X}$ est nulle. Dans ce cas, on calcule que 
$\O _{\X} \otimes  _{u ^{-1} \O _{\PP}} u ^{-1} \widetilde{\B} ^{(m)}  _{\PP} ( T) =
\O _{\X} \otimes  _{u ^{-1} \O _{\PP}} u ^{-1} \O _{\PP} \{ X \} /( f ^{p ^{m+1}}X -p)
=
\O _{\X} \{ X \} /( p)
= \O _{X} [X]$, qui est annulé par $p$. 
D'où le résultat.

2) Supposons maintenant la propriété vraie pour $r-1$ et prouvons-la pour $r$.
Notons $\X:= V (t _1)$. Si $T \cap X$ est un diviseur de $X$ (resp. si $T \supset X$), 
le lemme \ref{u*Odiv} (resp. le cas $r =1$) nous permet de conclure.

\end{proof}

\begin{prop}
\label{stab-coh-u!}
Soit $X' \subset X$ un sous-schéma fermé de $X$. 
On a alors
$ \R \underline{\Gamma} ^\dag _{X'} \circ u ^{(\bullet) !} (\widetilde{\B} ^{(\bullet)}  _{\PP} ( T) )
\in \smash{\underrightarrow{LD}} ^{\mathrm{b}} _{\Q,\mathrm{coh}} 
( \smash{\widetilde{\D}} _{\X} ^{(\bullet)})$.
\end{prop}

\begin{proof}
Comme $P$ et $X$ sont lisses, il ne coûte pas cher de les supposer intègres.
Dans ce cas, soit $T \cap X$ est un diviseur de $X$, soit
$T \supset X$. Le premier cas résulte alors des lemmes \ref{induction-div-coh} et \ref{u*Odiv}, 
tandis que le second cas découle du lemme \ref{annul-imm-ferm}.
\end{proof}

\begin{theo}
\label{2.2.18}
Soient $f \colon \PP' \to \PP$ un morphisme de $\V$-schémas formels lisses,
$\ZZ$ et $\ZZ'$ des diviseurs à croisements normaux stricts de 
respectivement $\PP$ et $\PP'$ tels que $f ^{-1} (\ZZ) \subset \ZZ'$.
On pose $\PP ^{\sharp}:= (\PP, \ZZ)$, $\PP ^{\prime \sharp}:= (\PP', \ZZ')$
et
$f ^{\sharp}
\colon 
\PP ^{\prime \sharp}
\to 
\PP ^{\sharp}$ le morphisme de $\V$-schémas formels logarithmiques induits par $f$. 
Soient $X$ un sous-schéma fermé de $P$, $X':= f ^{-1} (X)$, 
$\E ^{(\bullet)} \in \smash{\underrightarrow{LD}}
^\mathrm{b} _{\Q, \mathrm{qc}} (\overset{^g}{} \widetilde{\D} _{\PP ^{\sharp}}
^{(\bullet)} )$ et
$\E ^{\prime (\bullet)} \in\smash{\underrightarrow{LD}}  ^\mathrm{b}
_{\Q, \mathrm{qc}} (\overset{^g}{} \widetilde{\D} _{\PP ^{\prime \sharp}}
^{(\bullet)} )$. 
On a des isomorphismes fonctoriels en $X$ et compatibles à Frobenius :
\begin{gather}
\label{commutfonctcohlocal1}
  f ^{\sharp (\bullet)!}  \circ\R \underline{\Gamma} ^\dag _{X}(\E ^{(\bullet)}) 
  \riso
   \R \underline{\Gamma} ^\dag _{X' }\circ f ^{\sharp (\bullet)!}  (\E ^{(\bullet)}), \
f ^{\sharp (\bullet)!} \circ (\hdag X)(\E ^{(\bullet)} ) \riso (\hdag X') \circ f ^{\sharp (\bullet)!}(\E^{(\bullet)} )\\
\label{commutfonctcohlocal2}
\R \underline{\Gamma} ^\dag _{X}\circ f ^{\sharp (\bullet)} _{+} (\E ^{\prime (\bullet)})
\riso
f ^{\sharp (\bullet)}  _{+} \circ \R \underline{\Gamma} ^\dag _{X'}(\E ^{\prime (\bullet)})
 , \
(\hdag X)\circ f ^{\sharp (\bullet)} _{+} (\E ^{\prime (\bullet)}) 
\riso  f ^{\sharp (\bullet)} _{+} \circ (\hdag X')(\E ^{\prime (\bullet)} ).
\end{gather}

\end{theo}

\begin{proof}
Ce théorème se vérifie de manière identique à sa version non-logarithmique de \cite[2.2.18]{caro_surcoherent}. 
Pour la commodité du lecteur, nous allons rappeler sa preuve en ajoutant aussi quelques détails lorsque cela est nécessaire. 
On établit d'abord les isomorphismes \ref{commutfonctcohlocal1}. 
Via les triangles de localisation, il suffit de valider celui de droite ou celui de gauche.
Avec \ref{theo2.2.8}, on peut en outre supposer que 
$X$ et $X'$ sont des diviseurs.
Comme $f$ se décompose en son graphe suivi de la projection canonique 
$\PP ' \times \PP \to \PP$, comme l'isomorphisme de droite est immédiat lorsque $f$ est plat, 
on se ramène au cas où $f$ est une immersion fermée.
Comme le foncteur cohomologique local commute au produit tensoriel 
(voir \ref{fonctXX'Gamma}) et de même pour l'image inverse extraordinaire, 
on se ramène au cas où 
$\E ^{(\bullet)}=\O _{\PP} ^{(\bullet)}$.
Avec \ref{GammaX-oub-sharp-iso} et \ref{f!oubsharp}, on peut oublier les structures logarithmiques. 
Grâce à \ref{stab-coh-u!},
les morphismes canoniques
$\R \underline{\Gamma} ^\dag _{X' }\circ f ^{(\bullet)!} \circ\R \underline{\Gamma} ^\dag _{X}(\O _{\PP} ^{(\bullet)})
\to 
\R \underline{\Gamma} ^\dag _{X' }\circ f ^! (\O _{\PP} ^{(\bullet)})$
et
$\R \underline{\Gamma} ^\dag _{X' }\circ f ^! \circ\R \underline{\Gamma} ^\dag _{X}(\O _{\PP} ^{(\bullet)})
\to 
f ^! \circ\R \underline{\Gamma} ^\dag _{X}(\O _{\PP} ^{(\bullet)})$
sont dans $\smash{\underrightarrow{LD}} ^{\mathrm{b}} _{\Q,\mathrm{coh}} ( \smash{\widetilde{\D}} _{\PP} ^{(\bullet)})$.
Via l'équivalence de catégories 
de \ref{eq-catLDBer-LD-D} de la forme
$\underrightarrow{\lim} 
\colon 
\underrightarrow{LD} ^{\mathrm{b}}  _{\Q, \mathrm{coh}} (\smash{\widetilde{\D}} _{\PP} ^{(\bullet)} (T))
\cong
D ^{\mathrm{b}} _{\mathrm{coh}}( \smash{\D} ^\dag _{\PP} (\hdag T) _{\Q} )$, 
pour vérifier que ce sont des isomorphismes, il suffit de le valider après application de ce foncteur $\underrightarrow{\lim} $ pleinement fidèle, 
ce qui sont les lemmes 
\cite[2.2.19 et 2.2.20]{caro_surcoherent}.
D'où \ref{commutfonctcohlocal1}. 
L'isomorphisme \ref{2.1.4-caro-surcoh-iso} permet de déduire \ref{commutfonctcohlocal2}
à partir de \ref{commutfonctcohlocal1}.
\end{proof}

\begin{vide}
[On peut oublier les diviseurs]
\label{oub-div-opcoh}
Soient $f \colon \PP' \to \PP$ un morphisme de $\V$-schémas formels lisses,
$\ZZ$ et $\ZZ'$ des diviseurs à croisements normaux stricts de 
respectivement $\PP$ et $\PP'$ tels que $f ^{-1} (\ZZ) \subset \ZZ'$.
On pose $\PP ^{\sharp}:= (\PP, \ZZ)$, $\PP ^{\prime \sharp}:= (\PP', \ZZ')$
et
$f ^{\sharp}
\colon 
\PP ^{\prime \sharp}
\to 
\PP ^{\sharp}$ le morphisme de $\V$-schémas formels logarithmiques induits par $f$. 
Soient $T$ et $T'$ des diviseurs respectifs de $P$ et $P'$ tels que
$f ( P '\setminus T' ) \subset P \setminus T$.

\begin{enumerate}
\item   Soit $\E ^{\prime (\bullet)} \in \smash{\underset{^{\longrightarrow}}{LD}} ^{\mathrm{b}} _{\Q ,\mathrm{qc}}
( \smash{\widetilde{\D}} _{\PP ^{\prime \sharp }} ^{(\bullet)}(T'))$.
Comme pour \cite[1.1.9 ]{caro_courbe-nouveau}, on vérifie de manière élémentaire que 
l'on dispose de l'isomorphisme canonique
$oub _{T} \circ f ^{\sharp(\bullet)} _{T,T',+} (\E ^{\prime (\bullet)} )\riso f ^{\sharp(\bullet)} _+  \circ oub _{T'} (\E ^{\prime (\bullet)})$.
On notera alors simplement 
$f ^{\sharp(\bullet)} _{+}$ à la place de $f ^{\sharp(\bullet)} _{T, T',+}$.

D'après la remarque \ref{rema-fct-qcoh2coh}, 
les foncteurs 
$\mathrm{Coh} _{T'} (f ^{\sharp(\bullet)} _{T,T',+}) $
et
$\mathrm{Coh}  (f ^{\sharp(\bullet)} _{+}) $ (on n'indique pas dans les notations l'ensemble vide par convention)
sont isomorphes sur 
$D ^\mathrm{b} _\mathrm{coh} ( \smash{\D} ^\dag _{\PP ^{\prime \sharp },\Q} ) 
\cap D ^\mathrm{b} _\mathrm{coh} ( \smash{\D} ^\dag _{\PP ^{\prime \sharp }} (\hdag T') _{\Q} )$.
Or, on a vu dans le paragraphe \ref{coh-Qcoh} que 
l'on dispose de l'isomorphisme de foncteurs 
$\mathrm{Coh} _{T'} (f  ^{\sharp (\bullet)}_{T , T ', +}) \riso f  ^{\sharp} _{T , T ', +}$ 
et
$\mathrm{Coh} (f  ^{\sharp (\bullet)}_{+}) \riso f  ^{\sharp} _+$.
On pourra alors noter sans risque de confusion
$f  ^{\sharp} _+$ à la place de $f  ^{\sharp} _{T, T',+}$.

\item Supposons $T ' = f ^{-1} (T)$ et 
soit $\E ^{(\bullet)} \in \smash{\underset{^{\longrightarrow}}{LD}} ^{\mathrm{b}} _{\Q ,\mathrm{qc}}
( \smash{\widetilde{\D}} _{\PP ^\sharp} ^{(\bullet)} (T))$.
De manière analogue à \cite[1.1.10]{caro_courbe-nouveau},
on déduit du 
théorème \ref{2.2.18} ci-dessus, 
que l'on bénéficie de l'isomorphisme canonique
$oub _{T'} \circ f ^{\sharp (\bullet)!} _{T} (\E^{(\bullet)} )\riso f ^{\sharp(\bullet)!} \circ oub _T (\E^{(\bullet)})$.
On notera alors simplement 
$f ^{\sharp(\bullet)!} $ à la place de $f ^{\sharp(\bullet)!} _T$. 
Pour des raisons identiques au cas de l'image directe, on pourra aussi noter simplement
$f ^{\sharp!} $ à la place de $f ^{\sharp!} _T$. 

\end{enumerate}
\end{vide}

\subsection{Surcohérence dans un $\V$-schéma formel lisse, stabilité par image inverse extraordinaire par une immersion fermée}

\begin{defi}
\label{surcohP}
$\bullet$ Soit $\E ^{(\bullet)}
\in 
\smash{\underrightarrow{LD}} ^{\mathrm{b}} _{\Q,\mathrm{coh}} ( \smash{\widetilde{\D}} _{\PP ^\sharp} ^{(\bullet)} (T))$.
On dit que $\E ^{(\bullet)}$ est surcohérent dans $\PP ^\sharp$ (ou dans $\PP$ s'il n'y a pas d'ambiguïté avec la log structure) 
si, pour tout diviseur $T'$ 
de $P$, on a alors
$(\hdag T') (\E ^{(\bullet)}) \in 
\smash{\underrightarrow{LD}} ^{\mathrm{b}} _{\Q,\mathrm{coh}} ( \smash{\widetilde{\D}} _{\PP ^\sharp} ^{(\bullet)} (T))$.
On note $\smash{\underrightarrow{LD}} ^{\mathrm{b}} _{\Q,\mathrm{surcoh}, \PP} ( \smash{\widetilde{\D}} _{\PP ^\sharp} ^{(\bullet)} (T))$
la sous-catégorie pleine 
de 
$\smash{\underrightarrow{LD}} ^{\mathrm{b}} _{\Q,\mathrm{coh}} ( \smash{\widetilde{\D}} _{\PP ^\sharp} ^{(\bullet)} (T))$
des complexes surcohérents dans $\PP$.

$\bullet$ De même, on note $D ^{\mathrm{b}}  _{\mathrm{surcoh}, \PP} (\smash{\D} ^\dag _{\PP ^\sharp} (\hdag T) _{\Q})$ 
la sous-catégorie pleine de
$D ^{\mathrm{b}}  _{\mathrm{coh}} (\smash{\D} ^\dag _{\PP ^\sharp} (\hdag T) _{\Q})$
des complexes $\E$ tels que 
pour tout diviseur $T'$ de $P$, 
on ait 
$(\hdag T') (\E) \in 
D ^{\mathrm{b}}  _{\mathrm{coh}} (\smash{\D} ^\dag _{\PP ^\sharp} (\hdag T) _{\Q})$.

\end{defi}

\begin{prop}
\label{surcohP=surcohP}
Soit $\E ^{(\bullet)}
\in 
\underrightarrow{LD} ^{\mathrm{b}} _{\Q, \mathrm{coh}} (\smash{\widetilde{\D}} _{\PP ^\sharp} ^{(\bullet)} (T))$.
La propriété 
$\E ^{(\bullet)}
\in 
\smash{\underrightarrow{LD}} ^{\mathrm{b}} _{\Q,\mathrm{surcoh}, \PP} ( \smash{\widetilde{\D}} _{\PP ^\sharp} ^{(\bullet)} (T))$
est équivalente à la propriété 
$\underrightarrow{\lim}~
\E ^{ (\bullet)}
\in
D ^{\mathrm{b}}  _{\mathrm{surcoh}, \PP} (\smash{\D} ^\dag _{\PP ^\sharp} (\hdag T) _{\Q})$.
Ces deux propriétés sont locales en $\PP$ et
le foncteur $\underrightarrow{\lim}$ se factorise en 
l'équivalence de catégories de la forme
$\underrightarrow{\lim}
\colon 
\smash{\underrightarrow{LD}} ^{\mathrm{b}} _{\Q,\mathrm{surcoh}, \PP} ( \smash{\widetilde{\D}} _{\PP ^\sharp} ^{(\bullet)} (T))
\cong 
D ^{\mathrm{b}}  _{\mathrm{surcoh}, \PP} (\smash{\D} ^\dag _{\PP ^\sharp} (\hdag T) _{\Q})$.
\end{prop}

\begin{proof}
Soit $T'$ un diviseur de $P$. 
Comme 
$\underrightarrow{\lim} \circ (\hdag T') (\E ^{(\bullet)})
\underset{\ref{hdagT'T=cup}}{\riso} 
\underrightarrow{\lim} \circ (\hdag T' \cup T) (\E ^{(\bullet)})
\riso
(\hdag T' \cup T,T) \circ \underrightarrow{\lim} (\E ^{(\bullet)})$, 
il suffit d'appliquer le corollaire \ref{coro1limTouD}
au complexe 
$\E ^{\prime (\bullet)} := (\hdag T') (\E ^{(\bullet)})$.
\end{proof}

\begin{lemm}
\label{1ercrit-surcohP}
Soit $\E ^{(\bullet)}
\in 
\smash{\underrightarrow{LD}} ^{\mathrm{b}} _{\Q,\mathrm{coh}} ( \smash{\widetilde{\D}} _{\PP ^\sharp} ^{(\bullet)} (T))$.
Avec les notations et définitions de \ref{dfn-4.3.4} et \ref{hdagX}, 
l'objet $\E ^{(\bullet)}$ appartient à 
$\smash{\underrightarrow{LD}} ^{\mathrm{b}} _{\Q,\mathrm{surcoh},\PP} ( \smash{\widetilde{\D}} _{\PP ^\sharp} ^{(\bullet)} (T))$
si et seulement si, 
pour tout sous-schéma fermé $Z$ de $P$, 
les objets de $(\hdag Z) (\E ^{(\bullet)}) $ et 
$\R \underline{\Gamma} ^\dag _{Z} (\E ^{(\bullet)})$
sont dans 
$\smash{\underrightarrow{LD}} ^{\mathrm{b}} _{\Q,\mathrm{surcoh},\PP} ( \smash{\widetilde{\D}} _{\PP ^\sharp} ^{(\bullet)} (T))$.
\end{lemm}

\begin{proof}
On procède par dévissage de manière identique à \ref{induction-div-coh}.
\end{proof}

\begin{nota}
\label{nota-usharp}
 Soient $u \colon \X \hookrightarrow \PP$ une immersion fermée de $\V$-schémas formels lisses,
$\ZZ$ et $\mathfrak{D}$ des diviseurs à croisements normaux stricts de 
respectivement $\PP$ et $\X$ tels que $u ^{-1} (\ZZ) \subset  \mathfrak{D}$.
On pose $\PP ^{\sharp}:= (\PP, \ZZ)$, $\X ^{ \sharp}:= (\X, \mathfrak{D})$
et
$u ^{\sharp}
\colon 
\X ^{ \sharp}
\to 
\PP ^{\sharp}$ l'immersion fermée de $\V$-schémas formels logarithmiques induits par $u$. 
Soient $\U$ l'ouvert de $\PP$ complémentaire de $u (X)$, 
$T$ un diviseur de $P$ tel que $T \cap X$ soit un diviseur de $X$. 

\end{nota}

\begin{lemm}
\label{imm-fer-coh}
Avec les notations de \ref{nota-usharp}, 
on dispose des carrés commutatifs à isomorphisme canonique:
\begin{equation}
\label{imm-fer-coh-diag}
\xymatrix @=0,4cm{
{ \underrightarrow{LD} ^{\mathrm{b}} _{\Q, \mathrm{coh}} (\smash{\widetilde{\D}} _{\X ^\sharp} ^{(\bullet)} (T\cap X))}
\ar[r]  _-{u  ^{(\bullet)} _+}
\ar[d] ^-{\underrightarrow{\lim} } _-{\cong} 
& 
{ \underrightarrow{LD} ^{\mathrm{b}} _{\Q, \mathrm{coh}} (\smash{\widetilde{\D}} _{\PP ^\sharp} ^{(\bullet)} (T))}
\ar[d] ^-{\underrightarrow{\lim} } _-{\cong} 
\\ 
{ D ^{\mathrm{b}} _{\mathrm{coh}} (\smash{\D} ^\dag _{\X ^\sharp} (\hdag T \cap X) _{\Q})}
\ar[r]  _-{u _{T+}}
& 
{D ^{\mathrm{b}} _{\mathrm{coh}}  ( \smash{\D} ^\dag _{\PP ^\sharp} (\hdag T) _{\Q} ),} 
}
\xymatrix @=0,4cm{
{ \underrightarrow{LD} ^{0} _{\Q, \mathrm{coh}} (\smash{\widetilde{\D}} _{\X ^\sharp} ^{(\bullet)} (T\cap X))}
\ar[r]  _-{u  ^{(\bullet)} _+}
\ar[d] ^-{\underrightarrow{\lim} } _-{\cong} 
& 
{ \underrightarrow{LD} ^{0} _{\Q, \mathrm{coh}} (\smash{\widetilde{\D}} _{\PP ^\sharp} ^{(\bullet)} (T))}
\ar[d] ^-{\underrightarrow{\lim} } _-{\cong} 
\\ 
{ D ^{0} _{\mathrm{coh}} (\smash{\D} ^\dag _{\X ^\sharp} (\hdag T \cap X) _{\Q})}
\ar[r]  _-{u _{T+}}
& 
{D ^{0} _{\mathrm{coh}}  ( \smash{\D} ^\dag _{\PP ^\sharp} (\hdag T) _{\Q} ).} 
}
\end{equation}
\end{lemm}

\begin{proof}
Comme $u$ est en particulier propre, 
on obtient alors la commutativité à isomorphisme canonique du carré de gauche.
Comme le foncteur 
$u _{T+} \colon 
\mathrm{Coh} (\smash{\D} ^\dag _{\X ^\sharp} (\hdag T \cap X) _{\Q})
\to
\mathrm{Coh}  ( \smash{\D} ^\dag _{\PP ^\sharp} (\hdag T) _{\Q} )$
est exact, 
on obtient la factorisation 
$ D ^{0} _{\mathrm{coh}} (\smash{\D} ^\dag _{\X ^\sharp} (\hdag T \cap X) _{\Q})
\to D ^{0} _{\mathrm{coh}}  ( \smash{\D} ^\dag _{\PP ^\sharp} (\hdag T) _{\Q} )$.
On obtient alors le diagramme commutatif à isomorphisme canonique près 
identique à celui de droite de \ref{imm-fer-coh-diag} où le symbole {\og $0$\fg}
du terme en haut à droite est remplacé par {\og $\mathrm{b}$\fg}.
Grâce à \ref{lemm-lim-MD2square}.\ref{lemm1-lim-MD2square}, on en déduit la factorisation 
$u  ^{(\bullet)} _+\colon 
 \underrightarrow{LD} ^{0} _{\Q, \mathrm{coh}} (\smash{\widetilde{\D}} _{\X ^\sharp} ^{(\bullet)} (T\cap X))
\to 
\underrightarrow{LD} ^{0} _{\Q, \mathrm{coh}} (\smash{\widetilde{\D}} _{\PP ^\sharp} ^{(\bullet)} (T))$
voulue.
\end{proof}

\begin{theo}
[Version logarithmique du théorème de Berthelot-Kashiwara]
\label{dagu!u+=id}
Avec les notations de \ref{nota-usharp}, on suppose que $u ^{\sharp}$ est exacte. 
Les foncteurs $u ^{\sharp !}$ et $u ^{\sharp} _{+}$ induisent des équivalences quasi-inverses entre la catégorie 
des $\smash{\D} ^\dag _{\PP ^\sharp} (\hdag T ) _{\Q} $-modules cohérents à support dans $X$ 
et celle des $\smash{\D} ^\dag _{\X ^{\sharp}} (\hdag T \cap X) _{\Q} $-modules cohérents. 
Ces foncteurs $u ^{\sharp !}$ et $u ^{\sharp} _{+}$ sont acycliques sur ces catégories. 
\end{theo}

\begin{proof}
C'est un cas particulier de \ref{Berthelot-Kashiwara-full}.
\end{proof}

\begin{theo}
[Version système inductif et logarithmique du théorème de Berthelot-Kashiwara]
\label{u!u+=id}
Avec les notations de \ref{nota-usharp}, on suppose que $u ^{\sharp}$ est exacte. 
Soient
$\FF ^{(\bullet)} 
\in \smash{\underrightarrow{LD}}  ^\mathrm{b} _{\Q, \mathrm{coh}}
(\overset{^\mathrm{g}}{} \smash{\widetilde{\D}} _{\X ^{\sharp}} ^{(\bullet)} (T \cap X))$,
$\E ^{(\bullet)}  
\in \smash{\underrightarrow{LD}}  ^\mathrm{b} _{\Q, \mathrm{coh}}
(\overset{^\mathrm{g}}{} \smash{\widetilde{\D}} _{\PP ^\sharp} ^{(\bullet)} (T ))$
tel que 
$\E ^{(\bullet)} |\U \riso 0$
dans $\smash{\underrightarrow{LD}}  ^\mathrm{b} _{\Q, \mathrm{coh}}
(\overset{^\mathrm{g}}{} \smash{\widetilde{\D}} _{\PP ^\sharp} ^{(\bullet)} (T ))$. 

\begin{enumerate}
\item On dispose de l'isomorphisme canonique dans 
$\smash{\underrightarrow{LD}}  ^\mathrm{b} _{\Q, \mathrm{coh}}
(\overset{^\mathrm{g}}{} \smash{\widetilde{\D}} _{\X ^{\sharp}} ^{(\bullet)} (T \cap X))$:
\begin{equation}
\label{u!u+=id-iso}
u ^{\sharp (\bullet)!} \circ u _{+} ^{\sharp (\bullet)} (\FF ^{(\bullet)})
\riso \FF ^{(\bullet)}.
\end{equation}

\item On a $u ^{\sharp (\bullet)!} (\E ^{(\bullet)} )\in \smash{\underrightarrow{LD}}  ^\mathrm{b} _{\Q, \mathrm{coh}}
(\overset{^\mathrm{g}}{} \smash{\widetilde{\D}} _{\X ^{\sharp}} ^{(\bullet)} (T \cap X))$ et l'on bénéficie de l'isomorphisme canonique
\begin{equation}
\label{u!u+=id-isobis}
u _{+} ^{\sharp (\bullet)} \circ  u ^{\sharp (\bullet)!}  (\E ^{(\bullet)})
\riso \E ^{(\bullet)}.
\end{equation}

\item 
\label{u!u+=id-eq-cat}
Les foncteurs $u _{+} ^{\sharp (\bullet)} $ et $u ^{\sharp (\bullet)!}  $ induisent des équivalences quasi-inverses entre
la catégorie $\smash{\underrightarrow{LD}}  ^\mathrm{b} _{\Q, \mathrm{coh}}
(\overset{^\mathrm{g}}{} \smash{\widetilde{\D}} _{\X ^{\sharp}} ^{(\bullet)} (T \cap X))$
(resp. 
$\smash{\underrightarrow{LD}}  ^0 _{\Q, \mathrm{coh}}
(\overset{^\mathrm{g}}{} \smash{\widetilde{\D}} _{\X ^{\sharp}} ^{(\bullet)} (T \cap X))$)
et la sous-catégorie pleine de 
$\smash{\underrightarrow{LD}}  ^\mathrm{b} _{\Q, \mathrm{coh}}
(\overset{^\mathrm{g}}{} \smash{\widetilde{\D}} _{\PP ^\sharp} ^{(\bullet)} (T ))$
(resp. $\smash{\underrightarrow{LD}}  ^0 _{\Q, \mathrm{coh}}
(\overset{^\mathrm{g}}{} \smash{\widetilde{\D}} _{\PP ^\sharp} ^{(\bullet)} (T ))$) des complexes
$\E ^{(\bullet)} $ tels que 
$\E ^{(\bullet)} |\U \riso 0$.
\end{enumerate}

\end{theo}

\begin{proof}
Vérifions dans un premier temps l'isomorphisme \ref{u!u+=id-iso}. 
On dispose du morphisme canonique d'adjonction 
$u ^{\sharp (\bullet)!} \circ u _{+} ^{\sharp (\bullet)} (\FF ^{(\bullet)})
\to \FF ^{(\bullet)}$. 
Il s'agit de vérifier que cette flèche est un isomorphisme. 
D'après la proposition \ref{LDiso-local}, cela est local. 
Par transitivité des morphismes d'adjonction, 
on se ramène alors au cas où $P$ est muni de coordonnées locales $t _1, \dots, t _d$ telles que $X=V (t _1)$.
Dans ce cas, on déduit de \ref{u!alg=u!} que, 
pour tout entier $l \not \in \{0,1 \}$, on a $\mathcal{H} ^l u ^{\sharp (\bullet)!}=0$ sur 
$\smash{\underrightarrow{LD}}  ^\mathrm{b} _{\Q, \mathrm{coh}}
(\overset{^\mathrm{g}}{} \smash{\widetilde{\D}} _{\PP ^\sharp} ^{(\bullet)} (T ))$.
Comme le foncteur
$u _{+} ^{\sharp (\bullet)} $ 
préserve la cohérence et est exact, 
le foncteur 
$u ^{\sharp (\bullet)!} \circ u _{+} ^{\sharp (\bullet)}$ est donc 
way-out à gauche
(cela a un sens grâce à l'équivalence de catégories
$\underrightarrow{LD} ^{\mathrm{b}}  _{\Q, \mathrm{coh}} (\smash{\widetilde{\D}} _{\X ^{\sharp}} ^{(\bullet)} (T \cap X))
\cong 
D ^{\mathrm{b}} _{\mathrm{coh}}
(\underrightarrow{LM} _{\Q} (\smash{\widetilde{\D}} _{\X ^{\sharp}} ^{(\bullet)} (T \cap X)))$
de \ref{eq-coh-lim}).
En recopiant le début de la preuve de \ref{theo-fleche-Rhomcoh}, on se ramène alors au cas où 
$\FF ^{(\bullet)}=\lambda ^* \smash{\widetilde{\D}} _{\X ^{\sharp}} ^{(\bullet)} (T \cap X)$, pour un certain $\lambda \in L$.
On dispose de l'isomorphisme canonique
$\mathcal{H} ^0 u ^{\sharp (m)!} \circ u _{+} ^{\sharp (m)} (\FF ^{(m)})
\riso \FF ^{(m)}$ (calcul classique, voir par exemple \cite{surcoh-hol}).
De plus, 
via la formule \cite[1.7.1]{caro_log-iso-hol},
on calcule que la flèche 
canonique 
$\mathcal{H} ^1 u ^{\sharp (m)!} \circ u _{+} ^{\sharp (m)} (\FF ^{(m)})
\to 
\mathcal{H} ^1 u ^{\sharp (m+1)!} \circ u _{+} ^{\sharp (m+1)} (\FF ^{(m+1)})$
est le morphisme nul (grâce à \ref{lemm-ualg-u!nivm}, 
$\mathcal{H} ^1 u ^{\sharp (m)!} =\mathcal{H} ^1 u ^{\sharp (m+1)!}= u ^*$ car le module est cohérent).
On en déduit le résultat.

Vérifions maintenant \ref{u!u+=id-isobis}. 
Posons $\E:=  \underrightarrow{\lim}~ \E ^{(\bullet)}$.
D'après la version logarithmique du théorème de Berthelot-Kashiwara de \ref{dagu!u+=id}, 
comme $\E \in D ^\mathrm{b} _\mathrm{coh} ( \smash{\D} ^\dag _{\PP ^\sharp} (\hdag T ) _{\Q} )$ 
et est à support dans $X$, 
alors
$u ^{\sharp !} (\E) \in D ^\mathrm{b} _\mathrm{coh} ( \smash{\D} ^\dag _{\X ^{\sharp}} (\hdag T \cap X) _{\Q} )$. 
Soit 
$\H ^{(\bullet)}
\in \smash{\underrightarrow{LD}}  ^\mathrm{b} _{\Q, \mathrm{coh}}
(\overset{^\mathrm{g}}{} \smash{\widetilde{\D}} _{\X ^{\sharp}} ^{(\bullet)} (T \cap X))$
tel que 
$u ^{\sharp !} (\E) 
\riso 
 \underrightarrow{\lim}~
 \H ^{(\bullet)}$.
Via \ref{imm-fer-coh}, on en déduit 
$u ^{\sharp} _+ (u ^{\sharp !} (\E) ) 
\riso  
\underrightarrow{\lim} \, u _{+} ^{\sharp (\bullet)} \H ^{(\bullet)}$.
Comme d'après le théorème de Berthelot-Kashiwara
$\E \riso u _+ ^{\sharp}  (u ^{\sharp !} (\E) ) $, comme 
$\E ^{(\bullet)} ,~u _{+} ^{\sharp (\bullet)} \H ^{(\bullet)}
\in  \smash{\underrightarrow{LD}}  ^\mathrm{b} _{\Q, \mathrm{coh}}
(\overset{^\mathrm{g}}{} \smash{\widetilde{\D}} _{\PP ^\sharp} ^{(\bullet)} (T ))$
et comme le foncteur $ \underrightarrow{\lim} \, $ est pleinement fidèle sur
$ \smash{\underrightarrow{LD}}  ^\mathrm{b} _{\Q, \mathrm{coh}}
(\overset{^\mathrm{g}}{} \smash{\widetilde{\D}} _{\PP ^\sharp} ^{(\bullet)} (T ))$, 
on en déduit $\E ^{(\bullet)} \riso u _{+} ^{\sharp (\bullet)} \H ^{(\bullet)}$.
Cela entraîne 
$u ^{\sharp (\bullet)!}  (\E ^{(\bullet)} )\riso u ^{\sharp (\bullet)!}  \circ u _{+} ^{\sharp (\bullet)} \H ^{(\bullet)} \underset{\ref{u!u+=id-iso}}{\riso}
\H ^{(\bullet)}$.
En appliquant le foncteur 
$u _{+} ^{\sharp (\bullet)}$ à ce composé, on obtient
l'isomorphisme \ref{u!u+=id-isobis} voulu.
Le troisième point est évident une fois validés les deux premiers.
\end{proof}

\begin{coro}
\label{pre-loc-tri-B-t1T}
Avec les notations de \ref{nota-usharp}, on suppose que $u ^{\sharp}$ est exacte. 
Pour tout
$\E ^{(\bullet)} \in \smash{\underrightarrow{LD}} ^{\mathrm{b}} _{\Q,\mathrm{qc}} ( \smash{\widetilde{\D}} _{\PP ^\sharp} ^{(\bullet)})$,
on dispose de l'isomorphisme 
\begin{equation}
\label{pre-loc-tri-B-t1T-iso}
\R \underline{\Gamma} ^\dag _{X } (\E ^{(\bullet)}) 
\riso 
u _{+} ^{\sharp (\bullet)} \circ  u ^{\sharp (\bullet)!} (\E ^{(\bullet)}).
\end{equation}
\end{coro}

\begin{proof}
En utilisant les isomorphismes
$$u _+ ^{\sharp (\bullet)} (\O _{\X} ^{(\bullet)})
\smash{\overset{\L}{\otimes}}   ^{\dag}
_{\O  _{\PP,\Q}} 
\E ^{(\bullet)}
[d _{X/P}]
\underset{\ref{2.1.4-caro-surcoh-iso}}{\riso}
u _+ ^{\sharp (\bullet)} (\O _{\X} ^{(\bullet)}
\smash{\overset{\L}{\otimes}}   ^{\dag}
_{\O  _{\X,\Q}} 
 u  ^{\sharp (\bullet)!}(\E ^{(\bullet)}))
\riso 
u _+ ^{\sharp (\bullet)} \circ
 u  ^{\sharp (\bullet)!}(\E ^{(\bullet)}),$$
 on se ramène au cas où 
$ \E ^{(\bullet)} = \O _{\PP} ^{(\bullet)}$. 
Comme l'immersion $u ^{\sharp}$ est exacte, 
par \ref{GammaX-oub-sharp-iso}, \ref{f!oubsharp} et \ref{com-f+-oubsharp-iso}, on peut oublier les structures logarithmiques. 
Dans ce cas, 
comme $\R \underline{\Gamma} ^\dag _{X } (\O _{\PP} ^{(\bullet)}) \in \smash{\underrightarrow{LD}}  ^\mathrm{b} _{\Q, \mathrm{coh}}
(\overset{^\mathrm{g}}{} \smash{\widetilde{\D}} _{\PP } ^{(\bullet)} (T ))$ (voir \ref{stab-coh-u!})
et est à support dans $X$, 
le théorème de Kashiwara sous la forme
\ref{u!u+=id} implique que l'on dispose de l'isomorphisme canonique 
$u _+ ^{(\bullet)} \circ
 u  ^{(\bullet)!} (\R \underline{\Gamma} ^\dag _{X } (\O _{\PP} ^{(\bullet)}) )
\riso
\R \underline{\Gamma} ^\dag _{X } (\O _{\PP} ^{(\bullet)}) $.
Comme d'après \ref{commutfonctcohlocal1}
le morphisme canonique
$ u  ^{(\bullet)!} (\R \underline{\Gamma} ^\dag _{X } (\O _{\PP} ^{(\bullet)}) )
 \to 
  u  ^{(\bullet)!} (\O _{\PP} ^{(\bullet)}) $
  est un isomorphisme, on obtient le résultat voulu. 
\end{proof}

\begin{coro}
Avec les notations de \ref{nota-usharp}, on suppose que $u ^{\sharp}$ est exacte. 
\begin{enumerate}
\item Pour tout complexe
$\E ^{(\bullet)}
\in 
\smash{\underrightarrow{LD}} ^{\mathrm{b}} _{\Q,\mathrm{surcoh},\PP} ( \smash{\widetilde{\D}} _{\PP ^\sharp} ^{(\bullet)} (T))$,
on a alors $u ^{\sharp (\bullet)!} (\E ^{(\bullet)} )\in \smash{\underrightarrow{LD}}  ^\mathrm{b} _{\Q, \mathrm{surcoh}, \X}
(\overset{^\mathrm{g}}{} \smash{\widetilde{\D}} _{\X ^{\sharp}} ^{(\bullet)} (T \cap X))$.

\item Les foncteurs $u _{+} ^{\sharp (\bullet)} $ et $u ^{\sharp (\bullet)!}  $ induisent des équivalences quasi-inverses entre
la catégorie $\smash{\underrightarrow{LD}}  ^\mathrm{b} _{\Q, \mathrm{surcoh}, \X}
(\overset{^\mathrm{g}}{} \smash{\widetilde{\D}} _{\X ^{\sharp}} ^{(\bullet)} (T \cap X))$
et la sous-catégorie pleine de 
$\smash{\underrightarrow{LD}}  ^\mathrm{b} _{\Q, \mathrm{surcoh}, \PP}
(\overset{^\mathrm{g}}{} \smash{\widetilde{\D}} _{\PP ^\sharp} ^{(\bullet)} (T ))$
des complexes
$\E ^{(\bullet)} $ tels que 
$\E ^{(\bullet)} |\U \riso 0$.
\end{enumerate}
\end{coro}

\begin{proof}
Vérifions d'abord la première assertion.
Soient $\E ^{(\bullet)}
\in 
\smash{\underrightarrow{LD}} ^{\mathrm{b}} _{\Q,\mathrm{surcoh},\PP} ( \smash{\widetilde{\D}} _{\PP ^\sharp} ^{(\bullet)} (T))$
et
$X'$ une sous-variété fermée  de $X$.
Il résulte de \ref{1ercrit-surcohP} que 
$\R \underline{\Gamma} ^\dag _{X '} (\E ^{(\bullet)})
\in \smash{\underrightarrow{LD}}  ^\mathrm{b} _{\Q, \mathrm{surcoh}, \PP}
(\overset{^\mathrm{g}}{} \smash{\widetilde{\D}} _{\PP ^\sharp} ^{(\bullet)} (T ))$.
Comme ce dernier est en particulier cohérent et nul en dehors de $X$, 
on déduit du théorème \ref{u!u+=id} que 
$u ^{\sharp (\bullet)!} (\R \underline{\Gamma} ^\dag _{X'} (\E ^{(\bullet)})  )
\in \smash{\underrightarrow{LD}}  ^\mathrm{b} _{\Q, \mathrm{coh}}
(\overset{^\mathrm{g}}{} \smash{\widetilde{\D}} _{\X ^{\sharp}} ^{(\bullet)} (T \cap X))$.
Or, il découle de 
\ref{commutfonctcohlocal1}, 
que l'on dispose de l'isomorphisme canonique 
$u ^{\sharp (\bullet)!} \circ  \R \underline{\Gamma} ^\dag _{X'} (\E ^{(\bullet)}) 
\riso 
\R \underline{\Gamma} ^\dag _{X'} \circ u ^{\sharp (\bullet)!} (\E ^{(\bullet)} )$.
Cela implique
$u ^{\sharp (\bullet)!} (\E ^{(\bullet)} )
\in \smash{\underrightarrow{LD}}  ^\mathrm{b} _{\Q, \mathrm{surcoh}, \X}
(\overset{^\mathrm{g}}{} \smash{\widetilde{\D}} _{\X ^{\sharp}} ^{(\bullet)} (T \cap X))$
(il ne faut pas oublier de d'abord vérifier la cohérence, ce qui résulte du cas particulier $X' =X$ car dans ce cas 
$\R \underline{\Gamma} ^\dag _{X} \circ u ^{\sharp (\bullet)!} =
u ^{\sharp (\bullet)!} $).
Traitons à présent l'assertion 2). 
On vérifie la stabilité de la surcohérence par 
$u _{+} ^{\sharp (\bullet)} $ grâce à \ref{commutfonctcohlocal2} (et à la stabilité de la cohérence par $u _{+} ^{\sharp (\bullet)} $).
On en déduit alors 
les équivalences quasi-inverses voulues via celles de \ref{u!u+=id}.\ref{u!u+=id-eq-cat}.

\end{proof}

\subsection{Surcohérence}
\begin{defi}
\label{surcoh}

\begin{enumerate}
\item De manière analogue à  \cite[3]{caro_surcoherent},
on définit la sous-catégorie pleine 
$D ^{\mathrm{b}}  _{\mathrm{surcoh}} (\smash{\D} ^\dag _{\PP ^\sharp} (\hdag T) _{\Q})$ 
de
$D ^{\mathrm{b}}  _{\mathrm{coh}} (\smash{\D} ^\dag _{\PP ^\sharp} (\hdag T) _{\Q})$
des complexes surcohérents $\E$, i.e.,  tels que,
pour tout morphisme lisse de la forme $f \colon \PP ' \to \PP$, 
en notant 
$\PP ^{\prime \sharp }:= (\PP', f ^{-1} (\ZZ) ) $
et
$f ^{\sharp}\colon \PP ^{\prime \sharp }\to \PP ^{\sharp}$ le morphisme induit par $f$, 
on ait 
$f ^{\sharp !}(\E )\in
D ^{\mathrm{b}}  _{\mathrm{surcoh}, \PP ^{\prime \sharp }} (\smash{\D} ^\dag _{\PP ^{\prime \sharp }} (\hdag f ^{-1}(T)) _{\Q})$.

\item De manière analogue, 
on dit que $\E ^{(\bullet)}
\in 
\smash{\underrightarrow{LD}} ^{\mathrm{b}} _{\Q,\mathrm{coh}} ( \smash{\widetilde{\D}} _{\PP ^\sharp} ^{(\bullet)} (T))$ est surcohérent si, 
pour tout morphisme lisse de la forme $f \colon \PP ' \to \PP$, 
en notant 
$\PP ^{\prime \sharp }:= (\PP', f ^{-1} (\ZZ) ) $
et
$f ^{\sharp}\colon \PP ^{\prime \sharp }\to \PP ^{\sharp}$ le morphisme induit par $f$,
on ait 
$f ^{\sharp (\bullet)!}(\E ^{(\bullet)} )\in 
\smash{\underrightarrow{LD}} ^{\mathrm{b}} _{\Q,\mathrm{surcoh}, \PP ^{\prime \sharp }} ( \smash{\widetilde{\D}} _{\PP ^{\prime \sharp }} ^{(\bullet)} (f ^{-1}(T)))$.
On note $\smash{\underrightarrow{LD}} ^{\mathrm{b}} _{\Q,\mathrm{surcoh}} ( \smash{\widetilde{\D}} _{\PP ^\sharp} ^{(\bullet)} (T))$
la sous-catégorie pleine 
de 
$\smash{\underrightarrow{LD}} ^{\mathrm{b}} _{\Q,\mathrm{coh}} ( \smash{\widetilde{\D}} _{\PP ^\sharp} ^{(\bullet)} (T))$
des complexes surcohérents.

\end{enumerate}

\end{defi}

\begin{rema}
\begin{enumerate}
\item Dans la définition \ref{surcoh} lorsque $T$ est vide, 
pour garantir que $\O _{\PP,\Q}$ soit bien surcohérent,
il est important que le morphisme lisse $f ^{\sharp}$ soit exact.  
En effet, si $\ZZ$ est un diviseur lisse de $\PP$ et si 
$\PP ^{\sharp}:= (\PP,\ZZ)$, alors 
$\O _{\PP} (\hdag Z) _{\Q}$ n'est pas a priori 
$\smash{\D} ^\dag _{\PP ^\sharp, \Q}$-cohérent.

\item Avec la remarque précédente, la notion de surcohérence est seulement intéressante sans structure logarithmique.
 Par commodité de rédaction nous resterons cependant dans cette sous-section dans le cas général logarithmique. 
\end{enumerate}
\end{rema}

\begin{prop}
\label{surcoh=surcoh}
Soit $\E ^{(\bullet)}
\in 
\underrightarrow{LD} ^{\mathrm{b}} _{\Q, \mathrm{coh}} (\smash{\widetilde{\D}} _{\PP ^\sharp} ^{(\bullet)} (T))$.
La propriété 
$\E ^{(\bullet)}
\in 
\smash{\underrightarrow{LD}} ^{\mathrm{b}} _{\Q,\mathrm{surcoh}} ( \smash{\widetilde{\D}} _{\PP ^\sharp} ^{(\bullet)} (T))$
est équivalente à la propriété 
$\underrightarrow{\lim}~
\E ^{ (\bullet)}
\in
D ^{\mathrm{b}}  _{\mathrm{surcoh}} (\smash{\D} ^\dag _{\PP ^\sharp} (\hdag T) _{\Q})$.
Ces deux propriétés sont locales en $\PP$.
On dispose de l'équivalence de catégories
$\underrightarrow{\lim}
\colon 
\smash{\underrightarrow{LD}} ^{\mathrm{b}} _{\Q,\mathrm{surcoh}} ( \smash{\widetilde{\D}} _{\PP ^\sharp} ^{(\bullet)} (T))
\cong 
D ^{\mathrm{b}}  _{\mathrm{surcoh}} (\smash{\D} ^\dag _{\PP ^\sharp} (\hdag T) _{\Q})$.
\end{prop}

\begin{proof}
Cela résulte de la proposition \ref{surcohP=surcohP}
et du fait que, pour tout morphisme exact et lisse $f ^{\sharp}$ de $\V$-schémas formels logarithmiques lisses, 
 le foncteur $f ^{\sharp (\bullet)!}$ préservent la cohérence. 
\end{proof}

\begin{nota}
\label{nota5.4.4}
Soient $f \colon \PP' \to \PP$ un morphisme de $\V$-schémas formels lisses, 
$T$ et $T'$ des diviseurs respectifs de $P$ et $P'$ tels que
$T'= f ^{-1} (T)$,
$\ZZ$ et $\ZZ'$ des diviseurs à croisements normaux stricts de 
respectivement $\PP$ et $\PP'$ tels que $f ^{-1} (\ZZ) \subset \ZZ'$.
On pose $\PP ^{\sharp}:= (\PP, \ZZ)$, $\PP ^{\prime \sharp}:= (\PP', \ZZ')$
et
$f ^{\sharp}
\colon 
\PP ^{\prime \sharp}
\to 
\PP ^{\sharp}$ le morphisme de $\V$-schémas formels logarithmiques induits par $f$. 

\end{nota}

\begin{theo}
\label{3.1.7bis}
Avec les notations \ref{nota5.4.4}, 
on suppose que
$f$ est strict.
Pour tout $\E ^{(\bullet)} 
\in \smash{\underrightarrow{LD}} ^{\mathrm{b}} _{\Q,\mathrm{surcoh}} ( \smash{\widetilde{\D}} _{\PP ^\sharp} ^{(\bullet)} (T))$, 
on a alors 
$f ^{\sharp (\bullet)!}(\E ^{(\bullet)} )\in 
\smash{\underrightarrow{LD}} ^{\mathrm{b}} _{\Q,\mathrm{surcoh}} ( \smash{\widetilde{\D}} _{\PP ^{\prime \sharp}} ^{(\bullet)} (T'))$.
\end{theo}

\begin{proof}
Il s'agit de calquer la preuve du théorème analogue de \cite[3.1.7]{caro_surcoherent}.
\end{proof}

\begin{theo}
\label{theo-iso-chgtbase}
Soient $f \colon \X ' \to \X$, $g \colon \Y \to \X$ deux morphismes de $\V$-schémas formels lisses
tels que 
$g$ soit lisse et $f$ soit propre. 
On se donne $\ZZ$ (resp. $\ZZ'$) des diviseurs à croisements normaux stricts de 
$\X$ (resp. $\X'$) tels que $f ^{-1} (\ZZ) \subset \ZZ'$. 
On pose $\Y ' := \X ' \times _{\X} \Y$, 
$f ' \colon \Y ^{\prime } \to \Y$, $g ' \colon \Y ^{\prime}\to \X ^{\prime}$ 
les deux projections canoniques.
On note
$\T:= g ^{-1} (\ZZ)$, $\T ':= g ^{\prime -1} (\ZZ')$,
$\X ^\sharp= (\X ,\ZZ)$ (resp. $\X ^{\prime \sharp }:= (\X ',\ZZ')$, $\Y ^{\sharp}:= (\Y ,\T)$, $\Y^{\prime \sharp }:= (\Y ',\T')$)
les $\V$-schémas formels logarithmiques lisses induits
et
$f ^{\sharp} \colon \X ^{\prime \sharp } \to \X ^{\sharp}$, $g ^{\sharp} \colon \Y ^{\sharp}\to \X ^{\sharp}$,
$f ^{\prime \sharp } \colon \Y ^{\prime \sharp } \to \Y ^{\sharp}$ et $g ^{\prime \sharp } \colon \Y ^{\prime \sharp } \to \X ^{\prime \sharp }$ 
les deux morphismes de log-$\V$-schémas formels lisses induits. 
Soit
$\E ^{\prime (\bullet)}
\in  \smash{\underrightarrow{LD}} ^\mathrm{b} _{\Q, \mathrm{coh}}
(\overset{^\mathrm{g}}{} \smash{\widetilde{\D}} _{\X ^{\prime \sharp}} ^{(\bullet)})$. 
Il existe alors un isomorphisme canonique dans 
$\smash{\underrightarrow{LD}} ^\mathrm{b} _{\Q, \mathrm{coh}}
(\overset{^\mathrm{g}}{} \smash{\widetilde{\D}} _{\X ^{\prime \sharp}} ^{(\bullet)})$:
\begin{equation}
\label{iso-chgtbase}
g ^{\sharp (\bullet)!} \circ f ^{\sharp(\bullet)}_{+} (\E ^{\prime (\bullet)})
\riso
f  ^{\prime \sharp(\bullet)}_{+}  \circ g ^{\prime \sharp(\bullet) !} (\E ^{\prime (\bullet)}). 
\end{equation}
\end{theo}

\begin{proof}
Avec les hypothèses sur $f$ et $g$, il résulte de 
\ref{stab-coh} que les complexes de \ref{iso-chgtbase} sont bien cohérents.
Le morphisme $g$ (resp. $g'$) se décompose en son graphe $\gamma$ (resp. $\gamma'$) 
suivi de la projection canonique $\pi \colon  \X \times \Y \to \X$
(resp. $\pi '\colon  \X '\times \Y \to \X'$).
Posons 
$\U : = \X \times \Y $, $\U ': = \X '\times \Y $.
On note $\pi ^{\sharp} \colon \U^{\sharp}  \to \X ^{\sharp}$,
$\gamma ^{\sharp}\colon \Y ^\sharp \to \U ^{\sharp}$ les morphismes stricts de log-schémas formels 
dont les morphismes de schémas formels sont $\pi$ et $\gamma$ ; de même avec des primes.
On note $f ^{\prime \prime \sharp}= f ^{\sharp} \times \mathrm{id} _\Y \colon  \U ^{\prime \sharp} \to  \U ^{\sharp}$ le morphisme 
canoniquement induit. 
Grâce au théorème \ref{u!u+=id},
on se ramène à vérifier que l'on dispose 
d'un isomorphisme canonique de la forme
  $\gamma ^{\sharp (\bullet)} _{+}  \circ  g ^{\sharp (\bullet)!} \circ f ^{\sharp(\bullet)}_{+} (\E ^{\prime (\bullet)})
\riso 
\gamma ^{\sharp (\bullet)} _{+}  \circ  f  ^{\prime \sharp(\bullet)}_{+}  \circ g ^{\prime \sharp(\bullet) !} (\E ^{\prime (\bullet)})$.
Concernant le terme de gauche, par transitivité des images inverses extraordinaires, 
on obtient l'isomorphisme canonique
  $\gamma ^{\sharp (\bullet)} _{+}  \circ  g ^{\sharp (\bullet)!} \circ f ^{\sharp(\bullet)}_{+} (\E ^{\prime (\bullet)})
\riso 
\gamma ^{\sharp (\bullet)} _{+}  \circ  
\gamma ^{\sharp (\bullet)!} \circ
\pi ^{\sharp (\bullet)!} \circ
 f ^{\sharp(\bullet)}_{+} (\E ^{\prime (\bullet)})
 $.
Pour celui de droite, par transitivité des images directes ou images inverses extraordinaires et 
en utilisant \ref{commutfonctcohlocal2} et \ref{pre-loc-tri-B-t1T-iso} on obtient les isomorphismes canoniques 
$\gamma ^{\sharp (\bullet)} _{+}  \circ  f  ^{\prime \sharp(\bullet)}_{+}  \circ g ^{\prime \sharp(\bullet) !} (\E ^{\prime (\bullet)})
 \riso 
f  ^{\prime \prime \sharp(\bullet)}_{+}  \circ  \gamma ^{\prime \sharp (\bullet)} _{+}   
 \circ \gamma ^{\prime \sharp(\bullet) !} 
 \circ \pi ^{\prime \sharp(\bullet) !} 
 (\E ^{\prime (\bullet)})
 \riso
\gamma ^{\prime \sharp (\bullet)} _{+}   
 \circ \gamma ^{\prime \sharp(\bullet) !}
 \circ   
 f  ^{\prime \prime \sharp(\bullet)}_{+}  
 \circ \pi ^{\prime \sharp(\bullet) !} 
 (\E ^{\prime (\bullet)})
 $.
 On se ramène ainsi à vérifier l'isomorphisme
 $\pi ^{\sharp (\bullet)!} \circ
 f ^{\sharp(\bullet)}_{+} (\E ^{\prime (\bullet)})
 \riso
 f  ^{\prime \prime \sharp(\bullet)}_{+}  
 \circ \pi ^{\prime \sharp(\bullet) !} 
 (\E ^{\prime (\bullet)})$.

De manière analogue à \cite[2.4.2]{Beintro2}, pour tout entier $i \in \N$, on vérifie que l'on dispose dans 
$D^\mathrm{b} _{\mathrm{qc}}
(\smash{\D} _{ U ^{\sharp}_i} ^{(m)})$ de l'isomorphisme canonique
\begin{equation}
\label{pre-adj-morph-gen-Rham3}
\pi _i ^{\sharp (m)!}  \circ f ^{\sharp (m)} _{i+}( \E _i^{\prime (m)})
\riso 
f ^{\prime \prime  \sharp  (m)} _{i+}  \circ \pi _i ^{\prime \sharp  (m) !} (\E _i^{\prime (m)}).
\end{equation}
Par soucis pour le lecteur, donnons les arguments principaux
de Berthelot (pour prouver \cite[2.4.2]{Beintro2}, mais il n'y a rien à changer)
dans son cours de 1997 sur les $\D$-modules arithmétiques. 
Comme les foncteurs sont way-out à gauche, comme tout $\smash{\D} _{ X ^{\prime \sharp} _i} ^{(m)}$-module
est un quotient d'un module de la forme 
$\smash{\D} _{ X ^{\prime \sharp} _i} ^{(m)} \otimes _{\O _{ X ^{\prime } _i}} \M$
avec $\M$ un $\O _{ X ^{\prime} _i}$-module, on se ramène à supposer que 
$\E _i^{\prime (m)}=\smash{\D} _{ X ^{\prime \sharp} _i} ^{(m)} \otimes _{\O _{ X ^{\prime}_i}} \FF _i ^{\prime (m)}$
pour un certain
$\O _{X ^{\prime}_i}$-module $\FF _i ^{\prime (m)}$.
On vérifie alors les isomorphismes canoniques
\begin{gather}
\label{f_+-chgb}
f ^{\sharp (m)} _{i+}( \E _i^{\prime (m)})
\riso
\smash{\D} _{ X ^{ \sharp} _i} ^{(m)}
 \otimes _{\O _{ X  _i}}
 \R f _{*}( \omega _{X ^{\prime \sharp}_i /X ^{ \sharp}_i}
 \otimes _{\O _{ X ^{\prime }_i}} \FF _i ^{\prime (m)}),
\\
\label{g^!-chgb}
\pi _i ^{\prime \sharp  (m) !} (\E _i^{\prime (m)})
 \riso
 \pi _i ^{\prime *} \smash{\D} _{ X ^{\prime \sharp} _i} ^{(m)}
 \otimes _{\O _{ U ^{\prime} _i}}
 \pi _i ^{\prime *} (\FF _i^{\prime (m)}) [d _{U'/X'}].
\end{gather}
Comme pour \cite[2.3.1]{Beintro2}, on dispose de l'isomorphisme 
de $\smash{\D} _{ U ^{\prime \sharp} _i} ^{(m)} $-modules à gauche
$\smash{\D} _{ U ^{\prime \sharp} _i} ^{(m)} 
\riso 
 \pi _i ^{\prime *} \smash{\D} _{ X ^{\prime \sharp} _i} ^{(m)} 
\otimes _{\O _{U _i}}
f _i ^{\prime * }\varpi _i ^{*} \smash{\D} _{ Y ^{\prime \sharp} _i} ^{(m)} $.
Le morphisme  
$ \pi _i ^{\prime *} \smash{\D} _{ X ^{\prime \sharp} _i} ^{(m)} 
\to 
\smash{\D} _{ U ^{\prime \sharp} _i} ^{(m)}$ induit est 
 $\smash{\D} _{ U ^{\prime \sharp} _i} ^{(m)} $-linéaire et fait de
$ \pi _i ^{\prime *} \smash{\D} _{ X ^{\prime \sharp} _i} ^{(m)} $ une sous-algèbre
de  $\smash{\D} _{ U ^{\prime \sharp} _i} ^{(m)} $.
De même, 
$f ^{\prime \prime  *} \smash{\D} _{ U ^{\sharp} _i} ^{(m)} 
\riso 
 \pi _i ^{\prime *}  f _i ^{*} \smash{\D} _{ X ^{\sharp} _i} ^{(m)} 
\otimes _{\O _{U _i}}
f _i ^{\prime * }\varpi _i ^{*} \smash{\D} _{ Y ^{\prime \sharp} _i} ^{(m)} $.
Le morphisme canonique 
$\smash{\D} _{ U ^{\sharp} _i} ^{(m)} 
\to
\pi _i ^{*} \smash{\D} _{ X ^{\sharp} _i} ^{(m)}  
$ 
induit alors
$$(\omega _{U ^{\prime \sharp}_i /U ^{ \sharp}_i} \otimes _{\O _{U ' _i }}f ^{\prime \prime  *} \smash{\D} _{ U ^{\sharp} _i} ^{(m)})
\otimes ^{\L}_{\smash{\D} _{ U ^{\prime \sharp} _i} ^{(m)} } 
 \pi _i ^{\prime *} \smash{\D} _{ X ^{\prime \sharp} _i} ^{(m)} 
\to 
\omega _{U ^{\prime \sharp}_i /U ^{ \sharp}_i} \otimes _{\O _{U ' _i }}f ^{\prime \prime  *} 
\pi _i ^{*} \smash{\D} _{ X ^{\sharp} _i} ^{(m)}  .
$$
Par un calcul en coordonnées locales, 
on vérifie que ce morphisme est un isomorphisme.
Via \ref{g^!-chgb}, on en déduit l'isomorphisme
\begin{gather}
\notag
f ^{\prime \prime  \sharp  (m)} _{i+}  \circ \pi _i ^{\prime \sharp  (m) !} (\E _i^{\prime (m)})
\riso 
\R f '_{*}(
\omega _{U ^{\prime \sharp}_i /U ^{ \sharp}_i} \otimes _{\O _{U ' _i }}f ^{\prime \prime  *} 
\pi _i ^{*} \smash{\D} _{ X ^{\sharp} _i} ^{(m)}  
 \otimes _{\O _{ U ^{\prime }_i}} 
  \pi _i ^{\prime *} (\FF _i^{\prime (m)}))
   [d _{U'/X'}]
   \\
   \notag
   \riso
\pi _i ^{*} \smash{\D} _{ X ^{\sharp} _i} ^{(m)} 
 \otimes _{\O _{ U _i}} 
\R f '_{*}(
\omega _{U ^{\prime \sharp}_i /U ^{ \sharp}_i} \otimes _{\O _{U ' _i }}
  \pi _i ^{\prime *} (\FF _i^{\prime (m)}))   [d _{U'/X'}].
\end{gather}
On déduit de  \ref{f_+-chgb} l'isomorphisme
$\pi _i ^{\sharp  (m) !} f ^{\sharp (m)} _{i+}( \E _i^{\prime (m)})
 \riso 
\pi _i ^{*}  (\smash{\D} _{ X ^{ \sharp} _i} ^{(m)}) 
 \otimes _{\O _{ U  _i}}
\pi _i ^{*} \R f _{*}( \omega _{X ^{\prime \sharp}_i /X ^{ \sharp}_i}
 \otimes _{\O _{ X ^{\prime }_i}} \FF _i ^{\prime (m)})) [d _{U'/X'}]$.
  Comme $\pi _i$ est plat, 
 alors 
 $  \pi _i ^{ *} \circ \R f _{*}( \omega _{X ^{\prime \sharp}_i /X ^{ \sharp}_i}
 \otimes _{\O _{ X ^{\prime }_i}} \FF _i ^{\prime (m)})
 \riso
 \R f '_{*}  \circ \pi _i ^{ \prime *} ( \omega _{X ^{\prime \sharp}_i /X ^{ \sharp}_i}
 \otimes _{\O _{ X ^{\prime }_i}} \FF _i ^{\prime (m)})
 \riso 
 \R f '_{*}( \omega _{U ^{\prime \sharp}_i /U ^{ \sharp}_i}
 \otimes _{\O _{ U ^{\prime }_i}} 
  \pi _i ^{\prime *} (\FF _i^{\prime (m)}))$.
D'où le résultat. 
Par passage à la limite projective et compatibilité au changement de niveaux
de ces isomorphes \ref{pre-adj-morph-gen-Rham3}, on obtient le théorème.
\end{proof}

\begin{rema}
Lorsque l'on ne dispose pas de structure logarithmique, 
T. Abe a prouvé que l'isomorphisme de changement de base
\ref{iso-chgtbase} est compatible à Frobenius
(voir \cite[5.7]{Abe-Frob-Poincare-dual}).
\end{rema}

\begin{theo}
\label{3.1.9bis}
Avec les notations \ref{nota5.4.4}, 
on suppose $f$ propre.
Pour tout $\E ^{\prime (\bullet)} 
\in \smash{\underrightarrow{LD}} ^{\mathrm{b}} _{\Q,\mathrm{surcoh}} ( \smash{\widetilde{\D}} _{\PP ^{\prime \sharp}} ^{(\bullet)} (T'))$, 
on a alors 
$f ^{\sharp (\bullet)} _{+}(\E ^{\prime(\bullet)} )\in 
\smash{\underrightarrow{LD}} ^{\mathrm{b}} _{\Q,\mathrm{surcoh}} ( \smash{\widetilde{\D}} _{\PP ^\sharp} ^{(\bullet)} (T))$.
\end{theo}

\begin{proof}
Il s'agit de calquer la preuve du théorème analogue de \cite[3.1.9]{caro_surcoherent}, 
i.e. cela découle du théorème 
\ref{commutfonctcohlocal2}
et du théorème de changement de base 
\ref{theo-iso-chgtbase}.
\end{proof}

\begin{rema}
\label{rema-interet papier}
Soit $\E ^{(\bullet)}
\in 
\underrightarrow{LD} ^{\mathrm{b}} _{\Q, \mathrm{qc}} (\smash{\widetilde{\D}} _{\PP ^\sharp} ^{(\bullet)} (T))$.
Si 
$\underrightarrow{\lim}~
\E ^{ (\bullet)}
\in
D ^{\mathrm{b}}  _{\mathrm{coh}} (\smash{\D} ^\dag _{\PP ^\sharp} (\hdag T) _{\Q})$, 
il semble faux que cela implique
que 
$\E ^{(\bullet)}
\in 
\smash{\underrightarrow{LD}} ^{\mathrm{b}} _{\Q,\mathrm{coh}} ( \smash{\widetilde{\D}} _{\PP ^\sharp} ^{(\bullet)} (T))$.
Même sans structure logarithmique, 
la stabilité de la surcohérence par images directes ou images inverses extraordinaires
pour les systèmes inductifs est donc un outil plus puissant (et manipulable) que celle de la surcohérence définie dans
\cite{caro_surcoherent}. La notion de surcohérence pour les systèmes inductifs apparaît avec du recul comme la bonne notion.

\end{rema}

\numberwithin{prop}{section}
\appendix

\section{Théorème de Berthelot-Kashiwara avec log-structures et coefficients}
 
 Le but de cette section est le théorème de Berthelot-Kashiwara sous la forme
 \ref{Berthelot-Kashiwara-full}.  
 Comme ce théorème est donné dans un contexte log-géométrique un peu plus général, 
nous nous permettons de nous écarter dans cette section des notations générales.

\begin{vide}[Notations]
\label{sectionA.1}
Soient $n\geq r$ deux entiers naturels. 
On désigne par $A ^{n,r}$ le log-schéma dont l'espace sous-jacent est  $\A _\Z ^n= \Spec \Z [ t _1,\dots, t _n]$ et tel qu'une pré-log-structure soit
donné par  $\N ^{r} \to \Z [t _1,\dots, t _n]$ définie par $e _i \mapsto t _i$ pour $i= 1,\dots, r$.
Soit $\mathfrak{A} ^{n,r}$ le $\Z _p$-schéma formel égale à la complétion $p$-adique de  $A ^{n,r}$.

Soient  $T$ un log-schéma fin
et  $\T$ un log-$\V$-schéma formel.
On pose  $A ^{n,r} _T :=A ^{n,r} \times _{\Spec \Z} T$ et 
$\mathfrak{A} _\T  ^{n,r}:= \mathfrak{A} ^{n,r} \times _{\Spf \Z _p} \T$.
On dispose des immersions fermées exactes canoniques:  
$A ^{r,r} _T  \hookrightarrow A ^{n,r} _T $
et
$\mathfrak{A} _\T^{r,r}
\hookrightarrow 
\mathfrak{A} _\T ^{n,r}$
données par $t _{r+1}= 0, \dots, t _n=0$.

Soit $m\in \N \cup \{ \infty\}$.
On dispose des inclusions 
$\D ^{(m)} _{A ^{n,0} _T  /T}\subset \D ^{(m)} _{A ^{n,r} _T  /T} \subset \D ^{(m)} _{A ^{n,n} _T  /T}$
(lorsque $m=\infty$, la convention de Berthelot donnée dans \cite{Be2} signifie les opérateurs différentiels usuels). 
On note $\partial _1, \dots, \partial _n$ (resp. $\partial _{\sharp ,1}, \dots, \partial _{\sharp ,n}$)
les dérivations (resp. les dérivations logarithmiques) correspondant aux
coordonnées locales $t _1,\dots, t _n$ de $A ^{n,0} _T  /T$ (resp. coordonnées locales logarithmiques $t _1,\dots, t _n$  de $A ^{n,n} _T  /T$).
Le $\O _{\A ^{n} _T }$-module (pour les structures droite ou gauche)
$\D ^{(m)} _{A ^{n,0} _T  /T}$ 
(resp. $\D ^{(m)} _{A ^{n,n} _T  /T}$) 
est  libre 
de base $\underline{\partial} _{} ^{< \underline{i}> _{(m)}}$
(resp. $\underline{\partial} _{\sharp} ^{< \underline{i}> _{(m)}}$)
avec $\underline{i} \in \N ^n$.
On a la relation $\underline{\partial} _{\sharp} ^{< \underline{i}> _{(m)}}
 = \underline{t} ^{\underline{i}}\underline{\partial} _{} ^{< \underline{i}> _{(m)}}$, 
 où 
 $\underline{t} ^{\underline{i}}:= t _1 ^{i _1}\cdots t _n ^{i _n}$ pour 
 $(i _1,\dots, i _n)=\underline{i}$.
  Pour tout $(i _1,\dots, i _n)=\underline{i} \in \N ^n$, 
on pose
$\underline{\partial} _{(r)} ^{< \underline{i}> _{(m)}}:=
\underline{\partial} _{\sharp} ^{< (i_1,\dots, i _r, 0,\dots, 0)> _{(m)}}\underline{\partial} _{} ^{<(0,\dots, 0,i _{r+1},\dots, i _n)> _{(m)}}$.
 Le $\O _{\A ^{n} _T }$-module (pour les structures droite ou gauche)
$\D ^{(m)} _{A ^{n,r} _T  /T}$ 
est  libre 
de base $\underline{\partial} _{(r)} ^{< \underline{i}> _{(m)}}$
avec $\underline{i} \in \N ^n$.

 On remarque alors qu'un $T$-morphisme 
$f\colon X \to A ^{n,r} _T $
équivaut à la donnée de sections globales
$e _1, \dots , e _r$  de  $M _X$ (le faisceau de la log-structure de $X$) et 
$a _{r+1}, \dots , a _n$ de $\O _X$. 
Si $f$ est log-étale alors 
$\D ^{(m)} _{X /T} =f ^* \D ^{(m)} _{A ^{n,n} _T  /T}$ et l'on dispose du morphisme d'anneaux
$f ^{-1}\D ^{(m)} _{A ^{n,n} _T  /T} \to \D ^{(m)} _{X /T}$. 
On notera encore $\underline{\partial} _{(r)} ^{< \underline{i}> _{(m)}}$
les sections globales de
$\D ^{(m)} _{X /T}$ égales à l'image de $\underline{\partial} _{(r)} ^{< \underline{i}> _{(m)}}$ via le morphisme
$f ^{-1}\D ^{(m)} _{A ^{n,n} _T  /T} \to \D ^{(m)} _{X /T}$.
Lorsque $m=\infty$, on note plutôt 
$\underline{\partial} _{(r)} ^{[ \underline{i}]}:= \underline{\partial} _{(r)} ^{< \underline{i}> _{(\infty)}}$.
\end{vide}

\begin{lemm}
\label{SA1II.4.6}
Soit  $f \colon X \to Y$ un morphisme de $T$-log-schémas lisses. 
Le morphisms $f$ est log-étale si et seulement si le morphisme canonique
$f ^{*} \Omega _{Y/T} ^{1} \to \Omega _{X/T} ^{1}$ est un isomorphisme. 
\end{lemm}

\begin{proof}
D'après \cite[IV.3.2.3.2 et IV.3.1.3]{Ogus-Logbook}, 
si le morphisme canonique
$f ^{*} \Omega _{Y/T} ^{1} \to \Omega _{X/T} ^{1}$ est un isomorphisme alors
$f$ est log-étale.
Réciproquement, si  $f$ est log-étale alors  $f$ est log-lisse et alors
 (en utilisant \cite[IV.3.2.3.1]{Ogus-Logbook}) $f ^{*} \Omega _{Y/T} ^{1} \to \Omega _{X/T} ^{1}$ est injective. 
Puisque $f$ est non ramifié alors 
$\Omega _{X/Y} ^{1}=0$ (voir \cite[IV.3.1.3]{Ogus-Logbook}).
Il en résulte que  $f ^{*} \Omega _{Y/T} ^{1} \to \Omega _{X/T} ^{1}$ est surjective. 

\end{proof}

\begin{lemm}
\label{SA1II.4.8}
Soit
$f \colon X \to A ^{n,r} _T $ un morphisme de  $T$-log-schémas lisses donné
par les sections globales
$e _1, \dots , e _r$ de $M _X$ et $a _{r+1}, \dots , a _n$ de $\O _X$. 
Le morphisme $f$ est log-étale si et seulement si 
le $\O _X$-module 
$\Omega ^{1} _{X /T}$ est libre et si 
$(dlog e _{1}, \dots, dlog e _{r},
d a _{r+1}, \dots , d a _n)$
en forme une base.

\end{lemm}

\begin{proof}
D'après \ref{SA1II.4.6}, $f$ est log-étale si et seulement si 
le morphisme canonique 
$f ^{*} \Omega _{A ^{n,r} _T/T} ^{1} \to \Omega _{X/T} ^{1}$
est un isomorphisme.
Cette dernière propriété équivaut au fait que le 
 $\O _X$-module 
$\Omega ^{1} _{X /T}$ est libre et que 
$(dlog e _{1}, \dots, dlog e _{r}, d a _{r+1}, \dots , d a _n)$
en soit une base (pour calculer explicitement $ \Omega _{A ^{n,r} _T/T} ^{1}$, 
il vaut mieux utiliser la remarque \cite[IV.1.1.8]{Ogus-Logbook}).

\end{proof}

\begin{lemm}
\label{Log-SGA1-II.4.10}
Soit $u \colon Z \hookrightarrow X$ une immersion fermée de log-schémas lisses sur $T$.
Soit $x$ un point de $|Z|$. 
Alors, en remplaçant $X$ par un ouvert contenant $x$ si nécessaire, 
il existe des entiers $n \geq r$ et un diagramme cartésien de morphismes de $T$-log-schémas de la forme: 
$$\xymatrix{
{X} 
\ar[r] ^-{}
\ar@{}[rd] ^-{}|\square
& {A ^{n,r} _T } 
\\ 
{Z} 
\ar@{^{(}->}[u] ^-{u}
\ar[r] ^-{}
& 
{ A ^{r,r} _T} 
\ar@{^{(}->}[u] ^-{}
}
$$
tel que les morphismes horizontaux sont log-étale et
le morphisme de droite est l'immersion fermée stricte   canonique.
De même, en remplaçant {\og log-schémas\fg} par {\og log-schémas formels \fg},
on obtient une version analogue. 
\end{lemm}

\begin{proof}
La preuve est analogue à celle du théorème \cite[II.4.10]{sga1} :
Puisque $u$ est une immersion fermée stricte, 
en notant  $\I$ l'idéal définissant $u$,
d'après \cite[IV.3.2.2]{Ogus-Logbook}
on obtient la suite exacte
\begin{equation}
\label{IV.3.2.2Logbook}
0 
\to 
\I /\I ^{2}
\to 
u ^{*} \Omega ^{1} _{X/T}
\to 
 \Omega ^{1} _{Z/T}
 \to
 0.
\end{equation}
Puisque $Z$ et $X$ sont log-lisse sur $T$, 
d'après \cite[IV.3.2.1]{Ogus-Logbook},
$ \Omega ^{1} _{X/T}$ 
et  
$\Omega ^{1} _{Z/T}$
sont localement libre. 
En regardant la description de 
$\Omega ^{1} _{Z/T}$  de
\cite[IV,1.1.6]{Ogus-Logbook}, 
en remplaçant $X$ par un ouvert contenant $x$ si nécessaire (et $Z$ par la trace de cet ouvert sur $Z$),
on vérifie que 
$\Omega ^{1} _{Z/T}$ est engendré comme $\O _Z$-module on vérifie que 
les éléments de la forme $d log a$ avec $a$ une section globale  de $M _Z$.
Ainsi, (en remplaçant $X$ par un ouvert contenant $x$ si nécessaire)
on trouve des sections globales $a _{1},\dots, a _r$ de $M _X$ telles que
$d log \overline{a} _{1}
,\dots, d log \overline{a}  _r$ forment une base de 
$\Omega ^{1} _{Z/T}$, $\overline{a}$ signifiant l'image de $a$ via la surjection 
$u ^{-1} M _X \to M _Z$.
Via la suite exacte 
\ref{IV.3.2.2Logbook}, on obtient 
des sections globales
$a _{r+1}, \dots, a_{n}$  de $\I$ telles que
$1 \otimes d log a _{1},\dots, 1 \otimes dlog  a _{r}, 
1 \otimes d a  _{r+1} ,\dots, 1 \otimes d a  _n$ forment une base de 
$u ^{*}\Omega ^{1} _{X/T}$.
Comme $\Omega ^{1} _{X/T}$ est un $\O _X$-module localement libre, 
quitte à remplacer $X$ par un ouvert contenant $x$, 
on peut alors supposer que 
$d log a _{1},\dots, 1 \otimes dlog  a _{r}, 
d a  _{r+1} ,\dots, d a  _n$ forment une base de 
$\Omega ^{1} _{X/T}$.
Les sections $a _{1}, \dots, a _{n}$ (resp. $\overline{a} _{1}
,\dots, \overline{a}  _r$)
induisent  un morphisme de $T$-log-schémas lisses de la forme
$f \colon X \to A ^{n,r} _T $
(resp. 
$f '\colon Z \to A ^{r,r} _T $).
D'après \ref{SA1II.4.8}, $f$ et $f'$ sont log-étales.
Soit le diagramme
$$\xymatrix{
{}
&
{X} 
\ar[r] ^-{f}
\ar@{}[rd] ^-{}|\square
& {A ^{n,r} _T } 
\\ 
{Z} 
\ar@{^{(}->}[ur] ^-{u}
\ar@{.>}[r] ^-{}
\ar@/_0,3cm/[rr] _-{f'}
& 
{Z'} 
\ar@{^{(}->}[u] ^-{}
\ar[r] ^-{}
&
{ A ^{r,r} _T} 
\ar@{^{(}->}[u] ^-{}
}
$$
dont les trois immersions fermées sont strictes et dont le carré est cartésien. 
Puisque $f'$ et $f$ sont log-étales, 
il en est alors de même de  $Z \to Z'$. Puisque les immersions fermées sont exactes, le morphisme $Z \to Z'$ est en outre strict.
Le morphisme sous-jacents de schémas $\underline{Z} \to \underline{Z'} $
est alors une immersion fermée étale, et donc une immersion ouverte.
Quitte à rétrécir $X$ si nécessaire, on peut alors conclure. 
\end{proof}

 \begin{vide}
 Sot $\X$ un $\T$-log-schéma formel log lisse.
  Soit $\B$ une  $\O _{\X}$-algèbre commutative.
 On dira que $\B$ est complètement cohérente s'il satisfait les conditions suivantes: 
  \begin{enumerate}[(i)]
  \item Pour n'importe quel entier $i\geq 0$, $\B _i := \B / \pi ^{i+1}\B$ est un 
  $\O _{X _i}$-module quasi-cohérent et  l'homomorphisme canonique 
  $\B \to \underleftarrow{\lim} \, _i \B _i$ est un isomorphisme ; 
  \item pour n'importe quel ouvert affine $\U \subset \X$, l'anneau
  $\Gamma (\U, \B)$ est noetherian. 
  \end{enumerate}

Soit $\B$ une  $\O _{\X}$-algèbre commutative complètement cohérente
munie d'une structure compatible de $\D _{\X/\T} ^{(m)}$-module.
Alors, d'après \cite[3.3]{Be1},
l'anneau $  \B \widehat{\otimes} _{\O _\X}\D _{\X/\T} ^{(m)}$ 
est cohérent et satisfait les théorèmes  de type $A$ et $B$. 
Soit $\E$ 
un $  \B \widehat{\otimes} _{\O _\X}\D _{\X/\T} ^{(m)}$-module coherent (à gauche).
En utilisant le théorème de type $A$, on vérifie que l'ensemble $\U$ de $\X$ des éléments $x$ tels que $\E _x = 0$ est ouvert. 
Le support de $\E$ est par définition le complémentaire de $\U$ dans $\X$. 
Ainsi $\E |\U =0$.
 \end{vide}

\begin{theo}
[Berthelot]
\label{Berthelot-Kashiwara}
Soit $u \colon \ZZ \hookrightarrow \X$ une immersion fermée exacte de 
$\T$-log-schémas formels log lisses. Soit $\Y$ l'ouvert de $\X$ complémentaire à 
 l'espace topologique sous-jacent de $\ZZ$. 
Soit $(\B ^{(m)}) _{m\in \N}$ un système inductif de $\O _{\X}$-algèbres commutatives complètement cohérentes. 
On suppose que $\B ^{(m)}$ est muni d'une
structure compatible de $\D ^{(m)} _{\X/\T}$-module telle que l'homomorphisme de 
$\O _\X$-algèbres
$\B ^{(m)}\to \B ^{(m+1)}$ soit aussi un monomorphisme de
$\D ^{(m)} _{\X/\T}$-modules.  
On pose
$\widetilde{\D} ^{(m)} _{\X/\T}: = 
\B ^{(m)} \widehat{\otimes} _{\O _\X}
\widehat{\D} ^{(m)} _{\X/\T}$.
On suppose que la famille 
$(u ^*\B ^{(m)}) _{m\in \N}$ vérifie les mêmes propriétés et on  pose 
$\widetilde{\D} ^{(m)} _{\ZZ/\T}: = 
(u ^*\B ^{(m)})\widehat{\otimes} _{\O _\ZZ}
\widehat{\D} ^{(m)} _{\ZZ/\T}$.
Soit $\E $ un $\widetilde{\D} ^{(m)} _{\X/\T,\Q}$-module cohérent à support dans $\ZZ$ (i.e. tel que $\E | \Y = 0$).

Il existe alors un entier assez grand $m' \geq m$, un $\widetilde{\D} ^{(m')} _{\ZZ/\T,\Q}$-module cohérent $\FF$ et 
un isomorphisme de $\widetilde{\D} ^{(m')} _{\X/\T,\Q}$-modules de la forme
\begin{equation}
\notag
u _{+} ^{(m')} (\FF) \riso 
\widetilde{\D} ^{(m')} _{\X/\T,\Q} \otimes _{\widetilde{\D} ^{(m)} _{\X/\T,\Q}} \E.
\end{equation}

\end{theo}
\begin{proof}
La preuve suivante reprend très fidèlement celle donnée par 
Berthelot dans son cours à Paris de 1997 dans le cas logarithmiques et sans coefficients
(i.e. pour $\B ^{(m)} _\X= \O _\X$). Pour la commodité du lecteur, 
en voici une écriture. 
Pour n'importe quel 
$\widetilde{\D} ^{(m')} _{\ZZ/\T,\Q}$-module cohérent $\G$,
on dispose de l'isomorphisme canonique
$\mathcal{H} ^{0} u ^{!} u _+ ^{(m')} (\G) \riso 
\G$. Cela implique que le foncteur $u _+  ^{(m')}$ est pleinement fidèle. 
Le théorème est donc local et avec \ref{Log-SGA1-II.4.10} l'on peut supposer que 
 $\underline{\X}$ est affine et 
qu'il existe des entiers $n \geq r$ et 
un diagramme cartésien de $\T$-log-schémas formels de la forme: 
$$\xymatrix{
{\X} 
\ar[r] ^-{}
\ar@{}[rd] ^-{}|\square
& {\mathfrak{A} _\T  ^{n,r} } 
\\ 
{\ZZ} 
\ar@{^{(}->}[u] ^-{u}
\ar[r] ^-{}
& 
{\mathfrak{A} _\T ^{r,r}} 
\ar@{^{(}->}[u] ^-{}
}
$$
tel que les morphismes horizontaux sont log-étales et
le morphisme de droite est l'immersion fermée exacte canonique.
Quitte à faire une récurrence sur $n$, on peut supposer que 
$n = r+1$, i.e. $\ZZ = V (t)$ où $t$ est la section de $\O _\X$ 
image de $t _n$ via $\X \to \mathfrak{A} _\T  ^{n,n-1}$.
On reprend alors librement les notations de la fin de
\ref{sectionA.1} concernant notamment les dérivations 
$\partial _{\sharp,1},\dots, \partial _{\sharp,r}, \partial _n$. 
On notera simplement $\partial $ pour $\partial _n$.
Pour terminer la preuve du théorème, nous aurons besoin du 
lemme clé suivant.

\begin{lemm}
[Lemme clé de Berthelot]
\label{Key-Lemma}
Soient $s \geq 1$ un entier, 
$\widetilde{D} ^{(m)} _{\X/\T}: =
\Gamma (\X, \widetilde{\D} ^{(m)} _{\X/\T}) $,
$R \in M _s (\widetilde{D} ^{(m)} _{\X/\T})$ une matrice. 
Il existe alors un entier assez grand
$m'\geq m$ et une matrice
$P \in M _s (\widetilde{D} ^{(m')} _{\X/\T}) $ 
satisfaisant les propriétés suivantes :
\begin{enumerate}
\item $P \equiv I _s \mod \pi M _s (\widetilde{D} ^{(m')} _{\X/\T}) $, 
où 
$I _s$ est la matrice identité ;\\
\item $t ^{p ^{m}} P = P ( t ^{p ^{m}} I _s - \pi R)$,
où $R $ est considéré comme un élément de $M _s (\widetilde{D} ^{(m')} _{\X/\T}) $
via l'inclusion $M _s (\widetilde{D} ^{(m)} _{\X/\T}) \subset M _s (\widetilde{D} ^{(m')} _{\X/\T})$.
\end{enumerate}

\end{lemm}

\begin{proof}
Puisque la preuve pour un $s$ quelconque est identique, 
on peut supposer  $s=1$. 

0) {\it 
Notations}.  
Soit $S \in 
 \widehat{D}  _{\X/\T} \widehat{\otimes} _{O _\X}B ^{(m)}$ 
(on prendra garde au fait que $\widehat{D}  _{\X/\T} \widehat{\otimes} _{O _\X}B ^{(m)}$ 
est seulement un $B ^{(m)}$-module, on prend la structure droite de $O _\X$-module de $\widehat{D}  _{\X/\T}$
pour calculer le produit tensoriel car celui-ci est placé à gauche du produit tensoriel).
L'élément $S$ s'écrit de manière unique sous la formel 
$S 
= 
\sum _{\underline{i}\in \N ^n} 
 \underline{\partial} _{(r)} ^{[ \underline{i}]} a _{\underline{i}}$, 
avec $a _{\underline{i}} \in B ^{(m)} $ tendant $\pi$-adiquement vers $0$ lorsque $| \underline{i}|$ tend vers l'infini
et où la notation $\underline{\partial} _{(r)} ^{[ \underline{i}]}$ est définie à la fin de 
\ref{sectionA.1}.  
Posons $c _{\underline{i}} (S) := a _{\underline{i}}$.
Lorsque $S\in D  _{\X/\T} \otimes _{O _\X}B ^{(m)}$, la somme est finie et on peut alors définir
$\mathrm{ord} (S)$ comme étant le maximal des éléments $|\underline{i}|$ tels que 
$c _{\underline{i}} (S) \not =0$.
Posons $[S] _l  := \pi ^{-l}  \sum _{\underline{i} \in E _l } 
  \underline{\partial} _{(r)} ^{[ \underline{i}]}c _{\underline{i}} (S)  $, où $E _l $ est le sous-ensemble (fini) de  $\N ^n$ des éléments 
 $\underline{i}$ tels que $ v _{\pi } ( c _{\underline{i}} (S))= l$, où $v _{\pi } $ signifie la valuation $\pi$-adique sur $B ^{(m)}$.  
Posons
$\sigma   _l (S) := \sum _{l' \leq l} \pi ^{l'}[S ]_{l'} $.

Comme pour  \cite[2.4.4.2]{Be1}, on vérifie que 
les propriétés suivantes :
\begin{enumerate}[(i)]
\item il existe un entier $m ' \geq m$ tel que
$S \in  \widehat{D}  _{\X/\T} ^{(m')} \widehat{\otimes} _{O _\X}B ^{(m)}$ ;
\item il existe un nombre réel  $a>0$ tel que
$\mathrm{ord} ( \sigma _l  (S))\leq a (l+1)$ pour n'importe quel entier $l$ ;
\item il existe un nombre réel  $a>0$ tel que
$\mathrm{ord} ([S]  _l)\leq a (l+1)$ pour n'importe quel entier $l$
\end{enumerate}
sont équivalentes.
Si $a >0$ est un réel, on dira alors que  $S$ est $a$-borné si, pour n'importe quel $\underline{i}\in \N ^n$, l'inégalité suivante est satisfaite:
$v _\pi ( c _{\underline{i}}(S)) \geq \frac{| \underline{i}|}{a} -1$.
On remarque que 
$S$ est $a$-borné si et seulement si
on a $\mathrm{ord} ( [S] _l) \leq a (l+1)$ pour n'importe quel entier $l$.

1) En utilisant l'inégalité de gauche de \cite[2.4.3.1]{Be1} (et la formule \cite[2.2.3.2]{Be1} encore valable dans le cas logarithmique),
on calcule que pour tout $a > p ^{m}(p-1)$, il existe $b \geq a$ assez grand tel que 
pour tout $P\in  \widehat{D}  _{\X/\T} ^{(m)}\widehat{\otimes} _{O _\X} B ^{(m)} $
on ait $v _\pi ( c _{\underline{i}}(P))\geq \frac{| \underline{i}|}{a} -\frac{b}{a}$ pour n'importe quel entier $\underline{i}$ (il s'agit d'étudier une fonction
et vérifier qu'elle est minorée).
On en déduit qu'il existe $\alpha$ assez grand (e.g. $a+b$) tel que 
pour tout $P\in  \widehat{D}  _{\X/\T} ^{(m)}\widehat{\otimes} _{O _\X} B ^{(m)} $,
$P$ soit $\alpha$-borné. 
On fixe à présent un tel $\alpha$ et on pose 
$\beta := \alpha + p ^m$.

2) On remarque que si $P \in D _{\X/\T} \otimes _{O _\X} B ^{(m)} $ 
alors
$P R \in \widehat{D}  _{\X/\T} \widehat{\otimes} _{O _\X}B ^{(m)}$
(en effet $R \in \widehat{D}  _{\X/\T} ^{(m)}\widehat{\otimes} _{O _\X} B ^{(m)} $, comme ce dernier est un anneau, ses éléments sont en particulier
stables par multiplication à gauche par un élément de $B ^{(m)} $). 
Supposons de plus que $P$ soit  $\beta$-borné.
L'opérateur $P$ est alors une somme finie d'éléments de la forme
$\pi ^{l _2} \underline{\partial} _{(r)} ^{[ \underline{i} _2]}b $,
avec $b \in B ^{(m)}$ et $| \underline{i} _2| \leq \beta ( l _2 +1)$.
Comme $bR\in \widehat{D}  _{\X/\T} ^{(m)}\widehat{\otimes} _{O _\X} B ^{(m)} $ 
(car $\widehat{D}  _{\X/\T} ^{(m)}\widehat{\otimes} _{O _\X} B ^{(m)} $ est un anneau),
d'après 1), $bR$ est $\alpha$-borné. On peut donc l'écrire 
comme une somme d'éléments
de la forme $\pi ^{l _1} \underline{\partial} _{(r)} ^{[ \underline{i} _1]}a $
avec  $a \in B ^{(m)}$ et $| \underline{i} _1| \leq \alpha ( l _1+1)$.
Ainsi, $PR$ est égale à une somme d'éléments de la forme 
$\pi ^{l _1+l _2} \underline{\partial} _{(r)} ^{[ \underline{i} _1]}\underline{\partial} _{(r)} ^{[ \underline{i} _2]}a $
avec $a \in B ^{(m)}$, $| \underline{i} _1| \leq \alpha ( l _1+1)$ et $| \underline{i} _2| \leq \beta ( l _2 +1)$.
Comme $\alpha \leq \beta$, on a donc
$| \underline{i} _1| + | \underline{i} _2| \leq \alpha + \beta ( l _1+ l _2+1)$.
Ainsi $P R $ est quasiment $\beta$-borné modulo l'ajout de $\alpha$.

3) On vérifie
 pour n'importe quel entier naturel $N$ 
 la formule
$$t ^{p ^{m}} \partial ^{[N + p ^m]}
-
\partial ^{[N + p ^m]}t ^{p ^{m}}
\equiv
- \partial ^{[N]}
\mod \pi D _{\X/\T} .$$
Cela implique que pour n'importe quel opérateur
$U\in D  _{\X/\T} \otimes _{O _\X}B ^{(m)}$ 
il existe un opérateur
$Q\in D  _{\X/\T} \otimes _{O _\X}B ^{(m)}$
tel que 
$$t ^{p ^{m}} Q
-
Qt ^{p ^{m}}
\equiv
U \mod \pi \widehat{D}  _{\X/\T} \widehat{\otimes} _{O _\X}B ^{(m)}$$
et 
$\mathrm{ord} (Q) \leq \mathrm{ord} (U) + p^m$.

4) Dans cette étape, nous construisons par récurrence sur  $l\geq 0$, des opérateurs
$P _l \in D  _{\X/\T} \otimes _{O _\X}B ^{(m)}$ vérifiant pour n'importe quel  $l \geq 0$ les conditions suivantes:
\begin{enumerate}[(i)]
\item $P _0 = 1$ et pour $l \geq 1$ on  a $P _{l} \equiv P _{l-1} \mod \pi ^{l}  D  _{\X/\T} \otimes _{O _\X}B ^{(m)}$ ,
\item 
$P _l$ est $\beta$-borné, 
\item $ t ^{ p ^m} P _l \equiv P _l ( t ^{ p ^m} - \pi  R) 
\mod \pi ^{l+1}  \widehat{D}  _{\X/\T} \widehat{\otimes} _{O _\X}B ^{(m)}$.
\end{enumerate}
Remarquons que d'après l'étape 2), les éléments de la congruence de la propriété (iii) 
appartiennent bien à $\widehat{D}  _{\X/\T} \widehat{\otimes} _{O _\X}B ^{(m)}$.
On a forcément $P _0 = 1$. Soit $l \in \N$ et supposons construit
$P _0, \dots, P _l$ vérifiant les conditions (i), (ii) et (iii).
Posons 
$U:= 
-[t ^{ p ^m} P _l - P _l ( t ^{ p ^m} - \pi R )] _{l+1}\in  D  _{\X/\T} \otimes _{O _\X}B ^{(m)}$.
D'après 
la propriété (iii) satisfaite par $P _l$,
on a $t ^{ p ^m} P _l - P _l ( t ^{ p ^m} - \pi R )
\equiv 0 \mod \pi ^{l+1}  \widehat{D}  _{\X/\T} \widehat{\otimes} _{O _\X}B ^{(m)}$.
On en déduit la congruence :
\begin{equation}
\label{3}
-\pi ^{l+1}  U 
\equiv
t ^{ p ^m} P _l - P _l ( t ^{ p ^m} - \pi R)
\mod \pi ^{l+2}  \widehat{D}  _{\X/\T} \widehat{\otimes} _{O _\X}B ^{(m)}.
\end{equation}
D'après l'étape 2), comme $P _l$ est $\beta$-borné 
on vérifie que 
$\sigma _{l+1} (\pi P _l  R )= \pi \sigma _{l} (P _l  R )$
est
la somme d'éléments de la forme 
$\pi ^{1+l _1+l _2} \underline{\partial} _{(r)} ^{[ \underline{i} _1]}\underline{\partial} _{(r)} ^{[ \underline{i} _2]}a $
avec  $a \in B ^{(m)}$, 
$| \underline{i} _1| + | \underline{i} _2| \leq \alpha + \beta ( l _1+ l _2+1)$
et $1+l _1 +l _2 \leq l +1$.
Il en résulte 
$\mathrm{ord} ( U)  \leq \beta (l+1) + \alpha$.
D'après la partie 3) de la preuve, il existe alors 
$Q \in  D  _{\X/\T} \otimes _{O _\X}B ^{(m)}$ tel que 
$t ^{p ^m} Q - Q t ^{ p ^m} \equiv U \mod  \pi \widehat{D}  _{\X/\T} \widehat{\otimes} _{O _\X}B ^{(m)}$ et 
$\mathrm{ord} ( Q) \leq \mathrm{ord} ( U) + p ^{m}$.
On en déduit 
$\mathrm{ord} (Q) \leq \beta ( l +1)+ \alpha + p ^{m} = \beta (l +2)$
(on rappelle que $\beta :=  \alpha + p ^{m}$).
Ainsi 
$\pi ^{l+1} Q$ est $\beta$-borné.

D'après l'étape 2), $QR \in \widehat{D}  _{\X/\T} \widehat{\otimes} _{O _\X}B ^{(m)}$. 
En utilisant \ref{3}, 
on obtient les congruences 
\begin{align}
\notag
t ^{p ^m} (P _l + \pi ^{l+1}  Q) 
-(P _l + \pi ^{l+1}  Q) (  t ^{ p ^m} - \pi R )
&
=
t ^{ p ^m} P _l - P _l ( t ^{ p ^m} - \pi R )
+
\pi ^{l+1} (t ^{p ^m} Q - Q t ^{ p ^m})
+
\pi ^{l+2} QR
\\
&
\equiv 
-\pi ^{l+1}  U 
+
\pi ^{l+1} (t ^{p ^m} Q - Q t ^{ p ^m})
\mod \pi ^{l+2} \widehat{D}  _{\X/\T} \widehat{\otimes} _{O _\X}B ^{(m)}
\\
&
\equiv 
0\mod \pi ^{l+2} \widehat{D}  _{\X/\T} \widehat{\otimes} _{O _\X}B ^{(m)}.
\end{align}
Ainsi  $P _{l+1}:= P _l + \pi ^{l+1}  Q$ est $\beta$-borné et 
satisfait  
$ t ^{ p ^m} P _{l+1} \equiv P _{l+1} ( t ^{ p ^m} - \pi R) 
\mod \pi ^{l+2}  \widehat{D}  _{\X/\T} \widehat{\otimes} _{O _\X}B ^{(m)}$.

5) Finalement, d'après l'étape 4), on obtient l'élément $P := \lim _{l\to \infty} \, P _l$ de 
$ \widehat{D}  _{\X/\T} \widehat{\otimes} _{O _\X}B ^{(m)}$
qui est $\beta$-borné et vérifie $ t ^{ p ^m} P = P ( t ^{ p ^m} - \pi R) $.
D'après l'étape 0),  comme $P$ est $\beta $-borné, il existe alors $m' \geq m$ tel 
que 
$P \in  \widehat{D}  ^{(m')} _{\X/\T} \widehat{\otimes} _{O _\X}B ^{(m)} 
\subset  \widetilde{D}  _{\X/\T} ^{(m')}$.
Cet opérateur $P$ satisfait aux conditions requises.
\end{proof}

Revenons à présent à la preuve de \ref{Berthelot-Kashiwara}. 
Soit $\overset{\circ}{\E}$ un $\widetilde{\D} ^{(m)} _{\X/\T}$-module cohérent sans $\pi$-torsion 
tel que $\overset{\circ}{\E} _\Q \riso \E$.
Choisissons des générateurs $e _1, \dots, e _s$ de $\overset{\circ}{\E}$.
Posons $\underline{e}:= \left(
\begin{array}{c}
e _1\\
\vdots  \\
e _s   
\end{array}
\right)$.
Puisque $\overset{\circ}{\E} / \pi \overset{\circ}{\E}$ est un $\O _X$-module quasi-cohérent et à support dans $Z$ (i.e. 
$(\overset{\circ}{\E} / \pi \overset{\circ}{\E})|Y= 0$), alors
$ t ^{p ^m } e _i \equiv 0 \mod \pi \overset{\circ}{\E}$ (augmenter $m$ si nécessaire). 
Il existe donc une matrice $R \in M _s (\widetilde{D} ^{(m)} _{\X/\T})$
telle que
$t ^{p ^{m}} \underline{e} =
\pi 
R \underline{e}$.
Soit $m'$ et $P\in M _s (\widetilde{D} ^{(m')} _{\X/\T})$ vérifiant les conditions du Lemma \ref{Key-Lemma}.
Posons
$E ^{(m')}:= \widetilde{D} ^{(m')} _{\X/\T,\Q} \otimes _{\widetilde{D} ^{(m)} _{\X/\T,\Q}} E$. 
On note 
$1 \otimes e _1, \dots, 1 \otimes e _s$ les éléments de $E ^{(m')}$
correspondant à $e _1, \dots, e _s$.
Posons alors $\underline{e}' : = P (1 \otimes \underline{e})$.
On calcule $t ^{p ^{m}} \underline{e}' = 
t ^{p ^{m}} P (1 \otimes \underline{e})
=
P ( t ^{p ^{m}} I _s - \pi R)
(1 \otimes \underline{e})=
P ( 1 \otimes (t ^{p ^{m}} \underline{e} -
\pi 
R \underline{e}))=0 $.

Notons
$\mathcal{H} ^{0} u ^{ !} (\smash{{E}} ^{(m')})$ 
l'ensemble des éléments de 
$\smash{{E}} ^{(m')}$
qui sont annulés par $t$. 
Notons 
$\smash{{K}} ^{(m')}$ 
le sous-$\widetilde{D} ^{(m')} _{\X/\T,\Q} $-module
de $\smash{{E}} ^{(m')}$
engendré par  $\mathcal{H} ^{0} u ^{ !} (\smash{{E}} ^{(m')})$.
Vérifions à présent par récurrence sur $i \geq 1$ que pour tout 
$x\in \smash{{E}} ^{(m')}$ 
si $t ^{i} x= 0$ alors  $x \in \smash{{K}} ^{(m')}$.
Lorsque $i= 1$, c'est immédiat. Supposons $i \geq 2$ et la propriété vraie pour $i -1$. 
Soit $x\in \smash{{E}} ^{(m')}$  tel que $t ^{i} x= 0$. 
A partir de la formule $\partial t ^{i-1} = t ^{i-1} \partial + (i-1)t ^{i-2}$,
on obtient par multiplication à droite par $t$ la formule $\partial t ^{i} = t ^{i-1} \partial t + (i-1)t ^{i-1}
=t ^{i-1} ( \partial t + (i-1))$.
Puisque $\partial t ^{i} \cdot x =0$, par hypothèse de récurrence, on obtient 
$( \partial t + (i-1)) \cdot x \in \smash{{K}} ^{(m')}$ (resp. $t \cdot x \in \smash{{K}} ^{(m')}$ et donc $\partial t \cdot x \in \smash{{K}} ^{(m')}$).
D'où par soustraction $(i-1) x \in \smash{{K}} ^{(m')}$ et donc  $x\in \smash{{K}} ^{(m')}$, ce qui conclut la récurrence.
Comme $\smash{{E}} ^{(m')}$ est engendré comme $\widetilde{D}  ^{(m')} _{\X/\T,\Q}$-module par des éléments annulés par $t ^{p ^m}$, 
on en conclut que $\smash{{K}} ^{(m')}=\smash{{E}} ^{(m')}$, 
i.e. que  $\smash{{E}} ^{(m')}$ est engendré comme $\widetilde{D}  ^{(m')} _{\X/\T,\Q}$-module par des éléments annulés par $t $.

Comme pour  \cite[2.3.2]{surcoh-hol}, on dispose du morphisme canonique de $\widetilde{D}  ^{(m')} _{\X/\T,\Q}$-modules 
\begin{equation}
\label{u_+u^!toif-limproj}
\mathrm{adj}\,:\,u ^{(m')} _{+} \circ \mathcal{H} ^{0} u ^{ !} (\smash{{E}} ^{(m')})
\to 
\smash{{E}} ^{(m')}
\end{equation}
donné par 
$\sum _{k\in \N }[ \partial ^{<k>} ]\otimes x _{k}
\mapsto 
\sum _{k\in \N } \partial ^{<k>}  \cdot x _{k}$,
avec $x _{k} \in \mathcal{H} ^{0} u ^{ !} (\smash{{E}} ^{(m')})$ 
tel que $\lim _{|k |\to \infty} x _{k} =0$ (où
$[ \partial ^{<k>} ]$ désigne la classe de $\partial ^{<k>} $ dans le bimodule $\widetilde{D}  ^{(m')} _{\X\leftarrow \ZZ/\T,\Q}$). 
Puisque $\smash{{E}} ^{(m')}$ est engendré comme $\widetilde{D}  ^{(m')} _{\X/\T,\Q}$-module par un nombre fini d'éléments de 
$\mathcal{H} ^{0} u ^{ !} (\smash{{E}} ^{(m')})$, 
on en déduit un morphisme surjectif qui se décompose de la forme 
$u ^{(m')} _{+} ( L )
\to 
u ^{(m')} _{+} \circ \mathcal{H} ^{0} u ^{ !} (\smash{{E}} ^{(m')})
\to 
\smash{{E}} ^{(m')}$,
où $L$ est 
un $\widetilde{D}  ^{(m')} _{\ZZ/\T,\Q}$-module libre de rang fini. 
Notons $N$ le noyau de la surjection
$u ^{(m')} _{+} ( L )
\to 
\smash{{E}} ^{(m')}$.
Comme  $N$  est aussi à support dans $\ZZ$, 
en réitérant le procédé ci-dessus, 
quitte à augmenter $m'$ si nécessaire, 
$N$ est engendré comme
 $\widetilde{D}  ^{(m')} _{\X/\T,\Q}$-module par un nombre fini d'éléments de   
$\mathcal{H} ^{0} u ^{ !} (N)$. En particulier, le morphisme canonique 
$u ^{(m')} _{+}\mathcal{H} ^{0} u ^{ !} (N) \to N$ est surjective. 
Puisque $\mathcal{H} ^{0} u ^{ !} (N)$ est un sous-module de $L$, 
alors $\mathcal{H} ^{0} u ^{ !} (N)$ est un $\widetilde{D}  ^{(m')} _{\ZZ/\T,\Q}$-module cohérent
et 
la composée 
$u ^{(m')} _{+}\mathcal{H} ^{0} u ^{ !} (N) \to N \to u ^{(m')} _{+} L$ est injective
(parce que  $u ^{(m')} _{+}$ est exact sur la catégorie
$\widetilde{D}  ^{(m')} _{\ZZ/\T,\Q}$-modules cohérents).
En particulier, 
$u ^{(m')} _{+}\mathcal{H} ^{0} u ^{ !} (N) \to N$ est aussi injective. 
Puisqu'il est aussi surjectif, la flèche canonique $u ^{(m')} _{+}\mathcal{H} ^{0} u ^{ !} (N) \to N$ est donc un isomorphisme. 
Ainsi, par exactitude de  $u ^{(m')} _{+}$, on obtient
$E ^{(m')} \riso u ^{(m')} _{+} ( L) /N
\riso u ^{(m')} _{+} ( L) /u ^{(m')} _{+} \mathcal{H} ^{0} u ^{ !} (N) 
\riso u ^{(m')} _{+} ( L/\mathcal{H} ^{0} u ^{ !} (N) )$.
\end{proof}

\begin{theo}
[Berthelot-Kashiwara]
\label{Berthelot-Kashiwara-full}
On garde les notations de 
\ref{Berthelot-Kashiwara} et on pose 
$\smash{\widetilde{\D}} ^{\dag} _{\X/\T,\Q}:=\underset{\underset{m}{\longrightarrow}}{\lim}\,
\smash{\widetilde{\D}} ^{(m)} _{\X/\T,\Q}$
et
$\smash{\widetilde{\D}} ^{\dag} _{\ZZ/\T,\Q}:=\underset{\underset{m}{\longrightarrow}}{\lim}\,
\smash{\widetilde{\D}} ^{(m)}  _{\ZZ/\T,\Q}$.
Les foncteurs image inverse extraordinaire $u ^{!}$ et et image directe $u _{+}$ induisent des équivalences quasi-inverses entre la catégorie 
des  $\smash{\widetilde{\D}} ^{\dag} _{\X/\T,\Q}$-modules cohérents à support dans $Z$ 
et celle des $\smash{\widetilde{\D}} ^{\dag} _{\ZZ/\T,\Q} $-modules cohérents. 
Ces foncteurs $u ^{!}$ et $u  _{+}$ sont acycliques sur ces catégories. 

\end{theo}

\begin{proof}
On se ramène à la situation géométrique de la preuve de \ref{Berthelot-Kashiwara}, avec $n=r-1$.
Pour tout $\smash{\widetilde{\D}} ^{\dag} _{\X/\T,\Q}$-module cohérent $\FF$ à support dans $Z$,
un calcul immédiat en coordonnées locales donne  
$\mathcal{H} ^{0} u ^{!} u _+  (\FF) \riso 
\FF$. 
Grâce au théorème \ref{Berthelot-Kashiwara}, 
il en résulte que les foncteurs 
$\mathcal{H} ^{0} u ^{!} $ et $u _+  $ induisent des équivalences 
quasi-inverses entre la catégorie 
des  $\smash{\widetilde{\D}} ^{\dag} _{\X/\T,\Q}$-modules cohérents à support dans $Z$ 
et celle des $\smash{\widetilde{\D}} ^{\dag} _{\ZZ/\T,\Q} $-modules cohérents. 
L'acyclité de $u _+$ découle d'un calcul facile en coordonnées locales. 
Il reste à établir l'acyclité de $u ^!$. 
Via la formule \cite[1.7.1]{caro_log-iso-hol},
pour tout $\smash{\widetilde{\D}} ^{(m)}  _{\ZZ/\T,\Q}$-module cohérent $\FF ^{(m)}$, 
en notant $\FF ^{(m+1)}:= \smash{\widetilde{\D}} ^{(m+1)}  _{\ZZ/\T,\Q} \otimes _{\smash{\widetilde{\D}} ^{(m)}  _{\ZZ/\T,\Q}}\FF ^{(m)}$,
on calcule que la flèche 
canonique 
$\mathcal{H} ^1 u ^{ (m)!} \circ u _{+} ^{ (m)} (\FF ^{(m)})
\to 
\mathcal{H} ^1 u ^{ (m+1)!} \circ u _{+} ^{ (m+1)} (\FF ^{(m+1)})$
est le morphisme nul (grâce à \ref{lemm-ualg-u!nivm}, 
$\mathcal{H} ^1 u ^{ (m)!} =\mathcal{H} ^1 u ^{ (m+1)!}= u ^*$ pour des modules cohérents).
On en déduit que 
$\mathcal{H} ^1 u ^{ !} \circ u _{+}  (\FF )=0$.
Comme pour $i \not \in \{ 0,1\}$, on a $\mathcal{H} ^i u ^{ !} \circ u _{+}  (\FF )=0$,
on en déduit l'acyclicité de $u ^{ !}$.
\end{proof}

\section{Preuve direct de la cohérence du coefficient constant avec singularités surconvergentes pour les systèmes inductifs}

On propose ci-dessous une preuve directe (i.e., sans passer par le théorème \ref{limTouD} 
qui permet de se ramener au cas déjà traité par Berthelot dans  \cite{Becohdiff}) du théorème
\ref{coh-Bbullet} ci-dessus, dont l'énoncé est le même que \ref{coh-Bbulletbis} ci-dessous.
La preuve du théorème \ref{coh-Bbulletbis} reprend les étapes clés du théorème analogue de cohérence de Berthelot de \cite{Becohdiff}
(on  pourra comparer avec \cite{Becohdiff}). Il s'agit en quelque sorte d'une adaptation de sa preuve 
au cas des systèmes inductifs. 
On termine l'appendice par une 
vérification à la main du théorème
\ref{coh-Bbullet} ci-dessus lorsque le diviseur $T$ est lisse (voir \ref{lisse}).

Nous aurons d'abord besoin du lemme suivant. 
\begin{lemm}
\label{fact-dirc-coh}
Soient  $T' \supset T$ un diviseur,
$\E  ^{(\bullet)} \in \underrightarrow{LD} ^{\mathrm{b}} _{\Q, \mathrm{coh}} (\smash{\widetilde{\D}} _{\PP ^\sharp} ^{(\bullet)}(T'))$
et 
$\G  ^{(\bullet)} \in 
\underrightarrow{LD} ^{\mathrm{b}} _{\Q, \mathrm{coh}} (\smash{\widetilde{\D}} _{\PP ^\sharp} ^{(\bullet)}(T))
\cap 
\underrightarrow{LD} ^{\mathrm{b}} _{\Q, \mathrm{qc}} (\smash{\widetilde{\D}} _{\PP ^\sharp} ^{(\bullet)}(T'))$.
Si $\underrightarrow{\lim}(\E ^{(\bullet)})$ est un facteur direct de $\underrightarrow{\lim} (\G ^{(\bullet)})$ dans
$D ^\mathrm{b}  (\smash{\D} ^\dag _{\PP ^{\sharp}, \Q} )$,
alors
$\E ^{(\bullet)} \in 
\underrightarrow{LD} ^{\mathrm{b}}  _{\Q, \mathrm{coh}} (\smash{\widetilde{\D}} _{\PP ^\sharp} ^{(\bullet)}(T))$.
\end{lemm}

\begin{proof}
Comme $\E := \underrightarrow{\lim}~\E  ^{(\bullet)} $
est un facteur direct de 
$\G := \underrightarrow{\lim}\G  ^{(\bullet)} $
dans $D ^\mathrm{b} 
(\smash{\D} ^\dag _{\PP ^{\sharp}, \Q}   )$, 
comme 
$\G \in D ^\mathrm{b} _{\textrm{coh}}
(\smash{\D} ^\dag _{\PP ^\sharp}( \hdag T) _{\Q} )$, 
on obtient alors $\E \in D ^\mathrm{b} _{\textrm{coh}}
(\smash{\D} ^\dag _{\PP ^\sharp}( \hdag T) _{\Q} )$.
Comme le foncteur 
$\underrightarrow{\lim}$ 
est une équivalence de catégories sur les complexes cohérents et à cohomologie bornée, 
il existe $\H  ^{ (\bullet)} ,\, \FF  ^{(\bullet)}\in 
\underrightarrow{LD} ^{\mathrm{b}}  _{\Q, \mathrm{coh}} (\smash{\widetilde{\D}} _{\PP ^\sharp} ^{(\bullet)}(T))$
tels que 
$\E \riso  \underrightarrow{\lim}\H  ^{ (\bullet)} $
et 
$\H  ^{ (\bullet)} \oplus \FF  ^{(\bullet)} \riso 
\G  ^{(\bullet)}$ 
dans 
$\underrightarrow{LD} ^{\mathrm{b}} _{\Q, \mathrm{coh}} (\smash{\widetilde{\D}} _{\PP ^\sharp} ^{(\bullet)}(T))$.
Comme $\G  ^{(\bullet)} \in 
\underrightarrow{LD} ^{\mathrm{b}} _{\Q, \mathrm{qc}} (\smash{\widetilde{\D}} _{\PP ^\sharp} ^{(\bullet)}(T'))$,
on en déduit que 
le morphisme canonique
$\H  ^{ (\bullet)} \oplus \FF  ^{(\bullet)} 
\to 
(\hdag T ')  (\H  ^{ (\bullet)} \oplus \FF  ^{(\bullet)})
=
(\hdag T ')  (\H  ^{ (\bullet)} ) 
\oplus (\hdag T ') (\FF  ^{(\bullet)})$
est un isomorphisme. 
Il en résulte que le morphisme canonique
$\H  ^{ (\bullet)} 
\to 
(\hdag T ') (\H  ^{ (\bullet)} )$
est un isomorphisme. 
Or, 
$\underrightarrow{\lim} (\hdag T ') (\H  ^{ (\bullet)} )
\riso (\hdag T') \underrightarrow{\lim}(\H  ^{ (\bullet)} )
\riso (\hdag T')(\E )$.
De plus, 
 comme $\E \in 
D ^\mathrm{b} _{\textrm{coh}}
(\smash{\D} ^\dag _{\PP ^\sharp} (\hdag T) _{\Q} )
\cap 
D ^\mathrm{b} _{\textrm{coh}}
(\smash{\D} ^\dag _{\PP ^\sharp} (\hdag T') _{\Q} )$,
le morphisme canonique
$ (\hdag T')(\E ) \to \E $
est un isomorphisme (grâce à \cite[4.8]{caro_log-iso-hol}).
Par pleine fidélité de 
$\underrightarrow{\lim}$
sur 
$\underrightarrow{LD} ^{\mathrm{b}}  _{\Q, \mathrm{coh}} (\smash{\widetilde{\D}} _{\PP ^\sharp} ^{(\bullet)}(T'))$,
on en déduit que 
$ (\hdag T ') (\H  ^{ (\bullet)} ) \riso  \E  ^{ (\bullet)} $. D'où le résultat. 
\end{proof}

\begin{theo}
\label{coh-Bbulletbis}
On a
$\smash{\widetilde{\B}} _{\PP} ^{(\bullet)} (T ) \in 
\underrightarrow{LM}  _{\Q, \mathrm{coh}} (\smash{\widehat{\D}} _{\PP } ^{(\bullet)})
\cap \underrightarrow{LM}  _{\Q, \mathrm{coh}} (\smash{\widetilde{\D}} _{\PP } ^{(\bullet)} (T))$.
\end{theo}

\begin{proof}
Il est immédiat que 
$\smash{\widetilde{\B}} _{\PP} ^{(\bullet)} (T ) \in \underrightarrow{LM}  _{\Q, \mathrm{coh}} (\smash{\widetilde{\D}} _{\PP } ^{(\bullet)} (T))$.
Vérifions à présent 
$\smash{\widetilde{\B}} _{\PP} ^{(\bullet)} (T ) \in \underrightarrow{LM}  _{\Q, \mathrm{coh}} (\smash{\widehat{\D}} _{\PP } ^{(\bullet)})$.

{\it Étape I. Supposons à présent que $T$ soit un diviseur à croisements normaux.} 
On peut supposer $\PP$ affine et muni de coordonnées locales $t _1, \dots, t _d$ 
telles que $T= V (\overline{t} _1\cdots \overline{t} _r)$, où $\overline{t} _1,\dots, \overline{t} _r\in \O _P$ désignent 
respectivement la réduction de $t _1,\dots, t _r$ modulo $\pi \O _{\PP}$.
On procède alors par récurrence sur $r$, i.e. le nombre de composante irréductible de $T$. 
Notons $\X = V (t _1)$,
$u\colon \X \hookrightarrow \PP$ l'immersion fermée canonique
et
$T'= V (\overline{t} _2\cdots \overline{t} _{r})$. 

Comme d'après le lemme \ref{u*Odiv}
$u  ^{(\bullet)!}(\smash{\widetilde{\B}} _{\PP} ^{(\bullet)} (T'))[1]
\riso
\smash{\widetilde{\B}} _{\X} ^{(\bullet)} (X \cap T')$, 
comme 
d'après \ref{hdagT'T=cup} 
$(\hdag X) (\smash{\widetilde{\B}} _{\PP} ^{(\bullet)} (T'))
\riso 
\smash{\widetilde{\B}} _{\PP} ^{(\bullet)} (X \cup T')$, 
avec l'isomorphisme \ref{pre-loc-tri-B-t1T-iso}, 
le triangle de localisation de 
$\smash{\widetilde{\B}} _{\PP} ^{(\bullet)} (T ' )$ par rapport à $X$ 
donne la suite exacte courte:
\begin{equation}
\label{loc-tri-B-t1T}
0\to
\smash{\widetilde{\B}} _{\PP} ^{(\bullet)} (T')
\to
\smash{\widetilde{\B}} _{\PP} ^{(\bullet)} (T )
\to 
u _+ ^{(\bullet)} (\smash{\widetilde{\B}} _{\X} ^{(\bullet)} (T' \cap X))
\to 
0.
\end{equation}
Comme $T'\cap X$ est un diviseur à croisements normaux de $X$ avec $r-1$ composantes irréductibles, 
comme la cohérence est préservée par le foncteur $u _+ ^{(\bullet)}$, 
on conclut alors par hypothèse de récurrence.

{\it Étape II. Cas général.} 
Il résulte du théorème de désingularisation de de Jong qu'il existe
un morphisme de cadres 
de la forme 
$\alpha = (f, g, h)  \colon (\widetilde{\PP} , \widetilde{T} ,\widetilde{X} ,\widetilde{Y}) \to (\PP, T,P,Y)$
tel que
$\widetilde{X} $ soit lisse, 
$\widetilde{T} = f ^{-1} (T)$ et $\widetilde{T} \cap \widetilde{X} $ soit un diviseur à croisements normaux de $\widetilde{X}$,
$f$ soit un morphisme propre et lisse de $\V$-schémas formels séparés et lisses,
$g$ soit un morphisme propre, surjectif, génériquement fini et étale de $k$-variétés.
Posons $\widetilde{\E} ^{(\bullet)}:=
 \R \underline{\Gamma} ^\dag _{\widetilde{X}} f  ^{(\bullet)!} (\smash{\widetilde{\B}} _{\PP} ^{(\bullet)} (T ))$. 
Prouvons que 
$\widetilde{\E} ^{(\bullet)} \in 
\underrightarrow{LD} ^{0} _{\Q, \mathrm{coh}} (\smash{\widetilde{\D}} _{\widetilde{\PP}} ^{(\bullet)})
\cap 
\underrightarrow{LD} ^{0} _{\Q, \mathrm{coh}} (\smash{\widetilde{\D}} _{\widetilde{\PP}} ^{(\bullet)} (\widetilde{T}))$.
D'après \ref{QCoh-local}, comme cela est local en $\widetilde{\PP}$, 
on peut supposer qu'il existe une immersion fermée de $\V$-schémas formels lisses
$u \colon \widetilde{\X} \hookrightarrow \widetilde{\PP}$ qui relève
$\widetilde{X} \hookrightarrow \widetilde{\PP}$.
Dans ce cas, avec \ref{pre-loc-tri-B-t1T-iso} puis \ref{u*Odiv}, on obtient les isomorphismes
$\widetilde{\E} ^{(\bullet)}\riso 
u _{+} ^{(\bullet) } \circ u ^{(\bullet) !} \circ  f  ^{(\bullet)!} (\smash{\widetilde{\B}} _{\PP} ^{(\bullet)} (T ))
\riso 
u _{+} ^{(\bullet) }  (\smash{\widetilde{\B}} _{\widetilde{\X}} ^{(\bullet)} (\widetilde{T} \cap \widetilde{\X}))$.
D'après l'étape $II$, on sait $\smash{\widetilde{\B}} _{\widetilde{\X}} ^{(\bullet)} (\widetilde{T} \cap \widetilde{\X})
\in \underrightarrow{LD} ^{0} _{\Q, \mathrm{coh}} (\smash{\widetilde{\D}} _{\widetilde{\X}} ^{(\bullet)})
\cap 
\underrightarrow{LD} ^{0} _{\Q, \mathrm{coh}} (\smash{\widetilde{\D}} _{\widetilde{\X}} ^{(\bullet)} (\widetilde{T} \cap \widetilde{\X}))$.
Via la factorisation de
$u _{+} ^{(\bullet) } $ du carré de droite de \ref{imm-fer-coh-diag}, 
on en déduit 
$\widetilde{\E} ^{(\bullet)} \in 
\underrightarrow{LD} ^{0} _{\Q, \mathrm{coh}} (\smash{\widetilde{\D}} _{\widetilde{\PP}} ^{(\bullet)})
\cap 
\underrightarrow{LD} ^{0} _{\Q, \mathrm{coh}} (\smash{\widetilde{\D}} _{\widetilde{\PP}} ^{(\bullet)} (\widetilde{T}))$.

Notons 
$\widetilde{\E} := \underrightarrow{\lim} (\widetilde{\E} ^{(\bullet)})$.
Comme $f$ est propre, 
alors $f ^{(\bullet)} _{+} (\widetilde{\E} ^{(\bullet)} ) 
\in 
\underrightarrow{LD} ^{\mathrm{b}} _{\Q, \mathrm{coh}} (\smash{\widehat{\D}} _{\PP } ^{(\bullet)})
\cap 
\underrightarrow{LD} ^{\mathrm{b}} _{\Q, \mathrm{coh}} (\smash{\widetilde{\D}} _{\PP } ^{(\bullet)} (T))$.
On en déduit 
$f _{+} (\widetilde{\E} ) \riso 
\underrightarrow{\lim} (f ^{(\bullet)}  _{+} (\widetilde{\E} ^{(\bullet)} ) )
\in 
\underrightarrow{LD} ^{\mathrm{b}} _{\Q, \mathrm{coh}} (\smash{\widehat{\D}} _{\PP } ^{(\bullet)})$.
Comme $\O _{\PP} (\hdag T) _{\Q}$ est un facteur direct de $f _{+} (\widetilde{\E} )$, 
par pleine fidélité du foncteur 
$\underrightarrow{\lim}\colon 
\underrightarrow{LD} ^{\mathrm{b}} _{\Q, \mathrm{coh}} (\smash{\widetilde{\D}} _{\PP } ^{(\bullet)} (T))
\to 
D ^\mathrm{b} _\mathrm{coh} ( \smash{\D} ^\dag _{\PP } (\hdag T) _{\Q} )$, 
il en résulte que
$\smash{\widetilde{\B}} _{\PP} ^{(\bullet)} (T )$ est un facteur direct de 
 $f ^{(\bullet)} _{+} (\widetilde{\E} ^{(\bullet)} ) $
dans la catégorie 
$\underrightarrow{LD}  ^{\mathrm{b}} _{\Q, \mathrm{coh}} (\smash{\widetilde{\D}} _{\PP } ^{(\bullet)} (T))$, 
et donc dans $\underrightarrow{LD}  ^{\mathrm{b}} _{\Q} (\smash{\widehat{\D}} _{\PP } ^{(\bullet)})$ (en effet, 
le foncteur oubli du diviseur est pleinement fidèle d'après \ref{oub-pl-fid}).
Il résulte de \ref{fact-dirc-coh} que l'on obtient
$\smash{\widetilde{\B}} _{\PP} ^{(\bullet)} (T ) \in 
\underrightarrow{LD} ^{\mathrm{b}}  _{\Q, \mathrm{coh}} (\smash{\widehat{\D}} _{\PP } ^{(\bullet)})$.

\end{proof}

\begin{rema}
\label{lisse}
Lorsque $T$ lisse, on peut vérifier par un calcul direct le théorème 
\ref{coh-Bbulletbis}. 
Enfin, d'après \ref{QCoh-local}, 
comme le résultat à vérifier est local, 
on peut supposer $\PP$ affine et muni de coordonnées locales $t _1, \dots, t _d$ 
telles que $T= V (\overline{t} _1)$, où $\overline{t} _1\in \O _P$ désigne la réduction de $t _1$ modulo $\pi \O _{\PP}$.
Pour tout entier $m\geq 0$, notons 
$\smash{\widetilde{\B}} _{\PP} ^{(m)} (*T ):= \frac{1}{p} \smash{\widetilde{\B}} _{\PP} ^{(m)} (T ) \subset \O _{\PP} (\hdag T ) _{\Q}$ (pour l'injection, voir 
\cite[4.3.3]{Be1}).
Comme $\smash{\widetilde{\B}} _{\PP} ^{(m)} (T )$ est un
sous-$\smash{\widetilde{\D}} _{\PP } ^{(m)} $-module de $\O _{\PP} (\hdag T ) _{\Q}$ (voir \cite[4.2.4]{Be1}), il en est de même de 
$\smash{\widetilde{\B}} _{\PP} ^{(m)} (*T )$.
Comme $\frac{1}{t _1} \in \smash{\widetilde{\B}} _{\PP} ^{(m)} (*T )$, on dispose donc de l'application canonique
$\phi \colon \smash{\widetilde{\D}} _{\PP } ^{(m)}  \to \smash{\widetilde{\B}} _{\PP} ^{(m)} (*T )$
définie par la formule $\phi (P) = P \cdot (\frac{1}{t _1})$. 
Notons $\I $ l'idéal à gauche de 
$\smash{\widetilde{\D}} _{\PP } ^{(m)} $ engendré par $\partial _1 t _1$ et $\partial _2, \dots, \partial _d$
et
$\FF ^{(m)}:= (\smash{\widetilde{\D}} _{\PP } ^{(m)} / \I ) / (\smash{\widetilde{\D}} _{\PP } ^{(m)} / \I ) _{\text{p-tors}}$.
Comme les éléments de $\I$ agissent de manière triviale sur $\frac{1}{t _1}$, comme 
$\smash{\widetilde{\B}} _{\PP} ^{(m)} (*T )$ est sans $p$-torsion, le morphisme $\phi$ se factorise
en la flèche
$\FF ^{(m)} \to \smash{\widetilde{\B}} _{\PP} ^{(m)} (*T )$.
On remarque que comme $\FF ^{(m)} _\Q = \smash{\widetilde{\D}} _{\PP,\Q} ^{(m)} / \I _\Q \subset \smash{\widetilde{\B}} _{\PP } ^{(m)} (T ) _\Q$
(voir \cite[3.2.1.(i) et 4.2.2]{Be0}), 
alors 
$\FF ^{(m)} \subset \smash{\widetilde{\B}} _{\PP} ^{(m)} (*T )$.
On peut donc identifier les éléments de $\FF ^{(m)} $ avec les éléments $f$ de $\O _{\PP} (\hdag T ) _{\Q}$ tel qu'il existe 
$P \in \smash{\widetilde{\D}} _{\PP } ^{(m)}$ tel que $f= P \cdot (\frac{1}{t _1})$.

Vérifions à présent que 
$\smash{\widetilde{\B}} _{\PP} ^{(m)} (T ) \subset \FF ^{(m+1)}$.
Soit $f \in \smash{\widetilde{\B}} _{\PP} ^{(m)} (T ) $.
On a $f = \sum _{k \in \N} a _{k} \frac{p ^k}{t _1 ^{kp ^{m+1}}}$
où $a _k \in \O _{\PP}$ converge vers $0$ pour la topologie $p$-adique lorsque $k$ tend vers l'infini.
Soit $P = \sum _{k \geq 0} (-1) ^{kp ^{m+1} -1} a _k p ^{k} \partial _1 ^{[kp ^{m+1} -1]} \in \smash{\D} ^\dag _{\PP }$. 
On calcule alors dans $\O _{\PP} (\hdag T ) _{\Q}$ la formule $f= P \cdot (\frac{1}{t _1})$. Il suffit alors de vérifier que l'on a en fait 
$P \in \smash{\widetilde{\D}} _{\PP } ^{(m+1)}$.
Posons $kp ^{m+1} -1 = p ^{m+1} q _k + r _k$ avec $0 \leq r _k < p ^{m+1}$.
On a alors 
$\partial _1 ^{[kp ^{m+1} -1]}  = \frac{1}{ q _k !} \partial _1 ^{<kp ^{m+1} -1> _{(m+1)}}$ (voir la formule \cite[2.2.3.2]{Be1}).
Donc 
$P = \sum _{k \geq 0} (-1) ^{kp ^{m+1} -1} a _k  \frac{p ^{k}}{ q _k !} \partial _1 ^{<kp ^{m+1} -1> _{(m+1)}}$.
D'après \cite[2.4.3.1]{Be1} (cette formule est utilisée au niveau $m+1$), 
$v _p (q _k !) \leq \frac{kp ^{m+1} -1 }{p ^{m+1}(p-1)}  \leq \frac{k}{(p-1)}\leq k$.
D'où $P \in \smash{\widetilde{\D}} _{\PP } ^{(m+1)}$.

On en déduit que 
$\smash{\widetilde{\B}} _{\PP} ^{(\bullet)} (T )$
et 
$\FF ^{(\bullet)}$
sont isomorphes dans
$\underrightarrow{LM}  _{\Q} (\smash{\widehat{\D}} _{\PP } ^{(\bullet)})$.

\end{rema}

\bibliographystyle{smfalpha}
\newcommand{\etalchar}[1]{$^{#1}$}
\def\cprime{$'$}
\providecommand{\bysame}{\leavevmode ---\ }
\providecommand{\og}{``}
\providecommand{\fg}{''}
\providecommand{\smfandname}{et}
\providecommand{\smfedsname}{\'eds.}
\providecommand{\smfedname}{\'ed.}
\providecommand{\smfmastersthesisname}{M\'emoire}
\providecommand{\smfphdthesisname}{Th\`ese}

\bigskip
\noindent Daniel Caro\\
Laboratoire de Mathématiques Nicolas Oresme\\
Université de Caen
Campus 2\\
14032 Caen Cedex\\
France.\\
email: daniel.caro@unicaen.fr

\end{document}